\documentclass[english]{article}
\usepackage{preamble}

\begin{document}

\title{The Chromatic Fourier Transform}
\maketitle

\begin{abstract}
    We develop a general theory of higher semiadditive Fourier transforms that includes both the classical discrete Fourier transform for finite abelian groups at height $n=0$, as well as a certain duality for the $E_n$-(co)homology of $\pi$-finite spectra, established by Hopkins and Lurie, at heights $n\ge 1$. We use this theory to generalize said duality in three different directions. First, we extend it from $\ZZ$-module spectra to all (suitably finite) spectra and use it to compute the discrepancy spectrum of $E_n$. Second, we lift it to the telescopic setting by replacing $E_n$ with $T(n)$-local higher cyclotomic extensions, from which we deduce various results on affineness, Eilenberg--Moore formulas and Galois extensions in the telescopic setting. Third, we categorify their result into an equivalence of two symmetric monoidal $\infty$-categories of local systems of $K(n)$-local $E_n$-modules, and relate it to (semiadditive) redshift phenomena. 
\end{abstract}

\begin{figure}[H]
    \centering{}
    \includegraphics[scale=0.185]{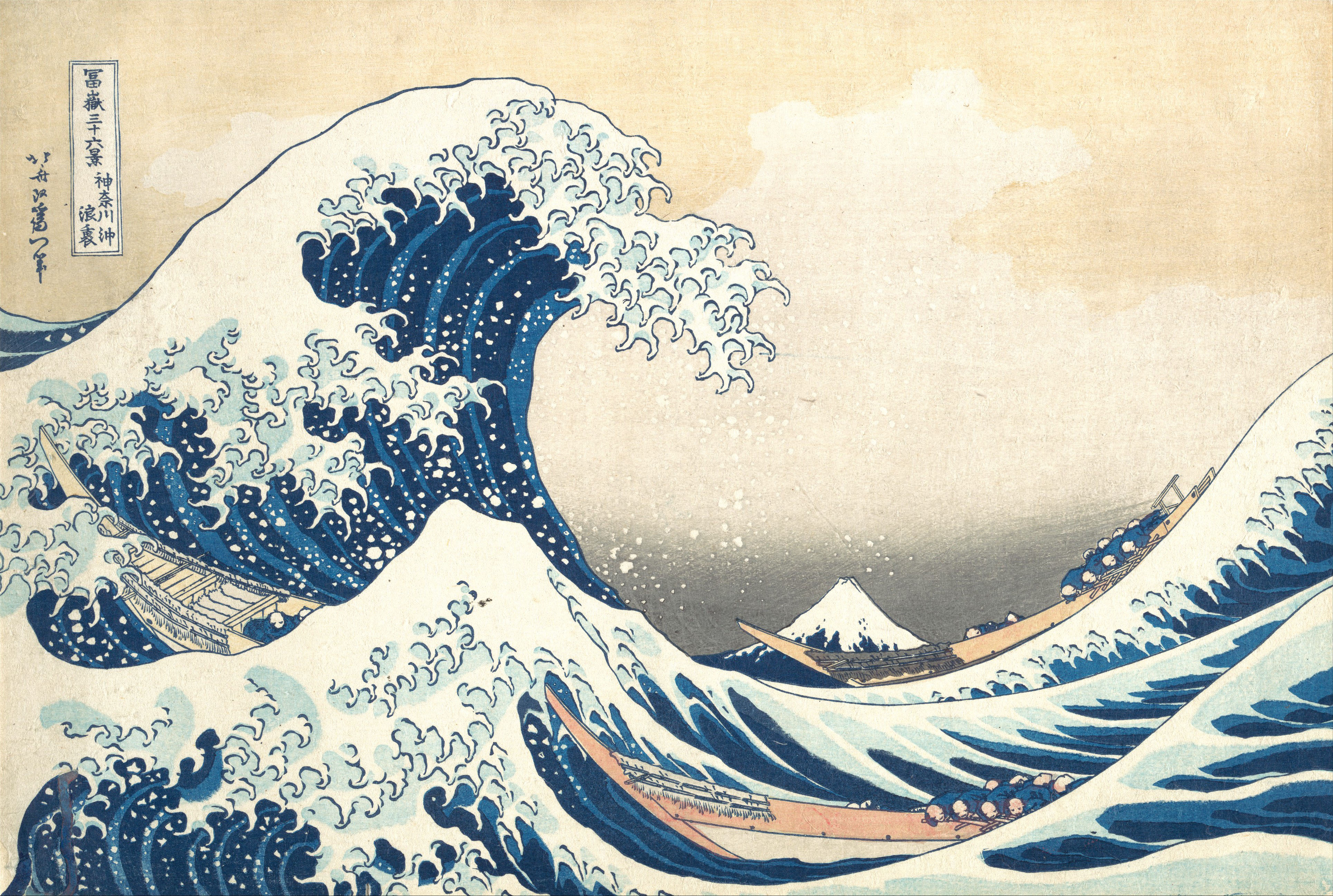}
    \caption*{The Great Wave off Kanagawa, Katsushika Hokusai.}
\end{figure}

\newpage
\tableofcontents{}

\section{Introduction}\label{sec:intro}

\subsubsection{Background \& overview}

The classical $m$-dimensional Discrete Fourier Transform (DFT) is a linear isomorphism
\[
    \Four_\omega \colon \CC^m \iso \CC^m,
\]
associated to a primitive $m$-th root of unity $\omega \in \CC$, whose characteristic property is transforming the convolution product on the source to the pointwise product on the target. More generally, one can associate to every commutative ring $R$ with an $m$-th root of unity $\omega \colon \ZZ/m \to R^\times$, a natural transformation of $R$-algebras,
\[
    \Four_\omega \colon R[M] \too R^{M^*},
\]
from the group $R$-algebra of an $m$-torsion abelian group $M$ to the algebra of $R$-valued functions on its Pontryagin dual $M^* = \hom(M, \QQ/\ZZ)$. Furthermore, $\Four_\omega$ is an isomorphism if and only if the image of $\omega$ is primitive in every residue field of $R$. The classical case is recovered by taking $R=\CC$ and $M = \ZZ/m$.

Passing from the ordinary category of abelian groups to the $\infty$-category of spectra, i.e., from classical commutative algebra to stable homotopy theory, introduces new ``characteristics''. The Morava $K$-theory ring spectra of heights $n=0,\dots,\infty$ at an (implicit) prime $p$,
\[
    \QQ = K(0) \:\: , \:\: K(1) \:\: , \:\: K(2) \:\: , \:\: \dots \:\:, \:\: K(n) \:\: , \:\: \dots \:\:, \:\: K(\infty) = \FF_p, 
\]
are in a precise sense the prime fields in the world of spectra, and can be thought of as providing an interpolation between the classical characteristics $0$ and $p$; see \cite{nilp2}. 
A central tool in the study of these intermediate characteristics is  Lubin--Tate spectra. For each $0<n<\infty$, this is a $K(n)$-local commutative algebra $E_n$ that can be realized as the algebraic closure of the $K(n)$-local sphere, and which has deep connections to the algebraic geometry of formal groups making it amenable to computations. For example, in \cite{AmbiKn}, Hopkins and Lurie prove the following theorem, which resembles the discrete Fourier transform, only in higher chromatic heights:

\begin{thm}[{{\cite[Corollary 5.3.26]{AmbiKn}}}] 
\label{HL_orientation_Intro}
    For all integers $n\ge1$, there is a natural isomorphism of $K(n)$-local commutative $E_n$-algebras
    \[
        E_n[M] \iso 
        E_n^{\Omega^{\infty - n} M^*},
    \]
    where $M$ is a connective $\pi$-finite (i.e., having only finitely many non-vanishing homotopy groups, all of which are finite)  $p$-local $\ZZ$-module and
    $M^*$ is its Pontryagin dual.
\end{thm}

Furthermore, they deduce from this result several fundamental structural properties of local systems of $K(n)$-local algebras on $\pi$-finite spaces, reproving among other things the convergence of the $K(n)$-based Eilenberg--Moore spectral sequence from \cite{bauer2008convergence}.

In this paper, we develop a general theory that formalizes and substantiates the analogy between \Cref{HL_orientation_Intro} and the classical Fourier transform. In particular, we reinterpret both in terms of a unified notion of a \textit{chromatic Fourier transform} isomorphism for all finite chromatic heights, and show that it shares many of the formal properties of the classical Fourier transform. We then apply this theory to generalize \Cref{HL_orientation_Intro} in three different directions:

\begin{enumerate}
    \item We \textit{lift} it to the telescopic world, by replacing $E_n$ with certain faithful Galois extensions of the $T(n)$-local sphere (\Cref{Tn_orientation_Intro}). By analogy with the $K(n)$-local case, we deduce several structural results for local systems of $T(n)$-local algebras over $\pi$-finite spaces (\Cref{Tn_Applications_Intro}). We also obtain an analogue of Kummer theory at heights $n\ge 1$ (\Cref{Higher_Kummer_Intro}).
    
    \item We \textit{extend} it over $E_n$ to all (i.e., not just $\ZZ$-module) connective $\pi$-finite $p$-local spectra  (\cref{Sphere_Or_Intro}), and deduce from this the conjectured description of the discrepancy spectrum of Ando--Hopkins--Rezk in terms of the Brown--Comenetz spectrum (\Cref{Discrepancy_Intro}). As another application, we construct a certain $K(n)$-local pro-$\pi$-finite Galois extension, which is a strong analogue of the classical $p$-typical cyclotomic extension (\Cref{Sph_Gal_Intro}).
    
    \item We \textit{categorify} it into a symmetric monoidal equivalence between  $\infty$-categories of local systems of $K(n)$-local $E_n$-modules on the underlying spaces of two dual $\pi$-finite spectra. Among other things, this generalizes the weight space decomposition of representations of finite abelian groups in characteristic zero (\Cref{E_n_Categorification_Intro}). We also explain how this categorification accords with semiadditive redshift phenomena.
\end{enumerate}

We shall now discuss each of these sets of results in some more detail, and outline along the way some of the key aspects of the general theory. 

\subsubsection{Telescopic lift}
  
Recall that the \textit{telescopic localization} $\Sp_{T(n)}$ is the Bousfield localization of $\Sp$ with respect to $T(n) = F(n)[v^{-1}]$, where $F(n)$ is (any) finite spectrum of type $n$ with a $v_n$-self map of the form $v\colon \Sigma^d F(n) \to F(n)$. 
It is a classical fact that $\Sp_{K(n)} \sseq \Sp_{T(n)}$, and a long standing conjecture of Ravenel, known as the telescope conjecture, states that the two localizations are in fact \textit{equal}. While proven to be true in heights $n=0,1$, the telescope conjecture is widely believed to be \textit{false} for all $n\ge 2$ and all primes $p$. In recent years, the telescopic localizations gained new interest (independently of the status of the telescope conjecture) due to their pivotal role in several remarkable developments, of which we mention two. First, the work \cite{heuts2021lie} of Heuts on unstable chromatic homotopy theory, which generalizes Quillen's classical rational homotopy theory to higher chromatic heights. And second, the works \cite{land2020purity, clausen2020descent}, which made a major progress on establishing the conjectural chromatic redshift philosophy pioneered by Rognes (see, e.g., \cite{Rognes_Redshift}). 

The $T(n)$-localizations are considerably less amenable to computations than the corresponding $K(n)$-localizations, largely due to the lack of a (faithful) telescopic lift of $E_n$. Nevertheless, we show that the isomorphism of \cref{HL_orientation_Intro} descends from $E_n$ to a deeper base, which does admit a faithful telescopic lift and over to which the chromatic Fourier transform lifts as well. To explain this in more detail, we first note that while the classical Fourier transform is not defined over $\QQ$, one does not need to go all the way up to $\CC$ or even $\cl{\QQ}$. Instead, for $\ZZ/m$-modules, it suffices to have a primitive $m$-th root of unity $\omega_m$, so one can construct the Fourier transform already over the cyclotomic field $\QQ(\omega_m)$, which is a finite Galois extension of $\QQ$. 
In the same spirit, we observe that natural transformations as in  \cref{HL_orientation_Intro} are in a canonical bijection with \textit{higher roots of unity} $\Sigma^n \ZZ/p^r \to E_n^\times$ of $E_n$. Moreover, the natural \textit{isomorphisms} are in a canonical bijection with those higher roots of unity that are \textit{primitive} in the sense of \cite[Definition 4.2]{carmeli2021chromatic}. 

\begin{rem}
    In \cite{AmbiKn}, the isomorphism of \Cref{HL_orientation_Intro} is constructed from a \textit{normalization} of the $p$-divisible group $\GG$ associated with $E_n$, namely, an isomorphism of the top alternating power $\Alt_n (\GG)$ with the constant $p$-divisible group $\QQ_p/\ZZ_p$. It can be verified directly, that such data are equivalent to compatible systems of primitive higher roots of unity of $E_n$,
    \[
        \Sigma^n\QQ_p/\ZZ_p \simeq 
        \colim \Sigma^n\ZZ/p^r \too 
        E_n^\times.
    \]
\end{rem}

We then proceed to show that, as in the classical case, the chromatic Fourier isomorphism exists already over the \textit{higher cyclotomic extensions} $R_{n,r}$, which are certain faithful $(\ZZ/p^r)^\times$-Galois extensions of the $K(n)$-local sphere classifying primitive \textit{higher roots of unity} in the sense of \cite{carmeli2021chromatic}. 
The key point now is that, by \cite[Theorem A]{carmeli2021chromatic}, the $R_{n,r}$-s admit faithful $T(n)$-local lifts $R^f_{n,r}$, the corresponding $T(n)$-local higher cyclotomic extensions. Consequently, the general theory developed in this paper, combined with the nilpotence theorem, allows us to lift the chromatic Fourier transform to the telescopic world. 

\begin{theorem}[\ref{Tn_Fourier}]
    \label{Tn_orientation_Intro}
    For every $n,r\ge 1$, there exists a faithful $(\ZZ/p^r)^\times$-Galois extension $R^f_{n,r}$ of the $T(n)$-local sphere and a natural isomorphism of $T(n)$-local commutative $R^f_{n,r}$-algebras
    \[
        \Four_{\omega \colon 
        }R_{n,r}^f[M] \iso 
        (R_{n,r}^f)^{\Omega^{\infty - n} M^*},
    \]
    where $M$ is a connective $\pi$-finite $\ZZ/p^r$-module and
    $M^*$ is its Pontryagin dual.
\end{theorem}

The natural isomorphisms of \Cref{Tn_orientation_Intro} are compatible with varying $r$. Thus, if we replace the individual $R^f_{n,r}$-s with the colimit $R_n^f := \colim R_{n,r}^f$,
we obtain a telescopic Fourier transform for all connective $\pi$-finite $\ZZ_{(p)}$-module (or equivalently, $p$-local $\ZZ$-module) spectra as in \Cref{HL_orientation_Intro}. The commutative ring spectrum $R_n^f$ can be viewed as the infinite $p$-typical higher cyclotomic extension and is a telescopic lift of Westerland's $K(n)$-local commutative ring spectrum $R_n$ (see \cite{Westerland}). However, in contrast with $R_n$, it is not known whether $R_n^f$ is \textit{faithful}. This subtle point might also shed some new light on (the failure of) the telescope conjecture. Localizing $\Sp_{T(n)}$ with respect to $R_n^f$ forms an interesting intermediate localization between $\Sp_{K(n)}$ and $\Sp_{T(n)}$.
In particular, if one speculates that $R_n^f$ is in fact $K(n)$-local, the telescope conjecture becomes equivalent to the faithfulness of $R_n^f$.

As in \cite{AmbiKn}, we deduce from \Cref{Tn_orientation_Intro} several structural properties of local systems of $T(n)$-local algebras over $\pi$-finite spaces.

\begin{theorem}
    \label{Tn_Applications_Intro}
    Let $A$ be a $\pi$-finite space such that $\pi_1(A,a)$ is a $p$-group and $\pi_{n+1}(A,a)$ is of order prime to $p$, for all $a\in A$.
    \begin{enumerate}
        \item (\ref{Tn_Affineness}) For every $R \in \alg(\Sp_{T(n)})^A$, the \textit{global sections} functor induces a symmetric monoidal equivalence
        \[
            \Mod_R(\Sp_{T(n)})^A \iso \Mod_{A_*R}(\Sp_{T(n)}).
        \]
        
        \item (\ref{Tn_EM}) For every $R \in \alg(\Sp_{T(n)})$ and spaces $B$ and $C$ mapping to $A$, one of which is $\pi$-finite, the canonical \textit{Eilenberg--Moore} map is an isomorphism:
        \[
            R^B \otimes_{R^A} R^C \iso R^{B\times_A C}.
        \]
        
        \item (\ref{Tn_Galois}) Assuming $A$ is connected, every $R \in \calg(\Sp_{T(n)})^A$ is \textit{$\Omega A$-Galois}, in the sense of Rognes, over the global sections (i.e., $\Omega A$-homotopy fixed points) algebra $A_*R \in \calg(\Sp_{T(n)})$. 
    \end{enumerate}
\end{theorem}

\begin{rem}
    The assumptions on $\pi_1$ and $\pi_{n+1}$ are necessary. The spaces $B^{n+1}C_p$ and $ BC_q$ for $q \neq p$ are counterexamples to all three claims. On the other hand, for $n\ge1$, \Cref{Tn_Applications_Intro} is interesting already for $A = BG$, where $G$ is a finite $p$-group.  
\end{rem}

We think of the first claim as an \textit{affineness} property (cf. \cite{mathew2015affineness}), and study it in a general context by abstracting the $K(n)$-local arguments of \cite[\S 5.4]{AmbiKn}. In particular, in an ambidextrous setting, it both implies and is implied by (a very special case of) the second claim regarding the Eilenberg--Moore isomorphism. Roughly speaking, the affineness of $A$ reduces the Eilenberg--Moore isomorphism to the \textit{K\"{u}nneth isomorphism} for the fibers of $B$ and $C$ over $A$, which in an ambidextrous setting is guaranteed by the $\pi$-finiteness assumptions. The suitable special case of the Eilenberg--Moore isomorphism is then deduced from the chromatic Fourier transform, replacing ``cohomology'' with ``homology'', for which the analogous claim holds for formal reasons. The third claim, which we learned from Dustin Clausen, that the Galois condition is automatically satisfied, is also a rather formal consequence of affineness, and might have been known to other experts (in the $K(n)$-local case), though we are not aware that it has appeared in the literature. 

As a further corollary of the above, we obtain an analogue of \textit{Kummer theory} at higher chromatic heights $n\ge1$. 
\begin{theorem} (\ref{Tn_Higher_Kummer})
    \label{Higher_Kummer_Intro}
    Let $R$ be a $T(n)$-local commutative algebra admitting a primitive higher $p^r$-th root of unity, and let $M$ be a connected $n$-finite $\ZZ/p^r$-module spectrum. Then,
    \[
        \{\text{\,$\Omega M$-Galois extensions of $R$\,}\} \:\simeq\:
        \Map(\Sigma^nM^*, R^\times).
    \]
\end{theorem}

That is, we obtain a classification of abelian Galois extensions of a commutative algebra in terms of its units, in the presence of enough primitive (higher) roots of unity.\footnote{For a more precise relation to classical Kummer theory in height $n=0$, see \Cref{Kummer_Lower}.} 

\subsubsection{Spherical orientations}

Generalizing \Cref{HL_orientation_Intro} in another direction, we show that when one does work $K(n)$-locally and over $E_n$, the chromatic Fourier transform extends to \textit{all} connective $\pi$-finite $p$-local spectra (i.e., not just $\ZZ$-modules), provided that we replace \textit{Pontryagin duality} by \textit{Brown--Comenetz duality}.

\begin{theorem}[\ref{Sphere_or_E_n}]
    \label{Sphere_Or_Intro}
    There is a natural isomorphism of $K(n)$-local commutative $E_n$-algebras
    \[
        E_n[M] \iso 
        E_n^{\Omega^{\infty - n} I M},
    \]
    where $M$ is a connective $\pi$-finite $p$-local spectrum and
    $I M$ is its Brown--Comenetz dual.
\end{theorem}

\Cref{Sphere_Or_Intro} generalizes \Cref{HL_orientation_Intro}, as for $\ZZ$-module spectra, the Brown--Comenetz dual identifies canonically with the Pontryagin dual. Conversely, it can be obtained from \Cref{HL_orientation_Intro} by a bootstrap procedure. To begin with, for a $K(n)$-local commutative algebra $R$, a compatible system of primitive higher $p^r$-th roots of unity defines a map
\[
    \Sigma^n \QQ_p/\ZZ_p \simeq
    \colim \Sigma^n \ZZ/p^r \too
    R^\times. 
\]
Denoting by $I_{\QQ_p/\ZZ_p}$ the Brown--Comenetz spectrum, we show that extensions of the chromatic Fourier transformation over $R$ to non-$\ZZ$-module spectra are in a natural bijection with solutions to the following extension problem: 

\[\begin{tikzcd}
	{\Sigma^n\QQ_p/\ZZ_p} && {\Sigma^nI_{\QQ_p/\ZZ_p}} \\
	& {R^\times.}
	\arrow[from=1-1, to=1-3]
	\arrow[from=1-1, to=2-2]
	\arrow[dashed, from=1-3, to=2-2]
\end{tikzcd}\]
We call such extensions \textit{spherical orientations} of $R$, and think of them as spherical  analogues of (compatible systems of) primitive higher roots of unity of $R$. Using a devissage argument, we show that the Fourier transform associated to a spherical orientation is an isomorphism for all connective $\pi$-finite $p$-local spectra. By studying the obstructions for extending higher roots of unity to spherical orientations, we construct a universal $R$ with a spherical orientation, the \textit{spherical cyclotomic extension}, and prove that it is $K(n)$-locally faithful. In fact, we even construct it $T(n)$-locally, and show that it is faithful over the intermediate localization,
\[
    \Sp_{K(n)} \:\sseq\: 
    (\Sp_{T(n)})_{R_n^f} \:\sseq\: 
    \Sp_{T(n)}.
\]

We further introduce a natural \textit{higher connectedness} property for commutative algebras, which guarantees the canonical vanishing of said obstructions, providing a practical criterion for extending the chromatic Fourier transform to non-$\ZZ$-module spectra for specific choices of $R$. As a consequence of the ``chromatic nullstellensatz'' of the third author with Burklund and Yuan \cite{Null}, this criterion is satisfied for $R = E_n$, thus yielding \Cref{Sphere_Or_Intro}.

\begin{rem}\label{HL_Trunc_Or_Intro}
    More specifically, the required ingredient is \cite[Proposition 8.14]{Null},
    \[
        \Map_\Sp(C_p , \pic(E_n)) \:\simeq\: B^{n+1} C_p. 
    \]
    Looping this isomorphism once yields
    \[
        \Map_\Sp(C_p , E_n^\times) \:\simeq\: B^{n} C_p,
    \]
    which was conjectured, and subsequently proven, by Hopkins and Lurie (see \cite[Conjecture 5.4.14]{AmbiKn}), though we are not aware of a written account of their proof. We note that the second isomorphism suffices for establishing \Cref{Sphere_Or_Intro} for all connective $\pi$-finite $p$-local spectra whose $n$-truncation admits a module structure over the $(n-1)$-truncated sphere spectrum.
\end{rem}

As a further consequence of \Cref{Sphere_Or_Intro}, we deduce that a spherical orientation of $E_n$ identifies the connective cover of $\Sigma^nI_{\QQ_p/\ZZ_p}$ with the universal right approximation of $E_n^\times$ by an \textit{ind-$\pi$-finite} connective $p$-local spectrum. We denote the latter by $\mu_{\Sph_{(p)}}(E_n)$ as a ``spherical'' analogue of the spectrum of ordinary $p$-typical roots of unity of $E_n$. We then reinterpret $\mu_{\Sph_{(p)}}(E_n)$ as the $p$-local part of the connective cover of the so-called \textit{discrepancy spectrum} of $E_n$, which was defined by Ando, Hopkins and Rezk to be the fiber of the localization map $E_n^\times \to L_n E_n^\times$. To summarize, we obtain the following result, originally announced by Hopkins and Lurie:

\begin{theorem} 
    \label{Discrepancy_Intro}
    (\ref{discrepency_fiber}) The $p$-localization of the connective cover of the discrepancy spectrum of $E_n$ is isomorphic to the connective cover of $\Sigma^n I_{\QQ_p/\ZZ_p}$.
\end{theorem}

As the discrepancy spectrum is constructed from $K(n)$-local ingredients, it is somewhat remarkable that one can read off of it the first $n$ stable homotopy groups of spheres, which are a ``global'' invariant. This relation also ties the \textit{chromatic} filtration with the \textit{Postnikov} filtration, which are generally speaking two quite ``orthogonal'' filtrations in stable homotopy theory. 

To systematize the study of the various flavours of the chromatic Fourier transform, and to facilitate d\'{e}vissage arguments, we introduce a general notion of an \textit{$\OR$-orientation} of height $n$ for every connective $p$-local commutative ring spectrum $\OR$. This is the type of data from which one can construct the chromatic Fourier transform for connective $\pi$-finite $\OR$-module spectra. The case $\OR = \ZZ/p^r$ recovers primitive height $n$ roots of unity, while the case $\OR = \Sph_{(p)}$ recovers spherical orientations. Furthermore, we show that if $\OR$ is itself $n$-truncated and $\pi$-finite, then the associated $\OR$-cyclotomic extension, classifying $\OR$-orientations, is $\OR^\times$-Galois. Note that in general, $\OR^\times$ is not a finite discrete group, but a \textit{$\pi$-finite} group in the homotopical sense. We deduce from this that the spherical cyclotomic extension is a \textit{pro-$\pi$-finite} Galois extension.

\begin{theorem}[\ref{Kn_Pro_Galois}]\label{Sph_Gal_Intro}
    The $K(n)$-local spherical cyclotomic extension is pro-$G$-Galois for $G = \tau_{\le n}\Sph_p^\times$, viewed as a pro-$\pi$-finite group. 
\end{theorem}

The classical $p$-typical cyclotomic extension $\QQ_p(\omega_{p^\infty})$, and the corresponding $p$-typical cyclotomic character classifying it
\[
    \chi \colon \Gal(\cl{\QQ}_p/\QQ_p) \too \ZZ_p^\times  = \tau_{\le 0}\Sph_p^\times,
\]
plays a fundamental role in number theory, in the formulation of various arithmetic duality theorems, via the construction of Tate twists. In \cite{carmeli2021chromatic}, it is shown that Westerland's ring spectrum $R_n$, which is a $\ZZ_p^\times$-Galois extension of $\Sph_{K(n)}$, can be similarly viewed as a $p$-typical \textit{higher cyclotomic extension}. As the pro-finite Galois extensions of $\Sph_{K(n)}$ are classified by the Morava stabilizer group $\GG_n$, associated to $R_n$ is the \textit{higher cyclotomic character}
\[
    \chi \colon \GG_n := \Gal(E_n/\Sph_{K(n)}) \too \ZZ_p^\times,
\] 
which is essentially the determinant map. This higher cyclotomic character plays a similarly fundamental role in chromatic homotopy theory, via the construction of the determinant sphere featuring in Gross-Hopkins duality. \Cref{Sph_Gal_Intro}, together with the theory developed in this paper, suggests that the spherical cyclotomic extension should assume a similarly pivotal role in $K(n)$-local \textit{higher Galois theory}, which deals with pro-\textit{$\pi$-finite} Galois extensions of $\Sph_{K(n)}$. A more systematic account of this circle of ideas and their applications is a subject for a future work.  

\subsubsection{Categorification}

Finally, we also extend \Cref{HL_orientation_Intro} by way of \textit{categorification}.
The key feature of $\Sp_{K(n)}$, that allows the chromatic Fourier transform to be an isomorphism, is \textit{higher semiadditivity} in the sense of \cite[Definition 4.4.2]{AmbiKn}. Moreover, the chromatic height $n$, which appears as the ``shift'' in the chromatic Fourier transform, can be interpreted as the \textit{semiadditive height} of $\Sp_{K(n)}$ in the sense of \cite{AmbiHeight}. 
Similarly, our telescopic lift of \Cref{HL_orientation_Intro} relies on the higher semiadditivity of $\Sp_{T(n)}$, which is the maximal higher semiadditive localization of $\Sp$ of height $n$ (see \cite{TeleAmbi}). We therefore construct and study the higher Fourier natural transformation in the general setting of higher semiadditive symmetric monoidal $\infty$-categories of a given semiadditive height $n$ (at an implicit prime $p$). However, it is not easy to determine, in this abstract setting, when the Fourier transform is an isomorphism, and a large portion of this paper is devoted to developing tools for answering this question.  

One interesting source of examples of higher semiadditive $\infty$-categories, outside of the stable realm, is higher category theory itself. As already observed in \cite{AmbiKn}, the $\infty$-category $\Prl$ of presentable $\infty$-categories is $\infty$-semiadditive. More generally, for any presentably symmetric monoidal $\infty$-category $\cC$, the $\infty$-category $\Mod_\cC$ of $\cC$-linear presentable $\infty$-categories is $\infty$-semiadditive. 
Regarding roots of unity, a height $n$ root of unity $\omega\colon \Sigma^n \ZZ/m \to R^\times$ of a commutative algebra $R$ in $\cC$ deloops uniquely to a height $n+1$ root of unity 
\[ 
    \cl{\omega}\colon
    \Sigma^{n+1} \ZZ/m \too 
    \pic(R) := 
    \Mod_R^\times
\] 
of $\Mod_R$, viewed as a commutative algebra in $\Mod_\cC$. Note that the ``shift by one'' in the height is consistent with the \textit{semiadditive redshift} phenomenon from \cite{AmbiHeight}. Namely, if $\cC$ happens to be itself $\infty$-semiadditive and $R$ is of semiadditive height $n$ in $\cC$, then $\Mod_R$ is of semiadditive height $n+1$ in $\Mod_\cC$. As the semiadditive height generalizes the chromatic height, this is strongly related to the \textit{chromatic redshift} philosophy. Now, given $\omega$ as above, our theory provides a Fourier transform of commutative $R$-algebras in $\cC$
\[
    \Four_{\omega}\colon
    R[M] \too
    R^{\Omega^{\infty-n}M^*},
\]
and also a ``categorified'' Fourier transform of symmetric monoidal $\Mod_R$-linear presentable $\infty$-categories
\[
    \Four_{\cl{\omega}}\colon
    \Mod_R[M] \too 
    \Mod_R^{\Omega^{\infty-(n+1)}M^*}.
\]
Unpacking the definitions, the right hand side is the $\infty$-category of $\Mod_R$-valued local systems on $\Omega^{\infty-(n+1)}M^*$ with the \textit{pointwise} symmetric monoidal structure, while the source is the $\infty$-category of $\Mod_R$-valued local systems on $\Omega^\infty M$ with the \textit{Day convolution} symmetric monoidal structure. 
Our main result regarding this situation is that $\Four_{\cl{\omega}}$ is an isomorphism if and only if $\Four_{\omega}$ is an isomorphism. 

Categorification of spherical orientations is more subtle. In general, there can be an obstruction for delooping a height $n$ spherical orientation  $\Sigma^nI_{\QQ_p/\ZZ_p} \to R^\times$ of $R$ to a height $n+1$ spherical orientation $\Sigma^{n+1}I_{\QQ_p/\ZZ_p} \to \pic(R)$ of $\Mod_R$. However, it does give an $\OR$-orientation for the truncated $p$-local sphere spectrum $\OR = \tau_{\le n}\Sph_{(p)}$. In the case of the $\infty$-category $\LocMod_{E_n}$ of $K(n)$-local $E_n$-modules, this partial result reads as follows:

\begin{theorem}[\ref{Ninga_Orientation}]
    \label{E_n_Categorification_Intro}
    For every $n\ge0$, there is a natural equivalence of symmetric monoidal $\infty$-categories:
    \[
        \Fun(\Omega^\infty M, \LocMod_{E_n})_\Day \iso 
        \Fun(\Omega^{\infty - (n+1) } IM, \LocMod_{E_n})_\Ptw,
    \]
    for $M$ a connective $\pi$-finite $p$-local spectrum, assuming the vanishing of the canonical map 
    \[
        \pi_{n+1}\Sph \otimes \pi_0 M \too \pi_{n+1} M.
    \] 
\end{theorem}

\begin{rem}
    For a connective $M$, the induced isomorphism on the endomorphism objects of the respective monoidal units in the equivalence of \Cref{E_n_Categorification_Intro} recovers \Cref{Sphere_Or_Intro}.
\end{rem}

The technical assumption on $M$ is equivalent to the $(n+1)$-truncation of $M$ having a (necessarily unique) module structure over the $n$-truncated sphere spectrum. This happens, for example, if $M$ admits a $\ZZ$-module structure, as in the original statement of \Cref{HL_orientation_Intro}, in which case the equivalence exists already over $R_n$ (or even $R_n^f$). We conjecture however that this hypothesis is  unnecessary, namely, that the spherical orientation of $E_n$ can be delooped to a spherical orientation of $\LocMod_{E_n}$. This would happen if, for example, 
\[
    \Map(C_p, \mathrm{br}(E_n)) \simeq B^{n+2}C_p,
\]
where $\mathrm{br}$ stands for the \textit{Brauer spectrum} (cf. \Cref{HL_Trunc_Or_Intro}).

The height $n=0$ case of \Cref{E_n_Categorification_Intro} recovers a classical fact. 
For a finite abelian group  $A$ and $M = A^*$, we get $\Omega^\infty M = A^*$ and $\Omega^{\infty - 1} IM = BA$. We therefore obtain an equivalence between the symmetric monoidal (derived) categories of $\cl{\QQ}$-representations of $A$, and of $A^*$-graded $\cl{\QQ}$-vector spaces. This equivalence is provided, unsurprisingly, by decomposition into weight spaces.     
Assuming $A$ is a $p$-group, for a general height $n$ and $M = A^*$, we get a similar ``weight space decomposition'' of the symmetric monoidal $\infty$-category of $\LocMod_{E_n}$-representations of the higher group $G = B^nA$,
\[
    \Fun(B^{n+1}A,\LocMod_{E_n}) \simeq
    \prod_{\varphi \in A^*} \LocMod_{E_n}.
\]
For an example of a different flavor, in height $n=1$ with $M = \Sigma A$, we get an equivalence between $\LocMod_{E_1}$-representations of $A$ and of $A^*$, but with different symmetric monoidal structures. A similar (non-monoidal) equivalence, was considered by Treumann in \cite{T2015Representations}.

\subsubsection{Organization}

In \cref{sec:affine}, we study the notion of \textit{affineness}. 
We begin in \ref{ssec:affinefunc}, by introducing it in the general setting of monoidal functors. We establish equivalent characterisations for affine functors and certain closure and monoidal properties thereof. 
Then, in \ref{ssec:affinelocsys}, we specialize to functors arising as pullbacks for local systems, and relate the property of affineness to the behavior of Eilenberg--Moore maps and Galois extensions. 
We end this section, in \ref{ssec:affineambi}, by studying the interaction of affineness with ambidexterity and semiadditive height, establishing among other things the mutual implications of affineness and the Eilenberg--Moore isomorphism. 

In \cref{sec:higherfourier}, we lay down the foundations of the abstract theory of \textit{Fourier transforms}. 
In \ref{ssec:pre-or}, we define \textit{$\OR$-pre-orientations} of height $n$ for a connective $p$-local commutative ring spectrum $\OR$, and the corresponding notion of Brown--Comenetz duality for $\OR$-modules.  
In \ref{ssec:fourier}, we proceed to construct the Fourier (not necessarily invertible) transformation associated to an $\OR$-pre-orientation and discuss its functoriality and duality invariance. 
We conclude, in \ref{ssec:co-mult}, by promoting the Fourier transform to a map of Hopf algebras, and deducing an analogue of the classical translation invariance property.

In \cref{sec:orientations}, we study \textit{$\OR$-orientations}, which are $\OR$-pre-orientations such that the associated Fourier transform is an isomorphism for all suitably finite $\OR$-modules. 
In \ref{ssec:orient}, after presenting the relevant definitions and functorialities, we establish for a given $\OR$-pre-orientation various closure properties of the class of \textit{oriented modules} (those for which the Fourier transform is an isomorphism). 
Next, in \ref{ssec:cyclext}, we introduce the universal $\OR$-oriented algebras, the \textit{$\OR$-cyclotomic extensions}, and show that they are $\OR^\times$-Galois under certain finiteness hypotheses. 
Then, in \ref{ssec:affor}, we discuss \textit{virtual orientability}, which is the property of admitting an $\OR$-orientation after a faithful extension of scalars, or equivalently, that the $\OR$-cyclotomic extension is faithful. We show that under the assumption of virtual $\OR$-orientability, the underlying space of a suitably finite $\OR$-module is affine and deduce a general form of higher Kummer theory. 
Finally, in \ref{ssec:localrings}, we focus on local rings $\OR$, and show that an $\OR$-pre-orientation can be checked to be an orientation after pushforward to the residue field of $\OR$. This result is essential in lifting primitive (higher) $p$-th roots of unity to spherical orientations. 

In \cref{sec:categorification}, we investigate the interaction of the Fourier transform with \textit{categorification}. 
In the preliminary subsection \ref{ssec:categorification}, we review the categorification-decategorification adjunction taking, in one direction, an $\EE_n$-algebra to its $\EE_{n-1}$-monoidal category of modules, and in the other, an $\EE_{n-1}$-monoidal $\infty$-category to the $\EE_n$-algebra of (enriched) endomorphisms of the unit. We reinterpret the property of affineness in terms of this adjunction and address certain set-theoretical issues related to $\Prl$ not being itself presentable. 
In \ref{ssec:catfourier}, we initiate the study of the ``categorified'' Fourier transform, by describing its source and target in explicit terms, and explaining how the ``decategorified'' Fourier transform can be recovered from it. 
We continue, in \ref{ssec:orcat}, to show that orientations categorify to orientations and that the categorical $\OR$-cyclotomic extension is the categorification of the usual $\OR$-cyclotomic extension. 

In \cref{sec:examplesforspecificrings}, we concentrate on $\OR$-(pre)-orientations for local ring spectra $\OR$ with residue field $\FF_p$, paying special attention to the following tower of rings:
\[
    \Sph_{(p)} \to \dots \to 
    \tau_{\le d}\Sph_{(p)} \to \dots \to 
    \ZZ_{(p)} \to \dots \to 
    \ZZ/p^r \to \dots \to 
    \FF_p.
\]
In \ref{ssec:F_p_orientations}, we study the consequences of virtual $\FF_p$-orientability (at height $n$). In particular, we show that it implies affineness for all $\pi$-finite spaces as in \cref{Tn_Applications_Intro}. We also show  that it implies virtual $\OR$-orientability for all connective $\pi$-finite commutative ring spectra with residue field $\FF_p$, such as $\OR = \ZZ/p^r$. 
In \ref{ssec:zpror+higherroots}, we relate $\ZZ/p^r$-orientations in the \textit{stable} setting to primitive higher $p^r$-th roots of unity in the sense of \cite{carmeli2021chromatic}, and deduce that virtual $\FF_p$-orientability is detected by nil-conservative functors. We also characterise virtual $\FF_p$-orientability in terms of the affineness of certain spaces and the Galois condition for the higher cyclotomic extensions. 
We proceed, in \ref{ssec:Z_p-orientations and hyper-completeness}, to study the consequences of virtual $\ZZ_{(p)}$-orientability. First, we show that it implies virtual $\OR$-orientability for \textit{all} local ring spectra $\OR$ with residue field $\FF_p$, and so in particular that the spherical cyclotomic extension is faithful (the case $\OR=\Sph_{(p)}$). Second, we show that in the stable setting, it implies that the localization with respect to the infinite $p$-typical higher cyclotomic extension (or equivalently, the spherical cyclotomic extension) is a \textit{smashing} localization, and provide a formula for its unit. 
Finally, in \ref{ssec:connectedness}, we study \textit{truncated spherical orientations}, i.e., $\OR$-orientations for $\OR = \tau_{\le d}\Sph_{(p)}$. Observing that the case $d = n$ is already equivalent to an $\Sph_{(p)}$-orientation, we study the obstructions for lifting a $\tau_{\le d-1}\Sph_{(p)}$-orientation to a $\tau_{\le d}\Sph_{(p)}$-orientation for $d = 1,\dots,n$. To this end, we introduce the notion of \textit{$d$-connectedness}, which ensures that the said obstructions vanish up to $d$, and relate this property to the connectedness of the (ordinary) roots of unity $\mu_p$ and ``spherical'' roots of unity $\mu_{\Sph_{p}}$ spectra. The last step $d = n$ requires a categorification argument from section 5, and accordingly involves the Picard spectrum. We conclude with a discussion of the spherical cyclotomic extension, showing it is pro-Galois. 

In \cref{sec:chromatic}, we apply the abstract theory developed in the previous sections to chromatic homotopy theory. In the preliminary subsection \ref{ssec:chromaticpreliminaries}, we briefly recall some material on the monochromatic $\infty$-categories $\Sp_{K(n)}$ and $\Sp_{T(n)}$, the Lubin--Tate spectra $E_n$, and the higher cyclotomic extensions $R_{n,r}$ and $R_{n,r}^f$. In \ref{ssec:orlubintate}, we study the Fourier theory over $E_n$. First, we interpret \Cref{HL_orientation_Intro} as the existence of a $\ZZ_{(p)}$-orientation on $E_n$ and bootstrap it to an $\Sph_{(p)}$-orientation, proving \Cref{Sphere_Or_Intro}. By the general theory, this readily implies \Cref{Sph_Gal_Intro} on the $K(n)$-local spherical cyclotomic extension, and \Cref{E_n_Categorification_Intro} on the categorified spherical Fourier transform. We conclude with the study of the discrepancy spectrum, proving \Cref{Discrepancy_Intro}.
Finally, in \ref{ssec:ortelescopic}, we apply the theory of orientations and the Fourier transform to the telescopic setting. We first prove that $\Sp_{T(n)}$ is virtually $\ZZ/p^r$-orientable and deduce \Cref{Tn_orientation_Intro}. From the general theory we immediately deduce all the properties of local systems of $K(n)$-local algebras on $\pi$-finite spaces stated in \Cref{Tn_Applications_Intro} and the higher Kummer theory as formulated in \Cref{Higher_Kummer_Intro}. We conclude with a short discussion about the universal $T(n)$-local virtually spherically oriented localization and its relation to the properties of $R_n^f$ and the telescope conjecture. 


\subsubsection{Notation and conventions}

Throughout the paper, we work in the framework of $\infty$-categories (a.k.a.~quasi-categories) as developed in \cite{htt} and \cite{HA}. We generally follow the terminology and notation therein. For all concepts related to semiadditivity,  semiadditive height, and higher cyclotomic extensions we refer the reader to \cite{AmbiHeight} and \cite{carmeli2021chromatic}; precise references are given in the main body of the text. 

In addition, we employ the following notation:
\begin{enumerate}
    \item The underlying space of a spectrum $X$ will be denoted by $\und{X} := \Omega^{\infty}X$. More generally, for an object $X$ in a monoidal $\infty$-category $\cC$, we write $\und{X} = \Map_{\cC}(\one,X)$.
    \item A square in an $\infty$-category is called \textit{exact} if it is both a pullback and a pushout square. 
    \item In our applications, we usually fix a prime $p$ and work $p$-locally. We usually indicate the prime through a subscript; for instance, we write $\Sp_{(p)}$ for the category of $p$-local spectra.
    \item A $p$-local spectrum $X$ is said to be \textit{$\pi$-finite} if it is connective and $\bigoplus_n\pi_nX$ is a finite abelian group. We write $\Sp_{(p)}^{\ptors{\pi}}$ for the category of $p$-local $\pi$-torsion spectra (\cref{def:pitor}), i.e., those \emph{connective} spectra which can be written as filtered colimits of $p$-local  $\pi$-finite spectra. The $p$-local $\pi$-torsion torsion part of a $p$-local spectrum $X$ is denoted by $(X)_{(p)}^{\ptors{\pi}}$.
\end{enumerate}

\subsubsection{Acknowledgments}
We would like to thank Robert Burklund, Gijs Heuts, Michael Hopkins and Allen Yuan for useful discussions, and Shai Keidar, Shay Ben Moshe, and Maxime Ramzi for helpful comments on an earlier version of this paper. The first author is supported by the European Research Council (ERC) under Horizon Europe (grant No.~101042990). The third author is supported by ISF1588/18 and BSF 2018389.   The first and fourth authors thank the Max Planck Institute for its hospitality. The first author would also like to thank the Hausdorff Center for Mathematics, and the third author would like to thank the Massachusetts Institute of Technology for its hospitality.


\section{Affineness and Eilenberg--Moore}\label{sec:affine}

In algebraic geometry, affine schemes are those which can be canonically recovered from their algebra of global regular functions. In this section, we study an analogous notion in homotopy theory and its interaction with Eilenberg--Moore type properties, Galois extensions and ambidexterity. Much of the material in this section is inspired by, and is an abstraction of, the results and arguments in \cite[\S 5.4]{AmbiKn} for $K(n)$-local spectra.

\subsection{Affine functors}\label{ssec:affinefunc}

We begin the study of affineness in the abstract generality of presentably monoidal $\infty$-categories. For $f^*\colon \cD \to \cC$ in $\alg(\Prl)$, i.e., a monoidal, colimit preserving functor between presentably monoidal $\infty$-categories, write $f_*\colon \cC \to \cD$ for the right adjoint to $f^*$, which exists by the adjoint functor theorem. 
By \cite[Corollary 7.3.2.7]{HA}, $f_*$ has a canonical structure of a \textit{lax} monoidal functor. Hence, for every $R \in \alg(\cC)$, we have $f_*R \in \alg(\cD)$ and an induced functor 
\[
    \mdef{f_\sharp} \colon 
    \LMod_R(\cC) \too \LMod_{f_*R}(\cD),
\]
taking $X\in \LMod_R(\cC)$ to $f_*X$, endowed with its canonical $f_*R$-module structure. This functor admits a left adjoint
\[
    \mdef{f^\sharp} \colon 
    \LMod_{f_*R}(\cD) \too \LMod_R(\cC),
\]
taking $Y \in \LMod_{f_*R}(\cD)$ to
\[
f^\sharp X:= R \relotimes{f^*f_*R} f^*X,
\]
using the algebra map $f^*f_* R \to R$ provided by the counit of the adjunction $f^* \dashv f_*$ (as in the proof of \cite[Theorem~5.4.3]{AmbiKn}).

\begin{defn}\label{def:affine}
    A functor $f^*\colon \cD\to \cC$ in $\alg(\Prl)$ is called \tdef{affine}, if the functor 
    \[
f_\sharp \:\colon\: 
        \cC \simeq
        \LMod_{\one_\cC}(\cC)\too 
        \LMod_{f_*\one_\cC}(\cD) 
    \]
    is an equivalence.
\end{defn}

This definition generalizes the usual notion of affineness from algebraic geometry in the following sense:
\begin{example}
    Consider a morphism of commutative rings $g\colon R \to S$, write $g^*\colon \Mod_R \to \Mod_S$ for the extension of scalars functor between the associated module categories, and let $g_*$ be its forgetful right adjoint. Then $g^*$ is affine. More generally, let $f\colon X\to Y$ be a morphism of schemes, and $f^*\colon \QCoh(Y)\to \QCoh(X)$ the functor of pullback of (ordinary) quasi-coherent sheaves along $f$. Then, $f^*$ is affine if and only if $f$ is an affine morphism of schemes. 
\end{example}

\subsubsection{Characterization}

Our first goal is to characterize affineness in terms of the intrinsic properties of $f^*$, or rather its right adjoint $f_*$.
A functor $f^*\colon \cD \to \cC$ in $\alg(\Prl)$ endows $\cC$ with a $\cD$-linear structure, by which we mean a structure of a right $\cD$-module in the symmetric monoidal $\infty$-category $\Prl$, and $f^*$ becomes canonically a $\cD$-linear functor. Furthermore, we get a \textit{projection formula map}
\[
    \varphi \colon f_*(X)\otimes Y \to f_*(X\otimes f^*Y)
    \qin \cD,
\]
which for every $X\in \cC$ and $Y\in \cD$ is given by the composition 
\[
    f_*(X)\otimes Y \oto{u} f_*f^*(f_*(X)\otimes Y) \simeq f_*(f^*f_*X\otimes f^*Y) \oto{c} f_*(X\otimes f^*Y),
\]
with $u$ and $c$ the unit and counit maps respectively of the adjunction $f^*\dashv f_*$. 
If $\varphi$ happens to be a natural \textit{isomorphism}, then $f_*$ is $\cD$-linear as well, and the whole adjunction $f^* \dashv f_*$ promotes to the world of $\cD$-linear categories (see \cite[Remark~7.3.2.9]{HA}). 

\begin{lem}\label{Affine_Properties}
    Let $f^* \colon \cD \to \cC$ in $\alg(\Prl)$. If $f^*$ is affine, then $f_*$ is $\cD$-linear, colimit preserving and conservative. 
\end{lem}

\begin{proof}
    If $f^*$ is affine, then up to isomorphism, the functor $f_*$ identifies with the forgetful functor $\LMod_{f_*\one_\cC}(\cD) \to \cD$. Such functors are $\cD$-linear, colimit preserving, and conservative by \cite[Remark 4.8.4.11, Corollary 4.2.3.7, and Corollary 4.2.3.2]{HA}.
\end{proof}

We shall show that the converse of \Cref{Affine_Properties} holds as well, giving a characterization of affine functors in terms of the properties of their right adjoint. We begin with the following more general fact:

\begin{prop}\label{colim_lin_fully_faithful}
    Let $f^*\colon \cD\to \cC$ in $\alg(\Prl)$. If $f_*$ is $\cD$-linear and colimit preserving, then for every algebra $R\in\alg(\cC)$, the functor $f^\sharp \colon \LMod_{f_*R}(\cD) \to  \LMod_R(\cC)$ is fully faithful. 
\end{prop}

\begin{proof}
    To prove that $f^\sharp$ is fully faithful, we have to show that the unit morphism 
    $
        u\colon X\to f_\sharp f^\sharp X 
    $
    is an isomorphism for every $X\in \LMod_{f_*S}(\cD)$. 
    First, note that $f_\sharp$ is colimit preserving. Indeed, we have a commutative diagram 
    \[
        \xymatrix{
        \LMod_R(\cC)\ar[d]\ar^{f_\sharp}[r] & \LMod_{f_*R}(\cD)\ar[d] \\
        \cC\ar^{f_*}[r]                     & \cD,
        }
    \]
    where the vertical maps are the conservative and colimit preserving forgetful functors, and the lower horizontal functor is colimit preserving by assumption. Therefore, the upper horizontal functor is colimit preserving as well.
    We get that the source and the target of $u\colon \Id \to f_\sharp f^\sharp$ are colimit preserving functors. Since $\LMod_{f_*R}(\cD)$ is generated under colimits by the modules of the form $f_*R \otimes Y$ for $Y\in \cD$, it suffices to show that $u$ is an isomorphism at such modules. 
    The image of $f_\sharp f^\sharp (f_*R\otimes Y)$ under the forgetful functor $\LMod_{f_*R}(\cD) \to \cD$ is given by
    \[
         f_*(R \otimes_{f^*f_*R} f^*(f_*R\otimes Y)) \simeq f_*(R\relotimes{f^*f_*R}f^*f_*R\otimes f^*Y) \simeq f_*(R\otimes f^*Y). 
    \]
    Via this identification, the map 
    $u\colon f_*R \otimes Y \to f_\sharp f^\sharp(f_*R \otimes Y)$ 
    corresponds to the projection morphism 
    \[
        \varphi \colon f_*R\otimes Y \to f_*(R\otimes f^*Y), 
    \]
    which is an isomorphism by our assumption that $f_*$ is $\cD$-linear. Since the forgetful functor is conservative, we deduce that $u$ is an isomorphism at $f_*R\otimes Y$ and the result follows.
\end{proof}

We deduce the following characterization and consequence of affineness (cf.  \cite[Corollary 3.7]{mathew2015affineness}):

\begin{prop}\label{criterion_affineness} 
    Let $f^*\colon \cD\to \cC$ in $\alg(\Prl)$. The functor $f^*$ is affine if and only if its right adjoint $f_*\colon \cC \to \cD$ is $\cD$-linear, colimit preserving and conservative. Moreover, in such a case, for every $R\in \alg(\cC)$, the functor
    \[
        f_\sharp \colon \LMod_{R}(\cC)\to \LMod_{f_*R}(\cD)
    \]
    is an equivalence.
\end{prop}

\begin{proof}
    The `only if' part is given by \Cref{Affine_Properties}. Now, if $f_*$ is $\cD$-linear, colimit preserving and conservative, then by \Cref{colim_lin_fully_faithful}, the functor 
    \(
        f_\sharp \colon \LMod_{R}(\cC)\to \LMod_{f_*R}(\cD)
    \)
admits a fully faithful left adjoint. To show that $f_\sharp$ is an equivalence it therefore suffices to show that it is conservative. Note that the composition of $f_\sharp$ with the forgetful functor $\LMod_{f_*R}(\cD) \to \cD$ is the functor $f_*$, which is conservative by assumption. Thus, $f_\sharp$ is conservative as well, and hence an equivalence. In particular, the `if' part follows by taking $R=\one_\cC$.
\end{proof}

\begin{rem}
    \Cref{criterion_affineness} above is closely related to the Barr--Beck--Lurie monadicity theorem \cite[Theorem 4.7.3.5]{HA} and could have been proved using it. Namely, the fact that $f_*$ is colimit preserving and conservative implies that the adjunction $f^*\dashv f_*$ is monadic. The condition that it is also $\cD$-linear identifies the monad $f_*f^*$ on $\cD$ with the monad of tensoring with $f_*\one_\cC \in \alg(\cD)$, so that we can identify $\cC$ with the category of left modules over $f_*\one_\cC$ in $\cD$. 
\end{rem}

In view of \Cref{criterion_affineness}, we adopt the following definition:

\begin{defn}\label{def:semi-affine}
    Let $f^*\colon \cD \to \cC$ in $\alg(\Prl)$. We say that $f^*$ is \tdef{semi-affine}\footnote{This notion has appeared in various guises in the literature before; for instance, it is sometimes referred to as \textit{twisted ambidextrous}.} if $f_*$ is colimit preserving and $\cD$-linear. 
\end{defn}

As an immediate consequence of \Cref{criterion_affineness}, $f^*$ is affine if and only if it is semi-affine and $f_*$ is conservative.   

\subsubsection{Closure properties}

We now describe some of the closure properties enjoyed by the collection of affine functors. First, affine functors are closed under composition and satisfy cancellation from the left. 

\begin{prop} \label{affine_comp}
    Let $f^*\colon \cD \to \cC$ and $g^*\colon \cE \to \cD$ be morphisms in $\alg(\Prl)$. If $g^*$ is affine, then $f^*$ is affine if and only if $f^*g^*$ is affine.
\end{prop}

\begin{proof}
Let $h^*= f^*g^*$. The functor $h_\sharp \colon \cC\to \LMod_{h_*\one_\cC}(\cE)$ can be identified with the composition 
    \[
\cC \oto{\:f_\sharp\:} 
\LMod_{f_*\one_\cC}(\cD)\oto{\:g_\sharp\:}
        \LMod_{g_*f_*\one_\cC}(\cE). 
    \]
    The first functor $f_\sharp$ is an equivalence if and only if $f_*$ is affine by definition, and the second functor $g_\sharp$ is an equivalence if $g_*$ is affine by \Cref{criterion_affineness}. Hence, the result follows by 2-out-of-3. 
\end{proof}

Our next goal is to study the closure properties of affine functors under limits in $\alg(\Prl)$. 
Let $I$ be a small $\infty$-category and let $\cC_{(-)}$ and $\cD_{(-)}$ be two functors from $I$ to $\alg(\Prl)$. Given a natural transformation $f_{(-)}^*\colon \cD_{(-)}\to \cC_{(-)}$ which is level-wise affine, we would like to know when the induced map on the limits over $I$ is affine as well. While we do not expect this to always be the case, we can show this under the assumption that $f_{(-)}^*$ is \tdef{right adjointable}, in the sense of \cite[Definition 4.7.4.16]{HA}. Namely, we need the lax natural transformation $\cC_{(-)}\to \cD_{(-)}$ assembled from the right adjoints of the $f_a^*$-s to be a (strict) natural transformation.

\begin{prop} \label{affine_limits}
    Let $f^*_{(-)}\colon \cD_{(-)}\to \cC_{(-)}$ be a natural transformation of $I$-shaped diagrams in $\alg(\Prl)$. Assume that, 
    \begin{enumerate}
        \item For every $a\in I$, the functor $f_a^*\colon \cD_a \to \cC_a$ is affine. 
        
        \item $f^*_{(-)}$ is right adjointable.
    \end{enumerate}
    Then, the induced functor on the limits 
    \[
        f^* \colon 
        \invlim_{a\in I} \cD_a \too 
        \invlim_{a\in I}\cC_a
    \]
    is affine. 
\end{prop}

\begin{proof}
    Since $f_{(-)}^*$ is right adjointable, the right adjoint of $f^*$ is given by the limit of the right adjoints
    \[
         f_*\simeq \invlim_{a\in I} (f_a)_* \colon 
         \invlim_{a\in I} \cC_a \too
         \invlim_{a\in I}\cD_a,
    \]
    see, e.g., \cite[Proposition 2.1.7]{arad2019tale}.
    By \Cref{criterion_affineness}, it suffices to show that $f_*$ is conservative, colimit preserving and $\cD$-linear. For the conservativity, since the projections $\cD\to \cD_a$ are jointly conservative, it suffices to show that the compositions 
    \[
\cC \oto{\:f_*\:}\cD \too \cD_a
    \]
    are jointly conservative. But this follows from the facts that these composites identify with
    \[
        \cC \too \cC_a \oto{(f_a)_*} \cD_a
    \]
    and that the $(f_a)_*$-s are all conservative.
    
    To see that $f_*$ is colimit preserving, note that, since all the transition functors in the diagrams $\cC_{(-)}$ and $\cD_{(-)}$ are colimit preserving functors, the projections $\cC \to \cC_a$ and $\cD \to \cD_a$ jointly detect (and preserve) colimits in $\cC$ and $\cD$ respectively. Hence, the result follows once again from the assumption that $(f_a)_*$ is colimit preserving.
     
    It remains to show that $f_*$ is $\cD$-linear.
Let $X= \{X_a\}_{a\in I} \in \cC$ and $Y = \{Y_a\}_{a\in I} \in \cD$. Then, the $a$-th component of the projection morphism 
    \[
        f_*X \otimes Y \too f_*(X\otimes f^*Y) 
    \]
    is the projection morphism 
    \[
        (f_a)_*X_a \otimes Y_a \too (f_a)_*(X_a\otimes f_a^*Y)
    \]
    which is an isomorphism by the assumption that all the $f_a^*$-s are affine. We deduce that $f_*$ is $\cD$-linear and hence that $f^*$ is affine. 
\end{proof}

\begin{cor}\label{affine_limits_spaces}
Affine functors are closed under limits over spaces in $\alg(\Prl)$.
\end{cor}
\begin{proof}
    For a space $A$, since all morphisms in $A$ are invertible, all natural transformations of $A$-shaped diagrams in $\Prl$ are right-adjointable. Hence, the result
    follows from \Cref{affine_limits}.
\end{proof}

\subsubsection{Monoidal structure}

By \Cref{criterion_affineness}, for an algebra $R\in \alg(\cC)$ and an affine functor $f^*\colon \cD \to \cC$, we can identify $R$-modules in $\cC$ with $f_*R$-modules in $\cD$. We shall show that this identification is compatible with the formation of relative tensor products of (left and right) modules. We first observe that for all $N\in \RMod_R(\cC)$ and $M\in \LMod_R(\cC)$, the lax monoidal structure on $f_*$ provides a canonical comparison map
\[
    \nu \colon 
    (f_*N)\relotimes{f_*R}(f_*M) \too 
    f_*(N\relotimes{R} M)
    \qin \cD,
\] 
given by the composition 
\[
    f_*N\relotimes{f_*R}f_*M \oto{u}
    f_*f^*(f_*N\relotimes{f_*R}f_*M) \simeq 
    f_*(f^*f_*N \relotimes{f^*f_*R} f^*f_*M) \to 
    f_*(N \relotimes{R} M),
\]
where the last map is the one induced on the relative tensor product from the counit $f^*f_* \to \Id$. 

\begin{prop}\label{affine_relative_tensor}
    Let $f^*\colon \cD\to \cC$ in $\alg(\Prl)$ and let $R\in \alg(\cC)$. If $f^*$ is affine, then for every $M \in \LMod_R(\cC)$  and $N\in \RMod_R(\cC)$, the map
    \[
        \nu \colon    
        (f_*N) \relotimes{f_*R} (f_*M) \too 
        f_*(N\relotimes{R} M)
    \]
    is an isomorphism.
\end{prop}

\begin{proof}
Since $f_*$ is colimit preserving, the source and target of $\nu$ preserve colimits in the $M$-variable. The $\infty$-category $\LMod_R(\cC)$ is generated under colimits by modules of the from $R\otimes X$ for $X\in \cC$. Also, since $f_*$ is conservative, $\cC$ is generated under colimits from the essential image of $f^*$. Consequently, it suffices to show that $\nu$ is an isomorphism at modules of the form $N=R\otimes f^*Y$ for $Y\in \cD$. 
    Hence, it suffices to show that the natural transformation
    \[
        \overline{\nu} \colon 
        (f_* N) \relotimes{f_*R}f_*(R\otimes f^*Y) \too 
        f_*(N\relotimes{R} (R\otimes f^*Y))
    \]
    obtained from $\nu$ via composing with the functor $R\otimes f^*(-)\colon\cD \to \LMod_R(\cC)$
in the $Y$-variable is a natural isomorphism. The assumption that $f_*$ is $\cD$-linear shows that the source and target of $\overline{\nu}$ are $\cD$-linear functors, and $\overline{\nu}$ is hence a natural transformation of $\cD$-linear functors. Thus, we are reduced to the case $Y= \one_\cD$, where we get the canonical isomorphism
    \[
        (f_*N) \otimes_{f_*R} (f_*R) \iso f_* N.\qedhere
    \]
\end{proof}

When dealing with \textit{symmetric} monoidal $\infty$-categories and \textit{commutative} algebras, the above has a very clean interpretation. 

\begin{prop}\label{Affine_SM}
    Let $f^*\colon \cC \to \cD$ in $\calg(\Prl)$. If $f^*$ is affine, then for every $R\in \calg(\cC)$ we have a natural symmetric monoidal equivalence
    \[
        f_\sharp\colon \Mod_R(\cC) \iso \Mod_{f_*R}(\cD).
    \]
    In particular, it induces an equivalence
    \[
        \calg_R(\cC) \iso \calg_{f_*R}(\cD).
    \] 
\end{prop}

\begin{proof}
    Since the tensor product in the $\infty$-category of modules over a commutative algebra is given by the relative tensor product, the first claim is a consequence of \Cref{affine_relative_tensor}. The second claim follows by taking commutative algebra objects and using the natural identification
    \[
        \calg(\Mod_R(\cC)) \simeq \calg_R(\cC).\qedhere
    \]
\end{proof}

For $\cC \in \calg(\Prl)$, pushouts in $\calg(\cC)$ are computed via the relative tensor product. Hence, for a symmetric monoidal functor $f^*\colon \cD \to \cC$ in $\calg(\Prl)$, \Cref{Affine_SM} implies that if $f^*$ is affine, then the functor $f_*\colon \calg(\cC) \to \calg(\cD)$ preserves pushout squares. We shall now discuss a generalization of this property to the context of non-commutative algebras.  
Let $\cC\in \alg(\Prl)$ and let 
\[
    \xymatrix{
    R_0\ar[r]\ar[d] & R_1\ar[d] \\
    S_0 \ar[r] & S_1
    }
\]
be a commutative square in $\alg(\cC)$. 
The right map $R_1\to S_1$ is a map of left $R_0$-modules.
The $R_0$-module structure of $S_1$ comes from restricting
the $S_0$-module structure along the left map. Hence, the restriction-extension
of scalars adjunction along $R_0\to S_0$ induces a map of left $S_0$-modules:
\[
    S_0\relotimes{R_0}R_1\too S_1.
\]

\begin{defn} \label{def:relative_tensor_square}
    We say that a square of algebras as above is a \tdef{relative tensor square} if the above map $S_0\otimes_{R_0}R_1\to S_1$ is an isomorphism.
\end{defn}


Every sifted colimit preserving monoidal functor preserves relative tensor squares, as the relative tensor product can be realized via the geometric realization of a bar construction. For an \emph{affine} functor, the same holds for its right adjoint.
\begin{prop}\label{affine_functor_EM_right_adjoint}
    Let $f^*\colon \cD \to \cC$ in $\alg(\Prl)$. If $f^*$ is affine, then a commutative square in $\alg(\cC)$ is a relative tensor square if and only if its image under $f_* \colon \cC \to \cD$ is a relative tensor square in $\alg(\cD)$. 
\end{prop}

\begin{proof}
    Assume that $f^*$ is affine and let 
    \[
        \xymatrix{
        R_0 \ar[r]\ar[d] & \ar[d] R_1 \\ 
        S_0 \ar[r] & S_1
        }
    \]
    be a commutative square in $\alg(\cC)$. By \Cref{affine_relative_tensor}, we have a canonical isomorphism
    \[
        (f_*R_1) \relotimes{f_*R_0} (f_*S_0) \simeq f_*(R_1\relotimes{R_0}S_0).
    \]
    Via this isomorphism, we can identity the map 
    $(f_*R_1) \otimes_{f_*R_0}(f_*S_0) \to f_*S_1$ 
    with the image under $f_*$ of the corresponding map $R_1 \otimes_{R_0} S_0 \to S_1$. Hence, if the map $R_1 \otimes_{R_0} S_0 \to S_1$ is an isomorphism, so is the map $(f_*R_1) \otimes_{f_*R_0}(f_*S_0) \to f_*S_1$. 
    By \Cref{criterion_affineness} the functor $f_*$ is conservative, so the converse of the above implication holds as well.
\end{proof}

\subsection{Affineness for local systems}\label{ssec:affinelocsys}

We now specialize our discussion of affineness to the setting we are mostly interested in, that of local systems. 

\begin{defn}
    Let $\cC\in \alg(\Prl)$. A map of spaces $f\colon A\to B$ is said to be \tdef{$\cC$-(semi-)affine}, if the pullback functor $f^*\colon \cC^B\to \cC^A$ is (semi-)affine in the sense of \cref{def:affine} and \cref{def:semi-affine}. 
\end{defn}

\subsubsection{Basic properties}



Given a map of spaces $f\colon A \to B$, and writing $q\colon B \to \pt$ for the terminal map, the unit of the adjunction $f^* \dashv f_*$ 
induces a map
\[
\one^B= q_*q^* \one \too 
q_*f_*f^*q^* \one= \one^A
    \qin \alg(\cC).
\]
We now show that $f^* \colon \cC^B \to \cC^A$ is compatible with the functor 
$\LMod_{\one^B}(\cC) \to \LMod_{\one^A}(\cC)$ given by extending scalars along the above map.

\begin{prop}\label{pullback_relative_tensor_affine}
    Let $\cC\in \alg(\Prl)$ and
    let $f\colon A\to B$ be a map of spaces with $q\colon B \to \pt$ the terminal map. We have a commutative square of $\infty$-categories
    \[\begin{tikzcd}
    	{\LMod_{\one^B}(\cC)} && {\LMod_{\one^A}(\cC)} \\
    	{\cC^B} && {\cC^A.}
\arrow["{f^*}", from=2-1, to=2-3]
\arrow["{q^\sharp}"', from=1-1, to=2-1]
\arrow["{\one^A\otimes_{\one^B}(-)}", from=1-1, to=1-3]
\arrow["{(qf)^\sharp}", from=1-3, to=2-3]
    \end{tikzcd}\]
    In particular, when $A$ and $B$ are $\cC$-affine, the top and bottom functors are isomorphic.
\end{prop}

\begin{proof} 
    It suffices to show that the square comprising from the right adjoints of all the functors commutes. As above, let $q\colon B \to \pt$ be the terminal map. 
    The composition
    \[
\cC^A \oto{\:\;f_*\:\;} \cC^B \oto{\:\;q_\sharp\:\;} \LMod_{\one^B}{\cC}
    \]
    is the functor that takes $X\in \cC$ to $q_*f_*X \in \cC$ with the induced 
$\one^B= q_* q^*\one$-module structure. This is evidently the same as restricting the $\one^A = f_*q_*q^*f^* \one$-module structure on $q_*f_*X \in \cC$ along the map
    \[
\one^B= q_*q^* \one \too 
q_*f_*f^*q^* \one= \one^A,
    \]
    which is the second composition
    \[
\cC^A \oto{\:(qf)_\sharp\:} \LMod_{\one^A}{\cC} \to \LMod_{\one^B}{\cC}.
    \]
    The last claim follows from the fact that if $A$ and $B$ are $\cC$-affine, then the vertical functors are equivalences. 
\end{proof}

For every point $b\in B$, let $\one_b$ be the $\one^B$ algebra structure on $\one$ given by evaluation at $b$.
\begin{cor} \label{Ev_Sharp}
    Let $\cC\in \alg(\Prl)$ and let $B$ be a space with $q\colon B \to \pt$ the terminal map. For every $b\in B$,  we have a commutative diagram, 
    \[\begin{tikzcd}
{\Mod_{\one^B}(\cC)} && {\: \cC^B\quad} \\
    	& \cC.
\arrow["{\one_b \otimes_{\one^B}(-)}"'{pos=0.3}, from=1-1, to=2-2]
\arrow["q^\sharp", from=1-1, to=1-3]
\arrow["{\ev_b}", from=1-3, to=2-2]
    \end{tikzcd}\]
\end{cor}

\begin{proof}
This is a special case of \Cref{pullback_relative_tensor_affine}, with $A=\pt$ and $f$ the inclusion $\pt \oto{\:b\:} B$.
\end{proof}

We next observe that the closure properties of affine functors imply corresponding closure properties for $\cC$-affine maps of spaces. 

\begin{prop} \label{affineness_extensions}
    Let $\cC\in \alg(\Prl)$ and let $f\colon A\to B$, $f'\colon A' \to B'$ and $g\colon B\to C$ be maps of spaces.
    \begin{enumerate}
        \item If $f$ is an isomorphism, then it is $\cC$-affine.
    
        \item If $g$ is $\cC$-affine, then $f$ is $\cC$-affine if and only if $g\circ f$ is $\cC$-affine.
        
        \item If all the fibers of $f$ are $\cC$-affine, then $f$ is $\cC$-affine.
        
        \item If $f$ and $f'$ are $\cC$-affine, then 
        $f\sqcup f' \colon A\sqcup A' \to B\sqcup B'$
        is $\cC$-affine.
    \end{enumerate}
\end{prop}

\begin{proof}
    (1) is clear and (2) is an immediate consequence of \Cref{affine_comp}. 
    For (3), we observe that the functor $f^*\colon \cC^B\to \cC^A$ can be written as a limit over $b\in B$ of the functors 
    $f_b^*\colon \cC \to \cC^{f^{-1}(b)},$
    where $f_b \colon f^{-1}(b) \to \pt$ is the terminal map. By assumption, each $f_b^*$ is affine, and by \Cref{affine_limits_spaces}, affine functors are closed under limits over spaces, so the result follows. The proof of (4) is similar to that of (3). It follows from the isomorphism 
    $(f\sqcup f')^* \simeq f^* \times f'^*$ and \Cref{affine_limits_spaces}.
\end{proof}

\begin{war}
Claim (3) can not be promoted to an `if and only if' statement. Namely, a $\cC$-affine map of spaces need not have  $\cC$-affine fibers. Indeed, if $A$ is a $\cC$-affine space, then the projection $\pi_1 \colon A\times A \to A$ is $\cC$-affine by (3). Since the composition $A \oto{\Delta} A\times A \oto{\pi_1} A$ is the identity, we deduce by (2) and (1) that the diagonal $\Delta \colon A \to A\times A$ is $\cC$-affine as well. As the fibers of $\Delta$ are the loop-spaces $\Omega A$, if the converse of (3) were to be true, it would have implied that if $A$ is $\cC$-affine, then $\Omega A$ is $\cC$-affine as well. However, for all $n\ge 0$, we have that $B^{n+2} C_p$ is $\Sp_{K(n)}$-affine while $B^{n+1} C_p= \Omega B^{n+2} C_p$ is not (\cref{affiness_Height_below}).
\end{war}

\subsubsection{Eilenberg--Moore}

Our next goal is to show that affineness is strongly related to Eilenberg--Moore
type formulas for the cohomology of pullbacks of spaces. 
Given a pullback square of spaces 
\[
    \xymatrix{B'\times_B A\ar[d]\ar[r] & A\ar[d]\\
    B'\ar[r] & B
    }
\]
and a ring $R\in \alg(\cC)$, we can form the square in $\alg(\cC)$,
\[
    \xymatrix{R^{B}\ar[d]\ar[r] & R^{A}\ar[d]\\
    R^{B'}\ar[r] & R^{B'\times_B A}.
    }
\]
In some cases, this square turns out to be a relative tensor square. For $\cC=\Sp$, this implies the existence of a spectral sequence computing the $R$-cohomology of 
$ B'\times_B A$ from the $R$-cohomologies of $B,B'$ and $A$, known as the \textit{Eilenberg--Moore spectral sequence}. This motivates the following definition:
\begin{defn}\label{Eilenberg_Moore} 
    Let $\cC\in\alg(\Prl)$ and let $R\in\alg(\cC)$. A map of spaces $f\colon A\to B$
    is said to be \tdef{Eilenberg--Moore} with respect to $R \in \alg(\cC)$, if for every map $g\colon B'\to B$, the canonical morphism
    \[
        R^{B'} \otimes_{R^B} R^A \too R^{B'\times_B A}
    \]
    is an isomorphism.
\end{defn}

In the special case where $B= \pt$, the Eilenberg--Moore property degenerates to the \textit{K\"{u}nneth isomorphism}, for which the following is a useful criterion.

\begin{prop}\label{Kunneth}
    Let $\cC \in \alg(\Prl)$, let $R\in \alg(\cC)$ and let $A$ be a space. If $R[A] \in \LMod(R)$ is left dualizable, then $A$ has the Eilenberg--Moore property with respect to $R$. That is, for every space $B$, we have a K\"{u}nneth isomorphism
    \[
        R^A\otimes_R R^B \iso R^{A\times B}.
    \] 
    In particular, if $\one[A] \in \cC$ is left dualizable, the above holds for every $R\in\alg(\cC)$.
\end{prop}

\begin{proof}
    Under the assumption that $R[A] \in \LMod(R)$ is left dualizable, the module $R^A \in \RMod_R(\cC)$ is its left dual. Unwinding the definitions, the canonical comparison map
    \[
R^A\otimes_R R^B \too (R^A \otimes_R R)^B= R^{A\times B},
    \]
    identifies with the assembly map for the functor
    \[
G:=R^A \otimes_R (-) \:\colon\: \LMod_R(\cC) \to \cC
    \] 
and the constant $B$-shaped limit. Now, by \cite[Proposition 4.6.2.1]{HA}, the functor $G$ is a right adjoint and hence preserves limits, so the claim follows. The last part follows from the fact that if $\one[A] \in \cC$ is left dualizable, then $R[A]= R\otimes \one[A] \in \LMod_R(\cC)$ is left dualizable,  
    by \cite[Example 4.6.2.5]{HA}.
\end{proof}

\begin{example}
Let $\cC= \Sp$ and let $R=\FF$ be an ordinary field considered as a ring spectrum. \cref{Kunneth} then recovers the classical fact, that if a space $A$ has finite-dimensional homology with coefficients in $\FF$, then for every space $B$ we have a K\"{u}nneth isomorphism for the cohomology of $A\times B$ with coefficients in $\FF$.
\end{example}

Using affineness we can reduce the general case of the Eilenberg--Moore property to the existence of K\"{u}nneth isomorphisms for the fibers. 

\begin{prop}\label{Affineness_Kunneth_EM}
    Let $\cC\in\alg(\Prl)$, let $R\in\alg(\cC)$, and let $f\colon A \to B$ be a map of spaces. If $B$ is $\cC$-affine, and all the fibers of $f$ are Eilenberg--Moore with respect to $R$, then $f$ is Eilenberg--Moore with respect to $R$.
\end{prop}
\begin{proof}
    Let $q\colon B\to \pt$ denote the projection, so that by our assumption $q$ is $\cC$-affine. For a map $g\colon B' \to B$, the square 
    \[
        \xymatrix{
        R^B \ar[r]  \ar[d]     & R^A \ar[d] \\
        R^{B'} \ar[r]   & R^{A\times_B B'}
        }
    \]
    in $\alg(\cC)$ is obtained from the square
    \[
        \xymatrix{
        q^*R \ar[r]  \ar[d]     & f_*f^*(q^*R) \ar[d] \\
        g_*g^*(q^*R) \ar[r]   & (f\times_B g)_*(f\times_B g)^*(q^*R)
        }
    \]
    in $\alg(\cC^B)$ by applying the functor $q_*$. Since $q^*$ is affine, by \Cref{affine_functor_EM_right_adjoint} it would suffice to show that the latter square is a relative tensor square. We can verify this after applying the functors $b^*$ for all $b\colon \pt \to  B$. 
    Using the Beck--Chevalley isomorphism for local systems, this reduces the claim to the case $B= \pt$, and $A =f^{-1}(b)$, which holds by assumption.
\end{proof}

\begin{rem}
We can informally summerize \Cref{Affineness_Kunneth_EM} by the slogan:
    \[
        \text{``Affineness (of the base) }+\text{ K\"{u}nneth (for the fibers) } \implies \text{Eilenberg--Moore''}.
    \]
\end{rem}

\subsubsection{Galois extensions}

We shall now discuss the implications of affineness to Galois theory in the sense of Rognes. Our main result is that if the classifying space $BG$ of a group $G$ is affine, then \emph{every} commutative algebra with a $G$-action is faithful Galois over its $G$-fixed points. We refer the reader to \cite{RognesGal} and \cite{AkhilGalois} for a discussion of Galois extensions in the context of stable homotopy theory. 

To make the connection with affineness more transparent, we shall rephrase Rognes's notion of a Galois extension in terms of the space $BG$ rather than the group $G$.  
For a space $A$ and a presentably symmetric monoidal $\infty$-category  $\cC$, we denote by 
\[
    \Delta\colon A\too A\times A \qquad,\qquad q\colon A\too \pt
\]
the diagonal and terminal maps of $A$, respectively. For an $A$-local system of commutative algebras $R\in \calg(\cC)^A$, the unit map $q^*\one \to R$ has a mate $\one \to q_*R$ with respect to the $q^* \dashv q_*$ adjunction. Considering the external product $R\boxtimes R \in \calg(\cC)^{A\times A}$, the multiplication map 
\[
    R\otimes R \simeq \Delta^*(R\boxtimes R) \too R
    \qin \calg(\cC)^A
\] 
has a mate  
\[ 
    R\boxtimes R \too \Delta_*R \qin \calg(\cC)^{A\times A}
\]
with respect to the adjunction $\Delta^* \dashv \Delta_*$.
\begin{defn} \label{def:Galois} 
Let $\cC\in \calg(\Prl)$ and let $A$ be a space. A local system $R\in \calg(\cC)^A$ is called an \tdef{$A$-Galois extension} (of the unit object $\one$) if it satisfies the following two properties: 
    \begin{enumerate}
        \item[(G1)] The mate $\one \to q_*R$ of the unit map is an isomorphism in $\calg(\cC)$.
        \item[(G2)] The mate 
        \(
            R\boxtimes R \to \Delta_*R
        \)
        of the multiplication map is an isomorphism in $\calg(\cC)^{A\times A}$. 
    \end{enumerate} 
    We say that a Galois extension is \textit{faithful} if the functor $R \otimes (-)\colon \cC \to \cC^A$ is conservative.  
\end{defn}

When $A$ is \textit{connected}, by choosing a basepoint for $A$, we get a grouplike $\EE_1$-space  $G=\Omega A$. Then, by definition, an $A$-local system of commutative algebras in $\cC$ is a commutative algebra endowed with a $G$-action. Via this identification, $q_*R$, which is the limit of $R$ over $A$, identifies with the fixed points $R^{hG}$ of the $G$-action on $R$. Similarly, $\Delta_*R$ identifies with the algebra of functions $R^G$, with a suitable $G\times G$-action. Under these identifications, the maps appearing in conditions (G1) and (G2) are easily seen to correspond to the maps 
\(
    \one \to R^{hG} 
\)
and 
\(
    R\otimes R \to R^G
\)
appearing in Rognes's definition of a Galois extension, see e.g., \cite[Definition 6.12]{AkhilGalois}. Thus, \Cref{def:Galois} is a base-point free reformulation of Rognes's notion of a Galois extension. Furthermore, allowing non-connected spaces $A$ provides a natural extension of this Galois theory ``downwards''.

\begin{example}
    For $\cC \in \calg(\Prl)$ and a finite discrete space $A$, it is easy to check that an $A$-Galois extension is a collection of idempotent rings $\{R_a \mid a \in A\}$ in $\calg(\cC)$, such that $\prod_{a \in A}R_a \simeq \one$ and $R_a \otimes R_b \simeq \pt$ for all $a\neq b$. In particular, if we denote by $\pi_0(\cC)$ the pro-finite set of connected components of $\one$, then $A$-Galois extensions of $\cC$ are classified by continuous maps $\pi_0(\cC) \to A$. In the stable case, this fits naturally into Akhil's Galois theory developed in \cite{AkhilGalois}.
\end{example}

The somewhat surprising observation is that if $A$ is $\cC$-affine, then it suffices to check only the condition (G1) to ensure that we have a faithful $A$-Galois extension.

\begin{prop} \label{automatic_Galois}
    Let $\cC\in \calg(\Prl)$ and let $A$ be a $\cC$-affine space. A local system $R\in \calg(\cC)^A$ is a faithful Galois extension if and only if the map $\one \to q_*R$ is an isomorphism. 
\end{prop}
\begin{proof}
    If $R$ is Galois then by condition (G1), the unit map $\one \to q_*R$ is an isomorphism. To prove the converse, we have to show that
    \begin{enumerate}
        \item The functor $R\otimes (-)\colon \cC \to \cC^A$ is conservative.
        \item The map $\varphi \colon R\boxtimes R \to \Delta_*R$, which is the mate of the multiplication, is an isomorphism.
    \end{enumerate}
    For $(1)$, by \Cref{criterion_affineness}, the functor $q_*\colon \cC^A \to \cC$ is $\cC$-linear. Consequently, the composition 
    \[
F\colon \cC \oto{R\otimes (-)} \cC^A \oto{\:\:q_*\:} \cC
    \]
    is a $\cC$-linear functor from $\cC$ to $\cC$, which is therefore given by tensoring with $F(\one)\simeq q_*R$. By our assumption this is the identity functor, so that $R\otimes (-)$ is left-invertible and in particular conservative. 
    
    For $(2)$, let $\pi_1,\pi_2 \colon A\times A \to A$ be the projections on the two factors. By \Cref{affineness_extensions}, the map $\pi_1$ is $\cC$-affine, and hence, by \Cref{criterion_affineness}, the functor $(\pi_1)_*$ is $\cC$-linear and conservative. Using the conservativity, it suffices to show that the map 
    \[
        (\pi_1)_*\varphi \colon (\pi_1)_*(R\boxtimes R) \too
        (\pi_1)_*\Delta_*R\simeq R
    \]
    is an isomorphism. 
    By the projection formula for $(\pi_1)_*$ and the Beck--Chevalley isomorphism, we have 
    \[
        (\pi_1)_*(R\boxtimes R) \simeq (\pi_1)_*(\pi_2^*R\otimes \pi_1^*R) \simeq  ((\pi_1)_*\pi_2^*R) \otimes R \simeq (q^*q_*R)\otimes R. 
    \]
    Unwinding the definition of $\varphi$, the map $(\pi_1)_*\varphi$ corresponds via this identification to the composition
    \[
        q^*q_*R \otimes R \oto{c\otimes 1}
R\otimes R \oto{\:m\:} R,
    \]
    where $c$ is the counit of the adjunction $q^*\dashv q_*$ and $m$ is the multiplication map. By the assumption that $q_*R\simeq \one_\cC$, we have an isomorphism $q^*q_*R\simeq \one_{\cC^A}$, so that this composite is an isomorphism and the result follows.  
\end{proof}

More generally, we define an $A$-Galois extension of any $S\in\calg(\cC)$, to be an $A$-Galois extension in the $\infty$-category $\Mod_S(\cC)$. That is, an object $R\in\calg_S(\cC)^A$ satisfying the analogues of (G1) and (G2) relative to $S$. From \Cref{automatic_Galois} we deduce the following:

\begin{cor}\label{Galois_Auto_Rel}
    Let $\cC\in \calg(\Prl)$ and let $A$ be a $\cC$-affine space. Every $R\in \calg(\cC)^A$ is a faithful Galois extension of $q_*R \in \calg(\cC)$ for $q\colon A \to \pt$ the terminal map. 
\end{cor}
\begin{proof}
Let $S= q_* R$. By construction, the map $S \to q_*R$ is an isomorphism and hence the claim follows from \Cref{automatic_Galois}.
\end{proof}

\begin{example}
The constant $A$-local system on $\one\in \calg(\cC)$, i.e., $q^* \one \in \calg(\cC)^A$, is Galois over $\one^A= q_*q^* \one$.
\end{example}

In fact, the above example is \textit{universal}. Let 
$\calg_S^{A-\gal}(\cC) \sseq \calg_S(\cC)^A$
denote the space of faithful $A$-Galois extensions of $S \in \calg(\cC)$.

\begin{prop}\label{Galois_Affine}
    Let $\cC\in \calg(\Prl)$ and let $A$ be a $\cC$-affine space. For every $S \in \calg(\cC)$, there is a natural isomorphism
    \[
        \calg_S^{A-\gal}(\cC) \simeq 
        \Map_{\calg(\cC)}(\one^A, S).
    \]
    In other words, the object $\one^A$ corepresents $A$-Galois extensions of commutative algebras in $\cC$. 
\end{prop}
\begin{proof}
    By \Cref{Affine_SM}, we have an equivalence of $\infty$-categories
    $\calg(\cC^A) \simeq \calg_{\one^A}(\cC),$
    under which the global sections functor $q_* \colon \calg(\cC^A) \to \calg(\cC)$ corresponds to the forgetful functor $\calg_{\one^A}(\cC) \to \calg(\cC)$. This is further isomorphic to the canonical projection functor
    $\calg(\cC)_{\one^A/} \to \calg(\cC).$
    The latter is a left fibration whose fiber over $S\in\calg(\cC)$ is the space $\Map_{\calg(\cC)}(\one^A, S)$. Thus, this space is also isomorphic to the fiber of $q_*$. Namely, the space of objects $R \in \calg(\cC^A)$ with an isomorphism of commutative algebras $q_* R \simeq S$. By \Cref{automatic_Galois}, it is isomorphic to $\calg_S^{A-\gal}(\cC)$. 
\end{proof}

\begin{rem}
Rognes develops the theory of $A$-Galois extensions under the additional assumption that $G= \Omega A$ is \textit{dualizable} in $\cC$. This occurs, for example, when $A$ is \textit{weakly $\cC$-ambidextrous} (see \cite[Corollary 3.3.10]{TeleAmbi}). 
\end{rem}

\subsection{Affineness and ambidexterity}\label{ssec:affineambi}

\subsubsection{Semi-affineness and ambidexterity}

For \emph{truncated} maps of spaces, (semi-)affineness turns out to be closely related to ambidexterity.

\begin{prop} \label{Ambi_Semi_Affine}
    Let $\cC\in \alg(\Prl)$. A truncated map of spaces $f\colon A\to B$ is 
    $\cC$-ambidextrous if and only if $f$ and all of its iterated diagonals are $\cC$-semi-affine.
\end{prop}

\begin{proof}
    We prove the claim by induction on the truncatedness level of $f$. For $m=-2$, the map $f$ is an isomorphism and the claim holds trivially. For $m\ge -1$, the diagonal of $f$ is $(m-1)$-truncated, so the claim holds for it by the inductive hypothesis. We are thus reduced to showing that if $f$ is weakly ambidextrous, then it is ambidextrous if and only if it is semi-affine. Recall that $f$ is $\cC$-semi-affine if and only if $f_*$ is colimit preserving and $\cC^A$-linear. By \cite[Proposition 4.3.9]{AmbiKn}, since $f$ is weakly ambidextrous, it is $\cC$-ambidextrous if and only if $f_*$ is colimit preserving. Hence, semi-affineness implies ambidexterity. Conversely, if $f$ is ambidextrous, then by \cite[Proposition 3.3.1]{TeleAmbi}, the functor $f_* \simeq f_!$ is $\cC^A$-linear, and hence $f$ is semi-affine.
\end{proof}

\begin{cor}\label{Affine_Sadd}
    Let $\cC\in \alg(\Prl)$ be semiadditive and let $p$ be a prime. The $\infty$-category $\cC$ is $p$-typically $m$-semiadditive if and only if the spaces $BC_p,B^2C_p,\dots ,B^mC_p$ are all $\cC$-semi-affine.
\end{cor}
\begin{proof}
    By \cite[Proposition 3.1.2]{AmbiHeight}, it suffices to show that the spaces $ BC_p,B^2C_p,\dots ,B^mC_p$ are all $\cC$-ambidextrous. This follows from \cref{Ambi_Semi_Affine} and \cref{affineness_extensions}, as the diagonal map $B^kC_p \to B^kC_p \times B^kC_p$ has fiber $B^{k-1}C_p$.
\end{proof}

Under the assumption of ambidexterity, affineness reduces to the conservativity of the global sections functor.
\begin{cor} \label{Ambi_Affine}
    Let $\cC\in\alg(\Prl)$ and let $f\colon A\to B$ be a $\cC$-ambidextrous map of spaces. Then, $f$ is $\cC$-semi-affine. Furthermore, $f$ is $\cC$-affine if and only if $f_*\colon \cC^A\to \cC^B$ is conservative.
\end{cor}

\begin{proof} 
     This follows from \Cref{Ambi_Semi_Affine} and the characterization of affine functors  given in \Cref{criterion_affineness}.  
\end{proof}

\begin{example} \label{affine_finite_sets}
For $\cC\in \alg(\Prl)$ semiadditive, every map of spaces $f\colon A\to B$ with finite discrete fibers is $\cC$-affine. Indeed, by \Cref{affineness_extensions} it suffices to show this when $B=\pt$ and by the above \Cref{Ambi_Affine} we only need to show that the functor $f_*\colon \cC^A\to \cC$ is conservative. This functor takes an $A$-indexed collection of objects $\left(X_a\right)_{a\in A}$ to their product $\prod_{a\in A} X_a$. The conservativity now follows from the fact that each projection $a^*\colon \cC^A\to \cC$ for $a\in A$ is a retract of $f_*$.
\end{example}
 
\begin{rem}
    For $f\colon A \to \pt$, the right adjoint $f_* \colon \cC^A \to \cC$ is conservative if and only if the image of the left adjoint $f^*\colon \cC \to \cC^A$ generates $\cC^A$ under colimits. That is if and only if every $\cC$-valued local system on $A$ can be constructed from constant ones by colimits. By analogy with representation theory, one might call such local systems \textit{unipotent}. Thus, we can rephrase the second part of \Cref{Ambi_Affine}, by saying that a $\cC$-ambidextrous space $A$ is $\cC$-affine if and only if all $\cC$-local systems on $A$ are unipotent.   
\end{rem}

\subsubsection{Ambidexterity and Eilenberg--Moore}

Under the assumption of ambidexterity, the relationship between affineness and the Eilenberg--Moore property can be further tightened. First, we have the following criterion for K\"{u}nneth isomorphisms:

\begin{prop}\label{Kunneth_Ambi}
    Let $\cC \in \alg(\Prl)$, and let $A$ be a space. If $A$ is $\cC$-ambidextrous, then $A$ has the Eilenberg--Moore property with respect to every $R \in \alg(\cC)$. That is, for every space $B$, we have a K\"{u}nneth isomorphism
    \[
        R^A\otimes_R R^B \iso R^{A\times B}.
    \]    
\end{prop}
\begin{proof}
    Since $A$ is $\cC$-ambidextrous, $\one[A] \in \cC$ is dualizable by \cite[Corollary 3.3.10]{TeleAmbi}. Hence, the claim follows from \Cref{Kunneth}.
\end{proof}

The combination of \Cref{Kunneth_Ambi} and \Cref{Affineness_Kunneth_EM} shows that if $B$ is  $\cC$-ambidextrous, then the $\cC$-affineness of $B$ implies the Eilenberg--Moore property for a large class of maps $A\to B$. In fact, the ambidexterity assumption guarantees that affineness is also \textit{implied} by a very special case of the Eilenberg--Moore property. The situation can be summerized as follows:

\begin{thm}\label{Affiness_Eilenberg_Moore} 
Let $\cC\in\alg(\Prl)$, and let $B$ be a $\cC$-ambidextrous space. The following are equivalent:
    \begin{enumerate}
        \item The space $B$ is $\cC$-affine.
    
        \item Every $\cC$-ambidextrous map $f\colon A\to B$ is Eilenberg--Moore        with respect to every $R\in\alg(\cC)$.
    
        \item For every pair of points $a,b\in B$, the canonical map $\one\otimes_{\one^{B}}\one\to\one^{\{a\}\times_{B}\{b\}}$
        is an isomorphism. 
    \end{enumerate}
\end{thm}

\begin{proof}
    We will show that (1)$\implies$(2)$\implies$(3)$\implies$(1). For every $\cC$-ambidextrous map $f\colon A \to B$, the fibers are Eilenberg--Moore with respect to every $R\in\alg(\cC)$ by \Cref{Kunneth_Ambi}. Thus, if $B$ is $\cC$-affine, then $f$ is Eilenberg--Moore with respect to every $R\in\alg(\cC)$ by \Cref{Affineness_Kunneth_EM}. That is, we have shown that (1) implies (2). Now, 
    (3) follows from (2) by taking $f\colon \{a\}\to B$ and $g\colon \{b\}\to B$. 
    
    It remains to show that (3) implies (1). To show that $B$ is $\cC$-affine, we need to show that the functor 
    \[
        q_{\sharp}\colon\cC^{B}\too\LMod_{\one^{B}}(\cC)
    \]
    is an equivalence, where $q\colon B\to \pt$ is the terminal map. By \cref{Ambi_Affine}, the map $q$ is $\cC$-semi-affine and hence by \Cref{colim_lin_fully_faithful}, the left adjoint $q^\sharp$ of $q_\sharp$ is fully faithful. 
    Hence, it remains to show that the counit map
    \[
\varepsilon \:\colon\: 
q^{\sharp}q_{\sharp}X= 
        q^{*}\one\relotimes{q^{*}q_{*}\one}q^{*}q_{*}X\too 
        X
        \tag{$*$}
    \]
    is an isomorphism for all $X\in\cC^{B}$. By \cite[Lemma 4.3.8]{AmbiKn}, the category
    $\cC^{B}$ is generated under colimits by objects of the form
    $b_{!}Y$ for $Y\in\cC$ and $b\colon \pt \to B$. By the $\cC$-ambidexterity of $q$, both sides of $(*)$
    preserve colimits and it therefore suffices to show that $\varepsilon$
is an isomorphism at local systems of the form $X= b_! Y$. Using the $\cC$-ambidexterity of the map $b\colon\pt\to B$, we can also identify $b_{!}Y$ with $b_{*}Y$. 
    
    Next, to show that $\varepsilon$ is an isomorphism at $b_*Y$, it suffices to
show that for every $a\colon \pt \to B$, the map $a^{*}\varepsilon$ is an isomorphism at $b_*Y$. Using the identities $b^{*}q^{*}=\Id$ and $q_{*}a_{*}=\Id$,
    the map 
    \[
        a^{*}\varepsilon \colon 
        a^*q^\sharp q_\sharp (b_*Y) \too 
        a^*(b_*Y)
    \]
    assumes the form
    \[
        \one\relotimes{\one^{B}}Y\simeq
        \one\relotimes{q_{*}\one}Y\too 
        b^{*}a_{*}Y\simeq Y^{\{a\}\times_{B}\{b\}}.
    \]
    Both the domain and the range of this map, when considered as functors in the $Y$-variable, are colimit preserving and $\cC$-linear. Indeed, for the domain it follows from the colimit preservation and $\cC$-linearity of the relative tensor product, and for the target by \Cref{Ambi_Affine} applied to the $\cC$-ambidextrous space $\{a\}\times_B \{b\}$. Moreover, $a^*\varepsilon$ is canonically a natural transformation of $\cC$-linear functors. Hence, it suffices to show that the above map is an isomorphism for $Y=\one$. In this case, we obtain precisely the map from condition (3)
    \[
        \one\relotimes{\one^{B}}\one \too
        \one^{\{a\}\times_{B}\{b\}},
    \]
    which is an equivalence by assumption. 
\end{proof}

\Cref{Affiness_Eilenberg_Moore}(3) provides a very practical criterion for checking affineness, which we shall use repeatedly. For now, we demonstrate its utility by deducing that affineness behaves well with respect to monoidal functors. 

\begin{prop} \label{affine_spaces_functors}
Let $F\colon \cC \to \cD$ be a functor in $\alg(\Prl)$ and let $B$ be a $\cC$-ambidextruous space. If $B$ is $\cC$-affine then it is $\cD$-affine. Conversely, if $B$ is $\cD$-affine and $F$ is conservative, then $B$ is $\cC$-affine. 
\end{prop}

\begin{proof}
Since $F$ is monoidal and colimit preserving, $B$ is also $\cD$-ambidextrous, see \cite[Corollary 3.3.2]{TeleAmbi}. Moreover, $F$ preserves $B$-shaped limits as well as $\{a\}\times_B \{b\}$-shaped limits for every $a,b\in B$, see \cite[Corollary 3.2.4]{TeleAmbi}. It follows that for all $a,b\in B$ the square
\[
\xymatrix{
\one_\cD^B \ar[r]\ar[d] & \one_\cD \ar[d] \\ 
\one_\cD \ar[r]       & \one_\cD^{\{a\}\times_B \{b\}}
}
\]
is the image under $F$ of the square
\[
\xymatrix{
\one_\cC^B \ar[r]\ar[d] & \one_\cC \ar[d] \\ 
\one_\cC \ar[r]       & \one_\cC^{\{a\}\times_B \{b\}}
}.
\]
Since $F$ is is colimit preserving and monoidal, the latter is a relative tensor square if the former is a relative tensor square, and the converse holds if $F$ is conservative. 
\end{proof}
\subsubsection{Affineness and height}

In \Cref{Ambi_Affine}, we have seen that $\cC$-semi-affineness is closely related to $\cC$-ambidexterity. In particular, the $\cC$-semi-affineness of the Eilenberg--Maclane spaces $B^kC_p$ is closely related to the $p$-typical higher semiadditivity of $\cC$ (\cref{Affine_Sadd}). We shall now see that \textit{$\cC$-affineness} of these spaces is closely related to the \textit{semiadditive height} of $\cC$ in the sense of \cite[\S 3]{AmbiHeight}.

\begin{prop}\label{Affiness_Height}\label{affiness_Height_below}
    Let $p$ be a prime and let $\cC\in\alg(\Prl)$ be $p$-typically $n$-semiadditive.
    \begin{enumerate}
        \item  If $\cC$ is of height $\le n$, then every $(n+1)$-connected $\pi$-finite $p$-space is $\cC$-affine. 
        
        \item If $\cC$ is of height $\le n$ and  $B^{n+1}C_{p}$ is $\cC$-affine, then $\cC$ is of height $\le (n-1)$.\footnote{By convention,  an $\infty$-category is of height $\le -1$ if and only if it is the zero category.}
        
\item If the spaces $B^kC_p$ are $\cC$-affine for $k=0,\dots,n$,  
        then $\cC$ is of height $\ge n$.  
    \end{enumerate}
\end{prop}

\begin{proof}
    For (1), let $A$ be an $(n+1)$-connected $\pi$-finite $p$-space and let $q\colon A\to\pt$ be the terminal map. Since $\cC$ is of height $\le n$, we get by \cite[Proposition 3.2.3]{AmbiHeight}, that $q^{*}\colon\cC\to\cC^{A}$ is an equivalence and hence clearly affine. 

    For (2), assuming $B^{n+1}C_p$ is $\cC$-affine, we get by \Cref{Affiness_Eilenberg_Moore}, an isomorphism 
    \[
        \one\relotimes{\one^{B^{n+1}C_{p}}}\one\iso\one^{B^{n}C_{p}}.
    \]
    
    Since $\cC$ is of height $\le n$, by \cite[Proposition 3.2.1]{AmbiHeight}
    we also have an isomorphism  $\one\simeq\one^{B^{n+1}C_{p}}$. Combining the two isomorphisms we get
    \[
        \one \simeq 
        \one \otimes_{\one} \one \simeq 
        \one \otimes_{\one^{B^{n+1}C_{p}}}\one \simeq 
        \one^{B^{n}C_{p}}.
    \]
    By \cite[Proposition 3.2.1]{AmbiHeight} again, this implies that $B^{n-1}C_{p}$ is $\cC$-amenable. Namely that $\cC$ is of height $\le n-1$. 

For $(3)$, let $\cC_{\le n-1} \sseq \cC$ be the full subcategory spanned by the objects of height $\le n-1$, so that $\cC$ is of height $\ge n$ if and only if $\cC_{\le n-1}\simeq \pt$. By \cite[Proposition 5.2.16]{AmbiHeight}, the inclusion of $\cC_{\le n-1}$ into $\cC$ admits a symmetric monoidal reflection $\cC \to \cC_{\le n-1}$. Hence, the $\infty$-category $\cC_{\le n-1}$ is itself $p$-typically $n$-semiadditive and the space $B^kC_p$ for $k=0,\dots,n$ are also $\cC_{\le n-1}$-affine, by \Cref{affine_spaces_functors}. Now, applying (2) inductively, we find that $\cC_{\le n-1}$ is of height $-1$ and hence trivial.
\end{proof}

For $\cC$ of height $\le n$, we can roughly summarize the content of \Cref{Affiness_Height} regarding $\cC$-affineness of $\pi$-finite $p$-spaces as follows: 
\begin{enumerate}
    \item We have affineness \textit{above} level $n+1$ for trivial reasons. 
    
    \item To have affineness \textit{below} level $n+1$, the height of $\cC$ must be exactly $n$, in which case,
    
    \item There is no affineness at level \textit{exactly} $n+1$.
\end{enumerate}

This sill leaves open, however, the question of whether for $\cC$ of height exactly $n$ we actually have affineness below level $n+1$. In the chromatic world, Hopkins and Lurie proved in \cite[Theorem 5.4.3]{AmbiKn}, that this is indeed the case for $\cC= \Sp_{K(n)}$. Their argument is rather specific though, as it relies on explicit computations with the Lubin--Tate spectrum $E_n$. One of the goals of this paper is to bootstrap \cite[Theorem 5.4.3]{AmbiKn} to the telescopic localizations $\cC = \Sp_{T(n)}$. We achieve this by placing the approach of Hopkins and Lurie in the context of a \textit{higher semiadditive Fourier transform}, which we develop in the next section.


\section{The Higher Fourier Transform}\label{sec:higherfourier}

Let $\OR$ be a connective $p$-local commutative ring spectrum and $\cC$ a symmetric monoidal $\infty$-category. In this section, we study natural transformations
\[
    \one[M] \too
    \one^{\und{\Sigma^n I_{\QQ_p/\ZZ_p}M}} \qin 
    \calg(\cC),
\]
from the group algebra of a suitably finite connective $\OR$-module spectrum $M$ to the algebra of functions on its $n$-suspended Brown--Comenetz dual. We show that such natural transformations can be viewed as a generalization of the classical discrete Fourier transform and share many of its basic properties. In particular, we show that such maps are parameterized by a certain datum of an \textit{$\OR$-pre-orientation} on $\cC$, which generalizes roots of unity of $\one_\cC$ in the case $\OR = \ZZ/p^r$, and satisfy familiar relations with respect to augmentations, translation and duality.

\subsection{Pre-orientations}\label{ssec:pre-or}

\subsubsection{Shifted Brown--Comenetz duality}

The first ingredient in the construction of the higher Fourier transform is a spectral lift of Pontryagin duality for abelian groups, known as \textit{Brown--Comenetz duality}. We shall work throughout with the $p$-local variant of this theory with respect to a fixed prime $p$. By Brown representability, the contravariant functor $\hom_\Ab(\pi_{-*}(-),\QQ_p/\ZZ_p)$, from spectra to graded abelian groups, is represented by a spectrum $I_{\QQ_p/\ZZ_p}$, the $p$-local Brown--Comenetz dual (of the sphere). It is characterized by the following property: There is a natural isomorphism 
\[
    \pi_*(\hom_\Sp(M,I_{\QQ_p/\ZZ_p}))\simeq \hom_\Ab(\pi_{-*}(M),\QQ_p/\ZZ_p)
\]
for all $M \in \Sp$.
Mapping into $I_{\QQ_p/\ZZ_p}$ gives a contravariant endofunctor on $\Sp$. For our applications, it will be useful to introduce a certain connective shifted version of it.

\begin{defn}
    For $n \in \NN$ and $M\in \Sp^\cn$, we define the \tdef{$n$-shifted Brown--Comenetz dual} of $M$ by 
    \[
        \mdef{\Dual{M}{n}}:= \tau_{\ge 0}\hom(M, \Sigma^n I_{\QQ_p/\ZZ_p}) 
        \qin \Sp^\cn.
    \]
\end{defn}

For a connective commutative $p$-local ring spectrum $\OR$ and a connective $\OR$-module $M$, the $n$-shifted Brown--Comenetz dual $\Dual{M}{n}$ admits a canonical $\OR$-module structure via the action of $\OR$ on the source of the mapping spectrum $\hom(M,\Sigma^n I_{\QQ_p/\ZZ_p})$. In particular, $\Dual{\OR}{n}$ itself is a connective $\OR$-module. In fact, it represents the functor $\Dual{}{n}$ internally to connective $\OR$-modules. In the following,  $\hom_\OR^\cn$ stands for the internal hom functor in \emph{connective} $\OR$-modules. 

\begin{lem}
    Let $R\in \calg(\Sp_{(p)}^{\cn})$. For every $M \in \Mod_\OR^\cn$, we have
    \[
        \Dual{M}{n} \:\simeq\: \hom_\OR^\cn(M,\Dual{\OR}{n})
        \qin \Mod_\OR^\cn.
    \]
\end{lem}
\begin{proof}
    We compute using the fact that $M$ is connective,
    \[
        \hom_\OR^\cn(M,\Dual{\OR}{n}) \simeq
        \tau_{\ge 0}\hom_\OR(M, \Dual{\OR}{n}) \simeq
        \tau_{\ge 0}\hom_\OR(M, \tau_{\ge 0}\hom(\OR, \Sigma^n I_{\QQ_p/\ZZ_p})) \simeq
    \]
    \[
        \tau_{\ge 0}\hom_\OR(M, \hom(\OR, \Sigma^n I_{\QQ_p/\ZZ_p})) \simeq
        \tau_{\ge 0}\hom(M, \Sigma^n I_{\QQ_p/\ZZ_p}) \simeq
        \Dual{M}{n}. \qedhere
    \]
\end{proof}

\begin{ter}\label{ter:characters}
With notation as above, we will refer to maps $M \to \Dual{\OR}{n}$ as \textit{characters} of $M$.
\end{ter}

Next, we observe that $\Dual{M}{n}$ is always $n$-truncated and depends only on the $n$-truncation of $M$. Thus, it makes sense to restrict $\Dual{}{n}$ to the full subcategory 
$\mdef{\Mod_\OR^{[0,n]}} \sseq \Mod_\OR^\cn$ 
of $n$-truncated (connective) $\OR$-modules, on which it is characterized by the property
\[
    \pi_k(\Dual{M}{n}) \:\simeq\: 
    \hom_\Ab(\pi_{n-k}M,\QQ_p/\ZZ_p),\qquad 
    k= 0,\dots,n.
\]

\begin{rem}\label{rem:n_trunc}
    For $\OR\in \calg(\Sp_{(p)}^\cn)$, both the spectrum $\Dual{\OR}{n}$ and the $\infty$-category $\Mod_\OR^{[0,n]}$ depend only on the $n$-truncation of $\OR$, 
    so we might as well assume that $\OR$ is $n$-truncated whenever it is convenient to do so. 
\end{rem}

When further restricting to the full subcategory 
$\mdef{\Modfin{\OR}{n}} \sseq \Mod_\OR^{[0,n]}$ 
of \tdef{$[0,n]$-finite} $\OR$-modules (i.e., connective $n$-truncated $\pi$-finite $\OR$-modules), the functor $\Dual{}{n}$ becomes a contravariant self-equivalence  
\[
    \Dual{}{n}\colon
    \Modfin{\OR}{n} \iso 
    (\Modfin{\OR}{n})^\op.
\]
Indeed, every module in $\Modfin{\OR}{n}$ is $\pi$-finite and $p$-local, hence its homotopy groups are finite $p$-groups, and the claim follows from the corresponding claim for the functor $\hom(-,\QQ_p/\ZZ_p)$. 

\begin{rem}
    In principle, we could work with the non-$p$-local Pontryagin dual $\hom(-,\QQ/\ZZ)$ and general (i.e., not necessarily $p$-local) connective ring spectra $\OR$. 
    However, in practice, the choice of the shift $n$ will match the semiadditive height of a given higher semiadditive $\infty$-category $\cC$, at the prime $p$ (introduced in \cite{AmbiHeight}). 
    As $\cC$ will usually have different heights at different primes, it will not make much sense to have a fixed shift $n$. instead, we observe that a connective $\pi$-finite spectrum $M$ decomposes as a direct sum of its $p$-localizations
    \(
        M \simeq \bigoplus_{p\text{ prime}} M_{(p)}.
    \)
    So for every vector $\vec{n}= (n_p)$ of integers, we can define the $\vec{n}$-shifted Brown--Comenetz dual of $M\in \Mod_\OR^{\pfin{\pi}}$ by
    \[
        \Dual{M}{\vec{n}}:= \bigoplus_{p\text{ prime}} \Dual{M_{(p)}}{n_p}
        \qin \Mod_\OR^{\pfin{\pi}}.
    \] 
    We chose to work $p$-locally in this paper to make things easier, but essentially everything can be generalized to this ``global'' setting. We note that for a stable $\infty$-category, the semiadditive height can be non-zero for at most one prime. On the other hand, there are interesting non-stable $\infty$-categories, such as those arising via categorification, for which the vector of semiadditive heights can be more complicated. 
\end{rem}

\subsubsection{Pre-orientations} 

Recall that the construction of the discrete Fourier transform for an $m$-torsion abelian group $M$ depends on a choice of an $m$-th root of unity. We can view the condition that $M$ is $m$-torsion as the existence of a (necessarily unique) structure of $\ZZ/m$-module on $M$. Analogously, for a connective ring spectrum $\OR$, the higher Fourier transform in an  $\infty$-category $\cC$, for  connective $\OR$-modules, requires a choice of auxiliary data. 

\begin{defn} \label{def:pre_orientations}
    Let $\cC \in \calg(\cat_\infty)$ and $\OR\in \calg(\Sp_{(p)}^\cn)$. For every $S \in \calg(\cC)$, the space of \tdef{$\OR$-pre-orientations} of height $n$ of $S$ is defined and denoted as follows:
    \[
        \mdef{\POr{\OR}{S;\cC}{n}}:= \Map_{\Sp^\cn}(\Dual{\OR}{n},S^\times),
    \]
    where $S^{\times}$ denotes the spectrum of units of $S$. An $\OR$ pre-orientation of $\cC$ is a pre-orientation of $\one_\cC$ and the space of such is denoted by $\POr{\OR}{\cC}{n}$. If $\OR$ is clear from contect, we will also refer to this data simply as a pre-orientation of $\cC$.
\end{defn}

\begin{example}\label{ex:rootsofunity}
    We have a canonical identification 
    $\Dual{\ZZ/p^r}{n}\simeq \Sigma^n \ZZ/p^r.$ 
    Hence, the notion of a $\ZZ/p^r$-pre-orientation of height $n$ identifies with that of a height $n$ root of unity of order $p^r$, in the sense of \cite[Defnition 4.2]{carmeli2021chromatic}. More precisely, for $\cC \in \calg(\cat_\infty)$ and $S\in \calg(\cC)$ we have 
    \[
        \POr{\ZZ/p^r}{S;\cC}{n} \simeq 
        \Map_\Sp(\Sigma^n \ZZ/p^r,S^\times) \simeq  
        \roots[S]{p^r}{n},
    \] 
    where $\roots[S]{p^r}{n}$ denotes the spectrum of height $n$ roots of unity of order $p^r$ in $S$. Note that, in particular, $\ZZ/p^r$-pre-orientations of height $0$ of a field $\FF$ correspond to $p^r$-th roots of unity in $\FF$.
\end{example}

\begin{rem}\label{rem:luriepreor}
    For a $p$-divisible group $\GG$ over an $\EE_\infty$-ring spectrum $R$, Lurie \cite[Definition 2.1.4]{Lurie_Ell3} defines a pre-orientation of $\GG$ as a map of $\ZZ$-modules $\Sigma\QQ_p/\ZZ_p \to \GG(R)$. 
    Since for the multiplicative group $\GG_m$ we have $\GG_m(E)\simeq \hom(\ZZ,R^\times)$, a pre-orientation of $\GG_m$ is the same thing as a map of spectra 
    \[
        \Sigma I_{\ZZ_{(p)}}\simeq \Sigma \QQ_p/\ZZ_p \to R^\times.
    \]
    Thus, we can identify the space of $\ZZ_{(p)}$-pre-orientations of height $1$ of $R$, with the space of pre-orientations of the $p$-divisible group $\GG_m$ over $R$. In general, we can view $\POr{\OR}{\cC}{n}$ as an ``$\OR$-linear, higher height analogue'' of pre-orientations for the multiplicative $p$-divisible group over the unit $\one$.
\end{rem}

There is an adjoint way of viewing pre-orientations. When $\cC$ is presentable, the functor
\[
    (-)^\times \colon \calg(\cC) \too \Sp^\cn
\] 
admits a left adjoint 
\[
    \mdef{\one[-]}\colon \Sp^\cn \too \calg(\cC),
\]
which takes a connective spectrum $M$ to its group-algebra $\one[M]$. 
Hence, we can identify a pre-orientation $\omega \colon \Dual{\OR}{n}\to \one^\times$ with an augmentation
\[
    \mdef{\varepsilon_\omega}\colon 
    \one[\Dual{\OR}{n}]\too 
    \one \qin \calg(\cC).
\]

Thus, $\OR$-pre-orientations of height $n$ for commutative algebras in $\cC$ are corepresented by $\one[\Dual{\OR}{n}]$.

\begin{example}
    Let $\cC$ be the category of complex vector spaces and let $R=\ZZ/p^r$. For an $m$-th root of unity $\omega \in \CC^\times$, the augmentation $\varepsilon_\omega\colon \CC[\ZZ/p^r]\to \CC$ is given by  
    \[
        \varepsilon_\omega\left (\sum_{k\in \ZZ/p^r} a_k [k]\right)= \sum_{k\in \ZZ/p^r} a_k \omega^k \qin \CC.
    \] 
    Hence, we can view $\varepsilon_\omega$ in general as the map which gives $\one$-valued ``exponential sums''.
\end{example}

\subsection{The Fourier transform}\label{ssec:fourier}

\subsubsection{Construction} 

Given an $\OR$-pre-orientation $\omega$ of height $n$, we now construct for every connective $\OR$-module $M$ a map of commutative algebras
$\Four_\omega \colon \one[M] \to \one^{\und{\Dual{M}{n}}}.$
In fact, we show that the space of such maps, that are natural in $M$, is \textit{parameterized} by pre-orientations.

\begin{prop}\label{Four_is_POr}
    Let $\cC \in \calg(\Prl)$ and let $\OR\in \calg(\Sp_{(p)}^\cn)$. The space of natural transformations $\one[-] \to \one^{\und{\Dual{(-)}{n}}}$ of functors $\Mod_\OR^{[0,n]} \to \calg(\cC)$ is naturally isomorphic to $\POr{\OR}{\cC}{n}$.
\end{prop}
\begin{proof}
    The functor that takes an $\OR$-module $M$ to the algebra $\charac{M}{n}$ is the right Kan extension of the functor $\pt \to \calg(\cC)$ corresponding to the object $\one$, along the functor $\pt \to \Mod_\OR^{[0,n]}$ corresponding to the object $\Dual{\OR}{n}$:
    \[\begin{tikzcd}
    	& {\pt} \\
    	{\Mod_\OR^{[0,n]}} && {\calg(\cC).}
        \arrow["\one", from=1-2, to=2-3]
        \arrow["{\Dual{\OR}{n}}"', from=1-2, to=2-1]
        \arrow["{\one^{\und{\Dual{(-)}{n}}}}"', dashed, from=2-1, to=2-3]
    \end{tikzcd}\]
    It follows that the space of natural transformations of the form
    $\one[-]\to \charac{-}{n}$
    is isomorphic to the space of augmentations of the algebra $\one[\Dual{\OR}{n}]$, which is the same as the space $\POr{\OR}{\cC}{n}$ of $\OR$-pre-orientations of height $n$. 
\end{proof}

\begin{defn}\label{Def_Four}
    Let $\cC \in \calg(\Prl)$ and let $\OR\in \calg(\Sp_{(p)}^\cn)$. For $\omega \in \POr{\OR}{\cC}{n}$, the associated \tdef{Fourier transform} is the natural transformation 
    \[
        \mdef{\Four_\omega} \colon \one[-] \too \one^{\und{\Dual{(-)}{n}}}
    \]
    of functors $\Mod_\OR^{[0,n]} \to \calg(\cC)$ corresponding to $\omega$ by \Cref{Four_is_POr}. We shall occasionally consider the Fourier transform as a functor $\Four_\omega \colon \Mod_\OR^{[0,n]} \to \calg(\cC)^{[1]}.$
\end{defn}

A couple of remarks on the definition of the Fourier transform are in order. 

\begin{rem}\label{Four_Explicit}
By unwinding the definition of the Kan extension, the passage from a pre-orientation to the associated Fourier transform and vice versa can be made more explicit. First, by evaluating $\Four_\omega$ on $M= \Dual{\OR}{n}$, we recover the augmentation that corresponds to $\omega \colon \Dual{\OR}{n} \to \one^\times$, as the composition
    \[
        \varepsilon_\omega \colon 
\one[\Dual{\OR}{n}] \oto{\:\Four_\omega\:}
\one^{\und{\Dual{{\Dual{\OR}{n}}}{n}}}=
        \one^{\Map_R(\Dual{\OR}{n},\Dual{\OR}{n})} \oto{\ev_{\Id}}
        \one.
    \]
    Conversely, let $\omega \colon \Dual{\OR}{n} \to \one^\times$ and let $M\in \Mod_\OR^{[0,n]}$. For every character $\varphi \colon M \to \Dual{\OR}{n}$, the composition $\omega \circ \varphi \colon M \to \one^\times$
    corresponds to an augmentation 
    $\varepsilon_{\omega \circ \varphi} \colon \one[M] \to \one.$ 
    These assemble into a map 
    \[
\und{\Dual{M}{n}}= \Map_R(M,\Dual{\OR}{n}) \too \Map_{\calg(\cC)}(\one[M],\one),
    \]
    which is the same data as the map 
    \(
        \Four_\omega\colon \one[M] \to \one^{\und{\Dual{M}{n}}}.
    \)
\end{rem}

\begin{rem}\label{Four_Pi_Finite}
    In \Cref{Four_is_POr}, one could replace $\Mod_\OR^{[0,n]}$ with a larger $\infty$-category such as $\Mod_R^{\cn}$ and thus define the Fourier transform for all connective $\OR$-modules. Though this might sometimes be technically convenient, we shall be interested primarily in the Fourier transform only for $n$-truncated modules. 
    In fact, eventually, we shall be interested only in the further restriction of the Fourier transform to the full subcategory 
    $\Modfin{\OR}{n} \sseq \Mod_\OR^{[0,n]}$
    of \textit{$[0,n]$-finite} modules.
    However, for establishing some of its formal properties, it is useful to use the characterization provided by \Cref{Four_is_POr}. Having said that, if $\OR$ itself happens to be $[0,n]$-finite (e.g., $\OR=\ZZ/p^r$), then 
    $\Dual{\OR}{n}\in \Modfin{\OR}{n}$ and one could also replace $\Mod_\OR^{[0,n]}$ with $\Modfin{\OR}{n}$.
\end{rem}

\Cref{Def_Four} generalizes the classical discrete Fourier transform in the following sense:
\begin{example}
    Let $\cC$ be the category of $\CC$-vector spaces. For an $m$-th root of unity $\omega \in \CC^\times$ and an $m$-torsion abelian group $M$, the map $\Four_\omega\colon \CC[M]\to \CC^{M^*}$ is given on $x\in M\subseteq \CC[M]$ by 
    \[
        \Four_\omega(x)\colon \varphi \mapsto \omega^{\varphi(x)}.
    \]
In particular, if $M\simeq \ZZ/p^r$ and we identify $M^*$ with $M$ using the generator $1\in M$, we recover the classical (inverse) Fourier matrix $\Four_\omega(k)(\ell)= \omega^{k\cdot \ell}$. 
\end{example}

The Fourier transform can be enhanced to take into account a bit more structure. A character
$\varphi \colon M \to \Dual{\OR}{n}$ induces a map $\omega \circ \varphi \colon M \to \one^\times$, which corresponds to an augmentation 
\[
    \varepsilon_{\omega \circ \varphi} \colon \one[M] \too \one
    \qin \calg(\cC).
\] 
On the other hand, thinking of $\varphi$ as a point in $\und{\Dual{M}{n}}$ yields an augmentation 
\[  
    \ev_{\varphi}\colon {\one^{\und{\Dual{M}{n}}}} \too \one
    \qin \calg(\cC).
\]
The Fourier transform intertwines these augmentations.

\begin{prop}\label{Four_Aug}
    Let $\cC \in \calg(\Prl)$ and let $\OR\in \calg(\Sp_{(p)}^\cn)$. For every $\omega \in \POr{\OR}{\cC}{n}$ and $\varphi \colon M \to \Dual{\OR}{n}$ in $\Mod_\OR^{[0,n]}$, the following diagram commutes:
    \[
        \begin{tikzcd}
    	{\one[M]} && {\one^{\und{\Dual{M}{n}}}} \\
    	& \one.
        \arrow["{\Four_\omega}", from=1-1, to=1-3]
        \arrow["{\varepsilon_{\omega \circ \varphi}}"'{pos=0.4}, from=1-1, to=2-2]
        \arrow["{\ev_\varphi}"{pos=0.4}, from=1-3, to=2-2]
        \end{tikzcd}   
    \]
    Moreover, $\Four_\omega$ promotes uniquely to a natural transformation of functors 
    \[
        (\Mod_\OR^{[0,n]})_{/\Dual{\OR}{n}} \too 
        \calg(\cC)_{/\one},
    \]
    whose component at $\varphi\colon M \to \Dual{\OR}{n}$ is given by the diagram above.
\end{prop}

\begin{proof}
    One the one hand, the Fourier transform
    \(
        \Four_\omega\colon \one[M] \to  \charac{M}{n},
    \)
    whiskered by the forgetful functor 
    $(\Mod_\OR^{[0,n]})_{/\Dual{\OR}{n}} \to \Mod_\OR^{[0,n]}$, can be viewed as a functor
    \[
        (\Mod_\OR^{[0,n]})_{/\Dual{\OR}{n}} \too
        \calg(\cC)^{[1]}.
    \]
    On the other hand, the construction taking $\varphi\colon M \to \Dual{\OR}{n}$ to $\ev_\varphi \colon \one^{\und{\Dual{M}{n}}} \to \one$ assembles into a functor 
    \[
        (\Mod_\OR^{[0,n]})_{/\Dual{\OR}{n}} \longrightarrow 
        \calg(\cC)_{/\one}.
    \]
    Since the forgetful functor $\calg(\cC)_{/\one} \to \calg(\cC)$ is a right fibration, we have a pullback square of $\infty$-categories
    \[\begin{tikzcd}
    	{(\calg(\cC)_{/\one})^{[1]}} & {(\calg(\cC))^{[1]}} \\
    	{(\calg(\cC)_{/\one})} & {\calg(\cC).}
        \arrow["{\ev_1}"', from=1-1, to=2-1]
        \arrow["{\ev_1}", from=1-2, to=2-2]
        \arrow[from=1-1, to=1-2]
        \arrow[from=2-1, to=2-2]
        \arrow["\lrcorner"{anchor=center, pos=0.125}, draw=none, from=1-1, to=2-2]
    \end{tikzcd}\]

    It follows that the two functors above, which agree after projecting to $\calg(\cC)$, uniquely lift to a functor
    \[
        (\Mod_\OR^{[0,n]})_{/\Dual{\OR}{n}} \to (\calg(\cC)_{/\one})^{[1]}.
    \]
It remains to show that for every $M\in \Mod_\OR^{[0,n]}$, the map $\ev_{\varphi } \circ \Four_\omega$ is homotopic to $\varepsilon_{\omega \circ \varphi}$. By the naturality of $\Four_\omega$, we have a commutative diagram:
    \[\begin{tikzcd}
    	{ \one[M]} && {\one^{\und{\Dual{M}{n}}}} \\
    	{\one[\Dual{\OR}{n}]} && {\one^{\und{\Dual{{\Dual{\OR}{n}}}{n}}}} && { \one.}
\arrow["\varphi"', from=1-1, to=2-1]
\arrow["\varphi", from=1-3, to=2-3]
\arrow["{\Four_\omega}", from=1-1, to=1-3]
\arrow["{\Four_{\omega}}", from=2-1, to=2-3]
\arrow["{\ev_\Id}", from=2-3, to=2-5]
\arrow["{\ev_\varphi}"{pos=0.6}, curve={height=-12pt}, from=1-3, to=2-5]
    \end{tikzcd}\]    
By \cref{Four_Explicit}, the composition of the two bottom maps is $\varepsilon_\omega$ and $\varepsilon_\omega \circ \varphi= \varepsilon_{\omega \circ \varphi}$, so the claim follows.
\end{proof}

\subsubsection{Functoriality}

We now explain how the Fourier transform $\Four$ is functorial in the various arguments it depends on. First, given a map $f\colon \ORone\to \ORtwo$ in $\calg(\Sp_{(p)}^\cn)$, we can regard an $S$-module as an $\OR$-module by restriction of scalars along $f$, which we denote by 
$f_* \colon \Mod_\ORtwo^{[0,n]} \to \Mod_\ORone^{[0,n]}.$ 
Similarly, given an $\OR$-pre-orientation $\omega \colon \Dual{\ORone}{n} \to \one^\times$, by pre-composition with the $\OR$-module map $\Dual{\ORtwo}{n} \to \Dual{\OR}{n}$, we obtain an $\ORtwo$-pre-orientation, which we denote by $f_*\omega \colon \Dual{\ORtwo}{n} \to \one^\times$. 

\begin{prop} \label{or_funct_ring}
    Let $\cC\in  \calg(\Prl)$, let $f\colon \ORone\to \ORtwo$ in $\calg(\Sp_{(p)}^\cn)$, and let $\omega \colon \Dual{\ORone}{n}\to \one^\times$ be an $\ORone$-pre-orientation. There is a canonical isomorphism of natural transformations
    \[
        \Four_{f_*\omega} \:\simeq\: \Four_{\omega} f_*
    \]
    of functors $\Mod_\ORtwo^{[0,n]} \to \calg(\cC).$
\end{prop} 

\begin{proof}
    First, observe that for every $M\in \Mod^\cn_\ORtwo$ we have
    \[
        \Map_\ORone(f_* M, \Dual{\ORone}{n}) \simeq
        \und{\Dual{M}{n}} \simeq
        \Map_\ORtwo( M, \Dual{\ORtwo}{n}).
    \]
    Hence, both $\Four_{f_*\omega}$ and $\Four_{\omega} f_*$ are maps $\one[-] \to \one^{\Dual{(-)}{n}}$ of functors $\Mod_\ORtwo^{[0,n]} \to \calg(\cC)$. Consequently, by \Cref{Four_is_POr}, it suffices to show that $\Four_{\omega} f_*$ is classified by the pre-orientation $f_*\omega$. For this, it suffices to identify the corresponding augmentations. By \Cref{Four_Explicit}, the augmentation associated with $\Four_{\omega} f_*$ is given by
    \[
        \varepsilon\colon \one[\Dual{\ORtwo}{n}] \oto{\Four_\omega} \one^{\Map_\ORone(\Dual{\ORtwo}{n},\Dual{\ORone}{n})} \iso \one^{\Map_\ORtwo(\Dual{\ORtwo}{n},\Dual{\ORtwo}{n})} \oto{\ev_\Id}
        \one.
    \]
    The isomorphism $\Map_\ORone(\Dual{\ORtwo}{n},\Dual{\ORone}{n})\iso \Map_\ORtwo(\Dual{\ORtwo}{n},\Dual{\ORtwo}{n})$ takes the map $\Dual{f}{n}\colon \Dual{\ORtwo}{n}\to \Dual{\ORone}{n}$, dual to $f\colon \ORone\to \ORtwo$, to the identity map of $\Dual{\ORtwo}{n}$. Hence, we can identify $\varepsilon$ with the composition
    \[
        \one[\Dual{\ORtwo}{n}] \oto{\Four_\omega} \one^{\Map_\ORone(\Dual{\ORtwo}{n},\Dual{\ORone}{n})}  \oto{\ev_{(\Dual{f}{n})}} \one.
    \]
    By \Cref{Four_Aug}, this composition is homotopic to the composition
    \[
        \one[\Dual{\ORtwo}{n}] \oto{\Dual{f}{n}}
        \one[\Dual{\ORone}{n}] \oto{\varepsilon_\omega} \one,
    \]
    which is by definition $\varepsilon_{f_*\omega}$. 
\end{proof}

\begin{example}
    Let $\cC\in \calg(\Prl)$, and let  $\omega \in \roots{p^{r+1}}{n}= \POr{\ZZ/p^{r+1}}{\cC}{n}$.
    The Pontryagin dual of the quotient map $\ZZ/p^{r+1}\onto \ZZ/p^r$ identifies with the inclusion $\ZZ/p^r\into \ZZ/p^{r+1}$. Hence, 
    if $M$ is a $\ZZ/p^r$-module, then by \Cref{or_funct_ring} we can identify the map
    \[
        \Four_{\omega^p}\colon \one[M]\too \one^{\und{\Dual{M}{n}}},
    \]
    with the map 
    \[
        \Four_\omega\colon \one[M]\too \one^{\und{\Dual{M}{n}}},
    \]
    obtained by regarding $M$ as a $\ZZ/p^{r+1}$-module.
\end{example}

Next, we discuss the functoriality of the Fourier transform with respect to the ambient $\infty$-category. A symmetric monoidal functor $F\colon \cC \to \cD$ induces a functor $\calg(\cC) \to \calg(\cD).$ Furthermore, given a pre-orientation $\omega \colon \Dual{\OR}{n} \to \one_\cC^\times$ in $\cC$, by post-composition with the map $F^\times \colon \one_\cC^\times \to \one_\cD^\times$, we obtain a pre-orientation in $\cD$, which we denote by $F(\omega) \colon \Dual{\OR}{n} \to \one_\cD^\times$. 

\begin{prop} \label{Four_Funct_C}
    Let $F\colon \cC\to \cD$ in $\calg(\Prl)$, let $\OR\in \calg(\Sp_{(p)}^\cn)$, and let $\omega \colon \Dual{\OR}{n}\to \one_\cC^\times$ be an $\OR$-pre-orientation of $\cC$. The following diagram commutes
    \[\begin{tikzcd}
    	{ F(\one_\cC[M])} && {F(\one_\cC^{\und{\Dual{M}{n}}})} \\
    	{ \one_\cD[M]} && {\one_\cD^{\und{\Dual{M}{n}}}}
\arrow[from=2-1, to=1-1]
\arrow[from=1-3, to=2-3]
\arrow["{F(\Four_\omega)}", from=1-1, to=1-3]
\arrow["{\Four_{F(\omega)}}", from=2-1, to=2-3]
    \end{tikzcd}\]
    naturally in $M \in \Mod_\OR^{[0,n]}$. Here, the vertical maps are induced by functoriality.
\end{prop}

\begin{proof}    
By \Cref{Four_is_POr}, it suffices to show that the following diagram commutes:
    \[\begin{tikzcd}
    	{ F(\one_\cC[\Dual{\OR}{n}])} && {F(\one_\cC)} \\
    	{ \one_\cD[\Dual{\OR}{n}]} && {\one_\cD}
\arrow[from=2-1, to=1-1]
\arrow[from=1-3, to=2-3]
\arrow["{F(\varepsilon_\omega)}", from=1-1, to=1-3]
\arrow["{\varepsilon_{F(\omega)}}", from=2-1, to=2-3]
    \end{tikzcd}\]
    Namely, that the augmentation $\varepsilon_{F(\omega)}$ is homotopic to the composition 
    \[
\one_\cD[\Dual{\OR}{n}] \simeq F(\one_\cC[\Dual{\OR}{n}]) \oto{\:F\circ \varepsilon_\omega\:} F(\one_\cC) \simeq \one_\cD.
    \]
    Via the adjunction $\one_\cD[-]\dashv (-)^\times$, this composition corresponds to the map 
    \[
\Dual{\OR}{n} \oto{\:\:\omega\:\:} \one_\cC^\times \oto{F^\times} \one_\cD^\times,
    \]
    which is by definition $F(\omega)$.
\end{proof}

\begin{rem}\label{Four_Iso_C}
    In the situation of \Cref{Four_Funct_C}, the left vertical map is always an isomorphism, since by assumption $F$ preserves colimits. If we additionally require $F$ to preserve $\und{\Dual{M}{n}}$-shaped limits, then the right vertical map is also an isomorphism and hence $F(\Four_\omega)$ identifies with $\Four_{F(\omega)}$. 
\end{rem}

We have one more type of functoriality to consider, which is given by rescaling the pre-orientation. For $\OR\in\calg(\Sp_{(p)}^\cn)$, the underlying space $\und{\OR}$ with its multiplicative monoid structure acts naturally on every $\OR$-module $M$. Thus, given $\cC \in \calg(\Prl)$ and a pre-orientation $\omega \in \POr{\OR}{\cC}{n}$, the naturality of the Fourier transform 
$\Four_\omega \colon \one[-] \to \one^{\und{\Dual{-}{n}}}$
makes it $\und{\OR}$-equivariant. In particular, for every $r \in \und{\OR}$, we have a commutative square 
\[\begin{tikzcd}
	{\one[M]} && {\one^{\und{\Dual{M}{n}}}} \\
	{\one[M]} && {\one^{\und{\Dual{M}{n}}}}
\arrow["{r \cdot}"', from=1-1, to=2-1]
\arrow["{r \cdot}", from=1-3, to=2-3]
\arrow["{\Four_\omega}", from=2-1, to=2-3]
\arrow["{\Four_\omega}", from=1-1, to=1-3]
\end{tikzcd}\]
natural in $M\in \Mod_\OR^{[0,n]}$. Now, the diagonal map given by (either of) the compositions is also a natural transformation of the form  
$\one[-] \to \one^{\und{\Dual{-}{n}}}$
and hence, by \Cref{Four_is_POr}, is classified by the associated augmentation
\[
\Dual{\OR}{n} \oto{\:\:r\cdot\:\:}\Dual{\OR}{n} \oto{\:\:\omega\:\:} \one^\times.
\]

\begin{defn}\label{def:Por_Scaling} 
    For $r\in \und{\OR}$, we denote by 
    \[
        \mdef{(-)^r} \colon 
        \POr{\OR}{\cC}{n} \too \POr{\OR}{\cC}{n}
    \]
        the pre-composition with
    \(
\Dual{\OR}{n}\oto{\:r\cdot\:} \Dual{\OR}{n}.
    \) 
\end{defn}

\begin{example}
    For $R= \ZZ/p^r$ and $n=0$, the operation $\omega \mapsto \omega^k$ for $k\in \ZZ/p^r$ is given by raising the $p^r$-th root of unity $\omega$ to the $k$-th power. 
\end{example}

For ease of reference we record the following immediate consequence of the above discussion:

\begin{prop} \label{Por_Scaling}
    Let $\cC\in \calg(\Prl)$, let $\OR\in \calg(\Sp_{(p)}^\cn)$, and let $\omega \in \POr{\OR}{\cC}{n}$. For every $r\in\und{\OR}$, we have a commutative triangle 
    \[\begin{tikzcd}
    	{\one[M]} && {\one[M]} \\
    	& {\one^{\und{\Dual{M}{n}}}}
        \arrow["r \cdot", from=1-1, to=1-3]
        \arrow["{\Four_{\omega^r}}"', from=1-1, to=2-2]
        \arrow["{\Four_{\omega}}", from=1-3, to=2-2]
    \end{tikzcd}\]    
    in $\calg(\cC)$ naturally in $M\in\Mod_\OR^{[0,n]}$.
\end{prop}

\begin{proof}
    This follows immediately from the definition of $\omega^r$. 
\end{proof}

\subsubsection{Duality} 

We conclude this section with a discussion of the symmetry of the Fourier transform with respect to monoidal duality. Given a presentably symmetric monoidal $\infty$-category $\cC$ and $\omega \in \POr{\OR}{\cC}{n}$, we have for every $M\in \Modfin{\OR}{n}$ a Fourier transform map
\[
    (\Four_\omega)_{M} \colon \one[M] \too
    \one^{\und{\Dual{M}{n}}}.
\]
Under the canonical identification $M \iso \Dual{({\Dual{M}{n}})}{n}$, we also have the Fourier transform map
\[
    (\Four_\omega)_{\Dual{M}{n}} \colon \one[\Dual{M}{n}] \too
    \one^{\und{M}}.
\]
We shall now show that when $\cC$ is $n$-semiadditive, $(\Four_\omega)_{\Dual{M}{n}}$ coincides with the monoidal \textit{dual} of $(\Four_\omega)_{M}$. In fact, we show more generally that even when $\cC$ is not assumed to be $n$-semiadditive, the two maps are the \textit{transpose} of one another in the following sense:

\begin{defn}
    Let $\cC \in \calg(\Prl)$.
    \begin{enumerate}
        \item For $X\in \cC$, we define the \tdef{(weak) dual} to be
        \[
            \mdef{X^\vee} := \hom(X,\one) \qin \cC.
        \]
        A map $f\colon X  \to Y^\vee$ can be identified with a map $X\otimes Y \to \one$.
        
        \item For a map $f\colon X\to Y^\vee$ in $\cC$, we define the \tdef{transpose} of $f$ to be the map 
        \[
            \mdef{f^t} \colon Y \too X^\vee \qin \cC,
        \]
        corresponding to $f$ via the isomorphism
        \[
            \Map_{\cC}(X,Y^\vee)\simeq \Map_{\cC}(X\otimes Y, \one) \simeq \Map_{\cC}(Y,X^\vee).
        \]
    \end{enumerate}
\end{defn}

Unwinding the definitions, $f^t$ is given by the composition
\[
f^t\colon Y \too (Y^\vee)^\vee \oto{\:\:f^\vee} X^\vee,
\]
where the first map is the canonical map from an object to its double dual. In particular, if $Y$ is dualizable, then the transpose $f^t$ coincides with the dual $f^\vee$ under the canonical identification $Y \iso Y^{\vee\vee}$. One can think of the transpose as a useful substitute for the dual when $Y$ is non-dualizable.

\begin{rem}
    In general, even if $f$ is an isomorphism, $f^t$ might not be. For example, for $V$ an infinite dimensional vector space over a field $k$, the transpose of the identity map $V^\vee \to V^\vee$ is the non-isomorphic embedding  $V\into (V^\vee)^\vee$ of $V$ into its double dual. This deficiency however disappears if we assume that $X$ and $Y$ are dualizable. 
\end{rem}

\begin{prop} \label{Four_Duality}
    Let $\cC \in \calg(\Prl)$, let $\OR\in \calg(\Sp_{(p)}^\cn)$ and let $\omega \in \POr{\OR}{\cC}{n}$. We have a natural isomorphism 
    \[
        (\Four_\omega)_{\Dual{M}{n}} \:\simeq\: (\Four_\omega)_{M}^t
    \] 
    as natural transformations of functors $\Mod_\OR^{[0,n]} \to \cC$. 
\end{prop}

\begin{proof}
    We have 
    \[
        \Map_\cC(\one[M],\one[\Dual{M}{n}]^\vee) \simeq 
        \Map_{\cC}(\one[M]\otimes \one[\Dual{M}{n}],\one) \simeq 
    \]
    \[  
        \Map_{\cC}(\one[M \times \Dual{M}{n} ],\one) \simeq
        \Map_{\Spc}(\und{M}\times \und{\Dual{M}{n}},\und{\one}).
    \]
    By construction, via this identification, the map $\Four_\omega \colon \one[M]\to \one[\Dual{M}{n}]^\vee$ corresponds to the map of spaces 
    \[
        \und{M}\times \und{\Dual{M}{n}} \oto{\ev} \und{\Dual{\OR}{n}} \oto{\omega} \und{\one}. 
    \]
    By the definition of the transpose, we deduce that $\Four_\omega^t$ corresponds to the map of spaces 
    \[
        \und{\Dual{M}{n}} \times \und{M} \iso 
        \und{M} \times \und{\Dual{M}{n}} \oto{\ev} 
        \und{\Dual{\OR}{n}} \oto{\omega} \und{\one}.
    \]
    Via the isomorphism $M \iso \Dual{(\Dual{M}{n})}{n}$, this map identifies with the composition 
    \[
        \und{\Dual{M}{n}} \times \und{\Dual{(\Dual{M}{n})}{n}} \oto{\ev} \und{\Dual{\OR}{n}}\oto{\omega} \und{\one}, 
    \]
    which is the map corresponding to the morphism 
    $\Four_{\omega}\colon \one[\Dual{M}{n}]\to \one^{\und{M}}$.   
\end{proof}

\begin{cor}\label{cor:Four_Duality}
In the situation of \cref{Four_Duality}, if $\cC$ is $n$-semiadditive and $M$ is $[0,n]$-finite, then 
    \[
(\Four_\omega)_{\Dual{M}{n}} \:\simeq\: (\Four_\omega)_{M}^\vee.
    \]
\end{cor}
\begin{proof}
    If $\cC$ is $n$-semiadditive then the source and the target of both maps are dualizable \cite[Proposition 3.3.6]{TeleAmbi} and hence the maps $(\Four_\omega)_{M}$ and $(\Four_\omega)_{\Dual{M}{n}}$ are dual to one another.
\end{proof}

\subsection{co-Multiplicative properties}\label{ssec:co-mult}

So far, for a presentably symmetric monoidal $\infty$-category $\cC$ we have constructed the Fourier transform 
\[
    \Four_\omega\colon \one[M]\longrightarrow \one^{\und{\Dual{M}{n}}}
\]
only as a map of \textit{commutative algebras} in $\cC$. However, as in the classical case, under suitable finiteness and semiadditivity assumptions, both sides also admit natural co-multiplication and co-unit maps making them into \textit{Hopf algebras}, and the Fourier transform preserves this structure. 

\subsubsection{Hopf algebras}

We begin by recalling the definitions of coalgebras, bialgebras and Hopf algebras. First, given $\cC \in \calg(\cat_\infty)$, the $\infty$-category of cocommutative \tdef{coalgebras} in $\cC$ is given by
\[
    \mdef{\cocalg(\cC)} := \calg(\cC^\op)^\op.
\]
\begin{defn}
    Let $\cC \in \calg(\cat_\infty)$. We define the $\infty$-category of \tdef{bialgebras} in $\cC$ by
    \[
        \mdef{\bicalg(\cC)}:= \cocalg(\calg(\cC)),
    \]
    where $\calg(\cC)$ has the induced symmetric monoidal structure from $\cC$.
\end{defn}

Since the induced monoidal structure on $\calg(\cC)$ is coCartesian, we have equivalently
\[
    \bicalg(\cC) \simeq 
    (\CMon(\calg(\cC)^\op))^\op.
\]

\begin{defn}\label{Def_Hopf}
    Let $\cC \in \calg(\cat_\infty)$. We define the $\infty$-category of \tdef{Hopf algebras} in $\cC$ to be the following full subcategory
    \[
        \mdef{\Hopf(\cC)} := 
        (\CMon^\gp(\calg(\cC)^\op))^\op
    \]
   of $\bicalg(\cC)$.
\end{defn}

\begin{rem}
    By the Yoneda lemma, \Cref{Def_Hopf} is equivalent to \cite[Definition 3.9.7]{Lurie_Ell1} in terms of the functor $\Spec^\cC$.
    One can think of $\Hopf(\cC)$ as opposite to the category of commutative group objects in the $\infty$-category $\calg(\cC)^\op$ of 
    ``affine $\cC$-schemes''. 
\end{rem}

We now have the following consequence of \cite[Corollary 2.10]{GepUniv}:

\begin{prop}\label{Add_Hopf_Lift}
    Let $\cC \in \calg(\cat_\infty)$. For every additive $\infty$-category $\cE$, composition with the forgetful functor $\Hopf(\cC) \to \calg(\cC)$ induces an equivalence
    \[
        \Fun^\sqcup(\cE,\Hopf(\cC)) \iso \Fun^\sqcup(\cE,\calg(\cC)),
    \]
    where $\Fun^\sqcup$ denotes the $\infty$-category of co-product preserving functors. 
\end{prop}

\begin{proof}
    This follows by applying the opposite of \cite[Corollary 2.10]{GepUniv} to the $\infty$-category $\calg(\cC)$.  
\end{proof}

\begin{cor}\label{Fourier_Hopf}
    Let $\cC \in \calg(\Prl^{\sad{n}})$ and let $\OR\in \calg(\Sp_{(p)}^\cn)$. For every $\omega \in \POr{\OR}{\cC}{n}$, the Fourier transform 
    \[
        \Four_\omega \colon \one[-] \too \one^{\und{\Dual{(-)}{n}}}
    \] 
    lifts uniquely to a natural transformation of functors 
    $\Modfin{\OR}{n} \to \Hopf(\cC)$, along the forgetful functor $\Hopf(\cC) \to \calg(\cC)$.
\end{cor}

\begin{proof}
    The $\infty$-category $\Modfin{\OR}{n}$ is additive. The functor $\one[-]$ is coproduct preserving as a (restriction of a) left adjoint and the functor $\one^{\und{\Dual{(-)}{n}}}$ is coproduct preserving by the $n$-semiadditivity assumption on $\cC$ and \Cref{Kunneth_Ambi}. Thus, the claim follows from \cref{Add_Hopf_Lift}.
\end{proof}

\begin{rem}
    The co-multiplication of $\one[M]$ comes from the diagonal of $M$, while the multiplication uses the addition of $M$. In contrast, the multiplication of $\one^{\und{\Dual{M}{n}}}$ uses only the diagonal map of $\und{\Dual{M}{n}}$, while the co-multiplication uses the addition. 
\end{rem}

\subsubsection{Translation equivariance}

One of the main features of the classical Fourier transform is that it intertwines the translation operations on the function space $\CC^{M^*}$ with multiplication by characters on $\CC[M]$. In other words, the Fourier transform simultaneously diagonalizes the shift operators. 
We now derive a similar result for the higher Fourier transform.

\begin{defn}\label{Def_Translation_Aut}
    Given $\cC \in \calg(\cat_\infty)$ and $H \in \Hopf(\cC)$ with an augmentation 
    $\varepsilon \colon H  \to \one$ as a commutative algebra,
    we define the \tdef{translation automorphism} of $H$ (as a commutative algebra) by the composition
    \[
        \mdef{T_{\varepsilon}} \colon H \oto{\Delta} H\otimes H \oto{\Id\otimes\varepsilon} H\otimes \one= H.
    \]    
\end{defn}

This construction naturally assembles into a functor 
\[
    T_{(-)}\colon \Hopf^\aug(\cC) \to \calg(\cC)^{[1]},
\]
where $\Hopf^\aug(\cC)$ is the category of pairs $(H,\varepsilon)$, formally given as a pullback
\[
\Hopf^\aug(\cC):= \Hopf(\cC)\times_{\calg(\cC)} \calg(\cC)_{/\one}.
\]

\begin{rem}
    In algebro-geometric terms, $\varepsilon \colon H \to \one$ corresponds to a global element of the group scheme $\Spec(H)$, and $T_\varepsilon$ is the translation map by that element from $\Spec(H)$ to itself.
\end{rem}

Using the compatibility of the Fourier transform with augmentations (\cref{Four_Aug}) and Hopf algebra structures (\cref{Fourier_Hopf}), we deduce that it intertwines the corresponding translation automorphisms.
\begin{prop}\label{Translation_Inv}
    Let $\cC \in \calg(\Prl^{\sad{n}})$ and let $\OR\in \calg(\Sp_{(p)}^\cn)$. For every $\omega \in \POr{\OR}{\cC}{n}$, the associated Fourier transform $\Four_\omega$ promotes uniquely to a functor
    \[
        (\Modfin{\OR}{n})_{/\Dual{\OR}{n}} \too 
        \calg(\cC)^{[1]\times[1]},
    \]
    whose component at $\varphi\colon M \to \Dual{\OR}{n}$ is given by
    \[\begin{tikzcd}
	{ \one[M]} && {\one^{\und{\Dual{M}{n}}}} \\
	{\one[M]} && {\one^{\und{\Dual{M}{n}}}.}
    \arrow["{T_{(\varepsilon_{\omega\circ\varphi})}}"', from=1-1, to=2-1]
    \arrow["{\;T_{(\ev_\varphi)}}"{pos=0.6}, from=1-3, to=2-3]
    \arrow["{\Four_\omega}", from=1-1, to=1-3]
    \arrow["{\Four_\omega}", from=2-1, to=2-3]
    \end{tikzcd}\]
\end{prop}
 
\begin{proof}
    We have a commutative diagram of functors
    \[\begin{tikzcd}
    	{(\Modfin{\OR}{n})_{/\Dual{\OR}{n}}} & {(\calg(\cC)_{/\one})^{[1]}} \\
    	{\Hopf(\cC)^{[1]}} & { \calg(\cC)^{[1]},}
\arrow[from=1-1, to=1-2]
\arrow[from=1-2, to=2-2]
\arrow[from=2-1, to=2-2]
\arrow[from=1-1, to=2-1]
    \end{tikzcd}\]
    where the top one is provided by \Cref{Four_Aug}, the bottom one by \Cref{Fourier_Hopf}, and the other two are the canonical forgetful functors (and both compositions are the Fourier transform map). This corresponds to a functor from the top left corner into the pullback of the remaining diagram. Since raising to the power of $[1]$ preserves pullbacks, this is a functor 
    \[
        (\Modfin{\OR}{n})_{/\Dual{\OR}{n}} \to
        \Hopf^\aug(\cC)^{[1]}.
    \]
    Composing this functor with $(T_{(-)})^{[1]}$, yields a functor
    \[
        (\Modfin{\OR}{n})_{/\Dual{\OR}{n}} \longrightarrow
        \calg(\cC)^{[1]\times [1]}.
    \]
    
    Unwinding the definitions, the component at $\varphi\colon M\to \Dual{\OR}{n}$ is as claimed.
\end{proof} 

\section{Orientations and Orientability}\label{sec:orientations} 

The construction of the Fourier transform in the previous section depends on a choice of a pre-orientation, which plays the role of a root of unity in classical Fourier theory. Usually, one chooses the root of unity to be \textit{primitive}, so that the Fourier transform is an \textit{isomorphism}. In this section, we consider an analogous property of a pre-orientation and its relation to affineness and the Eilenberg--Moore property.

\subsection{Orientations}\label{ssec:orient} 

\subsubsection{Definition and functoriality}

A pre-orientation for which the associated Fourier transform is an isomorphism on all $\pi$-finite modules will be called an orientation. More precisely, we have the following:

\begin{defn} \label{def:orientation}
    Let $\cC\in \calg(\Prl)$, let $\OR\in \calg(\Sp_{(p)}^\cn)$, and let $\omega\in \POr{\OR}{\cC}{n}$ be a pre-orientation.
    \begin{enumerate}
        \item $M\in \Modfin{\OR}{n}$ is called \tdef{$\omega$-oriented} if $\Four_\omega \colon \one[M]\to \charac{M}{n}$ is an isomorphism. 
        \item $\omega$ is called an \tdef{orientation} if every $M\in\Modfin{\OR}{n}$ is $\omega$-oriented.
    \end{enumerate}
    We denote the subspace of $\OR$-orientations by $\mdef{\Or{\OR}{\cC}{n}}\subseteq \POr{\OR}{\cC}{n}$, and say that $\cC$ is \textit{$(\OR,n)$-orientable} if $\Or{\OR}{\cC}{n}\neq \es$.
\end{defn}


The following is the motivating example from the classical theory:

\begin{example}
    Let $\cC$ be the category of (ordinary) modules over a field $\FF$ of characteristic $0$. Then, $\ZZ/p^r$-pre-orientations of height $0$ of $\cC$ are just $p^r$-th roots of unity in $\FF$ (see \cref{ex:rootsofunity}). A root of unity defines an \textit{orientation} of $\cC$ exactly when it is \emph{primitive}. We note that the situation with \textit{higher} roots of unity is more subtle, as we shall discuss in \cref{ssec:zpror+higherroots}.
\end{example}

\begin{example}
    As in \cref{rem:luriepreor}, for $R\in \calg(\cC)$ a $\ZZ$-pre-orientation $\omega \colon \Sigma \QQ/\ZZ \to R^\times$ of $R$, in the sense of \Cref{def:pre_orientations} is the same datum as a pre-orientation
    $\widetilde{\omega}\colon \QQ/\ZZ \to \GG_m(R)$ for the divisible group $\GG_m$ over $R$, in the sense of \cite[Definition 2.6.8]{Lurie_Ell3}. One may show that $\omega$ is an orientation of $R$ if and only if $\widetilde{\omega}$ is an orientation of $\GG_m$ over $R$ in the sense of \cite[Definition 4.3.9]{Lurie_Ell2}
\end{example}

We now examine how the space of orientations behaves under the various operations on pre-orientations considered in \cref{sec:higherfourier}. First, orientations are preserved under restriction of scalars.

\begin{prop}\label{Or_Push}
    Let $\cC\in \calg(\Prl)$ and let 
    $f\colon \ORone\to \ORtwo$ in $\calg(\Sp_{(p)}^\cn)$. A module $M\in\Modfin{\ORtwo}{n}$ is $f_*\omega$-oriented if and only if $f_* M$ is $\omega$-oriented. Consequently, if $\omega$ is an $\ORone$-orientation, then $f_*\omega$ is an $\ORtwo$-orientation. 
\end{prop}

\begin{proof}
    Given $M\in \Modfin{\ORtwo}{n}$, by  \Cref{or_funct_ring}, the component of $\Four_\omega$ at $f_*M \in \Modfin{\ORone}{n}$ is homotopic to the component of $\Four_{f_*\omega}$ at $M$. Thus, $M$ is $f_*\omega$-oriented if and only if $f_*M$ is $\omega$-oriented.
\end{proof}

Orientations are also preserved under symmetric monoidal functors which preserve $n$-finite limits.

\begin{prop}\label{push_orientations_functor}
    Let $F\colon \cC \to \cD$ in $\calg(\Prl)$, and let $\omega \in \POr{\OR}{\cC}{n}$. If $M \in \Modfin{\OR}{n}$ is $\omega$-oriented and $F$ preserves $\und{\Dual{M}{n}}$-shaped limits, then $M$ is $F(\omega)$-oriented. In particular, if $\omega$ is an orientation on $\cC$ and $F$ preserves $n$-finite limits, then $F(\omega)$ is an orientation on $\cD$. 
\end{prop}

\begin{proof}
    The assumption that $F$ preserves $\und{\Dual{M}{n}}$-shaped limits implies that the $M$-components of $F(\Four_\omega)$ and $\Four_{F(\omega)}$ are isomorphic, see \Cref{Four_Iso_C}. Thus, if $M$ is $\omega$-oriented then it is also $F(\omega)$-oriented. If $F$ preserves all $n$-finite limits, then this is true for all $M \in \Modfin{\OR}{n}$.
\end{proof}

When $\cC$ and $\cD$ are $n$-semiadditive, every colimit preserving functor $F\colon\cC \to \cD$ preserves $n$-finite limits. If we further assume that the $\infty$-categories are stable, \Cref{push_orientations_functor} admits a partial converse.

\begin{defn}\label{defn:nilconservative}
Recall from \cite[Definition 4.4.1]{TeleAmbi}, that a functor $F\colon \cC \to \cD$ in $\alg(\Prl_\st)$ is called \tdef{nil-conservative} if for every $S\in \alg(\cC)$ for which $F(S)=0$, we have $S=0$. 
\end{defn}

\begin{prop}\label{nil_reflect_orientations}
    Let $F\colon \cC \to \cD$ in  $\calg(\Prl^{\sad{n}}_\st)$ be nil-conservative and let $\OR\in \calg(\Sp_{(p)}^\cn)$. 
    An $\omega \in \POr{\OR}{\cC}{n}$ is an orientation on $\cC$ if and only if $F(\omega)$ is an orientation on $\cD$.
\end{prop}
\begin{proof}
    Since $F$ is $n$-semiadditive, it preserves $n$-finite limits
    and hence $F(\Four_\omega)$ is isomorphic to $\Four_{F(\omega)}$, see \Cref{Four_Iso_C}. We deduce that, for every $M\in \Modfin{\OR}{n}$, the map
    \[
        F(\Four_\omega)\colon F(\one[M]) \too F(\one^{\und{\Dual{M}{n}}})
    \] 
    is an isomorphism. Since $F$ is nil-conservative, it is conservative when restricted to the dualizable objects of $\cC$, see \cite[Proposition 4.4.4]{TeleAmbi}. By \cite[Proposition 2.5]{carmeli2021chromatic}, both $\one[M]$ and $\one^{\und{\Dual{M}{n}}}$ are dualizable in $\cC$, and we deduce that $\Four_\omega$ is an isomorphism at every $M\in \Modfin{\OR}{n}$ if and only $F(\Four_\omega)$ is.
\end{proof}

Finally, we have the following behaviour with respect to scaling: 

\begin{prop}\label{Or_Scaling}
    Let $\cC\in \calg(\Prl)$, let $\OR\in \calg(\Sp_{(p)}^\cn)$, and let $\omega\in \Or{\OR}{\cC}{n}$. For $r\in \und{\OR}$, the pre-orientation $\omega^r$ is an orientation, if $r \in \OR^\times$. The converse holds if $R$ is $\pi$-finite and $\cC$ is non-zero. 
\end{prop}

\begin{proof}
    By \Cref{Por_Scaling}, for every $M \in \Modfin{\OR}{n}$ we have a commutative diagram 
    \[\begin{tikzcd}
    	{\one[M]} && {\one[M]} \\
    	& {\one^{\und{\Dual{M}{n}}}.}
        \arrow["r \cdot", from=1-1, to=1-3]
        \arrow["{\Four_{\omega^r}}"', from=1-1, to=2-2]
        \arrow["{\Four_{\omega}}", from=1-3, to=2-2]
    \end{tikzcd}\]    
    Since $\omega$ is an orientation, the right diagonal map is an isomorphism. By 2-out-of-3 for isomorphisms, $\omega^r$ is an orientation, i.e., the left diagonal map is an isomorphism if and only if $\one[M]\oto{\:r\cdot\:} \one[M]$ is an isomorphism for all $M \in \Modfin{\OR}{n}$. Clearly, if $r$ is invertible, then this is the case. Conversely, assuming $R$ is $\pi$-finite, we get that the map $\one[R]\oto{\:r\cdot\:} \one[R]$ is an isomorphism. Applying $\Map_\cC(-,\one)$ we get an isomorphism
    $\und{\one}^{\und{\OR}} \iso \und{\one}^{\und{\OR}}$. Since $\pi_0(R)$ is a retract of $\und{\OR}$, we get an isomorphism
    $\und{\one}^{\pi_0(R)} \iso \und{\one}^{\pi_0(R)}$,
    to which we can apply $\pi_0$, which preserves products, and get an isomorphism
    $\pi_0(\one)^{\pi_0(R)} \iso \pi_0(\one)^{\pi_0(R)}$.
    Since multiplication by $r$ is not invertible on $\pi_0(R)$, the last map can be an isomorphism only if $\pi_0(\one) \simeq \pt$, which would imply that $\cC$ is zero. 
\end{proof}

\subsubsection{Oriented modules} 

To study the question of whether a given $\OR$-pre-orientation $\omega$ is in fact an orientation, it is useful to know that the collection of $\omega$-oriented $\OR$-modules is closed under a variety of operations. We shall assume higher semiadditivity throughout.
We begin by showing that orientability is preserved under shifted Brown--Comenetz duality.
 
\begin{prop}\label{oriented_dual}
    Let $\OR\in\calg(\Sp_{(p)}^\cn)$, let $\cC \in \calg(\Prl^{\sad{n}})$, and let $\omega \in \POr{\OR}{\cC}{n}$. An $\OR$-module $M\in \Modfin{\OR}{n}$ is $\omega$-oriented if and only if $\Dual{M}{n} \in \Modfin{\OR}{n}$ is $\omega$-oriented. 
\end{prop}
 
\begin{proof}
By \Cref{cor:Four_Duality}, the maps $(\Four_\omega)_{M}$ and $(\Four_\omega)_{\Dual{M}{n}}$
    are dual to one another.
    Hence one is an isomorphism if and only if the other is. 
\end{proof}

Oriented modules are also closed under finite direct sums.

\begin{prop} \label{Orientability_Sum}
    Let $\OR\in \calg(\Sp_{(p)}^\cn)$, let
    $\cC \in \calg(\Prl^{\sad{n}})$ and let $\omega \in \POr{\OR}{\cC}{n}$. For every $\omega$-oriented $M,N \in \Modfin{\OR}{n}$, the module $M\oplus N$ is also $\omega$-oriented.
\end{prop}
\begin{proof}
    The Fourier transform $\Four_\omega \colon \one[-] \to \one^{\und{\Dual{(-)}{n}}}$ is a natural transformation between two functors, which both preserve finite co-products. Indeed, $\one[-]$ by being a (restriction of a)  left adjoint, and $\one^{\und{\Dual{(-)}{n}}}$ by the $n$-semiadditivity assumption on $\cC$ and \Cref{Kunneth_Ambi}.
\end{proof}
 
Next, we consider the behavior of orientability under cofibers and extensions of $\OR$-modules. 

\begin{prop} \label{Orientability_Ext_Cof}
    Let $\OR\in \calg(\Sp_{(p)}^\cn)$, let
    $\cC \in \calg(\Prl^{\sad{n}})$ and let $\omega \in \POr{\OR}{\cC}{n}$. Consider an exact sequence in $\Modfin{\OR}{n}$ of the form
    \[
        M_0 \too M_1 \too M_2, 
    \]
    such that $M_0$ is $\omega$-oriented and $\und{\Dual{M_0}{n}}$ is $\cC$-affine. Then, $M_1$ is $\omega$-oriented if and only if $M_2$ is $\omega$-oriented. 
\end{prop}
 
\begin{proof}
Consider the following commutative diagram in $\calg(\cC)$, where the diagonal arrows represent the components of the Fourier transform $\Four_\omega$ for the objects $0, M_0, M_1$ and $M_2$:
    \[\begin{tikzcd}\tag{$*$}
    	{ \one[M_0]} && { \one[M_1]} \\[-10pt]
    	& {\one^{\und{\Dual{M_0}{n}}}} && {\one^{\und{\Dual{M_1}{n}}}} \\
    	{ \one} && {\one[M_2]} \\[-10pt]
    	& \one && {\one^{\und{\Dual{M_2}{n}}}}
\arrow[from=1-1, to=3-1]
\arrow["\:"{description, pos=0.55}, from=3-1, to=3-3]
\arrow[from=1-1, to=1-3]
\arrow["\quad"{description, pos=0.42}, from=1-3, to=3-3]
\arrow[from=2-2, to=4-2]
\arrow[from=2-2, to=2-4]
\arrow[from=2-4, to=4-4]
\arrow[from=4-2, to=4-4]
\arrow[dashed, from=1-3, to=2-4]
\arrow[shorten <=4pt, shorten >=4pt, from=3-1, to=4-2]
\arrow[from=1-1, to=2-2]
\arrow[dashed, from=3-3, to=4-4]
\arrow["\sim"{marking}, shift left=1.5, draw=none, from=3-1, to=4-2]
\arrow["\sim"{marking}, shift left=1.5, draw=none, from=1-1, to=2-2]
    \end{tikzcd}\]
    The two solid diagonal maps are isomorphisms by assumption and we have to show that if either of the dashed diagonal maps is an isomorphism, then so is the other. Since the functor $\one[-]$ is a left adjoint, the back face of the diagram is a pushout. Since we assumed that $\und{\Dual{M_0}{n}}$ is $\cC$-affine, the front face is also a pushout (\Cref{Affiness_Eilenberg_Moore}). Thus, if the top dashed diagonal $(\Four_\omega)_{M_1}$ is an isomorphism, then so is the bottom one $(\Four_\omega)_{M_2}$. 
    
    Now, assume that the bottom dashed diagonal is an isomorphism. We can view the diagram as living in commutative algebras in $\cC$ under $\one^{\und{\Dual{M_0}{n}}}$ or equivalently, as commutative algebras in $\Mod_{\one^{\und{\Dual{M_0}{n}}}}(\cC)$. By the definition of affineness, we have an equivalence of categories 
    \[
        \Mod_{\one^{\und{\Dual{M_0}{n}}}}(\cC) \simeq
        \cC^{\und{\Dual{M_0}{n}}}.
    \]
    The collection of functors 
    $\varphi^*\colon\cC^{\und{\Dual{M_0}{n}}}\to \cC$
    for all $\varphi \in \und{\Dual{M_0}{n}}$ is jointly conservative.  By \Cref{Ev_Sharp}, these functors correspond under the above equivalence to the extension of scalars functors
    \[
F_\varphi:=
\one_{\varphi} \otimes_{\one^{\und{\Dual{M}{n}}}} (-) \:\colon\:
        \Mod_{\one^{\und{\Dual{M_0}{n}}}}(\cC) \too \cC,
    \]
    where $\one_\varphi$ is the unit $\one$ with the $\one^{\und{\Dual{M_0}{n}}}$-algebra structure given by 
    $\ev_{\varphi}\colon\one^{\und{\Dual{M_0}{n}}} \to \one.$
    For example, in the above cubical diagram, the left vertical map in the front face is  $\one^{\und{\Dual{M}{n}}} \to \one_0$. Thus, the fact that the back and the front faces are pushouts implies that 
    \[
F_0((\Four_\omega)_{M_1})= (\Four_\omega)_{M_2},
    \]
    which is, by assumption, an isomorphism. To show that $F_\varphi((\Four_\omega)_{M_1})$
is an isomorphism for all $\varphi \in \und{\Dual{M_0}{n}}$, we shall use the translation invariance of the Fourier transform and the case $\varphi=0$. 
    By the long exact sequence of homotopy groups associated with the cofiber sequence 
    \[
        M_0 \too M_1 \too M_2,
    \]
    the map $\pi_n M_0 \to \pi_n M_1$ is \textit{injective}, which implies that the map $\pi_0(\Dual{M_1}{n}) \to \pi_0 (\Dual{M_0}{n})$ is \textit{surjective}. Hence, we can lift $\varphi$ to an element $\cl{\varphi} \in \und{\Dual{M_1}{n}}$. By \Cref{Translation_Inv}, we get a commutative diagram
    \[\begin{tikzcd}
    	{ \one[M_0]} && { \one[M_1]} \\[-10pt]
    	& {\one^{\und{\Dual{M_0}{n}}}} && {\one^{\und{\Dual{M_1}{n}}}} \\
    	{ \one[M_0]} && { \one[M_1]} \\[-10pt]
    	& {\one^{\und{\Dual{M_0}{n}}}} && {\one^{\und{\Dual{M_1}{n}}},}
\arrow["\:"{description}, from=3-1, to=3-3]
\arrow[from=1-1, to=1-3]
\arrow["\quad"{description,pos=0.425}, from=1-3, to=3-3]
\arrow["{T_{(\ev_{\varphi})}}"{pos= 0.4}, from=2-2, to=4-2]
\arrow[from=2-2, to=2-4]
\arrow["\wr"'{pos= 0.35}, from=2-2, to=4-2]
\arrow[from=2-2, to=2-4]
\arrow["{T_{(\ev_{\cl{\varphi}})}}"{pos= 0.5}, from=2-4, to=4-4]
\arrow["\wr"'{pos= 0.45}, from=2-4, to=4-4]
\arrow[from=4-2, to=4-4]
\arrow[from=1-3, to=2-4]
\arrow[from=3-1, to=4-2]
\arrow[from=1-1, to=2-2]
\arrow[from=3-3, to=4-4]
\arrow["{T_{(\varepsilon_{\omega\circ\varphi})}}"{pos= 0.6}, from=1-1, to=3-1]
\arrow["\wr"'{pos= 0.55}, from=1-1, to=3-1]
\arrow["{T_{(\varepsilon_{\omega\circ\cl{\varphi}})}}"{pos= 0.7}, draw=none, from=1-3, to=3-3]
\arrow["\wr"'{pos= 0.65}, draw=none, from=1-3, to=3-3]
    \end{tikzcd}\]
    where the vertical maps are the respective translation automorphisms.
    Pasting this diagram on top of $(*)$ yields a cubical diagram analogous to $(*)$, where the left vertical map of the front face is 
    $\ev_\varphi \colon \one^{\und{\Dual{M_0}{n}}} \to \one.$
    As before, we get that 
    $F_\varphi((\Four_\omega)_{M_1})$
    identifies with $(\Four_\omega)_{M_2}$ and hence is an isomorphism. Thus, $(\Four_\omega)_{M_1}$ is itself an isomorphism.
\end{proof}

Similarly, we have a dual statement for fibers.
\begin{prop} \label{Orientability_Ext_Fib}
    Let $\OR\in \calg(\Sp_{(p)}^\cn)$, let
    $\cC \in \calg(\Prl^{\sad{n}})$ and let $\omega \in \POr{\OR}{\cC}{n}$. Consider an exact sequence in $\Modfin{\OR}{n}$ of the form
    \[
        M_0 \too M_1 \too M_2,
    \]
    such that $M_2$ is $\omega$-oriented and $\und{M_2}$ is $\cC$-affine. Then, $M_1$ is $\omega$-oriented if and only if $M_0$ is $\omega$-oriented. 
\end{prop}

\begin{proof}
    By \cref{oriented_dual}, a module $\Modfin{\OR}{n}$ is $\omega$-oriented if and only if $\Dual{M}{n}$ is. Thus, the claim follows from \cref{Orientability_Ext_Cof} applied to the exact sequence
    \[
        \Dual{M_2}{n} \too \Dual{M_1}{n} \too \Dual{M_0}{n}.\qedhere
    \]
\end{proof}

\subsection{\texorpdfstring{$\OR$}{R}-Cyclotomic extensions}\label{ssec:cyclext} 

We now construct for every $\OR\in\calg(\Sp_{(p)}^\cn)$ and $\cC \in \calg(\Prl^{\sad{n}})$ a universal $\OR$-oriented commutative algebra (of height $n$) in $\cC$, denoted by $\orcyc{\OR}{n}$, which we call the \textit{$\OR$-cyclotomic extension} (of height $n$).

\subsubsection{Universally oriented categories} 

We begin by working one categorical level up, which does not require any higher semiadditivity assumptions. The adjunction unit $\Dual{\OR}{n} \to \one[\Dual{\OR}{n}]^\times$ exhibits $\one[\Dual{\OR}{n}]$ as the universal \textit{$R$-pre-oriented} (of height $n$) commutative algebra in $\cC$. Our first goal is to establish a categorification of this fact.

\begin{prop}\label{universal_pre_orientation}
    Let $\cC \in \calg(\Prl)$ and let $\OR\in \calg(\Sp_{(p)}^\cn)$. 
    The functor 
    \[
        \POr{\OR}{-}{n}\colon \calg(\Mod_\cC(\Prl)) \too \Spc
    \] 
    is co-representable by $\Mod_{\one[\Dual{\OR}{n}]}(\cC)$. 
\end{prop}
\begin{proof}
    Using the adjunctions 
    \[
        \Mod_{(-)}(\cC)\colon \calg(\cC)  \adj 
        \calg(\Mod_\cC(\Prl)) \noloc \End(\one_{(-)}) 
    \]
    and
    \[
        \one[-]\colon \Sp^\cn \adj 
        \calg(\cC) \noloc (-)^\times,
    \]
    we get 
    \[
        \Map_{\calg(\Mod_\cC(\Prl))}(\Mod_{\one_\cC[\Dual{\OR}{n}]}(\cC),\cD) \simeq
        \Map_{\calg(\cC)}(\one[\Dual{\OR}{n}], \End(\one_\cD)) \simeq
    \]
    \[
        \Map_{\Sp^\cn}(\Dual{\OR}{n},\End(\one_\cD)^\times ) \simeq
        \Map_{\Sp^\cn}(\Dual{\OR}{n},\one_\cD^\times)=:
        \POr{\OR}{\cD}{n}
    \]
    naturally in $\cD\in \Mod_\cC$, which finishes the proof.
\end{proof}

We shall thus use the following notation: 

\begin{defn}
    Let $\cC\in \calg(\Prl)$ and let $\OR\in \calg(\Sp_{(p)}^\cn)$. We define
    \[
        \mdef{\puniv{\cC}{\OR}{n}} :=
        \Mod_{\one[\Dual{\OR}{n}]}(\cC)
    \]
    and denote by 
    \[
        \mdef{\omega_\taut} \colon 
        \Dual{\OR}{n} \too 
        \one_{\puniv{\cC}{\OR}{n}}^\times =
        \one[\Dual{\OR}{n}]^\times
    \] 
    the \tdef{tautological $\OR$-pre-orientation} of $\puniv{\cC}{\OR}{n}$ given by the unit of the adjunction $\one[-]\dashv (-)^\times.$
\end{defn} 

\begin{rem}
    By \Cref{universal_pre_orientation}, $\puniv{\cC}{\OR}{n}$ co-represents $\OR$-pre-orientations of height $n$ for $\cC$-linear presentably symmetric monoidal $\infty$-categories. Explicitly, given $F\colon \puniv{\cC}{\OR}{n} \to \cD$ in $\Mod_\cC(\Pr)$, the associated $\OR$-pre-orientation of $\cD$ is $F(\omega_\taut)$.
\end{rem}

We now consider the co-representability of the subfunctor $\Or{\OR}{-}{n}\subseteq \POr{\OR}{-}{n}$.

\begin{defn}
    Let $\cC\in \calg(\Prl)$ and let $\OR\in \calg(\Sp_{(p)}^\cn)$. We define 
    \(
        \mdef{\univ{\cC}{\OR}{n}} \subseteq \puniv{\cC}{\OR}{n}
    \)
    to be the full subcategory of objects $X\in\puniv{\cC}{\OR}{n}$, such that
    \[
        \Four_{\omega_\taut}^* \colon 
        \hom(\one^{\und{\Dual{M}{n}}},X) \too 
        \hom(\one[M],X) \qin \cC
    \]
    is an isomorphism for all $M\in \Modfin{\OR}{n}$.
\end{defn}

Equivalently, by \Cref{Four_Duality}, $\univ{\cC}{\OR}{n}$ can be identified with the left localization of $\puniv{\cC}{\OR}{n}$ with respect to the collection of morphisms of the form
\[
    \Id_Y \otimes \Four_{\omega_\taut} \colon 
    Y\otimes \one[M]\too 
    Y\otimes \one^{\und{\Dual{M}{n}}}.
\]
In particular, it is a $\otimes$-localization of $\puniv{\cC}{\OR}{n}$, and hence can be seen as an object of $\calg(\Mod_\cC(\Prl))$, and it is equipped with a symmetric monoidal localization functor $\Lor\colon \puniv{\cC}{\OR}{n} \to \univ{\cC}{\OR}{n}$.

\begin{prop}\label{universal_oriented_category}
    Let $\cC\in \calg(\Prl)$ and let $\OR\in \calg(\Sp_{(p)}^\cn)$. The localization functor
    \[
        \Lor\colon \puniv{\cC}{\OR}{n} \too 
        \univ{\cC}{\OR}{n}
    \] 
    co-represents the fully faithful embedding $\Or{\OR}{-}{n}\subseteq \POr{\OR}{-}{n}$.
\end{prop}

\begin{proof}
    Since $\Lor$ is a localization, pre-composition with it exhibits $\Map_{\calg(\Mod_\cC(\Prl))}(\univ{\cC}{\OR}{n},-)$ as a subfunctor of 
    \[
    \Map_{\calg(\Mod_\cC(\Prl))}(\puniv{\cC}{\OR}{n},-)\simeq \POr{\OR}{-}{n}.
    \] 
    Thus, it would suffice to show that a functor $F\colon \puniv{\cC}{\OR}{n} \to \cD$ in $\calg(\Mod_\cC(\Prl))$ factors through the localization functor $\Lor$ if and only if $F(\omega_\taut)$ is an orientation of $\cD$. 
    
    By \Cref{Four_Funct_C} we have
    $\Four_{\Lor(\omega_\taut)} \simeq \Lor(\Four_{\omega_\taut})$
    on $\Modfin{\OR}{n}$, which is an isomorphism by the definition of $\univ{\cC}{\OR}{n}$. This implies that if $F$ factors through $\Lor$, then $F(\omega_\taut)$ is an orientation. 
    
    Conversely, assuming that $F(\omega_\taut)$ is an orientation we shall show that $F$ factors through $\Lor$. For this, it suffices to show that the essential image of the right adjoint $G\colon \cD \to \puniv{\cC}{\OR}{n}$ of $F$ lies in $\univ{\cC}{\OR}{n}\subseteq \puniv{\cC}{\OR}{n}$.
    By the definition of  $\univ{\cC}{\OR}{n}$, this is if and only if, for every $X\in \cD$ and every $M\in \Modfin{\OR}{n}$, the morphism
    \[
        \Four_{\omega_\taut}^*\colon 
        \hom(\one^{\und{\Dual{M}{n}}},G(X))\too
        \hom(\one[M],G(X)) 
    \]
    is an isomorphism. Since $F$ is symmetric monoidal, we have a natural isomorphism
    \[
        \hom_{\puniv{\cC}{\OR}{n}}(Y,G(X))\simeq 
        G(\hom_{\cD}(F(Y),X)) 
    \]
    for all $Y\in \puniv{\cC}{\OR}{n}$, so it would suffice to show that
    \[
        F(\Four_{\omega_\taut})^*\colon 
        \hom(F(\one^{\und{\Dual{M}{n}}}),X) \too
        \hom(F(\one[M]),X)
    \] 
    is an isomorphism for every $M\in \Modfin{\OR}{n}$. But this follows from our assumption that $F({\omega_\taut})$ is an orientation.
\end{proof}

\subsubsection{Universally oriented algebras} 

\Cref{universal_oriented_category} shows that there is a $\cC$-linear symmetric monoidal $\infty$-category $\univ{\cC}{\OR}{n}$ carrying a universal $\OR$-orientation of height $n$. We now show that if $\cC$ is $n$-semiadditive, then $\univ{\cC}{\OR}{n}$ is in fact the  $\infty$-category of modules over a universally $\OR$-oriented commutative algebra in $\cC$.

\begin{prop} \label{universal_oriented_ring} 
    Let $\cC\in \calg(\Prl^{\sad{n}})$ and let $\OR\in \calg(\Sp_{(p)}^\cn)$. Then, there is an idempotent  commutative $\one[\Dual{\OR}{n}]$-algebra $\orcyc{\OR}{n}$, such that 
    \[
        \univ{\cC}{\OR}{n}\simeq \Mod_{\orcyc{\OR}{n}}(\cC) \qin \calg(\Mod_{\puniv{\cC}{\OR}{n}}(\Prl)).
    \] 
\end{prop}

\begin{proof}
    By \cite[Theorem 7.6]{ragimovschlank2022}, it would suffice to show that $\univ{\cC}{\OR}{n}$ is closed under all limits, all colimits, tensoring with any object of $\puniv{\cC}{\OR}{n}$ and taking internal $\hom$ from any object of $\puniv{\cC}{\OR}{n}$. Now, the functor $(X,Y)\mapsto \hom(Y,X)$ is limit preserving in the $X$-argument and satisfies 
    \[
        \hom(Z,\hom(Y,X))\simeq \hom(Y,\hom(Z,X)). 
    \]
    These imply that $\univ{\cC}{\OR}{n}$ is closed under limits and applying $\hom(Y,-)$ for $Y\in \puniv{\cC}{\OR}{n}$.  On the other hand, by the $n$-seimadditivity assumption on $\cC$, for every 
    $M\in \Modfin{\OR}{n}$ the objects $\one[M]$ and $\one^{\und{\Dual{M}{n}}}$ are \emph{dualizable}. Hence, we can identify the morphisms 
    \[
        \hom(\one^{\und{\Dual{M}{n}}},X) \too \hom(\one[M],X)  
    \]
    of pre-composition with $\Four_\omega$, with the tensor product 
    $\Four_\omega^\vee \otimes X$. This implies that $\univ{\cC}{\OR}{n}$ is also closed under all colimits and tensoring with any object of $\puniv{\cC}{\OR}{n}$. 
    
    Finally, since the functor
    \[
        \Mod_{(-)}(\cC) \colon \calg(\cC) \too \calg(\Mod_\cC(\Prl))
    \] 
    is fully faithful, it follows from \Cref{universal_oriented_category} that indeed $\OR$-orientations of commutative algebras in $\cC$ are co-represented by $\orcyc{\OR}{n}$.
\end{proof}

From the above proposition follows that $\OR$-orientations of commutative algebras in $\cC$ are co-represented by $\orcyc{\OR}{n}$. In other words, $\orcyc{\OR}{n}$ carries a universal $\OR$-orientation of height $n$, which motivates the following:

\begin{defn}
    Let $\cC\in \calg(\Prl^{\sad{n}})$ and let $\OR\in \calg(\Sp_{(p)}^\cn)$. We refer to $\mdef{\orcyc{\OR}{n}} \in \calg(\cC)$ as the \tdef{$\OR$-cyclotomic extension} of height $n$.
\end{defn}

In general, the functoriality of the construction $\orcyc[\one_\cC]{\OR}{n}$ in $\cC \in \calg(\Prl^{\sad{n}})$ is rather subtle. However, there is one relatively simple, yet useful, case.

\begin{prop} \label{Orcyc_Functorial}
    Let $F\colon \cC \to \cD$ in $\calg(\Prl^{\sad{n}})$ and let $\OR\in \calg(\Sp_{(p)}^\cn)$. If $F$ admits a conservative right adjoint, then  $F(\orcyc[\one_\cC]{\OR}{n})\simeq \orcyc[\one_\cD]{\OR}{n}$.
\end{prop}
  
\begin{proof}
    Let $G$ be a right adjoint of $F$. Then, $G$ is lax symmetric monoidal and hence maps commutative algebras in $\cD$ to commutative algebras in $\cC$. Moreover, for $S\in \calg(\cD)$, we have a natural identification $G(S)^\times \simeq S^\times$ and hence we can identify pre-orientations of $S$ with pre-orientations of $G(S)$. 
      
    The object $\orcyc[\one_\cD]{\OR}{n}$ co-represents $\OR$-orientations of height $n$ in $\calg(\cD)$, while $F(\orcyc[\one_\cC]{\OR}{n})$ co-represents the functors $\calg(\cD)\to \Spc$ given by 
    \(
        S\mapsto \Or{\OR}{\cC;G(S)}{n}.
    \)
    Both these functors are sub-functors of $\POr{\OR}{\cD;-}{n}$. Hence, to identify them, it would suffice to show that, for $S\in \calg(\cD)$ with pre-orientation $\omega \colon \Dual{\OR}{n} \to S^\times \simeq G(S)^\times$ and every $M\in \Modfin{\OR}{n}$, the morphism
    \[
        \Four_\omega\colon S[M]\too 
        S^{\und{\Dual{M}{n}}} 
        \tag{$*$}
    \]
    is an isomorphism if and only if 
    \[
        \Four_\omega \colon G(S)[M]\too 
        G(S)^{\und{\Dual{M}{n}}}
        \tag{$**$}
    \]
    is an isomorphism.
    The functor $G$ is limit preserving and hence it preserves also $\pi$-finite colimits, see \cite[Corollary 3.2.4]{TeleAmbi}. Consequently, the map $(**)$ is the image under $G$ of the map $(*)$. Since $G$ is assumed to be conservative, we deduce that $(**)$ is an isomorphism if and only if $(*)$ is an isomorphism and the result follows.
\end{proof}

\begin{cor} \label{Orcyc_Scalar_Ext}
    Let $\cC \in \calg(\Prl^{\sad{n}})$ and let $\OR\in \calg(\Sp_{(p)}^\cn)$. For every $S \in \calg(\cC)$ we have, 
    \[
        \orcyc[S]{\OR}{n} \:\simeq\: S\otimes \orcyc{\OR}{n}
        \qin \calg_S(\cC).
    \]
\end{cor}
\begin{proof}
    Apply \Cref{Orcyc_Functorial} to the extension of scalars functor
    \(
        S\otimes (-) \colon \cC \to \Mod_S(\cC).
    \)
\end{proof}  

\subsubsection{Equivariance and Galois}

Recall that $\one[\Dual{\OR}{n}]$ admits an action of the multiplicative monoid $\und{\OR}$. This induces an action of $\und{\OR}$ on the functor $\POr{\OR}{-;\cC}{n}$ it co-represents. This action is given by scaling and hence, by \Cref{Or_Scaling}, the action of the submonoid $\OR^\times \sseq \und{\OR}$ preserves the subspace 
\[
    \Or{\OR}{-;\cC}{n} \sseq \POr{\OR}{-;\cC}{n}.
\]
We thus obtain an action of $\OR^\times$ on the co-representing object $\orcyc{\OR}{n}$ and an $\OR^\times$-equivariant map 
\[
    \one[\Dual{\OR}{n}] \too \orcyc{\OR}{n}
    \qin \calg(\cC)^{B\OR^\times}.
\]
 
\begin{prop}\label{Orcyc_Galois}
    Let $\cC\in \calg(\Prl^{\sad{(n+1)}})$ and let $\OR\in \calg(\Sp_{(p)}^\cn)$ be $n$-finite. If the $\OR$-cyclotomic extension $\orcyc{\OR}{n}$ is faithful, then it is $\OR^\times$-Galois.
\end{prop}

\begin{proof}
By our assumptions, both $B\OR^\times$ and $\OR^\times$ are $\cC$-ambidextrous spaces. Thus, the tensor product of $\cC$ preserves the respective limits in conditions (G1) and (G2) of a Galois extension (\Cref{def:Galois}). Hence, since $\orcyc{\OR}{n}$ is faithful, it suffices to show that it is Galois after base-change along itself. Namely, after applying the functor 
    \[
        \orcyc{\OR}{n}\otimes (-)\colon \cC \too \Mod_{\orcyc{\OR}{n}}(\cC).
    \]
    In other words, we may assume without loss of generality that $\cC$ is $(\OR,n)$-orientable. We shall show that in this case, $\orcyc{\OR}{n}$ is in fact \textit{split Galois}. The Fourier transform, associated with any $\OR$-orientation $\omega $ of height $n$, provides an isomorphism
    \[
        \Four_\omega \colon \one[\Dual{\OR}{n}] \iso \one^{\und{\OR}},
    \]
    which is equivariant with respect to the multiplicative monoid $\und{\OR}$, and hence in particular with respect to $\OR^\times$. Consider the composition
    \[
        \one[\Dual{\OR}{n}] \oto{\:\Four_\omega\:} 
        \one^{\und{\OR}} \too
        \one^{\OR^\times}
        \qin \calg(\cC)^{B\OR^\times}.
    \]
where the second map is given by restriction along the inclusion $\OR^\times \into \und{\OR}$. Since $\one^{\OR^\times}$ is co-induced as an $\OR^\times$-object, this map corresponds to a non-equivariant map $\one[\Dual{\OR}{n}] \to \one$ given by evaluation at $1 \in \OR^\times$. By unwinding the definitions, this is exactly the orientation $\omega$, and hence factors through $\one[\Dual{\OR}{n}] \to \orcyc{\OR}{n}$. Consequently, we get the following commutative square in $\calg(\cC)^{B\OR^\times}$:
    \[\begin{tikzcd}
    	{\one[\Dual{\OR}{n}]} && {\one^{\und{\OR}}} \\
    	{\orcyc{\OR}{n}} && {\one^{\OR^\times}.}
        \arrow["{\Four_\omega}", from=1-1, to=1-3]
        \arrow[from=1-1, to=2-1]
        \arrow[from=2-1, to=2-3]
        \arrow[from=1-3, to=2-3]
    \end{tikzcd}\tag{$*$}\]
    Since the top map is an isomorphism, it would suffice to show that $(*)$ is a pushout square, as this would imply that the bottom map is an isomorphism as well. Furthermore, as the forgetful functor $\calg(\cC)^{B\OR^\times} \to \calg(\cC)$ is colimit preserving and conservative, it suffices to check that $(*)$ is a pushout square non-equivariantly.  
    Namely, we need to show that for every $f\colon \one^{\und{\OR}} \to S$ in $\calg(\cC)$, the composition 
    \[
\omega_f\:\colon\: 
        \one[\Dual{\OR}{n}] \oto{\Four_\omega} 
        \one^{\und{\OR}}\oto{f} 
        S,
    \]
    is an orientation on $S$ if and only if $f$ factors through the projection $\one^{\und{\OR}} \to \one^{\OR^\times}$. The `if' part is clear from the existence of the commutative square $(*)$. For the `only if' part, it suffices to show that if $\omega_f$ is an orientation, then the pushout
    \(
        \one^{\und{\OR} \smallsetminus \OR^\times} 
        \otimes_{\one^{\und{\OR}}} S
    \)
    vanishes.
    Since the space $\und{\OR}\smallsetminus \OR^\times$ is $\cC$-ambidextrous, it is $\cC$-semi-affine (\Cref{Ambi_Affine}), so by \Cref{colim_lin_fully_faithful} we have a fully faithful embedding 
    \[
        \Mod_{\one^{\und{\OR}\smallsetminus \OR^\times}}(\cC) \longhookrightarrow 
        \cC^{\und{\OR}\smallsetminus \OR^\times}.
    \]
    Furthermore, by \Cref{Ev_Sharp}, the jointly conservative functors 
    $r^* \colon \cC^{\und{\OR}\smallsetminus \OR^\times} \to \cC$
    correspond to the functors $\one_r\otimes_{\one^{\und{\OR}\smallsetminus \OR^\times}}(-)$. 
    Therefore, by the associativity of the relative tensor product, it suffices to show that
    \(
S_r:= \one_r \otimes_{\one^{\und{\OR}}} S
    \)
vanishes for every $r\in \und{\OR} \smallsetminus \OR^\times$. Now, consider the diagram, where the right square a pushout square and the left square commutes by \Cref{Por_Scaling} and \Cref{Four_Aug}:
    \[\begin{tikzcd}
    	{\one[\Dual{\OR}{n}]} && {\one^{\und{\OR}}} && S \\
    	{\one[\Dual{\OR}{n}]} && \one && {S_r.}
        \arrow["r\cdot"', from=1-1, to=2-1]
        \arrow["{\ev_r}", from=1-3, to=2-3]
        \arrow["{\Four_\omega}", from=1-1, to=1-3]
        \arrow["\omega", from=2-1, to=2-3]
        \arrow["f", from=1-3, to=1-5]
        \arrow["{f_r}", from=2-3, to=2-5]
        \arrow["{g_r}", from=1-5, to=2-5]
    \end{tikzcd}\]
    Comparing the composition along the top and then right maps with the composition along the left and the bottom maps provides an isomorphism
    \[
(g_r)_* \omega_f \:\simeq\: ((f_r)_*\omega_f)^r 
        \qin \POr{\OR}{S;\cC}{n}.
    \]
Since we assumed that $\omega_f$ is an orientation, it follows by \Cref{Or_Push}, that the left hand side is an orientation. However, since $R$ is $\pi$-finite and $r$ is assumed to be non-invertible, it follows by \Cref{Or_Scaling}, that the right hand side is \textit{not} an orientation unless $S_r=0$. 
\end{proof}

\subsection{Virtual orientability and affineness}
\label{ssec:affor}

In \Cref{Affiness_Height}, we have seen that having semiadditive height $n$ for an $\infty$-category $\cC$ implies the affineness of $(n+1)$-\emph{connected} $\pi$-finite spaces with respect to $\cC$. We shall now show that the existence of an $\OR$-orientation of height $n$ implies the affineness of all \textit{$n$-truncated} $\pi$-finite spaces which admit an $\OR$-module structure. 
In fact, many $\infty$-categories of interest $\cC$ are not themselves $(\OR,n)$-orientable, yet possess an $(\OR,n)$-orientable commutative algebra $S$, which is \textit{faithful} in the sense that the functor $S\otimes - \colon \cC \to \cC$ is conservative. This allows one to use the Fourier transform over $\Mod_{S}(\cC)$ together with ``faithful descent'' along the morphism $\one \to S$ to derive structural properties, such as affineness, for $\cC$. 

\subsubsection{Virtual orientability}
The above discussion leads to the following definition: 

\begin{defn}\label{def:virtually_R_orientable}
    Let $\cC\in \calg(\Prl)$ and let $\OR\in \calg(\Sp_{(p)}^\cn)$. We say that the $\infty$-category $\cC$ is \tdef{virtually $(\OR,n)$-orientable} if there exists a faithful commutative algebra $S$ in $\cC$, which admits an $\OR$-orientation of height $n$.
\end{defn}

In the higher semiadditive setting, we have the universal commutative algebra $\orcyc{\OR}{n}$ in $\cC$ that caries an $\OR$-orientation. In this case, its faithfulness is equivalent to virtual orientability of $\cC$. 

\begin{prop} \label{virtually_oriented_universal_faithful}
    Let $\cC\in \calg(\Prl^{\sad{n}})$ and let $\OR\in \calg(\Sp_{(p)}^\cn)$. Then, $\cC$ is virtually $(\OR,n)$-orientable if and only if $\orcyc{\OR}{n}$ is faithful. 
\end{prop}

\begin{proof}
    Since $\orcyc{\OR}{n}$ admits an $\OR$-orientation of height $n$, if it is faithful, then $\cC$ is virtually $(\OR,n)$-orientable.     
    Conversely, assume that $\cC$ is virtually $(\OR,n)$-orientable, and let $S\in \calg(\cC)$ be a faithful algebra which admits an $\OR$-orientation of height $n$.
    The extension of scalars functor $S\otimes -\colon\cC \to \Mod_S(\cC)$ admits a conservative right adjoint, which, by \Cref{Orcyc_Functorial}, implies
    that $\orcyc[S]{\OR}{n}\simeq S\otimes \orcyc{\OR}{n}$. 
    Now, $S$ admits an $\OR$-orientation, which in turn provides an augmentation 
    \[
        S\otimes \orcyc{\OR}{n} \simeq \orcyc[S]{\OR}{n} \too S.
    \]
    Since $S$ is faithful, we deduce that $S\otimes \orcyc{\OR}{n}$ is faithful, and hence that $\orcyc{\OR}{n}$ is faithful.  
\end{proof}

\begin{example}\label{ex:Q_Virtually_Or}
    The field of rational numbers $\QQ$ does not admit a $\ZZ/p^r$-orientation of height 0, i.e., a primitive $p^r$-th root of unity, unless $p=2$ and $r=1$. However, it is still \textit{virtually} $(\ZZ/p^r ,0)$-orientable, as we can pass to the cyclotomic extension $\QQ(\omega_{p^r})$, which is clearly faithful over $\QQ$.
\end{example}

By localizing with respect to $\orcyc{\OR}{n}$, we obtain the universal virtually $(\OR,n)$-orientable localization of $\cC$.
 
\begin{defn}\label{def:Virt_Univ}
    Let $\cC\in \calg(\Prl^{\sad{n}})$ and let $\OR\in \calg(\Sp_{(p)}^\cn)$. We define $\mdef{\virt{\cC}{\OR}{n}}$ to be the Bousfield localization of $\cC$ with respect to $\orcyc{\OR}{n}$. 
\end{defn}

Namely, $\virt{\cC}{\OR}{n}$ is obtained from $\cC$ by inverting all the morphisms $X\to Y$ in $\cC$ for which the induced morphism $X\otimes \orcyc{\OR}{n}\to Y\otimes \orcyc{\OR}{n}$ is an isomorphism. The $\infty$-category $\virt{\cC}{\OR}{n}$ classifies the property of being virtually $(\OR,n)$-orientable among localizations of $\cC$. 
 
\begin{prop} \label{initial_virtually_R_oriented}
    Let $\cC\in \calg(\Prl^{\sad{n}})$ and let $\OR\in \calg(\Sp_{(p)}^\cn)$. The $\infty$-category $\virt{\cC}{\OR}{n}$ is the initial symmetric monoidal localization of $\cC$ in $\calg(\Prl^{\sad{n}})$ which is virtually $(\OR,n)$-orientable. 
\end{prop} 
 
\begin{proof}
    First, since $\virt{\cC}{\OR}{n}$ is a Bousfield localization of $\cC$, it is a symmetric monoidal localization of $\cC$. Let $L\colon \cC \to \virt{\cC}{\OR}{n}$ be the localization functor. 
    By \Cref{Orcyc_Functorial}, and since $L$ admits a fully faithful (and in particular conservative) right adjoint,  we have $L(\orcyc[\one_\cC]{\OR}{n})\simeq \orcyc[\one_{\virt{\cC}{\OR}{n}}]{\OR}{n}$.
    Since, by construction, $L(\orcyc[\one_\cC]{\OR}{n})$ is faithful in $\calg(\virt{\cC}{\OR}{n})$, we deduce from \Cref{virtually_oriented_universal_faithful} that $\virt{\cC}{\OR}{n}$ is virtually $(\OR,n)$-orientable. 
    It remains to show that $\virt{\cC}{\OR}{n}$ is initial with respect to being virtually $(\OR,n)$-orientable. Let $L_1\colon \cC \to \cD$ be any symmetric monoidal localization in $\calg(\Prl^{\sad{n}})$ for which $\cD$ is virtually $(\OR,n)$-orientable. By \Cref{Orcyc_Functorial} again, we have that $L_1(\orcyc[\one_\cC]{\OR}{n})\simeq \orcyc[\one_\cD]{\OR}{n}$, which is faithful by our assumption on $\cD$. This implies that $L_1$ factors through the Bousfield localization with respect to $\orcyc[\one_\cC]{\OR}{n}$, that is, through $L\colon \cC \to \virt{\cC}{\OR}{n}$. 
\end{proof}

\subsubsection{Affineness of modules}

The existence of an $\OR$-orientation for an $\infty$-category $\cC$ allows one to relate group algebras and algebras of functions on $\OR$-modules. Since the functor $\one[-]\colon \Mod_\OR \to\calg(\cC)$ preserves pushouts, we can use the Fourier transform to deduce Eilenberg--Moore type properties for $\cC$. In fact, \textit{virtual} orientability suffices.

\begin{prop} \label{Orientation_Then_EM}
    Let $\OR\in \calg(\Sp_{(p)}^\cn)$ and
    let $\cC\in\calg(\Prl^{\sad{n}})$ be virtually $(\OR,n)$-orientable. For every exact square $(*)$ in $\Modfin{\OR}{n}$, the associated square $(**)$ is a pushout in $\calg(\cC)$.    
    \[\begin{tikzcd}
    	{ M_0} & {M_1} & {(*)} && {\one^{\und{M_3}}} & { \one^{\und{M_1}}} & {(**)} \\
    	{ M_2} & {M_3,} &&& {\one^{\und{M_2}}} & {\one^{\und{M_0}}.}
        \arrow[from=1-1, to=1-2]
        \arrow[from=1-1, to=2-1]
        \arrow[from=2-1, to=2-2]
        \arrow[from=1-2, to=2-2]
        \arrow[from=1-5, to=2-5]
        \arrow[from=1-5, to=1-6]
        \arrow[from=1-6, to=2-6]
        \arrow[from=2-5, to=2-6]
    \end{tikzcd}\]
\end{prop}

\begin{rem}
    The square $(*)$ is exact (see the conventions) in $\Modfin{\OR}{n}$ if and only if it is a pushout square in $\Mod_\OR(\Sp)$. 
\end{rem}

\begin{proof}
    Let $S\in \calg(\cC)$ be a faithful algebra which admits an $\OR$-orientation of height $n$. 
    Since the tensor product in $\cC$ preserves $n$-finite limits in each coordinate, tensoring the square $(**)$ with $S$ gives the analogous square in $\Mod_S(\cC)$. Since $S$ is faithful, we may replace $\cC$ by $\Mod_S(\cC)$ and assume without loss of generality that $\cC$ admits an $\OR$-orientation $\omega\colon \Dual{\OR}{n}\to \one^\times$.
    
The exact square $(*)$ in $\Modfin{\OR}{n}$ induces another exact square in $\Modfin{\OR}{n}$:
    \[
        \xymatrix{\Dual{M_3}{n}\ar[d]\ar[r] & \Dual{M_1}{n}\ar[d]\\
        \Dual{M_2}{n}\ar[r] & \Dual{M_0}{n},
        }
    \]
    which is also exact in $\Sp^\cn$. The functor 
    \[
        \one[-]\colon\Sp^\cn\to\calg(\cC) 
    \]
    preserves colimits, so that the square 
\begin{equation*}\label{eq:***}
        \xymatrix{\one[\Dual{M_0}{n}]\ar[d]\ar[r] & \one[\Dual{M_1}{n}]\ar[d] \\
        \one[\Dual{M_2}{n}]\ar[r] & \one[\Dual{M_3}{n}]
        }\tag{$***$}
    \end{equation*}
    is a pushout square in $\calg(\cC)$. 
    Now, the (inverse of the) Fourier transform 
    $\Four_\omega\colon \one[M] \iso \one^{\und{\Dual{M}{n}}}$ 
    identifies the pushout square $(***)$ with the square $(**)$ in the claim, which is therefore a pushout square as well.
\end{proof}

This implies the following affineness result:

\begin{prop} \label{orientation_affineness_for_R_modules}
    Let $\OR\in \calg(\Sp_{(p)}^\cn)$ and
    let $\cC \in \calg(\Prl^{\sad{n}})$ be virtually $(\OR,n)$-orientable. For every $M\in \Modfin{\OR}{n}$, the space $\und{M}$ is $\cC$-affine. 
\end{prop}

\begin{proof}
    First, we have a map $\und{M}\to \pi_0 \und{M}$ whose fibers are all isomorphic to  $\und{\tau_{\ge 1} M}$. Thus, by \Cref{affineness_extensions}((2) and (3)), it suffices to prove that $\und{\pi_0 M}$ and $\und{\tau_{\ge 1}M}$ are $\cC$-affine. Since $\und{\pi_0 M}$ is finite, it is $\cC$-affine by \Cref{affine_finite_sets}. Hence, we are reduced to the case that $M$ is connected. 
    For connected $M$, all the path spaces $P_{a,b}\und{M}$ for $a,b\in \und{M}$ are isomorphic to $\und{\Omega M}\simeq \Omega \und{M}$. By \Cref{Affiness_Eilenberg_Moore}, to show that $\und{M}$ is $\cC$-affine, it suffices to show that the square 
    \[
        \xymatrix{
        \one^{\und{M}} \ar[r]\ar[d] & \one \ar[d] \\ 
        \one \ar[r]               & \one^{\und{\Omega M}} 
        }
    \]
    is a pushout square in $\calg(\cC)$. This in turn follows from \Cref{Orientation_Then_EM}. 
\end{proof}

\begin{rem}
    The conclusion of \Cref{orientation_affineness_for_R_modules} is far from being the best possible. Since $\cC$-affine spaces are closed under extensions (\Cref{affineness_extensions}), to deduce that an $n$-finite space $A$ is $\cC$-affine, it suffices to be able to build $A$ from $\OR$-modules by iterated extensions. We exploit this fact in \cref{p_Affineness}, to show that already virtual $(\FF_p,n)$-orientability implies the $\cC$-affineness of \textit{all} $n$-finite $p$-spaces.    
\end{rem}

\subsubsection{Higher Kummer theory}

When one actually has an $\OR$-orientation of height $n$ we obtain a variant of Kummer theory, providing a classification of certain abelian Galois extensions of a ring in terms of its units.

\begin{prop}\label{Kummer}
    Let $\OR\in \calg(\Sp_{(p)}^\cn)$ and
    let $\cC \in \calg(\Prl^{\sad{n}})$ be $(\OR,n)$-orientable. For every $M\in \Modfin{\OR}{n}$ and $S\in\calg(\cC)$, we have an isomorphism of spaces
    \[
        \calg^{\und{M}-\gal}(S;\cC) \simeq 
        \Map_{\Sp^\cn}(\Dual{M}{n},S^\times).
    \]
\end{prop}

\begin{proof}
    By \Cref{orientation_affineness_for_R_modules}, the space $\und{M}$ is $\cC$-affine. Hence, by \Cref{Galois_Affine}, we have
    \[
        \calg_S^{\und{M}-\gal}(S;\cC) \simeq 
        \Map_{\calg(\cC)}(\one^{\und{M}}, S).
    \]
    Since $\cC$ is $(\OR,n)$-orientable, the Fourier transform, associated with any $\OR$-orientation $\omega$ of height $n$, provides an isomorphism
    \[
        \Four_\omega \colon 
        \one[\Dual{M}{n}] \iso
        \one^{\und{M}}
        \qin \calg(\cC).
    \]
    Plugging this into the above we get,
    \[
        \Map_{\calg(\cC)}(\one^{\und{M}}, S) \simeq 
        \Map_{\calg(\cC)}(\one[\Dual{M}{n}], S) \simeq
        \Map_{\Sp^\cn}(\Dual{M}{n}, S^\times).
    \]
\end{proof}

\begin{example}\label{ex:Kummer}
If $\cC$ is $(\FF_p,n)$-orientable for some $n\ge 1$, then applying the above to $M= \Sigma C_p$, we get
    \[
        \calg^{BC_p-\gal}(S;\cC) \simeq 
        \Map_{\Sp^\cn}(\Sigma^{n-1}C_p, S^\times) \simeq
\Omega^{n-1}\mu_p(S)=:
        \roots[S]{p}{n-1}.
    \]
    That is, $C_p$-Galois extensions are classified by $p$-th roots of unity of height $n-1$. 
\end{example}

\begin{rem}\label{Kummer_Lower}
    The case $n=0$ is excluded, and does not even make sense, in \Cref{ex:Kummer}. However, we have shown in \cite[Theorem 3.18]{carmeli2021chromatic}, that if $\cC$ is additive and admits a primitive $m$-th root of unity, then we have an isomorphism
    \[
        \calg^{BC_m-\gal}(S;\cC) \simeq \Map_{\Sp^\cn}(C_m, \pic(S)).
    \]
    Furthermore, we have explained how this recovers and extends classical Kummer theory. The above isomorphism and the one in \Cref{ex:Kummer} can be combined into a single uniform claim (see \Cref{Tn_Higher_Kummer}).
\end{rem}

\subsection{Detection for local rings}\label{ssec:localrings}

We shall now show that \cref{Or_Push} admits a partial converse, when $\cC$ is higher semiadditive and $f\colon \ORone \to \ORtwo$ is a strict map of local ring spectra in the following sense:

\begin{defn}\label{def:Local_Ring}
    Let $\OR$ be a connective commutative ring spectrum. We say that $\OR$ is \tdef{local} if $\pi_0(\OR)$ is local. In this case, we refer to the residue field $k$ of $\pi_0(\OR)$ as the residue field of $\OR$. We shall implicitly assume that $k$ is of characteristic $p$, which implies that $\OR$ is $p$-local.
    We say that a map $\ORone \to \ORtwo$ of local ring spectra is \tdef{strict}, if it induces an isomorphism on residue fields.
\end{defn}

For a local ring spectrum $\OR$ with residue field $k$, the $\infty$-categories $\Modfin{\OR}{n}$ and $\Modfin{k}{n}$ are closely related. 

\begin{defn}
    Let $\cE$ be a pointed $\infty$-category and let $\cE_0$ be a full subcategory of $\cE$. 
    \begin{enumerate}
        \item We say that $\cE_0 \subseteq \cE$ is \tdef{closed under extensions} if for every exact sequence 
        \[
            M \to N \to L
        \]
        in $\cE$ such that $M,L\in \cE_0$, we also have $N\in \cE_0$. 
        \item We say that $\cE_0$ \tdef{generates} $\cE$ \tdef{under extensions}, if the only subcategory of $\cE$ containing $\cE_0$ and closed under extensions is $\cE$ itself.
    \end{enumerate}
\end{defn}

\begin{prop} \label{local_ring_generators_modules}
    Let $\OR\in \calg(\Sp_{(p)}^\cn)$ be local with residue field $k$. 
    The $\infty$-category $\Modfin{\OR}{n}$ is generated under extensions from the essential image of the restriction of scalars functor 
    \[
        f_*\colon\Modfin{k}{n}\too \Modfin{\OR}{n}
    \]
    induced by the residue map $f\colon \OR \to k$. 
\end{prop}

\begin{proof}
    Let $\cE$ denote the minimal subcategory of $\Modfin{\OR}{n}$ which is closed under extensions and contains the essential image of $f_*$. Given $M\in \Modfin{\OR}{n}$, we wish to show that $M\in \cE$. Consider the Postnikov tower 
    \[
M= \tau_{\le n} M \to \tau_{\le n-1} M \to \dots\to \tau_{\le 0} M \to  0 
    \]
    of $M$. For every $1\le t \le n$, we have an exact sequence 
    \[
        \Sigma^t \pi_t M \to \tau_{\le t} M \to \tau_{\le t-1} M 
    \]
    Hence, by induction on the tower, and using that $\cE$ is closed under extensions, it would suffice to show that for every $0\le t \le n$, the $\OR$-module $\Sigma^t\pi_t(M)$ belongs to $\cE$. In particular, it would suffice to show that for every discrete, finite $\pi_0(R)$ module $N$ and every $0\le t\le n$, we have $\Sigma^t N \in \cE$. 
    
    Let $\mathfrak{m}$ be the maximal ideal of $\pi_0(R)$.
Since $\pi_0(R)$ is a local ring and $N$ is finite, we have $\mathfrak{m}^\ell N= 0$ for some $\ell \ge 0$. We shall proceed by induction on $\ell$, where the case $\ell = 0$ holds trivially.  
    Let $N[\mathfrak{m}]$ denote the $\mathfrak{m}$-torsion in $N$, i.e., the submodule of elements killed by $\mathfrak{m}$. Then, we have an exact sequence 
    \[
        \Sigma^t N[\mathfrak{m}] \to\Sigma^t  N \to \Sigma^t N/N[\mathfrak{m}].
    \]
The object $\Sigma^tN[\mathfrak{m}]$ is a restriction of scalars of a $\kappa$-module, so it belongs to $\cE$. Also, by construction we have $\mathfrak{m}^{\ell - 1} \left(N/N[\mathfrak{m}]\right)= 0$. By our inductive hypothesis, this implies that $ \Sigma^t N/N[\mathfrak{m}]\in \cE$, and since $\cE$ is closed under extensions, we deduce that $\Sigma^t N\in \cE$.  
\end{proof}

In particular, this allows us to bootstrap affineness from $k$-modules to $\OR$-modules. 
\begin{cor}\label{orientable_residue_affine_modules}
    Let $\OR\in \calg(\Sp_{(p)}^\cn)$ be local with residue field $k$ and let $\cC \in \calg(\Prl^{\sad{n}})$ be virtually $(k,n)$-orientable. Then, for every $M\in \Modfin{\OR}{n}$, the space $\und{M}$ is $\cC$-affine. 
\end{cor}

\begin{proof}
    By \Cref{orientation_affineness_for_R_modules}, the space $\und{M}$ is $\cC$-affine for $\OR$-module $M$ in the essential image of the functor $\Modfin{k}{n} \to \Modfin{\OR}{n}$. By \Cref{affineness_extensions}, the $\OR$-modules $M$ for which $\und{M}$ is $\cC$-affine are closed under extensions. The result now follows from \Cref{local_ring_generators_modules}. 
\end{proof}

Finally, we can show that the orientation property of a pre-orientation is detected at the residue field. 

\begin{thm}\label{local_ring_orientation}
    Let $\cC\in \calg(\Prl^{\sad{n}})$ and let $f\colon \ORone \to \ORtwo$ be a strict map in $\calg(\Sp_{(p)}^\cn)$. A pre-orientation $\omega \in \POr{\ORone}{\cC}{n}$ is an $\ORone$-orientation if and only if $f_*\omega \in \POr{\ORtwo}{\cC}{n}$ is an $\ORtwo$-orientation. 
\end{thm}
 
\begin{proof}
    The `only if' direction follows from \cref{Or_Push}. Thus, for the `if' direction, it suffices to consider the case where $f\colon R \to k$ is the quotient map to the residue field.
    We now have to show that all objects of $\Modfin{\OR}{n}$ are $\omega$-oriented assuming that all objects of $\Modfin{k}{n}$ are $f_*\omega$-oriented. By \cref{Or_Push}, all the objects in the image of $f_*$ are $\omega$-oriented. Since these generate $\Modfin{\OR}{n}$ under extensions (\Cref{local_ring_generators_modules}), it would suffice to show that the collection of $\omega$-oriented $\OR$-modules is closed under extensions. Since for every $M\in \Modfin{\OR}{n}$, the space $\und{M}$ is $\cC$-affine (\Cref{orientable_residue_affine_modules}), this follows from \Cref{Orientability_Ext_Cof}. 
\end{proof}
 
\begin{rem}\label{Or_Liftintg}
    The above result reduces the verification of the $\OR$-orientability of an $S$-oriented $\infty$-category $\cC$ to a lifting problem. Namely, to $\ORone$-orient $\cC$, it will suffice to have \textit{any} map $\cl{\omega}\colon\Dual{\ORone}{n} \to \one^\times$ which makes the following diagram commutative:
    \[\begin{tikzcd}
    	{\Dual{\ORtwo}{n}} && {\Dual{\ORone}{n}} \\
        & {\one^\times.}
        \arrow[from=1-1, to=1-3]
        \arrow["\omega"', from=1-1, to=2-2]
        \arrow["\cl{\omega}",dashed, from=1-3, to=2-2]
    \end{tikzcd}\]
\end{rem}

\Cref{local_ring_orientation} has the following consequence regrading the functoriality of the construction $\orcyc[\one_\cC]{\OR}{n}$ in $\OR\in\calg(\Sp_{(p)}^\cn)$: 

\begin{prop} \label{local_map_universal_tensor}
    Let $\cC \in \calg(\Prl^{\sad{n}})$. For every strict map $f\colon \ORone\to \ORtwo$ in $\calg(\Sp_{(p)}^\cn)$, we have 
    \[
        \orcyc{\ORone}{n}\simeq \one[\Dual{\ORone}{n}]\otimes_{\one[\Dual{\ORtwo}{n}]} \orcyc{\ORtwo}{n}.
    \]
\end{prop}

\begin{proof}
    Passing to the co-representable functors, the claim is equivalent to the existence of a natural isomorphism 
    \[
        \Or{\OR}{T;\cC}{n} \simeq 
        \POr{\OR}{T;\cC}{n}\times_{\POr{\ORtwo}{T;\cC}{n}} \Or{\ORtwo}{T;\cC}{n}. 
    \]
    for $T\in \calg(\cC)$.
    Both sides are naturally subspaces of $\POr{\OR}{T;\cC}{n}$; the left hand side consisting of those pre-orientations which are orientations, and the right hand side those pre-orientations $\omega$ for which $f_*\omega$ is an orientations. By \Cref{local_ring_orientation}, these two subspaces are equal for every $T\in \calg(\cC)$, and the result follows.   
\end{proof}

\begin{cor}\label{Virt_Or_Push}
    Let $\cC\in \calg(\Prl^{\sad{n}})$ and let $f\colon \ORone\to \ORtwo$ be a strict map in $\calg(\Sp_{(p)}^\cn)$. If $\cC$ is virtually $(\ORone,n)$-orientable, then it is virtually $(\ORtwo,n)$-orientable.
\end{cor}

\begin{proof}
    By \Cref{local_map_universal_tensor}, there is a map 
    $\orcyc{\ORtwo}{n} \to \orcyc{\ORone}{n}$.
    Therefore, if $\orcyc{\ORone}{n}$ is faithful, then so is $\orcyc{\ORtwo}{n}$, and the claim follows from \Cref{virtually_oriented_universal_faithful}.
\end{proof}



\section{Categorification and Redshift}\label{sec:categorification}

In this section, we shall study the interaction of the various notions developed in this paper (affineness, (pre)orientations and the Fourier transform) with \textit{categorification}.

\subsection{Categorification}\label{ssec:categorification}

For $\cC$ a presentably symmetric monoidal $\infty$-category, the $\infty$-category $\Mod_\cC(\Prl)$ of presentably $\cC$-linear $\infty$-categories admits a canonical symmetric monoidal structure (see \cite[Corollary 5.1.2.6]{HA}). Hence, we could try to apply the theory developed in the previous sections to $\Mod_\cC(\Pr)$ in place of $\cC$. However, $\Mod_\cC(\Pr)$ is usually \textit{not presentable} itself. To avoid set-theoretical complications, we follow the strategy of \cite[\S 5.3.2]{HA}
and adopt the following convention: 

\begin{convention}\label{Convention}
    For every presentable $\infty$-category $\cC$, there exists an uncountable regular cardinal $\kappa$, such that $\cC$ is $\kappa$-compactly generated. We shall always implicitly choose such $\kappa$ and treat $\cC$ as an object of the $\infty$-category $\Prl_\kappa$, which is presentable by \cite[Lemma 5.3.2.9]{HA}. If needed, we shall allow ourselves to implicitly replace $\kappa$ by some larger $\kappa'$ using the canonical (non-full) inclusion $\Prl_\kappa \into \Prl_{\kappa'}.$
    If $\cC$ is furthermore presentably $\EE_n$-monoidal, we let $\Mod_\cC$ be the presentably  $\EE_{n-1}$-monoidal $\infty$-category $\Mod_\cC(\Prl_\kappa)$, for $\kappa$ as in \cite[Lemma 5.3.2.12]{HA}. 
\end{convention}

\subsubsection{Semiadditivity of \(\Mod_\cC\)}

An important feature, for our discussion, of $\Prl$ is that it is $\infty$-semiadditive (see \cite[Example 4.3.11]{AmbiKn}). In view of \cref{Convention}, we shall need the analogous property of $\Prl_\kappa$.

\begin{prop}\label{Pr_k_Sadd}
    For every uncountable regular cardinal $\kappa$, the $\infty$-category $\Prl_\kappa$ is $\infty$-semiadditive. 
\end{prop}

\begin{proof}
    Let $\cat_{\psmall{\kappa}} \subset \cat_\infty$ be the (non-full) subcategory of $\infty$-categories that admit $\kappa$-small colimits and functors preserving them. Since $\kappa$ is assumed to be uncountable, we have an equivalence of $\infty$-categories
    $\Prl_\kappa \simeq \cat_{\psmall{\kappa}},$ 
    by \cite[Proposition 5.5.7.10]{htt}. Similarly, for every integer $m$, let $\cat_{\pfin{m}} \subset \cat_\infty$ be the (non-full) subcategory of $\infty$-categories that admit $m$-finite colimits and functors preserving them. By \cite[Propoistion 5.26]{harpaz2020ambidexterity}, the $\infty$-category $\cat_{\pfin{m}}$ is $m$-semiadditive. Recall from \cite[\S 5.3.6]{htt}, the functor 
    \[
        \mathcal{P}_{\pfin{m}}^{\psmall{\kappa}} \colon 
        \cat_{\pfin{m}} \too \cat_{\psmall{\kappa}},
    \]
    which adds formally $\kappa$-small colimits while fixing the $m$-finite ones. This functor preserves all small, and in particular $m$-finite, colimits by \cite[Corollary 5.3.6.10]{htt}. Moreover, by \cite[Remark 4.8.1.8]{HA}, it is also symmetric monoidal with respect to the canonical symmetric monoidal structures on the source and target, which preserve colimits in each coordinate by \cite[Remark 4.8.1.6]{HA}. Thus, by \cite[Corollary 3.3.2]{TeleAmbi}, the $\infty$-category $\cat_{\psmall{\kappa}} \simeq \Prl_\kappa$ is $m$-semiadditive as well.
\end{proof}

\begin{cor}\label{Mod_C_Sadd}
    For every $n\ge2$ and $\cC \in \alg_{\EE_n}(\Prl)$, the $\infty$-category $\Mod_\cC$ is $\infty$-semiadditive. 
\end{cor}
\begin{proof}
Let $\kappa$ be the cardinal for which $\Mod_\cC= \Mod_\cC(\Prl_\kappa)$. By \cref{Pr_k_Sadd}, the $\infty$-category $\Pr_\kappa$ is $\infty$-semiadditive. Therefore, the claim follows from \cite[Corollary 3.3.2]{TeleAmbi}, applied to the monoidal functor
    \[
        \cC \otimes (-) \colon 
        \Prl_\kappa \too
        \Mod_\cC(\Prl_\kappa). \qedhere
    \]
\end{proof}

\subsubsection{Affineness revisited}

By \cite[\S 4.8.5 and \S 5.3.2]{HA}, for every $1\le n \le \infty$ and a presentably $\EE_n$-monoidal $\infty$-category $\cC$, we have a monoidal, fully faithful embedding 
\[
    \Mod_{(-)} \colon \alg_{\EE_n}(\cC) \into
    \alg_{\EE_{n-1}}(\Mod_\cC),
\]
taking an $\EE_n$ algebra $R$ in $\cC$ to the $\infty$-category of left modules $\Mod_R$ in $\cC$. 
Furthermore, $\Mod_{(-)}$ admits a right adjoint, which takes an $\EE_{n-1}$-monoidal $\cC$-linear $\infty$-category $\cD$ to the endomorphism object of the unit 
$\End(\one_\cD) \in \alg_{\EE_n}(\cC).$ 
We shall reserve a special notation for the value of this right adjoint on morphisms.

\begin{notation}
    For $\cC\in \alg_{\EE_n}(\Prl)$ and $F\colon \cD \to \cE$ in $\alg_{\EE_{n-1}}(\Mod_\cC)$, we denote by 
    \[
        \mdef{\Dec{F}}\colon \End(\one_{\cD}) \to \End(\one_{\cE})\qin \alg_{\EE_n}(\cC)
    \]
    the map induced by $F$ between the endomorphism objects of the units of $\cD$ and $\cE$, and refer to it as the \tdef{decategorifcation} of $F$.
\end{notation}

\begin{rem}\label{Mod_Base_Change}
    By \cite[Proposition 4.8.5.1]{HA}, the functor $\Mod_{(-)}$ is compatible with base-change in the sense that for every $\cC \to \cD$ in $\alg_{\EE_n}(\Prl)$, we have a commutative square
    \[\begin{tikzcd}
    	{\alg_{\EE_{n}}(\cC)} && {\alg_{\EE_{n-1}}(\Mod_\cC)} \\
    	{\alg_{\EE_{n}}(\cD)} && {\alg_{\EE_{n-1}}(\Mod_\cD).}
\arrow["{\Mod_{(-)}}", from=1-1, to=1-3]
\arrow[from=1-1, to=2-1]
\arrow[from=1-3, to=2-3]
\arrow["{\Mod_{(-)}}", from=2-1, to=2-3]
    \end{tikzcd}\]
\end{rem}

Specializing the above discussion to $n=2$, we can characterize the essential image of $\Mod_{(-)}$ in terms of the \textit{affineness}.

\begin{prop} \label{affine_in_image_Mod}
    Let $\cC \in \alg_{\EE_2}(\Prl)$. An $\infty$-category $\cD\in \alg_{\EE_1}(\Mod_\cC)$ belongs to the essential image of the functor
    \[
        \Mod_{(-)}\colon 
        \alg_{\EE_2}(\cC) \to 
        \alg_{\EE_1}(\Mod_\cC),
    \]
    if and only if
    the unit functor $u^*\colon \cC \to \cD$ is affine.
\end{prop}

\begin{proof}
    Since $\Mod_{(-)}$ is fully faithful, $\cD$ belongs to its essential image if and only if the counit map
    \(
        \Mod_{\End(\one_\cD)}(\cC) \to \cD
    \)
    is an isomorphism. Moreover, we have 
    \[
        \End(\one_\cD) \simeq 
        \hom(\one_\cD,\one_\cD) \simeq 
        \hom(u^*\one_\cC,\one_\cD) \simeq 
        \hom(\one_\cC,u_*\one_\cD) \simeq 
        u_*\one_\cD
    \]
    and the above counit map identifies with  
    \(
        u^\sharp\colon \Mod_{u_*\one_{\cD}}(\cC) \to \cD.
    \) 
    Thus, the claim follows from the very definition of affineness.
\end{proof}

The above characterisation of module categories has an immediate consequence for detecting equivalences of $\cC$-linear $\infty$-categories. 
\begin{prop}\label{decat_iso_affine_iso}
    Let $\cC \in \alg_{\EE_2}(\Prl)$ and let $F\colon \cD \to \cE$ in $\alg_{\EE_1}(\Mod_\cC)$, such that the unit functors $\cC \to \cD$ and $\cC \to \cE$ are affine. Then, $F$ is an equivalence if and only if the map 
    \[
        \Dec{F}\colon  \End(\one_\cD) \to \End(\one_\cE) \qin {\alg_{\EE_2}(\cC)}
    \]
    is an isomorphism. 
\end{prop}

\begin{proof}
    By \Cref{affine_in_image_Mod}, the $\infty$-categories $\cD$ and $\cE$ are in the essential image of the fully faithful embedding 
    $\Mod_{(-)}\colon \alg_{\EE_2}(\cC)\into \alg_{\EE_1}(\Mod_\cC).$
    Hence, $F$ is an equivalence if and only if its image under the right adjoint of $\Mod_{(-)}$ is an equivalence. Finally, this right adjoint is given by taking the endomorphism object of the unit and takes $F$ to $\Dec{F}$.
\end{proof}

We conclude this subsection by comparing affineness with respect to $\cC$ and affineness with respect to $\Mod_\cC$. 

\begin{prop}\label{affineness_categorification}
    Let $\cC \in \calg(\Prl)$ 
    and let $A$ be a $\pi$-finite $\cC$-ambidextrous space. If $A$ and $\Omega_a A$, for every $a\in A$, are $\cC$-affine, then $A$ is $\Mod_\cC$-affine.
\end{prop}

\begin{proof}
    By \cref{Mod_C_Sadd}, the space $A$ is $\Mod_\cC$-ambidextrous, and so by \Cref{Affiness_Eilenberg_Moore},
    it would suffice to show that for every $a,b\in A$, the square 
    \[
        \xymatrix{
        \Mod_\cC^A \ar[r]\ar[d] & \Mod_\cC \ar[d] \\ 
        \Mod_\cC \ar[r]   & \Mod_\cC^{\{a\}\times_A \{b\}}
        } \quad (*)
    \]
    is a relative tensor square in $\alg_{\EE_1}(\Mod_\cC)$. 
    Now, the space $A$ is $\cC$-affine by assumption and the space $\{a\}\times_A \{b\}$ is either empty or isomorphic to $\Omega_a A$, and hence $\cC$-affine as well. Since $\pt$ is obviously $\cC$-affine, we can identify the square $(*)$ with the image under the functor 
    \[
        \Mod_{(-)}\colon \calg(\cC) \to \calg(\Mod_\cC)
    \]
    of the square 
    \[
        \xymatrix{
        \one^A \ar[r]\ar[d] & \one \ar[d] \\ 
        \one \ar[r] & \one^{\{a\}\times_A\{b\}}.
        } \quad (**)
    \]
    Since $\Mod_{(-)}$ is colimit preserving, and relative tensor squares of commutative algebras are pushout squares, $\Mod_{(-)}$ takes relative tensor squares in $\calg(\cC)$ to relative tensor squares in $\calg(\Mod_\cC)$. Thus, it would suffice to show that $(**)$ is a relative tensor square. Since, by our assumption, $A$ is $\cC$-affine, this follows again from \Cref{Affiness_Eilenberg_Moore}.
\end{proof}

\subsection{The categorical Fourier transform}\label{ssec:catfourier} 

In this subsection, we compare the Fourier transform for a presentably symmetric monoidal $\infty$-category $\cC$, with the Fourier transform for its categorification $\Mod_\cC$. 

\subsubsection{Looping pre-orientations}

We begin with the observation that height $n$ pre-orientations for $\cC$ are essentially the same thing as height $n+1$ pre-orientations for $\Mod_\cC$. More precisely, for $\OR\in \calg(\Sp_{(p)}^\cn)$, an $\OR$-pre-orientation of height $n+1$ for $\Mod_\cC$ is a map 
\[
    \omega \colon 
    \Dual{\OR}{n+1} \too 
\cC^\times=: \pic(\cC),
\] 
where $\pic(\cC)$ is the Picard spectrum of $\cC$, consisting of $\otimes$-invertible objects. 
By applying the functor $\Omega \colon \Sp^\cn \to \Sp^\cn$ to $\omega$,  we get a morphism
\[
    \Omega \omega \colon 
    \Dual{\OR}{n}\simeq \Omega \Dual{\OR}{n+1} \too 
\Omega \pic(\cC)= \one_\cC^\times,
\] 
which we can view as an $\OR$-pre-orientation of $\cC$ of height $n$. 
Under the (natural) assumption that $R$ is $n$-truncated, taking loops provides an \textit{isomorphism} between $\POr{\OR}{\Mod_{\cC}}{n+1}$ and $\POr{\OR}{\cC}{n}$.

\begin{prop}\label{Por_Cat_Obstruction}
    Let $\cC \in \calg(\Prl)$ and let $\OR\in \calg(\Sp_{(p)}^\cn)$. If $\OR$ is $n$-truncated, we get an isomorphism of spaces
    \[
        \Omega\colon 
        \POr{\OR}{\Mod_\cC}{n+1} \iso
        \POr{\OR}{\cC}{n}.
    \] 
\end{prop}

\begin{proof}
Since $\OR$ is $n$-truncated we get $\pi_0(\Dual{\OR}{n+1})= 0$. Consequently, for every $M \in \Sp^\cn$, 
    \[
        \Map_{\Sp^\cn}(\Dual{\OR}{n+1},M) \simeq 
        \Map_{\Sp^\cn}(\Dual{\OR}{n+1},\tau_{\ge1 }M) \simeq 
        \Map_{\Sp^\cn}(\Dual{\OR}{n},\Omega M). 
    \]
The result follows by taking $M= \pic(\cC)$. 
\end{proof}

\subsubsection{Categorical group algebras}

Given $\omega \in \POr{\OR}{\Mod_\cC}{n+1}$, we get for every $M\in \Modfin{\OR}{n+1}$
the \tdef{categorical Fourier transform}:
\[
    \Four_\omega \colon 
    \cC[M] \too
    \cC^{\und{\Dual{M}{n+1}}}
    \qin \calg(\Mod_\cC).
\]
The range is the functor category $\Fun(\und{\Dual{M}{n+1}}, \cC)$ endowed with the \textit{pointwise} symmetric monoidal structure. We shall now show that the domain, i.e., the categorical group algebra $\cC[M]$, is also the functor category $\Fun(\und{M},\cC)$, albeit with the \textit{Day convolution} symmetric monoidal structure. That is, the categorical Fourier transform is a $\cC$-linear, colimit preserving symmetric monoidal functor
\[
    \Four_\omega \colon
    \Fun(\und{M},\cC)_{\Day} \too
    \Fun(\und{\Dual{M}{n+1}},\cC)_{\Ptw},
\]
where the subscripts `Day' and `Ptw' stand for the Day convolution and pointwise symmetric monoidal structures respectively.

\begin{prop} \label{cat_group_alg_is_day}
    For every $\cC\in \calg(\Prl)$ and $M\in\Sp^\cn$, we have a natural isomorphism
    \[
\cC[M] \:\simeq\: \Fun(\und{M},\cC)_{\Day}
        \qin \calg(\Mod_{\cC}).
    \]
\end{prop}

\begin{proof}
    On the one hand,
    \[
        \cC[M] \simeq \cC \otimes \Spc[M]
        \qin \calg(\Mod_{\cC}),
    \]
    and on the other, by \cite[Proposition 3.10]{moshe2021higher}, 
    \[
        \Fun(\und{M},\cC)_{\Day} \simeq 
        \cC \otimes \Fun(\und{M},\Spc)_{\Day}
        \qin \calg(\Mod_{\cC}).
    \]
Thus, the general case follows from the case $\cC= \Spc$. It remains to show that the functor
    \[
        \Fun(\und{-},\Spc)_{\Day} \colon
        \Sp^\cn \too \calg(\Prl) 
    \]
    is left adjoint to the functor
    \[
(-)^\times= \pic \colon \calg(\Prl) \to \Sp^\cn.
    \]
    By \cite{hinich2021colimits}, for every $\cD \in \calg(\Prl)$, we have 
    \[
        \Map_{\calg(\Pr)}(\Fun(\und{M},\Spc)_\Day,\cD) \simeq
        \Map_{\calg(\widehat{\cat}_{\infty})}(\und{M},\cD).
    \]
    Since as a symmetric monoidal $\infty$-category, $\und{M}$ is an $\infty$-groupoid with all objects $\otimes$-invertible, we have
    \[
        \Map_{\calg(\widehat{\cat}_{\infty})}(\und{M},\cD) \simeq
        \Map_{\Sp^{\cn}}(M,\pic(\cD)).
    \]
    The claim follows by stringing together the two isomorphisms above.
\end{proof}

When $M$ is connected, we can also identify the categorical group algebra $\cC[M]$ with the $\infty$-category $\Mod_{\one_\cC[\Omega M]}(\cC)$ of modules over the ordinary group algebra of $\Omega M$ in $\cC$ (as \textit{symmetric monoidal} $\infty$-categories).

\begin{prop}\label{Group_Algebra_Cat}
    For every $\cC \in \calg(\Prl)$ and a connected $M\in\Sp^\cn$, we have a natural isomorphism
    \[
\cC[M] \:\simeq \:
        \Mod_{\one[\Omega M]}(\cC)
        \qin \calg(\Mod_\cC).
    \]
\end{prop}
\begin{proof}
    We shall show that both objects co-represent naturally isomorphic functors from $\Mod_\cC$ into $\Spc$. For every $\cD \in \Mod_\cC$ we have a natural isomorphism
    \[
        \Map_{\calg(\Mod_\cC)}(\cC[M],\cD) \simeq
        \Map_{\Sp^\cn}(M,\pic(\cD)).
    \]
    Since $M$ is connected we have a natural isomorphism
    \[
        \Map_{\Sp^\cn}(M,\pic(\cD)) \simeq 
        \Map_{\Sp^\cn}(\Omega M, \Omega \pic(\cD)) \simeq 
        \Map_{\Sp^\cn}(\Omega M, \one_{\cD}^\times).
    \]
    On the other hand, we have 
    \[
        \Map_{\calg(\Mod_\cC)}(\Mod_{\one_{\cC}[\Omega M]}(\cC),\cD) \simeq
        \Map_{\calg(\cC)}(\one_{\cC}[\Omega M],\End(\one_\cD)) \simeq
        \Map_{\Sp^\cn}(\Omega M, \one_\cD^\times).
    \]
    Thus, both $\cC[M]$ and $\Mod_{\one_{\cC}[\Omega M]}(\cC)$ co-represent the functor $\Map_{\Sp^\cn}(\Omega M , \one_{(-)}^\times)$ and hence isomorphic by the Yoneda lemma.
\end{proof}

\begin{rem}\label{Group_Algebra_Cat_Explained}
    The isomorphism provided by the proof of \Cref{Group_Algebra_Cat} can be succinctly summarized as follows. The object $\cC[M]$ corepresents maps $M \to \pic(-)$ and the object $\Mod_{\one_{\cC}[\Omega M]}$ co-represents maps $\Omega M \to \one_{(-)}^\times$. When $M$ is connected, the two types of data are equivalent by taking $\Omega$ and using the identification $\one^\times \simeq \Omega\, \pic$. 
\end{rem}

\subsubsection{Decategorifying the Fourier transform}

We shall now explain how the Fourier transform for $\cC$ is essentially the decategorification of the Fourier transform for $\Mod_\cC$. Given $\omega \in \POr{\OR}{\cC}{n+1}$ and $M\in \Mod_R^{[0,n+1]}$, the categorical Fourier transform decategorifies to a natural transformation
\[
    \Dec{\Four}_\omega \colon
    \End(\one_{\cC[M]}) \too
    \End(\one_{\cC^{\und{\Dual{M}{n+1}}}})
    \qin \calg(\cC).
\]
We now interpret the source and target in terms of familiar commutative algebras. First, we have,
\[
\End(\one_{\cC^{\und{\Dual{M}{n+1}}}}) \:\simeq\:
    \one^{\und{\Dual{M}{n+1}}}
    \qin \calg(\cC).
\]
And second, when $M$ is connected, we have by \Cref{Group_Algebra_Cat},
\[
\End(\one_{\cC[M]}) \:\simeq\:
    \one[\Omega M]
    \qin \calg(\cC).
    \footnote{In fact, it can be easily deduced that this isomorphism also holds without the assumption that $M$ is connected.}
\]

Furthermore, for every $M \in \Mod_\OR^{[0,n]}$, we also have a natural isomorphism
\[
\Dual{(\Sigma M)}{n+1}=
    \hom^\cn(\Sigma M, \Dual{\OR}{n+1}) \simeq 
    \hom^\cn(M, \Omega (\Dual{\OR}{n+1})) \simeq 
\hom^\cn(M, \Dual{\OR}{n})=
    \Dual{M}{n}.
\]

Using these isomorphisms we have the following:

\begin{prop} \label{Fourier_Cat}
    Let $\cC \in \calg(\Prl)$, let $\OR\in \calg(\Sp_{(p)}^\cn)$ and let $\omega \in  \POr{\OR}{\Mod_\cC}{n+1}$. The following diagram of functors $\Mod_\OR^{[0,n]} \to \calg(\cC)$ commutes:
    \[\begin{tikzcd}
    	{\one[M]} && {\one^{\und{\Dual{M}{n}}}} \\
    	{\one[\Omega\Sigma M]} && {\one^{\und{\Dual{(\Sigma M)}{n+1}}}.}
        \arrow["\wr"', from=1-1, to=2-1]
        \arrow["\wr", from=1-3, to=2-3]
        \arrow["{\Four_{\Omega \omega}}", from=1-1, to=1-3]
        \arrow["{\Dec{\Four}_{\omega}}", from=2-1, to=2-3]
    \end{tikzcd}\]    
\end{prop}

\begin{proof}
    By \Cref{Four_is_POr}, it suffices to show that the pre-orientation $\Dual{\OR}{n} \to \one^\times$ associated with the path down--right--up is also $\Omega \omega$. This follows from \Cref{Group_Algebra_Cat} and 
    \Cref{Group_Algebra_Cat_Explained}
\end{proof}

\subsection{Orientations and categorification}\label{ssec:orcat}

For every $\cC \in \calg(\Prl)$, we have constructed a map 
\[
    \Omega\colon \POr{\OR}{\Mod_{\cC}}{n+1} \too \POr{\OR}{\cC}{n}
\]
between the respective spaces of $\OR$-pre-orientations.
Both the domain and the range have distinguished subspaces consisting of the $\OR$-\textit{orientations}. We shall now show that in the higher semiadditive setting these are preserved and detected by $\Omega$. 

\begin{prop} \label{orientability_categorification}
    Let $\cC \in \calg(\Prl)$, let $\OR\in\calg(\Sp_{(p)}^\cn)$ and let $\omega \in \POr{\OR}{\Mod_{\cC}}{n+1}$. For every $M\in\Modfin{\OR}{n}$, if $\Sigma M$ is $\omega$-oriented, then  $M$ is $\Omega \omega$-oriented. The converse holds if $\und{\Dual{M}{n}}$ is $\cC$-affine.
\end{prop}

\begin{proof}
    By \Cref{Fourier_Cat}, $\Four_{\Omega \omega}$ is an isomorphism at $M$ if and only if $\Dec{\Four}_\omega$ is an isomorphism at $\Sigma M$. Now, $\Dec{\Four}_\omega$ is an isomorphism at $\Sigma M$, if $\Four_\omega$ is an isomorphism at $\Sigma M$ and the converse holds if the source and target of 
    \[
        \Four_\omega \colon
        \cC[\Sigma M] \too
        \cC^{\und{\Dual{(\Sigma M)}{n+1}}}
    \] 
    have affine units (\Cref{decat_iso_affine_iso}). For the source, since $\Sigma M$ is connected, the unit $\cC \to \cC[\Sigma M]$ is affine by \Cref{Group_Algebra_Cat} and \Cref{affine_in_image_Mod}. For the target, by assumption, the space 
    $\und{\Dual{(\Sigma M)}{n+1}} \simeq \und{\Dual{M}{n}}$  
    is $\cC$-affine, and hence the unit 
    $\cC \to \cC^{\und{\Dual{(\Sigma M)}{n+1}}}$
    is affine.
\end{proof}

\begin{thm}\label{Or_Cat}
    Let $\cC \in \calg(\Prl^{\sad{(n+1)}})$, let $\OR\in\calg(\Sp_{(p)}^\cn)$ and let $\omega \in \POr{\OR}{\Mod_{\cC}}{n+1}$. Then, $\omega$ is an orientation of $\Mod_\cC$ if and only if $\Omega \omega$ is an orientation of $\cC$. 
\end{thm}

\begin{proof}
    By \Cref{orientability_categorification}, if $\omega$ is an orientation on $\Mod_\cC$, then $\Omega \omega$ is an orientation on $\cC$. Conversely, assume that $\omega$ is an orientation on $\cC$. We first observe that by \Cref{orientation_affineness_for_R_modules}, for every $M \in \Modfin{\OR}{n}$, the space $\und{\Dual{M}{n}}$ is $\cC$-affine, and hence by \Cref{orientability_categorification} again, $\Sigma M$ is $\omega$-oriented. In other words, we get that all \textit{$1$-connective} $N \in \Mod_\OR^{[1,n+1]\text{-fin}}$ are $\omega$-oriented. By \Cref{oriented_dual}, $\omega$-oriented modules are closed under the operation $M \mapsto \Dual{M}{n+1}=N$, and hence also all $M \in \Modfin{\OR}{n}$ are $\omega$-oriented. Finally, for a general 
    $M \in \Modfin{\OR}{n+1}$, we have an exact sequence
    \[
        \tau_{\ge 1} M \to M \to \tau_{\le 0} M 
        \qin \Modfin{\OR}{n+1}.
    \]
    The modules $\tau_{\ge 1} M$ and $\tau_{\le 0} M$ are $\omega$-oriented by the above. Moreover, the finite set $\und{\tau_{\le 0} M}$ is $\Mod_\cC$-affine, as $\Mod_\cC$ is 0-semiadditive (\Cref{affine_finite_sets}). We conclude by \Cref{Orientability_Ext_Fib} that $M$ is $\omega$-oriented.
\end{proof}

\begin{cor}\label{Orientability_Cat}
    Let $\cC \in \calg(\Prl^{\sad{(n+1)}})$ and let $\OR\in\calg(\Sp_{(p)}^\cn)$. If $\OR$ is $n$-truncated, we get an isomorphism of spaces
    \[
        \Omega \colon 
        \Or{\OR}{\Mod_\cC}{n+1} \iso
        \Or{\OR}{\cC}{n}.
    \]
    In particular, $\cC$ is $(\OR,n)$-orientable if and only if $\Mod_\cC$ is $(\OR,n+1)$-orientable.
\end{cor}
\begin{proof}
    By \Cref{Or_Cat}, the analogous isomorphism of the corresponding spaces of pre-orientations, provided by \Cref{Por_Cat_Obstruction}, restricts to an isomorphism between the subspaces of orientations. In particular, the domain is non-empty if and only if the range is non-empty.
\end{proof}

When $\cC$ is $(\OR,n)$-orientable, we get for every 
$M \in \Modfin{\OR}{n+1}$ 
an equivalence of $\cC$-linear presentable symmetric monoidal $\infty$-categories
\[
    \Four_\omega \colon
    \Fun(\und{M},\cC)_{\Day} \iso
    \Fun(\und{\Dual{M}{n+1}},\cC)_{\Ptw}.
\]
By \Cref{Fourier_Cat}, we recover the ordinary Fourier transform at $M \in \Modfin{\OR}{n}$,
\[
    \Four_{\Omega \omega} \colon
    \one[M] \iso
    \one^{\und{\Dual{M}{n}}},
\]
by applying the functor $\End(\one_{(-)})$ to $\Four_\omega$ at $\Sigma M$. However, in the oriented case, we can also go the other way around. By \Cref{Group_Algebra_Cat}, we have
\[
    \Fun(\und{\Sigma M},\cC)_{\Day} \simeq 
    \Mod_{\one[M]}(\cC),    
\]
and by \Cref{orientation_affineness_for_R_modules}, the space $\und{\Dual{M}{n}}$ is $\cC$-affine, so that we have
\[
    \Fun(\und{\Dual{M}{n}},\cC)_{\Ptw} \simeq
    \Mod_{\one^{\und{\Dual{M}{n}}}}(\cC).
\]
Thus, the categorical Fourier transform $\Four_\omega$ at $\Sigma M$ can be recovered from the ordinary Fourier transform $\Four_{\Omega \omega}$ at $M$ by applying the functor $\Mod_{(-)}(\cC)$. 

\begin{rem}
    To be precise, the above procedure recovers $\Four_\omega$ at $\Sigma M$ from 
    $\Four_{\Omega \omega}$ at $M$ only as an equivalence of plain ($\cC$-linear) $\infty$-categories. The equivalence as \textit{symmetric monoidal} $\infty$-categories can be deduced from the fact that $\Four_{\Omega \omega}$ is an isomorphism of \textit{Hopf algebras} (\Cref{Fourier_Hopf}), though we shall not prove nor use this.
\end{rem}

In contrast, for \textit{non-connected} modules $M \in \Modfin{R}{n+1}$, the categorical Fourier transform provides some new information. 

\begin{example}
    When $M$ is discrete (i.e., a finite abelian group), the inverse of the categorical Fourier transform at its Pontryagin dual $M^*$ assumes the form
    \[
        \Four_\omega^{-1} \colon 
        \cC^{B^{n+1}M} \iso 
        \prod\nolimits_{\varphi \in M^*} \cC.
    \]
    This should be thought of as providing a decomposition of every $\cC$-valued representation of the group $\Omega B^{n+1} M= B^n M$ into a sum of characters. For $R = \ZZ/p^r$ and height $n=0$, this reproduces the Fourier transform considered in \cite[Definition 3.12]{carmeli2021chromatic}.
\end{example}

Another consequence of \Cref{Orientability_Cat} is that the categorical $\OR$-cyclotomic extension is simply the $\infty$-category of modules over the usual $\OR$-cyclotomic extension. 

\begin{cor}\label{Or_Cyc_Cat}
    Let $\cC \in \calg(\Prl^{\sad{(n+1)}})$ and let $\OR\in \calg(\Sp_{(p)}^\cn)$ be $n$-truncated. We have a natural isomorphism
    \[
        \orcyc[\cC]{\OR}{n+1} \:\simeq \:
        \Mod_{\orcyc{\OR}{n}}(\cC)
        \qin \calg(\Mod_\cC).
    \]
\end{cor}

\begin{proof}
    Since $R$ is $n$-truncated, $\Dual{\OR}{n}$ is connected and hence by \Cref{Group_Algebra_Cat}, we have an isomorphism
    \[
        \cC[\Dual{\OR}{n+1}] \:\simeq \:
        \Mod_{\one[\Dual{\OR}{n}]}(\cC)
        \qin \calg(\Mod_\cC).
    \]
    This isomorphism corresponds the natural isomorphism of the associated co-represntable functors 
    \[
        \Omega \colon 
        \POr{\OR}{\cD ; \Mod_\cC}{n+1} \iso
        \POr{\OR}{\End(\one_\cD) ; \cC}{n}
    \]
for $\cD \in \calg(\Mod_\cC)$. \Cref{Orientability_Cat} implies that the above isomorphism restrict to an isomorphism between the corresponding subspaces of \textit{orientations}, when $\cD= \cC$. However, by base-changing from $\cC$ to $\cD$, this implies that the same holds for an arbitrary $\cD \in \calg(\Mod_\cC)$. From this we deduce the isomorphism in the claim as these are the objects co-representing the corresponding subspaces of orientations. 
\end{proof}

\section{Orientations for Thickenings of \texorpdfstring{$\FF_p$}{Fp}}\label{sec:examplesforspecificrings}

In the definition of an $\OR$-orientation of height $n$ of an $\infty$-category $\cC$, the main role of the ring spectrum $\OR$ is to determine which objects of $\Spfin{n}$ admit a Fourier transform, namely those admitting an $\OR$-module structure. It is natural to restrict attention to the subcategory of local rings $\OR$ with residue field $\FF_p$ and strict maps in the sense of \Cref{def:Local_Ring}. The  minimal choice $\OR = \FF_p$, i.e., the terminal object, gives rise to a Fourier transform for $\FF_p$-vector spaces. The maximal choice $\OR = \Sph_{(p)}$, i.e., the initial object, gives rise to a Fourier transform for all $p$-local spectra. In between, we also have the finite rings $\ZZ/p^r$, the $p$-local integers $\ZZ_{(p)}$, and the Postnikov truncations of the $p$-local sphere $\tau_{\le d} \Sph_{(p)}$. In this section, we study these particular cases, their special features, interrelations and implications. 

\subsection{\texorpdfstring{$\FF_p$}{Fp}-Orientations and affineness} 
\label{ssec:F_p_orientations}

We start with the minimal case $\OR= \FF_p$ and show that virtual $(\FF_p,n)$-orientability already has several significant consequences. 

\subsubsection{Semiadditive height} 

To begin with, in the definition of a pre-orientation the height was a free parameter. However, in the higher semiadditive setting, if the pre-orientation is an \textit{orientation}, then its height is strongly constrained by the higher semiadditive structure of the $\infty$-category. 

\begin{prop} \label{Fp_orientation_height}
    Let $\cC\in \calg(\Prl^{\sad{\infty}})$. If $\cC$ is virtually $(\FF_p,n)$-orientable, then $\cC$ is of semiadditive height $n$ at $p$.  
\end{prop}

\begin{proof}
First, by \Cref{orientation_affineness_for_R_modules}, the spaces $B^kC_p= \und{\Sigma^k C_p}$ are $\cC$-affine for all $k = 0,\dots ,n$. Hence, by \Cref{affiness_Height_below}(3), $\cC$ is of height $\ge n$ at $p$. To show that $\cC$ is also of height $\le n$ at $p$, it would suffice, by \cite[Proposition 3.2.1]{AmbiHeight}, to show that $\one[\Sigma^{n+1}C_p] \simeq \one \in \calg(\cC)$. Furthermore, it suffices to show this after extending scalars along a faithful commutative algebra, so without loss of generality, we may assume that $\cC$ is itself is $(\FF_p,n)$-orientable. 
    Now, the commutative algebra $\one[\Sigma^{n+1}C_p]$ is the pushout of the following diagram
    \[
        \one \xleftarrow{\quad\varepsilon\quad} 
        \one[\Sigma^n C_p] \xrightarrow{\quad\varepsilon\quad}
        \one.
    \]
    Applying the Fourier transform associated to any $\FF_p$-orientation of height $n$ of $\cC$, we get that $\one[\Sigma^{n+1}C_p]$ is also the pushout of the isomorphic diagram 
    \[
        \one \xleftarrow{\quad \ev_0\quad} 
        \one^{\und{C_p}} \xrightarrow{\quad \ev_0\quad}
        \one.
    \]
    Finally, since $\und{C_p}$ is $\cC$-affine (see \Cref{affine_finite_sets}), by \Cref{Affiness_Eilenberg_Moore}, we have 
    \[
        \one[\Sigma^{n+1}C_p] \simeq
        \one\otimes_{\one^{\und{C_p}}}\one \simeq 
        \one^{\Omega \und{C_p}}\simeq 
        \one.\qedhere
    \] 
\end{proof}

\subsubsection{Affineness for $p$-spaces} 

Virtual $(\FF_p,n)$-orientability also implies the affineness (and non-affineness) of a large class of spaces.

\begin{thm}\label{p_Affineness} 
    Let $\cC\in\calg(\Prl^{\sad{\infty}})$
be non-zero and virtually $(\FF_p,n)$-orientable. A $\pi$-finite $p$-space $A$  is $\cC$-affine if and only if $\pi_{n+1}(A,a)= 0$ for all $a\colon \pt \to A$.
\end{thm}

\begin{proof}
    By \Cref{affineness_extensions}, the collection of $\cC$-affine spaces is closed under extensions. That is, given  a map of spaces $f\colon B\to B'$, if $B'$ and all the fibers of $f$ are $\cC$-affine, then so is $B$. Combined with \Cref{affine_finite_sets},   
    we are reduced to considering only connected $A$. Now, assuming that $\pi_{n+1}(A,a)= 0$, the Postnikov tower of $A$ can be refined to a tower
    \[
        A= \tau_{\le r} A \too 
        \tau_{\le {r-1}} A \too \dots \too 
        \tau_{\le 0} A \simeq 
        \pt,
    \]
    in which the fiber of each map is of the form $B^k C_p$ for $k \neq n+1$. By \cref{Fp_orientation_height}, we get that $\cC$ is of height $n$ at $p$. Therefore,
    for $k\ge n+2$, the spaces $B^k C_p$ are $\cC$-affine by \Cref{Affiness_Height}(1). For $k \le n$, the spaces $B^k C_p= \und{\Sigma^k C_p}$ are $\cC$-affine by \Cref{orientation_affineness_for_R_modules}. Thus, by \Cref{affineness_extensions} again, we conclude that $A$ is $\cC$-affine. 
    
    Conversely, by \Cref{Fp_orientation_height}, $\cC$ is of height $n$ at $p$ and hence by \cite[Proposition 3.2.3]{AmbiHeight}, we have $\cC^A \simeq \cC^{\tau_{\le n+1}A}$. We can thus assume without loss of generality that $A$ is $(n+1)$-finite. Hence,
    if $\pi_{n+1} A \neq 0$, then by refining the Postnikov tower of $A$ as above, we have a fiber sequence
    \[
        B^{n+1}C_p \too A \too B,
    \]
    with $B$ also $(n+1)$-finite. In particular, $\Omega B$, being $n$-finite, is $\cC$-affine by \cref{p_Affineness}. Thus, if $A$ were $\cC$-affine, by applying \cref{affineness_extensions}(2) to the map $B^{n+1}C_p \to A$, we would deduce that $B^{n+1}C_p$ is $\cC$-affine, contradicting \Cref{Affiness_Height}(2).
\end{proof}

\begin{rem}
\label{rem:non_p_spaces_affine}
    In \Cref{p_Affineness}, if $\cC$ is further assumed to be \textit{$p$-local}, as is often the case (e.g., if $\cC$ is stable and $n\ge1$), then the collection of $\cC$-affine spaces includes the larger class of $\pi$-finite spaces $A$, such that for every $a\colon \pt \to A$,
    \begin{enumerate}
        \item $\pi_{1}(A,a)$ is a $p$-group,
        \item  $\pi_{n+1}(A,a)$ is of order prime to $p$.
    \end{enumerate}
    The proof proceeds by the exact same argument as in the proof of \Cref{p_Affineness}, with the additional ingredient that for every prime $\ell \neq p$, the spaces $B^k C_\ell$ are $\cC$-affine for all $k\ge 2$. Indeed, if $\cC$ is $p$-local, then it is of semiadditive height $0$ at $\ell$, so the said claim follows again from \Cref{Affiness_Height}.
\end{rem}

\subsubsection{Bootstrapping virtual orientability}

By \Cref{Virt_Or_Push}, for every local ring spectrum $\OR$ with residue field $\FF_p$, virtual $(\OR,n)$-orientability implies virtual $(\FF_p,n)$-orientability. Conversely, we have the following bootstrap result:

\begin{prop}\label{Virt_Fp_Bootstrap_Rel}
Let $\cC\in\calg(\Prl^{\sad{\infty}})$ and let $\ORone \to \ORtwo$ be a strict map of local rings (in the sense of \cref{def:Local_Ring}) with residue field $\FF_p$, such that:
    \begin{enumerate}
        \item Its fiber is $\pi$-finite. 
        \item It is surjective on $\pi_0$ and $\pi_n$.
    \end{enumerate}
    Then, $\cC$ is virtually $(\ORone,n)$-orientable if and only if it is virtually $(\ORtwo,n)$-orientable.
\end{prop}

\begin{proof} 
    By \Cref{Virt_Or_Push}, if $\cC$ is virtually $(\ORone,n)$-orientable, then it is virtually $(\ORtwo,n)$-orientable, so we only need to prove the converse. 
    By \Cref{virtually_oriented_universal_faithful}, we need to show that $\orcyc{\ORone}{n}$ is faithful under the assumption that $\orcyc{\ORtwo}{n}$ is faithful. Using \Cref{Virt_Or_Push} for the map $\ORtwo\to \FF_p$, we deduce that $\cC$ is virtually $(\FF_p,n)$-orientable. Moreover, since $\orcyc{\FF_p}{n}$ is faithful (again, by \Cref{virtually_oriented_universal_faithful}), it suffices to show that $\orcyc{\ORone}{n} \otimes \orcyc{\FF_p}{n}$ is faithful, so we can replace $\cC$ with the $(\FF_p,n)$-orientable $\infty$-category of $\orcyc{\FF_p}{n}$-modules in $\cC$. In other words, we can assume without loss of generality that $\cC$ itself is $(\FF_p,n)$-orientable. 
    
    We shall say that a map $A \to B$ in $\calg(\cC)$ is faithful, if the functor $B\otimes_A - \colon \Mod_A \to \Mod_B$ is conservative. In particular, $A$ is faithful if the unit map $\one \to A$ is faithful. Since faithful maps are clearly closed under composition, and $\orcyc{\ORtwo}{n}$ is faithful by assumption, it suffices to show that the map $\orcyc{\ORtwo}{n} \to \orcyc{\ORone}{n}$ is faithful.
    By \Cref{local_map_universal_tensor}, the map $\ORone \to \ORtwo$ induces the following pushout square in $\calg(\cC)$:
    \[\begin{tikzcd}
    	{\one[\Dual{\ORtwo}{n}]} && {\one[\Dual{\ORone}{n}]} \\
    	{\orcyc{\ORtwo}{n} } && {\orcyc{\ORone}{n}.}
        \arrow[from=1-1, to=2-1]
        \arrow[from=1-1, to=1-3]
        \arrow[from=1-3, to=2-3]
        \arrow[from=2-1, to=2-3]
    \end{tikzcd}\]
    It is easy to see that faithful maps are also closed under cobase-change, so it actually suffices to show that the map $\one[\Dual{\ORtwo}{n}] \to \one[\Dual{\ORone}{n}]$ is faithful. By \Cref{rem:n_trunc}, we may assume without loss of generality that both $\ORone$ and $\ORtwo$ are $n$-truncated, in addition to being connective. The assumption that $\ORone\to \ORtwo$ is surjective on $\pi_0$ and $\pi_n$ implies that $\Dual{\ORtwo}{n} \to \Dual{\ORone}{n}$ is a map of connective $n$-truncated spectra, which is injective on $\pi_n$ and $\pi_0$. 
    
    We shall prove, more generally, that for every map of $n$-truncated connective spectra $f\colon M \to N$, such that,
    \begin{enumerate}
        \item the fiber of $f$ is $p$-local and $\pi$-finite, and
        \item $f$ is injective on $\pi_0$ and $\pi_n$,
    \end{enumerate}
    the induced map $\one[M] \to \one[N]$ is faithful. The first stage of the relative Postnikov tower factors $f$ as $M\to N_0 \to N$, where $M\to N_0$ is a surjection on $\pi_0$ and $N_0 \to N$ is an injection on $\pi_0$ and an isomorphism on $\pi_k$ for $k\ge1$. 
    It follows that $\one[N]$ is a free $\one[N_0]$-module and hence the map $\one[N_0] \to \one[N]$ is faithful. It thus remains to show that the map $\one[M] \to \one[N_0]$ is faithful. Since $f$ is an injection on $\pi_0$, it follows that $M\to N_0$ is in fact an \textit{isomorphism} on $\pi_0$. In other words, we can assume without loss of generality that $f$ itself is an isomorphism on $\pi_0$. Consequently, the fiber of $f$ is connective. Moreover, since $M$ and $N$ are $n$-truncated and $f$ is injective on $\pi_n$, the fiber of $f$ is also $(n-1)$-truncated. Therefore, by refining the Postnikov tower of $f$ , we get a tower
    \[
        M= M_0 \too M_1 \too \dots \too M_k = N
    \]
    of connective $n$-truncated spectra, where the fiber of each map $M_i \to M_{i+1}$ is isomorphic to $\Sigma^d C_p$ for some $0\le d \le n-1$. Since faithful maps are closed under composition, we can assume without loss of generality that the fiber of $f$ itself is of this form. Now, the exact sequence
    \[
        \Sigma^d C_p \too M \too N
    \]
    induces a a pushout square in $\calg(\cC)$,
    \[\begin{tikzcd}
    	{\one[\Sigma^d C_p]} && {\one} \\
    	{\one[M]} && {\one[N].}
        \arrow[from=1-1, to=2-1]
        \arrow["\varepsilon",from=1-1, to=1-3]
        \arrow[from=1-3, to=2-3]
        \arrow[from=2-1, to=2-3]
    \end{tikzcd}\]
    By the same consideration as above, it would suffice to show that the map 
    \(
        \one[\Sigma^d C_p] \oto{\:\:\varepsilon\:\:} 
        \one
    \)
    is faithful. Applying the Fourier transform associated with any $\FF_p$-orientation of $\cC$, this map identifies, by \Cref{Four_Aug}, with  
    \[
        \one^{B^{n-d} C_p} \simeq 
        \one^{\und{\Sigma^{n-d} C_p}} \oto{\:\ev_e\:} 
        \one
    \]
    for $e\colon \pt \to B^{n-d}C_p$ the base point. Finally, by \Cref{orientation_affineness_for_R_modules}, the space $B^{n-d}C_p$ is $\cC$-affine and so the functor
    \[
        \one \otimes_{\one^{B^{n-d}C_p}} (-) \colon 
        \Mod_{\one^{B^{n-d}C_p}}(\cC) \too
        \cC
    \]  
    identifies, by \Cref{Ev_Sharp}, with
    \[
        e^* \colon \cC^{B^{n-d}C_p} \too \cC,
    \]
    which is conservative because $n-d\ge1$ and hence $B^{n-d}C_p$ is connected. 
\end{proof}

In particular, virtual $(\FF_p,n)$-orientability bootstraps to $(\OR,n)$-orientability for a large class of ring spectra $\OR$.

\begin{cor}\label{Virt_Fp_Bootstrap}
    If $\cC\in\calg(\Prl^{\sad{\infty}})$ is virtually $(\FF_p,n)$-orientable, then it is virtually $(\OR,n)$-orientable for every $\pi$-finite local ring spectrum $\OR$ with residue field $\FF_p$ (e.g., $\OR=\ZZ/p^r$).
\end{cor}

\begin{proof}
    Apply \Cref{Virt_Fp_Bootstrap_Rel} to the map $\OR \to \FF_p$.
\end{proof}

\subsection{\texorpdfstring{$\ZZ/p^r$}{Z/pr}-Orientations and higher roots of unity}\label{ssec:zpror+higherroots}

\subsubsection{Higher roots of unity}

We now turn to the case $\OR= \ZZ/p^r$ for some $r\in\NN$. Since the shifted Brown--Comenetz dual of $\ZZ/p^r$ is given by 
\[
    \Dual{{(\ZZ/p^r)}}{n} \:\simeq\: \Sigma^n\ZZ/p^r
    \qin \Sp^\cn,
\]
a $\ZZ/p^r$-pre-orientation of height $n$ of $\cC$ is the same thing as a $p^r$-th root of unity of height $n$ in $\one_\cC$, in the sense of \cite[Definition 4.2]{carmeli2021chromatic} (see \Cref{ex:rootsofunity}). To compare the theory developed in this paper with the one in \cite{carmeli2021chromatic}, we shall further assume that $\cC$ is \textit{stable}.
At height $n=0$, a $\ZZ/p^r$-pre-orientation is an \textit{orientation} if and only if the corresponding root of unity is \textit{primitive}. At higher heights the situation is in general more subtle, but this does continue to hold under the assumption that $\cC$ is virtually $(\FF_p,n)$-orientable. More precisely, we have the following:

\begin{prop} \label{Orientation_Primitive}
    Let $\cC \in \calg(\Prl^{\sad{\infty}}_\st)$ and let 
    $\omega \colon \Sigma^n\ZZ/p^r \to \one^\times$. If $\omega$ is an orientation, then it is a primitive $p^r$-th root of unity. The converse holds if $\cC$ is virtually $(\FF_p,n)$-orientable.
\end{prop}

\begin{proof}
    Assume that $\omega$ is an orientation. First, by \Cref{Fp_orientation_height}, $\cC$ is of height $n$ at $p$. Second, it follows by \Cref{push_orientations_functor}, that for every commutative algebra
    $S \in \calg(\cC)$, the composition 
    $\Sigma^n\ZZ/p^r \oto{\omega } \one^\times \to S^\times$, which we denote by $\omega_S$, is an orientation on $S$. Now, if 
    \[
        \omega_S^{p^{r-1}}= 1 \qin \roots[S]{p^r}{n},
    \]
    then we have a commutative triangle in $\calg(\cC)$,
    \[\begin{tikzcd}
    	{S[C_{p^r}]} && {S^{B^n C_{p^r}}} \\
    	& S.
        \arrow["{\Four_{\omega_S^{p^{r-1}}}}", from=1-1, to=1-3]
        \arrow["\varepsilon_0"', from=1-1, to=2-2]
        \arrow[from=2-2, to=1-3]
    \end{tikzcd}\]
    But then, by \Cref{Por_Scaling}, the map $S[C_{p^r}] \oto{p^{r-1}} S[C_{p^r}]$ factors through the augmentation $S[C_{p^r}] \oto{\;\varepsilon_0\;} S$, which implies $S=0$. 
    
    Conversely, assume that $\omega$ is primitive and that $\cC$ is virtually $(\FF_p,n)$-orientable. By \Cref{Virt_Fp_Bootstrap}, $\cC$ is virtually $(\ZZ/p^r,n)$-orientable and since we can check that $\omega$ is an orientation after a faithful extension of scalars, we can in fact assume without loss of generality that $\cC$ is $(\ZZ/p^r,n)$-orientable. Choose some $\zeta \in \Or{\ZZ/p^r}{\cC}{n}$ and consider the composition
    \[
        \one^{\ZZ/{p^r}} \xrightarrow{\:\Four_\zeta^{-1}\:}
        \one[\Sigma^n \ZZ/{p^r}] \oto{\;\:\varepsilon_\omega \;\:}
        \one.
    \]
    This map is represented by a sequence of orthogonal idempotents $e_0,\dots,e_{p^r-1} \in \pi_0(\one)$, with the property $e_0 + \dots + e_{p^r-1}= 1$. Thus, we get a decomposition
    \[
        \one \: \simeq \: 
        \one[e_0^{-1}] \times \dots \times \one[e_{p^r-1}^{-1}] 
        \qin \calg(\cC).
    \]
    Since the subfunctor 
    \[
        \Or{\ZZ/p^r}{-;\cC}{n} \sseq \POr{\ZZ/p^r}{-;\cC}{n}
    \] 
    is co-representable, it preserves products. It therefore suffices to show that each of the compositions
    \[
        \one[\Sigma^n \ZZ/p^r] \oto{\:\varepsilon_\omega\:} 
        \one \oto{\:\pi_j \:} \one[e_j^{-1}],
    \]
    for $j=0,\dots ,p^r-1$, corresponds to an orientation. We observe that by \Cref{Por_Scaling}, for each such $j$, we have 
    \[
        \pi_j(\omega) \:=\: 
        \pi_j(\zeta^j) 
        \qin \roots[{\one[e_j^{-1}]}]{p^r}{n}. 
    \]
    If $j$ is divisible by $p$, then 
    \[
        \pi_j(\omega)^{p^{r-1}}= 
        \pi_j(\omega^{p^{r-1}})= 
        \pi_j(\zeta^{j\cdot p^{r-1}})= 
        \pi_j(1)= 1.
    \]
    By the primitivity of $j$, we get that $\one[e_j^{-1}]= 0$ and hence $\pi_j(\omega)$ is trivially an orientation. For $j$ that is not divisible by $p$, the induced map 
    $\one[\Sigma^n \ZZ/p^r] \oto{\: j \:} \one[\Sigma^n \ZZ/p^r]$ 
    is an isomorphism and hence $\zeta^j$ is an orientation. Consequently, $\pi_j(\omega)= \pi_j(\zeta^j)$ is an orientation as well. 
\end{proof}

\subsubsection{Higher cyclotomic extensions}

In \Cref{universal_oriented_ring}, we have shown that $\ZZ/p^r$-orientations of height $n$ are classified by a certain commutative algebra called $\orcyc{\ZZ/p^r}{n}$, which was constructed in a rather indirect manner. In contrast, primitive $p^r$-th roots of unity of height $n$ are classified by the higher cyclotomic extension $\orcyc{p^r}{n}$, which was constructed in \cite[Definition 4.7]{carmeli2021chromatic} by a fairly explicit formula; namely, by splitting a certain idempotent in the group algebra $\one[\Sigma^n \ZZ/p^r]$. As a consequence of \Cref{Orientation_Primitive} these two algebras coincide under the assumption of virtual $(\FF_p,n)$-orientability. 

\begin{cor}\label{Cyclo_Univ_Or}
    Let $\cC \in \calg(\Prl^{\sad{\infty}}_\st)$ be virtually $(\FF_p,n)$-orientable. For every $r\in\NN$,
    \[
\orcyc{\ZZ/p^r}{n} \:\simeq\: \orcyc{p^r}{n} 
        \qin \calg(\cC).  
    \]
\end{cor}

\begin{proof}
    By \Cref{Orientation_Primitive}, both objects co-represent the same subfunctor of 
    \[
        \POr{\OR}{-;\cC}{n} \colon \calg(\cC) \too \Spc.
    \]
    Hence, the claim follows from the Yoneda lemma.
\end{proof}

Using higher cyclotomic extensions and the comparison between higher roots of unity and orientations, we can also show that nil-conservative functors detect virtual $(\FF_p,n)$-orientability. We first need a lemma regarding the functoriality of primitive roots of unity.

\begin{lem}\label{Prim_Roots_Functor}
    Let $F\colon \cC \to \cD$ in $\calg(\Prl^{\sad{\infty}}_\st)$. For every $R \in \calg(\cC)$ of height $n$ at $p$, the induced map 
    $\roots[R]{p^r}{n} \to \roots[F(R)]{p^r}{n}$
    takes primitive roots to primitive roots. 
\end{lem}
\begin{proof}
    Using the adjunction 
    \[
        \one_\cC[-] \colon \Sp^\cn \adj \calg(\cC) \colon (-)^\times,
    \]
    a root $\omega\colon \one_\cC[\Sigma^n C_{p^r}] \to R$ is primitive if and only if 
    \[
        \one_\cC[\Sigma^n C_{p^{r-1}}] 
        \otimes_{\one_\cC[\Sigma^n C_{p^r}]} R \simeq 0.
    \]
    Since $F$ is symmetric monoidal and colimit preserving, applying $F$ to the left hand side we obtain 
    \[
        \one_\cD[\Sigma^n C_{p^{r-1}}] 
        \otimes_{\one_\cD[\Sigma^n C_{p^r}]} F(R),
    \]
    which is zero if and only if $F(\omega)$ is primitive. 
\end{proof}

\begin{prop} \label{nil_cons_virt_F_p}
    Let $F\colon \cC \to \cD$ in $\calg(\Prl^{\sad{\infty}}_\st)$ be nil-conservative. If $\cD$ is virtually $(\FF_p,n)$-orientable, then so is $\cC$.  
\end{prop}

\begin{proof}
    First, by \Cref{Fp_orientation_height}, $\cD$ is of height $n$ at $p$ and therefore so is $\cC$, by \cite[Proposition 4.4.2]{AmbiHeight}. It follows that we can consider $\orcyc[\one_\cC]{p}{n} \in \calg(\cC)$, the $p$-th cyclotomic extension of height $n$ in $\cC$. By \cite[Proposition 4.9(2)]{carmeli2021chromatic}, the commutative algebra $\orcyc[\one_\cC]{p}{n}$ is faithful, hence if we show that it is $(\FF_p,n)$-orientable, the claim will follow. Consider the tautological $\FF_p$-pre-orientation, i.e., root of unity, 
    $\omega \colon \Sigma^n\FF_p \to \orcyc[\one_\cC]{p}{n}^\times$. Since $F$ is nil-conservative, by \Cref{nil_reflect_orientations}, for $\omega$ to be an orientation it suffices that 
    $F(\omega) \colon \Sigma^n\FF_p \to \orcyc[\one_\cD]{p}{n}^\times$
    is an orientation. Since $\omega$ is a primitive root of unity, by \Cref{Prim_Roots_Functor} so is $F(\omega)$, which is therefore an orientation by \Cref{Orientation_Primitive}.
\end{proof}

\begin{rem}
    In fact, under the assumptions of \Cref{Cyclo_Univ_Or}, one can write an ``explicit formula'' for $\orcyc{\OR}{n}$, for any local ring spectrum $\OR$ with residue field $\FF_p$. Indeed, let $\varepsilon_p \in \pi_0 (\one[\Sigma^n \FF_p])$ be the idempotent for which $\orcyc{p}{n} \simeq \one[\Sigma^n \FF_p][\varepsilon_p^{-1}]$ and denote by $\varepsilon_R \in \pi_0(\one[\Dual{\OR}{n}])$ its image under the map $\one[\Sigma^n \FF_p ]\simeq\one[\Dual{\FF_p}{n}] \to \one[\Dual{\OR}{n}].$ 
    By \Cref{local_map_universal_tensor}, we have
    \[
        \orcyc{\OR}{n} \simeq 
        \one[\Dual{\OR}{n}] \otimes_{\one[\Dual{\FF_p}{n}]} \orcyc{\FF_p}{n} \simeq
        \one[\Dual{\OR}{n}] \otimes_{\one[\Sigma^n \FF_p]} \one[\Sigma^n \FF_p][\varepsilon_p^{-1}] \simeq
        \one[\Dual{\OR}{n}][\varepsilon_R^{-1}].
    \]
\end{rem}

The results above show also that virtual $(\FF_p,n)$-orientability implies that the higher cyclotomic extensions are \textit{Galois}. 

\begin{prop}\label{Cyc_Galois}
     Let $\cC \in \calg(\Prl^{\sad{\infty}}_\st)$ be virtually $(\FF_p,n)$-orientable. For every $r\in\NN$, the cyclotomic extension $\orcyc{p^r}{n}$ is faithful $(\ZZ/p^r)^\times$-Galois.
\end{prop}

\begin{proof}
    By \Cref{Cyclo_Univ_Or}, we have $\orcyc{p^r}{n} \simeq \orcyc{\ZZ/p^r}{n} $. By \Cref{Virt_Fp_Bootstrap}, $\cC$ is virtually $(\ZZ/p^r,n)$-orientable, so $\orcyc{\ZZ/p^r}{n}$ is faithful  by \cref{virtually_oriented_universal_faithful}. Thus, the claim follows from \Cref{Orcyc_Galois}.
\end{proof}

\begin{war}\label{war:_Allen}
    For a general, non virtually $(\FF_p,n)$-orientable, $\cC \in \calg(\Prl^{\sad{\infty}}_\st)$, it may happen that $\orcyc{p^r}{n}$ fails to be Galois. In particular, it is possible that
    $\orcyc{\ZZ/p^r}{n} \not\simeq \orcyc{p^r}{n}$, and so that a primitive root of unity fails to be an orientation. This occurs in some examples constructed by Allen Yuan using the Segal conjecture, see \cite{Yuan2022TheSO}.
\end{war}

We have shown that a virtually $(\FF_p,n)$-orientable 
$\cC \in \calg(\Prl^{\sad{\infty}}_\st)$ enjoys the following two, seemingly unrelated, properties:
\begin{enumerate}
    \item All the higher cyclotomic extensions $\orcyc{p^r}{n}$ are $(\ZZ/p^r)^\times$-Galois (\Cref{Cyc_Galois}).
    
    \item All $n$-finite $p$-spaces are $\cC$-affine (\Cref{p_Affineness}).
\end{enumerate}
However, taken together, they turn out to \textit{characterize} virtual $(\FF_p,n)$-orientability in the stable setting. In fact, the converse implication requires only the following a priori weaker versions of the above two properties:

\begin{prop}\label{Fp_Virt_Or_Char}
Let $\cC \in \calg(\Prl^{\sad{\infty}}_\st)$ be of height $n$ at $p$. Then $\cC$ is virtually $(\FF_p,n)$-orientable if and only if the following holds:
    \begin{enumerate}
        \item The higher cyclotomic extension $\orcyc{p}{n}$ is $\FF_p^\times$-Galois.
        \item The spaces $B^k C_p$, for 
$k= \lceil\frac{n}{2}\rceil + 1,\dots,n$, are $\cC$-affine.
    \end{enumerate}
\end{prop}

\begin{proof}
    As noted above, virtual $(\FF_p,n)$-orientability implies conditions (1) and (2) by \Cref{Cyc_Galois} and \Cref{p_Affineness} respectively, so it remains to prove the converse. Since $\cC$ is $1$-semiadditive and $\orcyc{p}{n}$ is Galois, it is also faithful (\cite[Remark 2.3]{carmeli2021chromatic} using \cite[Proposition 6.3.3]{RognesGal}). Thus, it will suffice to show that the canonical primitive $p$-th root of unity of height $n$ of $\orcyc{p}{n}$ is an $\FF_p$-orientation. Equivalently, by base-changing to $\orcyc{p}{n}$, it suffices to show that if $\cC$ itself admits a primitive $p$-th root of unity 
    $\omega \colon \orcyc{p}{n} \to \one$, then $\omega$ is an $\FF_p$-orientation. We will first use condition (1) to show that $\Sigma^n \FF_p$ is an $\omega$-oriented $\FF_p$-module and then use condition (2), and the various closure properties of oriented modules, to deduce that all the modules in $\Modfin{\FF_p}{n}$ are $\omega$-oriented.
    
    Since $\orcyc{p}{n}$ is assumed to be Galois, the existence of the augmentation $\omega \colon \orcyc{p}{n} \to \one$ implies that it is \textit{split Galois}. That is, we have an isomorphism
    $\orcyc{p}{n} \iso \one^{\FF_p^\times},$
    which in the $k$-th coordinate is given by 
    $\omega^k \colon \orcyc{p}{n} \to \one.$ 
    Thus, we get the following composite isomorphism
    \[
        \one[\Sigma^n \FF_p] \iso 
        \orcyc{p}{n} \times \one \iso
        \one^{\FF_p^\times} \times \one \iso
        \one^{\und{\FF_p}}.
    \]
    Observing that the canonical augmentation $\one[\Sigma^n \FF_p] \to \one$ can be thought of as $\omega^0$, it follows by \Cref{Por_Scaling}, that the above isomorphism coincides with $\Four_\omega$ at $\Sigma^n \FF_p$. In other words, the module $\Sigma^n \FF_p$ is $\omega$-oriented. 
    
Next, for each $k= \lceil{\frac{n}{2}}\rceil + 1,\dots, n$, consider the exact sequence
    \[
        \Sigma^{k-1} \FF_p \too 0 \too \Sigma^k \FF_p
        \qin \Mod_{\FF_p}^{[0,n]}.
    \]
By assumption, $B^k C_p= \und{\Sigma^k \FF_p}$ is $\cC$-affine. Hence, by \Cref{Orientability_Ext_Fib}, if $\Sigma^k \FF_p$ is $\omega$-orientable, then $\Sigma^{k-1}\FF_p$ is $\omega$-orientable as well. Thus, by descending induction starting form $k=n$, we get that $\Sigma^k \FF_p$ is $\omega$-oriented for all $k=\lceil{\frac{n}{2}}\rceil,\dots,n$. 
Furthermore, by \Cref{oriented_dual}, if $\Sigma^k \FF_p$ is $\omega$-oriented, then so is $\Sigma^{n-k}\FF_p \simeq \Dual{(\Sigma^k \FF_p)}{n}$. Therefore,  $\Sigma^k \FF_p$ is $\omega$-oriented for all $k=0,\dots, n$. 
Finally, every object of $\Modfin{\FF_p}{n}$ can be written as a finite direct sum of objects of the form $\Sigma^k \FF_p$ for $k=0,\dots,n$. Thus, we deduce from the above and \Cref{Orientability_Sum}, that all such modules are $\omega$-oriented and hence that $\omega$ is an orientation.  
\end{proof}

\begin{rem}
Every $\cC \in \calg(\Prl^{\sad{\infty}}_\st)$ of height $n=0$ is virtually $(\FF_p,0)$-orientable. At height $n=1$, condition (2) of \Cref{Fp_Virt_Or_Char} is vacuous. Hence, $\cC$ is virtually $(\FF_p,1)$-orientable if and only if $\orcyc{p}{1}$ is Galois, which is however not always the case (see \Cref{war:_Allen}). At heights $n\ge2$, condition (2) is non-vacuous, but we don't know whether it is implied by condition (1) or not. 
\end{rem}

\subsection{\texorpdfstring{$\ZZ_{(p)}$}{Z(p)}-Orientations}
\label{ssec:Z_p-orientations and hyper-completeness}

\subsubsection{$\ZZ_{(p)}$-orientability}

In this subsection, we shall consider orientations for the ring $\OR = \ZZ_{(p)}$.
We begin with the observation that, for every height, the data of a $\ZZ_{(p)}$-pre-orientation is nothing but a compatible sequence of $\ZZ/p^r$-pre-orientations for all $r\in \NN$. The quotient maps $\ZZ_{(p)} \onto \ZZ/p^r$ induce a system of compatible maps $\Dual{\ZZ/p^r}{n} \to \Dual{\ZZ_{(p)}}{n}$.

\begin{lem}\label{lem:bcdcolim}
    The assembly map $\colim \Dual{{(\ZZ/p^r)}}{n}\to \Dual{\ZZ_{(p)}}{n}$ is an isomorphism, so that we have 
    \[
        \Dual{\ZZ_{(p)}}{n} \:\simeq\: \Sigma^n \QQ_p/\ZZ_p
        \qin \Sp^\cn.
    \]
\end{lem}
\begin{proof}
    Under the identification $\Dual{{(\ZZ/p^r)}}{n} \simeq \Sigma^n \ZZ/p^r$, the tower of maps $\ZZ/p^{r+1} \onto \ZZ/p^r$ induces the sequence of maps $\Sigma^n \ZZ/p^r \into \Sigma^n \ZZ/p^{r+1}$, whose colimit is $\Sigma^n \QQ_p/\ZZ_p$.
    Since the canonical map
    \[
        \colim \hom_\Ab(\ZZ/p^r, \QQ_p/\ZZ_p) \too 
        \hom_\Ab(\ZZ_{(p)}, \QQ_p/\ZZ_p) = \QQ_p / \ZZ_{p}
    \]
    is an isomorphism, the map
    \(
        \colim \Dual{{(\ZZ/p^r)}}{n} \to
        \Dual{\ZZ_{(p)}}{n}
    \)
    induces an isomorphism on homotopy groups and is thus an isomorphism as well.
\end{proof}

\begin{rem}\label{Cnt_Dual}
    It might seem more natural to speak of $\ZZ_p$-(pre-)orientations, rather than $\ZZ_{(p)}$-(pre-)orientations, but that would require us to take into account the $p$-adic topology of $\ZZ_p$ in the construction of the Pontryagin/Brown--Comenetz dual, which is a technical complication we preferred to avoid. Namely, while the composition 
    \[
        \colim \hom_\Ab(\ZZ/p^r, \QQ_p/\ZZ_p) \too 
        \hom_\Ab(\ZZ_p, \QQ_p/\ZZ_p) \too
        \hom_\Ab(\ZZ_{(p}), \QQ_p/\ZZ_p)
    \]
    is an isomorphism, the first map is merely an inclusion, which identifies the left hand side, and hence $\hom_\Ab(\ZZ_{(p)}, \QQ_p/\ZZ_p)$, with the subgroup of \textit{continuous} group homomorphisms $\ZZ_p \to \QQ_p/\ZZ_p$.
    The most adequate general framework appears to be that of commutative  \textit{pro-$[0,n]$-finite} ring spectra, but it is outside the scope of this paper. 
\end{rem}

We deduce the corresponding statement for orientations.

\begin{prop}\label{universal_algebra_for_Zp}
    Let $\cC\in \calg(\Prl^{\sad{\infty}})$. We have a canonical isomorphism
    \[
        \orcyc{\ZZ_{(p)}}{n} \:\simeq\: \colim \orcyc{\ZZ/p^r}{n} \qin \calg(\cC).
    \]
\end{prop}

\begin{proof}
By \Cref{lem:bcdcolim} and the fact that the group algebra functor preserves colimits, we have 
    \[
\one[\Dual{\ZZ_{(p)}}{n}]\:\simeq\: \colim \one[\Dual{{(\ZZ/p^r)}}{n}]
        \qin \calg(\cC).
    \]
    Thus, by applying \Cref{local_map_universal_tensor} to $\ZZ_{(p)} \onto \FF_p$, we have
    \[
        \orcyc{\ZZ_{(p)}}{n} \simeq
        \one[\Dual{\ZZ/p}{n}] \otimes_{\one[\Dual{\FF_p}{n}]} 
        \orcyc{\FF_p}{n} \simeq
     \]  
    \[
        (\colim \one[\Dual{{(\ZZ/p^r)}}{n}]) \otimes_{\one[\Dual{\FF_p}{n}]} 
        \orcyc{\FF_p}{n} \simeq
        \colim (\one[\Dual{{(\ZZ/p^r)}}{n}] \otimes_{\one[\Dual{\FF_p}{n}]} 
        \orcyc{\FF_p}{n}).
    \]
    Applying \Cref{local_map_universal_tensor} to $\ZZ/p^r \onto \FF_p$, we get
    \[
        \one[\Dual{{(\ZZ/p^r)}}{n}] \otimes_{\one[\Dual{\FF_p}{n}]} \orcyc{\FF_p}{n} \simeq
        \orcyc{\ZZ/{p^r}}{n}.
    \]
    Combining this with the above we obtain the claimed isomorphism. 
\end{proof}

As in \cite[Definition 4.10]{carmeli2021chromatic}, we set     
\[
\orcyc{p^\infty}{n} \::=\: \colim \orcyc{p^r}{n} \qin \calg(\cC).
\]
In the stable setting, we then obtain the following: 

\begin{cor}\label{Zp_Univ_Cyclo}
    Let $\cC\in \calg(\Prl^{\sad{\infty}}_\st)$ be virtually $(\FF_p,n)$-orientable. We have a canonical isomorphism
    \[
\orcyc{\ZZ_{(p)}}{n} \:\simeq\: \orcyc{p^\infty}{n} \qin \calg(\cC).
    \]
\end{cor}
\begin{proof}
    The claim follows from the combination of \Cref{universal_algebra_for_Zp} and \Cref{Cyclo_Univ_Or}.
\end{proof}

\subsubsection{Virtual $\ZZ_{(p)}$-orientability}

At height $n=0$, \Cref{Zp_Univ_Cyclo} implies that if $\cC$ is stable and virtually $(\FF_p,0)$-orientable, then it is also virtually $(\ZZ_{(p)},0)$-orientable. Indeed, 
$\orcyc{\ZZ_{(p)}}{0}$ is then the ordinary infinite cyclotomic extension $\one[\omega_{p^\infty}]$, which is always faithful (since, say, it has $\one$ as a retract). However, for higher heights $n$, virtual $(\FF_p,n)$-orientability need \textit{not} imply virtual $(\ZZ_{(p)},n)$-orientability, even in the stable setting. Namely, even though each $ \orcyc{p^r}{n}$ is faithful, the filtered colimit $\orcyc{p^\infty}{n}= \colim \orcyc{p^r}{n}$
may no longer be. Nevertheless, there are still some general observations that can be made. 
Recall from \cref{def:Virt_Univ}, that the Bousfield localization of $\cC$ with respect to $\orcyc{p^\infty}{n}$ is denoted by $\virt{\cC}{\ZZ_{(p)}}{n}$. We shall  show that this localization is smashing and identify its unit. 
The commutative algebra $\orcyc{p^\infty}{n}$ is acted on by the group $\ZZ_p^\times$, whose torsion subgroup is 
\[
T_p= 
    \begin{cases}
        \ZZ/(p-1) & p \text{ is odd}. \\
        \ZZ/2 & p=2.
    \end{cases}
\]

Furthermore, using the $p$-adic logarithm we get an isomorphism 
$\ZZ_p^\times \simeq T_p \times \ZZ_p,$
which provides a distinguished dense subgroup
\[
    G:= T_p \times \ZZ \:\sseq\: 
    T_p \times \ZZ_p \:\simeq\: 
    \ZZ_p^\times.
\]

 \begin{prop} \label{Z_p_virt_smashing}
    Let $\cC\in \calg(\Prl^{\sad{\infty}}_\st)$ be virtually $(\FF_p,n)$-orientable for some $n\ge1$. Then,  $\virt{\cC}{\ZZ_{(p)}}{n}$ is a smashing localization of $\cC$ with unit
    $\orcyc{p^\infty}{n}^{h G}.$
    In other words, the corresponding localization functor 
    $L\colon \cC \to \virt{\cC}{\ZZ_{(p)}}{n}$
    is given by 
    \[
L(X)= \orcyc{p^\infty}{n}^{h G} \otimes X 
        \qin \cC.
    \]
\end{prop}

\begin{proof}
We need to show that $R:= \orcyc{p^\infty}{n}^{h G}$ is an idempotent algebra in the same Bousfield class as $\orcyc{p^\infty}{n}$. We first observe that since $\cC$ is both $\infty$-semiadditive and stable, we have
    \[
        (-)^{hT_p} \simeq (-)_{hT_p} \qquad\text{and}\qquad
        (-)^{h\ZZ} \simeq \Sigma^{-1}(-)_{h\ZZ},
    \]
    hence the operation 
    $(-)^{h G} \simeq ((-)^{hT_p})^{h\ZZ}$ 
    commutes with colimits and the tensor product. 
    In particular, for every $X\in \cC$, we have
    \[
R \otimes X=
        \orcyc{p^\infty}{n}^{h G} \otimes X \simeq 
        (\orcyc{p^\infty}{n} \otimes X)^{h G},
    \]
    where we set $R:= \orcyc{p^\infty}{n}^{h G}$. Now, on the one hand, we have a map 
    \[
R= \orcyc{p^\infty}{n}^{h G} \too
        \orcyc{p^\infty}{n}
        \qin \calg(\cC).
    \]
    Hence, every $R$-acyclic object is also $\orcyc{p^\infty}{n}$-acyclic. On the other hand, if $X\in\cC$ is $\orcyc{p^\infty}{n}$-acyclic, then we have
    \[
        R \otimes X \simeq 
        (\orcyc{p^\infty}{n} \otimes X)^{h G} \simeq 0,
    \]
    so $X$ is also $R$-acyclic. Consequently, $R$ and $\orcyc{\ZZ_{(p)}}{n}$ are Bousfield equivalent. 
    
    Next, we show that $R$ is an idempotent algebra. Using the notation
    \[
        C(\ZZ_p^\times ; X):= \colim X^{(\ZZ/p^r)^\times} \quad,\quad
        C(\ZZ_p ; X):= \colim X^{\ZZ/p^r},
    \]
    we have
    \[
        \orcyc[R]{p^\infty}{n}:=
        R\otimes \orcyc{p^\infty}{n} \simeq 
        (\orcyc{p^\infty}{n} \otimes \orcyc{p^\infty}{n})^{h G} \simeq
        C(\ZZ_p^\times; \orcyc{p^\infty}{n})^{h G} \simeq 
        C(\ZZ_p^\times; R),
    \] 
    and thus
    \[
        R\otimes R \simeq 
        (\orcyc[R]{p^\infty}{n})^{h G} \simeq 
        C(\ZZ_p^\times; R)^{h G} \simeq 
        C(\ZZ_p^\times / T_p; R)^{h\ZZ} \simeq 
        C(\ZZ_p; R)^{h\ZZ}.
    \] 
    It remains to show that the commutative $R$-algebra unit map $R \to C(\ZZ_p; R)^{h\ZZ}$ is an isomorphism. In fact, this holds for every stable $p$-complete presentably symmetric monoidal $\infty$-category $\cC$ and $R\in\calg(\cC)$, such as our $\cC$ (by \Cref{Fp_orientation_height} and the assumption $n\ge1$).    
    It suffices to check this in the universal case $\cC= \widehat{\Sp}_p$ and for the unit $R= \Sph_{(p)}$. Furthermore, since both sides are connective, it suffices to check this after tensoring with $\FF_p \in \calg(\widehat{\Sp}_p)$. Denoting by $\sigma$ the action of $1\in \ZZ$ on $C(\ZZ_p; \FF_p)$, the claim becomes equivalent to the exactness of the short sequence of abelian groups
    \[
        0 \too \FF_p \too C(\ZZ_p; \FF_p) \oto{\Id - \sigma} C(\ZZ_p; \FF_p) \too 0,
    \]
    which can be easily verified by an explicit computation.
\end{proof}

Finally, virtual $(\ZZ_{(p)},n)$-orientability can be bootstrapped ``all the way up''.  

\begin{prop}\label{Virt_Zp_Bootstrap}
    Let $\cC\in\calg(\Prl^{\sad{\infty}})$. If $\cC$ is virtually $(\ZZ_{(p)},n)$-orientable, then it is virtually $(\OR,n)$-orientable for every local ring spectrum $R$ with residue field $\FF_p$.
\end{prop}

Namely, by replacing virtual $(\FF_p,n)$-orientability with virtual $(\ZZ_{(p)},n)$-orientability, we can remove the \textit{$\pi$-finiteness} assumption from \Cref{Virt_Fp_Bootstrap}.

\begin{proof}
    By \Cref{Virt_Or_Push}, it suffices to consider the initial case $R = \Sph_{(p)}$, and by \Cref{rem:n_trunc}, we can further reduce to $R = \tau_{\le n} \Sph_{(p)}$. The result now follows from \Cref{Virt_Fp_Bootstrap_Rel} applied to the map
    \[
        \tau_{\le n} \Sph_{(p)} \too \tau_{\le0} \Sph_{(p)} \simeq \ZZ_{(p)}. \qedhere
    \]
\end{proof}

\subsection{\texorpdfstring{$\tau_{\le d}\Sph_{(p)}$}{tau<=d S(p)}-Orientations and connectedness}\label{ssec:connectedness}

By \Cref{Virt_Zp_Bootstrap}, a $\ZZ_{(p)}$-orientation of height $n$ for $\cC\in\calg(\Prl^{\sad{\infty}})$ can be lifted to an $\Sph_{(p)}$-orientation, after a faithful extension of scalars. To study the existence of such a lift in $\cC$ itself, we proceed in steps by climbing up the Postnikov tower of $\Sph_{(p)}$. In general, for each $d=0,\dots,n$, there will be an obstruction for lifting a $\tau_{\le d-1}\Sph_{(p)}$-orientation to a $\tau_{\le d}\Sph_{(p)}$-orientation. Here, by convention, we set 
\[
\tau_{\le-1}\Sph_{(p)} := \FF_p.
\]
To study this extension problem, we shall introduce a certain \textit{$d$-connectedness} property of the $\infty$-category $\cC$ that will be closely related to the vanishing of these obstructions.

\subsubsection{Spherical cyclotomic extensions}

We begin by showing that the truncated spherical cyclotomic extensions are \textit{pro-Galois} extensions. We proceed by approximating $\Sph_{(p)}$ by $\pi$-finite local rings, using an argument we learned form Dustin Clausen. Consider the standard cosimplicial resolution (aka Amitsur complex) of the map $\Sph_{(p)}\to\FF_p$:
\[\begin{tikzcd}
	{\Sph_{(p)}} & {\FF_p} & {\FF_p\otimes_{\Sph}\FF_p} & {\FF_p\otimes_{\Sph}\FF_p\otimes_{\Sph}\FF_p} & {}
	\arrow[shift right=1, from=1-2, to=1-3]
	\arrow[shift left=1, from=1-2, to=1-3]
	\arrow[shorten <=4pt, shorten >=4pt, from=1-3, to=1-2]
	\arrow[from=1-3, to=1-4]
	\arrow[shift right=2, from=1-3, to=1-4]
	\arrow[shift left=2, from=1-3, to=1-4]
	\arrow[shift left=1, shorten <=4pt, shorten >=4pt, from=1-4, to=1-3]
	\arrow[shift right=1, shorten <=4pt, shorten >=4pt, from=1-4, to=1-3]
	\arrow[dotted, no head, from=1-4, to=1-5]
	\arrow[from=1-1, to=1-2]
\end{tikzcd}
\]
The partial totalizations
$\Tot{s} := \mathrm{Tot}^s (\FF_p^{\otimes_\Sph (\bullet + 1)})$ for $s\ge0$ assemble into a tower in $\calg(\Sp_{(p)}^\cn)$ under $\Sph_{(p)}$.

\begin{lem}\label{Sph_Tot}
    For every $s \ge 0$, the partial totalization $\Tot{s}$ has finite homotopy groups and is a local ring spectrum with residue field $\FF_p$. Moreover, for every $t\ge 0$, we have $\invlim \pi_t(\Tot{s}) \simeq \pi_t(\Sph_p)$.
\end{lem}
\begin{proof}
    By the computation of the dual Steenrod algebra, the ring spectrum  $\mathcal{A}^\vee = \FF_p \otimes_\Sph \FF_p$ has finite homotopy groups and $\pi_0 \mathcal{A}^\vee \simeq \FF_p$. 
    Hence, the same holds for
    $\FF_p^{\otimes_\Sph (k+1)} \simeq (\mathcal{A}^\vee)^{\otimes_{\FF_p} k}$ for every $k \ge 0$.  Now, each 
    $\Tot{s}$
    can be written as a finite limit of the rings $\FF_p^{\otimes_\Sph (k+1)}$ for $k=0,\dots s$. From this we deduce that the homotopy groups of $\Tot{s}$ are also finite. Moreover, since the functor of units $(-)^\times$ is a right adjoint, it preserves limits. It follows that an element in (the underlying space of) $\Tot{s}$ is invertible if and only if its image is invertible in each $\FF_p^{\otimes_\Sph (k+1)}$. We conclude that $\Tot{s}$ is local with residue field $\FF_p$. 
    Alternatively, from a more computational perspective, the above facts may be deduced from the multiplicative spectral sequence arising from the finite filtration of $\Tot{s}$ by the lower partial totalizations.
    
    Finally, by the convergence of the Adams spectral sequence at $\Sph_{p}$, we have $\invlim \Tot{s} \simeq \Sph_p$. Since all the homotopy groups of all the $\Tot{s}$-s are finite, the Mittag-Leffler condition is satisfied and therefore for all $t \ge0$, we have $\invlim \pi_t(\Tot{s}) \simeq \pi_t(\Sph_p)$.
\end{proof}

Next, we show that the Brown--Comenetz duals of (the truncations of) $\Tot{s}$ also approximate the Brown--Comenetz dual of (the truncations of) $\Sph_{(p)}$.

\begin{lem}\label{Spherical_Dual_Approx}
    For all $0 \le d \le n$, the augmented tower 
    \[
        \tau_{\le d}(\Sph_{(p)}) \too (\dots \too
        \tau_{\le d}(\Tot{2}) \too 
        \tau_{\le d}(\Tot{1}) \too 
        \tau_{\le d}(\Tot{0})) ,
    \]
    induces an isomorphism 
    \[
        \colim_s (\Dual{\tau_{\le d}(\Tot{s})}{n}) \iso 
        \Dual{\tau_{\le d}(\Sph_{(p)})}{n}
        \qin \Sp^\cn.
    \]
\end{lem}
\begin{proof}
    By \Cref{Sph_Tot}, for all $t \ge0$,
    \[
        \invlim_s (\pi_t\Tot{s}) \simeq 
        \pi_t\Sph_p =
            \begin{cases}
                \ZZ_p & t = 0 \\
                p\mathrm{-finite} & t > 0.
            \end{cases}
    \]
    Since both $\ZZ_p$ and all finite abelian groups are topologically finitely generated, we deduce that the above isomorphism holds as pro-finite abelian groups, and therefore induces an isomorphism on continuous Pontryagin duals. This implies that the composition 
    \[
        \colim_s \hom_\Ab(\pi_t\Tot{s},\QQ_p/\ZZ_p) \to
        \hom_\Ab(\pi_t \Sph_p,\QQ_p/\ZZ_p) \to 
        \hom_\Ab(\pi_t \Sph_{(p)},\QQ_p/\ZZ_p)
    \]
    is an isomorphism (see \Cref{Cnt_Dual}), which implies the claim.
\end{proof}

We deduce that the (truncated) spherical cyclotomic extension can be well approximated by the (truncated) $\Tot{s}$-cyclotomic extensions. 

\begin{prop}\label{Spherical_Cyc_Approx}
    Let $\cC \in \calg(\Prsad)$. For all $0 \le d \le n$, there is an isomorphism 
    \[
        \orcyc{\tau_{\le d}\Sph_{(p)}}{n} \:\simeq\:
        \colim \orcyc{\tau_{\le d}\Tot{s}}{n} 
        \qin \calg(\cC).
    \]
\end{prop}
\begin{proof}
    Applying the colimit preserving functor $\one[-]$ to the isomorphism in \Cref{Spherical_Dual_Approx}, we get 
    \[
        \colim \one[\Dual{\tau_{\le d}\Tot{s}}{n}] \iso 
        \one[\Dual{\tau_{\le d}\Sph_{(p)}}{n}] 
        \qin \calg(\cC).
    \]
    Since all the maps in the augmented tower are strict maps of local rings, by \Cref{local_map_universal_tensor}, we get the desired isomorphism by tensoring with $\orcyc{\FF_p}{n}$ over $\one[\Dual{\FF_p}{n}]$.
\end{proof}

\begin{thm}\label{Spherical_Cyc_Galois}
    Let $\cC \in \calg(\Prsad)$ be virtually $(\FF_p.n)$-orientable. For every $0 \le d \le n$, the truncated spherical cyclotomic extension $\orcyc{\tau_{\le d}\Sph_{(p)}}{n}$ is pro-Galois with respect to the pro-$\pi$-finite group $\tau_{\le d}\Sph_p^\times = \invlim_s (\tau_{\le d}\Tot{s}^\times)$.
\end{thm}
\begin{proof}
    Each $\tau_{\le d}\Tot{s}$ is $n$-truncated and $\pi$-finite. Hence, by \cref{Virt_Fp_Bootstrap}, $\cC$ is virtually $(\Tot{s},n)$-orientable, and by \Cref{Orcyc_Galois}, $\orcyc{\tau_{\le d}\Sph_{(p)}}{n}$ is $\tau_{\leq d}\Tot{s}^\times$-Galois. By naturality, these actions are compatible when $s$ varies. It remains to observe that the limit of $\tau_{\le d}\Tot{s}^\times$ is $\tau_{\le d}\Sph_p^\times$. Indeed, by \Cref{Sph_Tot}, $\invlim \tau_{\le d}\Tot{s} \simeq \tau_{\le d}\Sph_p$, and taking the spectrum of units $(-)^\times$ preserves limits.
\end{proof}

\subsubsection{Categorical connectedness}

To study the problem of constructing truncated spherical orientations we introduce the following notion:

\begin{defn}
    Let $\cC\in \calg(\Prl)$ and let $d\ge -2$ be an integer. We say that,
    \begin{enumerate}
        \item A space $A$ is \tdef{$\cC$-reflective} if the canonical map  
        \[
            A \too \Map_{\calg(\cC)}(\one^A,\one) 
            \qin \Spc
        \]
        is an isomorphism.
        
        \item $\cC$ is said to be \tdef{$d$-connected} at a prime $p$ if every $d$-finite $p$-space $A$ is $\cC$-reflective.
    \end{enumerate} 
\end{defn}
 
\begin{rem}
    If $A$ happens to be $\cC$-affine, then by \Cref{Galois_Affine}, we have an isomorphism
    \[
\Map_{\calg(\cC)}(\one^A,\one) \:\simeq\:
        \calg^{A-\gal}(\cC)
    \]
    and the canonical inclusion of $A$ into the above space corresponds to the \textit{trivial} $A$-Galois extensions. Thus, $A$ is $\cC$-reflective if and only if all the $A$-Galois extensions of $\one$ are trivial. Furthermore, if \textit{every} $d$-finite $p$-space $A$ is $\cC$-affine, for example if $\cC$ is virtually $(\FF_p,n)$-orientable for some $n\ge d$ (\Cref{p_Affineness}), then $\cC$ is $d$-connected at $p$ if and only if it has no non-trivial Galois extensions over such spaces. 
\end{rem} 
 
The first few values of $d$ recover some familiar notions:

\begin{example}
    First, since $\one^{\pt}= \one$ is initial in $\calg(\cC)$, every $\cC$ is ($-2$)-connected. Second, since $\one^{\emptyset} = 0$ is the zero ring, $\cC$ is ($-1$)-connected if and only if $\cC \neq 0$. 
\end{example}

Next, we have the following characterization of $0$-connectedness:
    
\begin{prop}\label{0_connected}
For $\cC\in \calg(\Prl^{\sad{0}})$, the following are equivalent:
    \begin{enumerate}
        \item $\cC$ is $0$-connected.
        \item $\one_\cC$ is indecomposable.
        \item For every decomposition $1 = \varepsilon + \delta$  in $\pi_0(\one_\cC)$ such that $\varepsilon$ and $\delta$ are idempotents with $\varepsilon\delta = 0$, we have either $\varepsilon = 1$ and $\delta = 0$ or vice versa.
    \end{enumerate}
\end{prop}

\begin{proof}
    To show that (1) implies (2), we assume by contradiction that $\one$ decomposes as $\one \simeq R\times S $ with both $R$ and $S$ non-zero. Then, the set 
    \begin{align*}
        \pi_0\Map_{\calg(\cC)}(\one\times \one ,\one) & \simeq  \pi_0\Map_{\calg(\cC)}(R^2\times S^2 ,R \times S) \\
        & \simeq \pi_0\Map_{\calg(\cC)}(R^2\times S^2 ,R )\times \pi_0\Map_{\calg(\cC)}(R^2\times S^2 , S)
    \end{align*} 
    contains at least 4 elements. It follows that the set with two elements is not $\cC$-reflective, contradicting (1). 

    Now, (2) implies (3) because for every pair of idempotents $\varepsilon,\delta \in \pi_0(\one)$ as in (3), we have a decomposition of commutative rings
    $\one \simeq \one[\varepsilon^{-1}]\times \one[\delta^{-1}]$, and (2) implies that one of the factors is zero and hence $\varepsilon = 1$ and $\delta = 0$ or vice versa. 

    It remains to prove that (3) implies (1). For a finite set $A$, the ring $\pi_0(\one^A)$ admits a collection of orthogonal idempotents $\varepsilon_a \in \pi_0(\one^A)$ for $a\in A$, such that $\sum_{a\in A}\varepsilon_a= 1$. Thus, by (3),  every map $\one^A \to \one$
    has to send exactly one of the $\varepsilon_a$-s to $1 \in \pi_0(\one)$. We get 
    \[
        \Map_{\calg(\cC)}(\one^A ,\one) \simeq  \coprod_{a\in A} \Map_{\calg(\cC)}(\one^A[\varepsilon_a^{-1}] ,\one) \simeq \coprod_{a\in A} \Map_{\calg(\cC)}(\one ,\one) \simeq A.\qedhere
    \] 
\end{proof}

\begin{example}
    For an ordinary commutative ring $R$, \Cref{0_connected} implies that the $\infty$-category $\Mod_R$ is $0$-connected if and only if the scheme $\Spec(R)$ is connected. 
\end{example}

As an application of \cref{0_connected}, we deduce that $0$-connectedness interacts well with categorification. 

\begin{cor}\label{lem:0_connected_cat}
    Let $\cC\in \calg(\Prl^{\sad{0}})$. Then, $\cC$ is $0$-connected if and only if $\Mod_{\cC}$ is $0$-connected.
\end{cor}
    
\begin{proof}
By \Cref{0_connected} it suffices to show that $\one$ is indecomposable in $\cC$ if and only if $\cC$ is indecomposable in $\Mod_\cC$. First, assume $\one$ is indecomposable and let $\cC= \cC_0\times \cC_1$ in $\calg(\Mod_\cC)$. We get 
    \[
        (\one_{
\cC_0} ,0) \times (0,\one_{\cC_1})= (\one_{
        \cC_0},\one_{
\cC_1})= \one_{\cC} \in \calg(\cC)
    \]
By our assumption either $\one_{\cC_0}= 0$ or $\one_{\cC_1} = 0$ and hence either $\cC_0\simeq 0$ or $\cC_1 \simeq 0$.
    
    Conversely, given $R \simeq R_0\times R_1$ in $\calg(\cC)$, then as in the proof of \cite[Proposition 5.1.11]{AmbiHeight}, the 0-semiadditivity  of $\cC$ implies that  
    \[
        \cC \simeq \Mod_{R_0}(\cC) \times \Mod_{R_1}(\cC). 
    \] 
By the indecomposability of $\cC$, we get either $\Mod_{R_0}(\cC)= 0$ or $\Mod_{R_1}(\cC) = 0$, which implies either $R_0 = 0$ or $R_1 = 0$. 
\end{proof}

The higher notions of connectedness depend on the prime $p$ and are more subtle. A classical example is provided by Mandell's theorem \cite{mandell}.
\begin{example}[Mandell]
    The $\infty$-category $\Mod_{\cl{\mathbb{F}}_p}$ is $\infty$-connected at $p$.
\end{example}

\subsubsection{Connectedness and higher semiadditivity}

We shall, however, be interested in higher connectedness primarily in the higher semiadditive setting. In this case, we first observe that the semiadditive height gives an \textit{upper bound} on connectedness.

\begin{prop}\label{Conn_Height}
    Let $\cC \in \calg(\Prsad)$. If the height of $\cC$ at $p$ is $\le n$, then $\cC$ is not ($n+1$)-connected at $p$. 
\end{prop}
\begin{proof}
Since the height of $\cC$ at $p$ is at most $n$, the $n$-connected space $B^{n+1}C_p$ is $\cC$-acyclic, so that $\one^{B^{n+1}{C_p}}= \one$, see \cite[Proposition 3.2.1]{AmbiHeight}. Thus, the canonical map 
    \[
B^{n+1}{C_p} \to \Map_{\calg(\cC)}(\one^{B^{n+1}{C_p}},\one)= \Map_{\calg(\cC)}(\one,\one) = \pt
    \]
    is \textit{not} an isomorphism. 
\end{proof}

To further analyze the collection of $\cC$-reflective spaces, we use the Eilenberg--Moore property of affine spaces. 
\begin{prop}\label{Mandell_Pullback}
    Let $\cC \in \calg(\Prsad)$. The functor
    \[
\Map_{\calg(\cC)}(\one^{(-)},\one) \:\colon\:
        \Spc \too \Spc
    \]
    preserves pullbacks of diagrams 
    \(
        A \to B \from C
    \)
    of $\pi$-finite spaces, where $B$ is $\cC$-affine. 
\end{prop}

\begin{proof}
    As pushouts in $\calg(\cC)$ are given by relative tensor products, this follows from \Cref{Affiness_Eilenberg_Moore} and the fact that representable functors take pushouts to pullbacks. 
\end{proof}

This implies that under the assumption of affineness $\cC$-reflective spaces are closed under extensions and formation of fibers.

\begin{prop}\label{Mandell_Ext}
    Let $\cC \in \calg(\Prsad) $ and let $f\colon A \to B$ be a map of $\pi$-finite spaces, where $B$ is $\cC$-reflective and $\cC$-affine. Then, $A$ is $\cC$-reflective if and only if all the fibers of $f$ are $\cC$-reflective.
\end{prop}

\begin{proof}
    Consider the canonical natural transformation 
    \[
        (-) \too \Map_{\calg(\cC)}(\one^{(-)},\one)
    \]
    of functors $\Spc \to \Spc$. The domain, being the identity functor, clearly preserves pullbacks. By \Cref{Mandell_Pullback}, the codomain preserves the pullbacks of all diagrams of the form 
$A \oto{f} B \ofrom{\:b\:} \pt.$ 
    Now, consider the commutative diagram 
    \[\begin{tikzcd}
    	{ A} && {\Map_{\calg(\cC)}(\one^A,\one)} \\
    	B && {\Map_{\calg(\cC)}(\one^B,\one),}
\arrow[from=1-1, to=1-3]
\arrow[from=1-3, to=2-3]
\arrow[from=1-1, to=2-1]
\arrow["\sim", from=2-1, to=2-3]
    \end{tikzcd}\]
    where the bottom arrow is an isomorphism since $B$ is $\cC$-reflective. It follows that the induced map on the fibers of the vertical maps at $b\in B$ can be identified with the canonical map
    \[
        f^{-1}(b) \too \Map_{\calg(\cC)}(\one^{f^{-1}(b)},\one).
    \]
    Thus, $A$ is $\cC$-reflective if and only if $f^{-1}(b)$ is $\cC$-reflective for all $b\in B$.
\end{proof}

Assuming the affineness of \textit{all} $d$-finite $p$-spaces, $d$-connectedness can be reduced to the $\cC$-reflectivity of a single space.

\begin{prop}\label{Proto_Hurewicz}
    Let $\cC \in \calg(\Prsad)$, such that all $d$-finite $p$-spaces are $\cC$-affine. Then, $\cC$ is $d$-connected at $p$ if and only if  $B^d C_p$ is $\cC$-reflective.
\end{prop}

\begin{proof}
    If $\cC$ is $d$-connected at $p$, then by definition all $d$-finite $p$-spaces are $\cC$-reflective and in particular $B^d C_p$.
    Conversely, assume that $B^d C_p$ is $\cC$-reflective.
    The functor 
    \[
\Map_{\calg(\cC)}(\one^{(-)},\one) \:\colon\:
        \Spc \too \Spc
    \] 
    takes $\pt$ to $\pt$, and preserves pullbacks of $\pi$-finite spaces with a $\cC$-affine base (\cref{Mandell_Pullback}).
Hence, it commutes with taking loops for  $d$-finite $p$-spaces. It follows that  $B^\ell C_p$ is $\cC$-reflective for all $\ell= 0,\dots d$. Finally, by \Cref{Mandell_Ext}, the $\cC$-reflective $d$-finite $p$-spaces are closed under extensions. Since all the $d$-finite $p$-spaces are generated under extensions by the spaces $C_p, BC_p,\dots,B^d C_p,$ it follows that all of them are $\cC$-reflective, so $\cC$ is $d$-connected at $p$.
\end{proof}

It will be useful for the sequel to have a slight variant of the above.

\begin{prop}\label{Hurewicz}
    Let $\cC \in \calg(\Prsad)$, such that all $d$-finite $p$-spaces are $\cC$-affine. Then, $\cC$ is $d$-connected at $p$ if and only if $\cC$ is 0-connected and the space $\Map_{\calg(\cC)}(\one^{B^d C_p},\one)$
    is $d$-connected, which then in particular implies that 
    \[
        B^d C_p \:\simeq\: \Map_{\calg(\cC)}(\one^{B^d C_p},\one).
    \]  
\end{prop}

The subtle difference from \Cref{Proto_Hurewicz} is that we do not, a priori, require the above isomorphism to be provided by the canonical map, at the expense of requiring $0$-connectedness in advance.

\begin{proof}
    The `only if' part is clear. For the `if' part, consider the canonical map 
    \[
        B^d C_p \too \Map_{\calg(\cC)}(\one^{B^d C_p},\one).
    \]  
    By assumption, both the source and target are $(d-1)$-connected, so it suffices to show that we get an isomorphism after applying the $d$-fold loop space functor. Now, as in the proof of \Cref{Proto_Hurewicz}, the functor $\Map_{\calg(\cC)}(\one^{(-)},\one)$ commutes with taking loops for $d$-finite $p$-spaces. Thus, applying $\Omega^d$ we get the canonical map
    \[
        C_p \too \Map_{\calg(\cC)}(\one^{C_p},\one),
    \]  
    which is an isomorphism by the assumption that $\cC$ is $0$-connected.
\end{proof}

\subsubsection{Connectedness and orientations}

The notions of connectedness at $p$ and $\FF_p$-orientability interact in a non-trivial way. 

\begin{prop}\label{Or_Conn_Bound}
    Let $\cC \in \calg(\Prsad)$. If $\cC$ is virtually $(\FF_p,n)$-orientable, then it is not $(n+1)$-connected at $p$.  
\end{prop}
\begin{proof}
    Combine \Cref{Conn_Height} and \Cref{Fp_orientation_height}.
\end{proof}

Furthermore, by \Cref{p_Affineness}, when $\cC$ is virtually $(\FF_p,n)$-orientable, all $n$-finite $p$-spaces are $\cC$-affine. Hence, by \Cref{Hurewicz}, for every $d\le n$, the $d$-connectedness of $\cC$ depends only on the properties of the commutative algebra $\one^{B^d C_p} \in \calg(\cC)$.
When $\cC$ is actually $(\FF_p,n)$-oriented, we can further reformulate the $d$-connectedness property in terms of the space 
\[
\mu_p(\one):= \Map_{\Sp^\cn}(C_p, \one^\times)
\] 
of $p$-th roots of unity of $\one \in \cC$. 

\begin{prop}\label{Conn_Mu_p}
    Let $\cC \in \calg(\Prsad)$ be $0$-connected and $(\FF_p,n)$-orientable. For every $d\le n$, the $\infty$-category $\cC$ is $d$-connected at $p$ if and only if
    \[
\tau_{\ge n-d}(\mu_p(\one_\cC)) \:\simeq\:  B^n C_p
        \qin \Spc.
    \]
\end{prop}

\begin{proof}
    From the $(\FF_p,n)$-orientability, we deduce that all the $d$-finite $p$-spaces are $\cC$-affine (\Cref{p_Affineness}). Thus, by \Cref{Hurewicz}, $\cC$ is $d$-connected if and only if it is 0-connected, and the space $\Map_{\calg(\cC)}(\one^{B^d C_p},\one)$ is isomorphic to $B^d C_p$. 
    Using the Fourier transform we get
    \[
        \Map_{\calg(\cC)}(\one^{B^d C_p},\one) \simeq
        \Map_{\calg(\cC)}(\one[{\Sigma^{n-d} C_p}],\one).
    \]
    Now, applying the various adjunctions, we get
    \[
        \Map_{\calg(\cC)}(\one[{\Sigma^{n-d} C_p}],\one) \simeq 
        \Map_{\Sp^\cn}({\Sigma^{n-d}C_p},\one^{\times}) \simeq
    \]
    \[
        \Omega^{n-d}\Map_{\Sp^\cn}(C_p,\one^\times) \simeq
        \Omega^{n-d}\mu_p(\one).
    \]
    Thus, $\cC$ is $d$-connected if and only if it is $0$-connected and the space $\Omega^{n-d}\mu_p(\one)$ is isomorphic to $B^d C_p$, which completes the proof.
\end{proof}

\begin{rem}\label{Or_Torsor}
    In particular, for $\cC \in \calg(\Prsad)$ which is $0$-connected and $(\FF_p,n)$-oriented, the space of $\FF_p$-pre-orientations of height $n$ is isomorphic to the discrete set $C_p$. Its subspace of \textit{orientations} consists of the non-zero elements $C_p \smallsetminus \{0\}$, and is a torsor for the action of $\FF_p^\times$ by scaling. 
\end{rem}

The characterization of $d$-connectedness in terms of $p$-th roots of unity also implies that it interacts well with categorification. 

\begin{cor}\label{Conn_Cat}
    Let $\cC \in \calg(\Prl)$ be $0$-connected and $(\mathbb{F}_p,n)$-orientable. For every $d\leq n$, the $\infty$-category $\cC$ is $d$-connected at $p$ if and only if $\Mod_\cC$ is $d$-connected at $p$. 
\end{cor}

\begin{proof}
By \cref{lem:0_connected_cat}, $\Mod_\cC$ is $0$-connected and by \cref{Orientability_Cat}, $\Mod_\cC$ is $(\mathbb{F}_p,n+1)$ orientable. In addition, we have  
    \[
\Omega (\one_{\Mod_\cC})^{\times}= 
        \Omega \Pic(\cC) \simeq
        \one_{\cC}^{\times}.
    \] 
    Hence,
    \[
        \Omega^{(n+1) - d} \mu_p(\one_{\Mod_\cC}) \simeq 
        \Omega^{n - d} \mu_p(\one_{\cC}).
    \]
    So the claim follows from \Cref{Conn_Mu_p}.
\end{proof}

\begin{rem}
    While an $(\FF_p,n)$-orientable $\cC$ is at most $n$-connected, its categorification $\Mod_\cC$ can be $(n+1)$-connected. Furthermore, we shall see that $(n+1)$-connectedness of $\Mod_\cC$ can have interesting implications for $\cC$ itself.
\end{rem}

\subsubsection{Extending orientations}

Our interest in $d$-connectedness steams from the fact that it allows one to extend $\FF_p$-orientations to truncated $\Sph_{(p)}$-orientations. We begin with the somewhat more general setting. In \Cref{local_ring_orientation}, we have seen that given a strict map of local ring spectra $f\colon \ORone\to \ORtwo$, an $\ORone$-preorientation $\omega$ is an orientation if and only if $f_*\omega$ is an orientation. We shall now show, that under more restrictive assumptions on $f$, the property of $d$-connectedness implies the stronger conclusion that every $\ORtwo$-orientation can be lifted to an $\ORone$-orientation.  
\begin{defn}
    Let $f\colon \ORone \to \ORtwo$ be a strict map of local rings in $\Sp^\cn$. We say that $f$ is \tdef{$d$-small} if the fiber of $f$ is $[0,d-1]$-finite and admits an $\ORtwo$-module structure.
\end{defn}

\begin{prop}\label{Small_Ext}
    Let $\cC \in \calg(\Prsad)$ be $d$-connected at $p$ and let $f\colon \ORone \to \ORtwo$ be a $d$-small map between local ring spectra with residue field $\FF_p$. For $n\ge d$, every $\ORtwo$-orientation of $\cC$ of height $n$ extends to an $\OR$-orientation.
\end{prop}
\begin{proof}
    Let $\omega \colon \Dual{\ORtwo}{n} \to \one^\times$ be an $\ORtwo$-orientation. 
By \Cref{local_ring_orientation}, it would suffice to construct a lift as in the diagram:
    \[\begin{tikzcd}
    	{\Dual{\ORtwo}{n}} && {\Dual{\OR}{n}} \\
    	& {\one^\times,}
\arrow[from=1-1, to=1-3]
\arrow["\omega"', from=1-1, to=2-2]
\arrow["\cl{\omega}",dashed, from=1-3, to=2-2]
    \end{tikzcd}
    \]
    see \cref{Or_Liftintg}.
    Let $X$ be the fiber of $f\colon \ORone \to \ORtwo$, as a map of spectra. 
    The obstruction for the existence of $\cl{\omega}$ lies in the group 
    \[
        \pi_0\Map_{\Sp}(\Sigma^{-1}\Dual{X}{n},\one^{\times}).
    \] 
    By the assumption that $f$ is $d$-small, $X$ is in particular connective and $(n-1)$-truncated. Hence, 
    $\Sigma^{-1}(\Dual{X}{n}) \simeq \Dual{(\Sigma X)}{n}$ 
    is connective. Consequently, we have 
    \[
        \Map_{\Sp}(\Dual{(\Sigma X)}{n},\one^{\times}) \simeq \Map_{\Sp^{\cn}}(\Dual{(\Sigma X)}{n},\one^{\times}) \simeq \Map_{\calg(\cC)}(\one[\Dual{\Sigma X}{n}],\one).
    \]
    By assumption, $X$ is $\pi$-finite and admits an $\ORtwo$-module structure. Thus, by the Fourier transform associated with $\omega$, we have
    \[
        \Map_{\calg(\cC)}(\one[\Dual{(\Sigma X)}{n}],\one)  \simeq \Map_{\calg(\cC)}(\one^{\und{\Sigma X}},\one).
    \]    
    Finally, by our assumptions, $\und{X}$ is a $(d-1)$-finite $p$-space, so $\und{\Sigma X}$ is a $d$-finite $p$-space, and hence the $d$-connectedness of $\cC$ implies that
    \[
        \Map_{\calg(\cC)}(\one^{\und{\Sigma X}},\one) \simeq \und{\Sigma X}.
    \]
Since $\pi_0(\und{ \Sigma X})= \pi_0(\Sigma X) \simeq 0$, we see that there is no obstruction to construct $\cl{\omega}$. 
\end{proof}

This specializes to the following criterion for extending $\FF_p$-orientations to truncated $\Sph_{(p)}$-orientations:

\begin{prop}\label{Conn_Ext}
    Let $\cC \in \calg(\Prsad)$ be $d$-connected at $p$. Every $\FF_p$-orientation of height $n$ of $\cC$ extends to a $\tau_{\leq d-1}\Sph_{(p)}$-orientation.
\end{prop}
\begin{proof}  
    The claim holds vacuously for $d=0$, so we may assume $d \ge 1$. Let $\omega\colon \Dual{\FF_p}{n} \to \one^\times$ be an $\FF_p$-orientation of $\cC$. 
    First, we extend $\omega$ to a $\ZZ_{(p)}$-orientation. For all $r\ge 1$, the quotient map $\ZZ/p^{r+1} \onto \ZZ/p^r$ satisfies the assumptions of \Cref{Small_Ext}. Thus, applying it iteratively, we obtain a compatible sequence of $\ZZ/p^r$-orientations extending $\omega$. Since $\Dual{\ZZ_{(p)}}{n} \simeq \colim \Dual{{(\ZZ/p^r)}}{n}$ by \cref{lem:bcdcolim}, and 
    \[
        \Modfin{\ZZ_{(p)}}{n} = 
        \bigcup\nolimits_{r \in \NN} \Modfin{\ZZ/p^r}{n},
    \] 
    the colimit of the associated diagram is the desired extension of $\omega$ to a $\ZZ_{(p)}$-orientation:
    \[
    \begin{tikzcd}
    	{\Dual{{(\ZZ/p)}}{n}} & {\Dual{{(\ZZ/p^2)}}{n}} & {\dots} & {\Dual{{(\ZZ/p^r)}}{n}} & {\dots } & {\Dual{\ZZ_{(p)}}{n}} \\
    	&&&&& {\one^\times}
        \arrow[from=1-1, to=1-2]
        \arrow[from=1-2, to=1-3]
        \arrow[from=1-3, to=1-4]
        \arrow[from=1-4, to=1-5]
        \arrow[dashed, from=1-5, to=1-6]
        \arrow["\omega"',pos=0.3,curve={height=12pt}, from=1-1, to=2-6]
        \arrow[curve={height=6pt}, from=1-2, to=2-6]
        \arrow[curve={height=6pt}, from=1-4, to=2-6]
        \arrow["\omega_{\le 0}",dashed, from=1-6, to=2-6]
    \end{tikzcd}
    \]
    
    Next, we proceed by induction on the Postnikov tower 
    \[
        \tau_{\leq d-1}\Sph_{(p)} \too 
        \dots \too \tau_{\leq 1}\Sph_{(p)} \too 
        \tau_{\leq 0}\Sph_{(p)}= \ZZ_{(p)}.
    \]
    Again, each map in the tower satisfies the conditions of \Cref{Small_Ext}, and thus the $\tau_{\leq 0}\Sph_{(p)}$-orientation $\omega_{\le 0}$ can be extended inductively all the way to a $\tau_{\leq d-1}\Sph_{(p)}$-orientation.
\end{proof}

If $\cC$ admits an $\FF_p$-orientation of height $n$, then it is \textit{not} $(n+1)$-connected (\Cref{Or_Conn_Bound}). Thus, the most we can get from \Cref{Conn_Ext} is when $d=n$, which when combined with \Cref{Conn_Mu_p} gives us the following:

\begin{cor}\label{Conn_Ext_Mu_p}
    Let $\cC \in \calg(\Prsad)$, such that
    \[
        \mu_p(\one)= 
        \Map_{\Sp^\cn}(C_p,\one^\times) 
        \simeq B^n C_p.
    \]
    If $\cC$ is $(\FF_p,n)$-orientable, then it is 
    $(\tau_{\le n-1}\Sph_{(p)},n)$-orientable.
\end{cor}

This raises a natural question of when does a spectrum 
$M \in \Spfin{n}_{(p)}$ admit an action of $\tau_{\leq n-1}\Sph_{(p)}$.
In general, being a module over a truncated sphere is a (rather subtle) structure. However, for spectra concentrated in degrees $0$ to $n$, this structure degenerates to a property, which can moreover be detected on the level of homotopy groups.

\begin{prop}\label{Truncated_Sph_Crit}
    A spectrum $M\in \Sp^{[0,n]}$ admits a module structure over $\tau_{\le n-1}\Sph$ if and only if the action map 
    \[
        \pi_n\Sph\otimes \pi_0 M  \oto{\:\alpha_M\:} \pi_n M
        \qin \Ab
    \] 
    is zero. Moreover, this   module structure is then unique. 
\end{prop}
\begin{proof}
    The $\infty$-category $\Sp^{[0,n]}$ admits a symmetric monoidal structure for which the truncation functor $\tau_{\le n}\colon \Sp^\cn \to \Sp^{[0,n]}$ is symmetric monoidal. In particular, the unit is $\tau_{\le n}\Sph \in \Sp^{[0,n]}$, and the tensor product of $X, Y \in \Sp^{[0,n]}$ is given by $\tau_{\le n} (X\otimes Y) \in \Sp^{[0,n]}$. 
    We divide the claim into two parts:
        \begin{enumerate}
            \item $\tau_{\le n-1}\Sph$ is an idempotent algebra in $\Sp^{[0,n]}$, in the sense of \cite[Definition 4.8.2.8]{HA};
            \item $\tau_{\le n-1}\Sph$ classifies the property that $\alpha_M$ is zero, in the sense that $M \in \Sp^{[0,n]}$ is an $\tau_{\le n-1}\Sph$-module if and only if $\alpha_M=0$.
        \end{enumerate}
    Consider the exact sequence
    \[
        \Sigma^n \pi_n \Sph \oto{\:\:f\:\:}
        \tau_{\le n}\Sph \oto{\:\:u\:\:} 
        \tau_{\le n-1}\Sph,
    \]
    where the truncation map $u$ is the unit map for $\tau_{\le n-1}\Sph$ as a commutative algebra in $\Sp^{[0,n]}$.
    For every $M\in \Sp^{[0,n]}$, we get an exact sequence
    \[
        M \otimes \Sigma^n \pi_n \Sph \oto{1_M\otimes f} 
        M \otimes \tau_{\le n}\Sph \oto{1_M \otimes u} 
        M \otimes \tau_{\le n-1}\Sph.
    \]
    The left most spectrum is $n-1$ connected, and on $\pi_n$ the map $1_M\otimes f$ induces the map
    \[
        \pi_n(M \otimes \Sigma^n \pi_n \Sph) \simeq 
        \pi_0 M \otimes \pi_n \Sph \oto{\:\alpha_M\:}
        \pi_n M \simeq 
        \pi_n(M \otimes \tau_{\le n}\Sph).
    \]
    Thus, by the long exact sequence in homotopy groups, the map $1_M \otimes u$ becomes an isomorphism after applying $\tau_{\le n}$ if and only if $\alpha_M = 0$. Now, for $M=\tau_{\le n-1}\Sph$, the target of $\alpha_M$ is the zero group and hence it is the zero map. Thus $1_{\tau_{\le n-1}\Sph} \otimes u\colon \tau_{\le n-1}\Sph \otimes \tau_{\le n}\Sph \to \tau_{\le n-1}\Sph \otimes \tau_{\le n-1}\Sph$ is an isomorphism, so $\tau_{\le n-1}\Sph$ is an idempotent algebra in $\Sp^{[0,n]}$; this verifies (1). By \cite[Proposition 4.8.2.10]{HA}, this implies that the forgetful functor
    \[
        \Mod_{\tau_{\le n-1}\Sph}(\Sp^{[0,n]}) \too 
        \Sp^{[0,n]}
    \]
    is fully faithful, and the essential image is precisely the objects $M \in \Sp^{[0,n]}$, whose tensor with $u$ in $\Sp^{[0,n]}$ is an isomorphism, which by the above is equivalent to $\alpha_M = 0$. Therefore, we have shown that (2) holds as well.
\end{proof}

\begin{cor}\label{Truncated_Sph_p_Crit}
    A spectrum $M \in \Spfin{n}_{(p)}$ admits a module structure over $\tau_{\le n-1}\Sph_{(p)}$ if and only if the action map
    \[
        \pi_n\Sph_{(p)} \otimes \pi_0 M  \oto{\:\alpha_M\:} \pi_n M
        \qin \Ab
    \] 
    is zero. Moreover, this module structure is then unique. 
\end{cor}
\begin{proof}
    In view of \cite[Proposition 4.8.2.10]{HA}, this follows immediately from \Cref{Truncated_Sph_Crit}.
\end{proof}

Sadly, \Cref{Conn_Ext_Mu_p} stops short from implying 
$(\tau_{\leq n}\Sph_{(p)},n)$-orientability, which will be already the same as $(\Sph_{(p)},n)$-orientability. This obstacle, however, can be overcome using \textit{categorification}. 

\begin{prop}\label{Conn_Ext_Cat}
    Let $\cC \in \calg(\Prsad)$ such that 
    \[
        \Map_{\Sp^\cn}(C_p,\pic(\cC)) \simeq B^{n+1} C_p.
    \]
    If $\cC$ is $(\FF_p,n)$-orientable, then it is $(\Sph_{(p)},n)$-orientable.
\end{prop}

\begin{proof}
    By \Cref{Orientability_Cat}, if $\cC$ is $(\FF_p,n)$-orientable, then $\Mod_\cC$ is $(\FF_p,n+1)$-orientable. Since we can identify $\Map_{\Sp^\cn}(C_p,\pic(\cC))$ with $\mu_p(\one_{\Mod_\cC})$, we get by \Cref{Conn_Mu_p}, that $\Mod_\cC$ is $(n+1)$-connected. Hence, by \Cref{Conn_Ext}, the $\infty$-category $\Mod_\cC$ is $(\tau_{\le n}\Sph_{(p)},n+1)$-orientable. It follows, by \Cref{Orientability_Cat} again, that $\cC$ is $(\tau_{\le n}\Sph_{(p)},n)$-orientable and hence, by \Cref{rem:n_trunc}, $(\Sph_{(p)},n)$-orientable. 
\end{proof}

\subsubsection{Torsion units and $(d+\frac{1}{2})$-connectedness}

To further study extensions of $\FF_p$-orientation to truncated $\Sph_{(p)}$-orientations, it will be convenient to use the following terminology:

\begin{defn}\label{def:Conn_Frac}
    Let $\cC \in \calg(\Prsad)$ be $(\FF_p,n)$-orientable. We say that $\cC$ is \tdef{$(d+\frac{1}{2})$-connected}, if it is $d$-connected and $(\tau_{\le d}\Sph_{(p)},n)$-orientable. 
\end{defn}  

\begin{rem}
    If $\cC \in \calg(\Prsad)$ is $(\FF_p,n)$-oriented and $0$-connected, then by \Cref{Or_Torsor}, the space of its $\FF_p$-orientations of height $n$ is a  torsor for the group $(\ZZ/p)^\times$. Thus, if any of the $\FF_p$-orientations of $\cC$ extends to a $\tau_{\le d}\Sph_{(p)}$-orientation, then all of them do. 
\end{rem}

As a sanity check, we observe that 
$(d+1)$-connectedness implies $(d+\frac{1}{2})$-connectedness by \Cref{Conn_Ext}, which in turn implies $d$-connectedness by definition. We further note that while an $(\FF_p,n)$-orientable $\cC$ can not be $(n+1)$-connected (\Cref{Or_Conn_Bound}), it \textit{can} be $(n+\frac{1}{2})$-connected. The following is a useful criterion for that:

\begin{prop}\label{Pic_Conn_Frac}
    Let $\cC \in \calg(\Prsad)$ such that 
    \[
        \Map_{\Sp^\cn}(C_p,\pic(\cC)) \simeq B^{n+1} C_p.
    \]
    If $\cC$ is $(\FF_p,n)$-orientable, then $\cC$ is $(n+\frac{1}{2})$-connected.
\end{prop}
\begin{proof}
    By \Cref{Conn_Ext_Cat}, we get that $\cC$ is $(\tau_{\le n}\Sph_{(p)},n)$-orientable. By looping once the isomorphism in the hypothesis, we get 
    \[
        \mu_p(\one_\cC) = \Map_{\Sp^\cn}(C_p,\one_\cC^\times) \simeq B^{n} C_p.
    \]
    Thus, $\cC$ is $n$-connected by \cref{Conn_Mu_p}.
\end{proof}

Our next goal is to show that just like $d$-connectedness can be characterized in terms of the $p$-th roots of unity of $\one$ (\Cref{Conn_Mu_p}), the property of $(d+\frac{1}{2})$-connectedness can be similarly characterized in terms of all of the ``$p$-local $\pi$-finite units'' of $\one$. 

\begin{defn}\label{def:pitor}
    Let $\mdef{\Sp_{(p)}^{\ptors{\pi}}} \sseq \Sp^\cn$ be the full subcategory generated under (filtered) colimits by $\Sp_{(p)}^{\pfin{\pi}} \sseq \Sp^\cn$, and denote the right adjoint of the embedding by 
    \[
        (-)_{(p)}^{\ptors{\pi}} \colon 
        \Sp^\cn \too 
        \Sp_{(p)}^{\ptors{\pi}}.
    \]
    We say that a spectrum $X \in \Sp^\cn$ is \tdef{($p$-local) $\pi$-torsion} if it belongs to $\Sp_{(p)}^{\ptors{\pi}}$. 
\end{defn}

The above notion of `$p$-local $\pi$-torsion' is related to $p$-torsion in the usual sense by the following:

\begin{prop}\label{Cp_Tor_Crit}
    Let $X \in \Sp^\cn$. If $X\in \Sp_{(p)}^{\ptors{\pi}}$, then $X[p^{-1}]= 0$. The converse holds if $X$ is bounded above. 
\end{prop}
\begin{proof}
    For the `if' part, observe that the collection of spectra $X$ for which $X[p^{-1}]= 0$, contains $C_p$ and is closed under colimits in $\Sp$. Since all $p$-local $\pi$-finite spectra are generated by $C_p$ under colimits, the same holds for all $p$-local $\pi$-torsion spectra and the claim follows. 
    Conversely, if $X\in\Sp^\cn$ is $p$-torsion and bounded above, it is the colimit of its Postnikov truncations, for which the successive fibers are (de)suspensions of $p$-torsion abelian groups, which in turn are all generated from $C_p$ by colimits in $\Sp$.
\end{proof}

\begin{rem}
    Some additional hypothesis is necessary for the converse part in \Cref{Cp_Tor_Crit}, as $X= \Sph/p$ is an example of a connective spectrum with $X[p^{-1}] = 0$, but $X \notin \Sp_{(p)}^{\ptors{\pi}}$.
\end{rem}

\begin{defn}
    For $\cC \in \calg(\cat_\infty)$ and $S\in \calg(\cC)$, the \tdef{$p$-local $\pi$-torsion units} of $S$ are given by 
    \[
        \mdef{\disc[S]}:= (S^\times)_{(p)}^{\ptors{\pi}}
        \qin \Sp_{(p)}^{\ptors{\pi}}.
    \]
\end{defn}

In many examples of interest, such as 
$\OR= \tau_{\le d} \Sph_{(p)},$ the spectrum $\Dual{\OR}{n}$ belongs to $\Sp_{(p)}^{\ptors{\pi}}$. In these cases, pre-orientations $\Dual{\OR}{n} \to S^\times$ factor uniquely through $\disc[S] \to S^\times$ and are thus in bijection with maps $\Dual{\OR}{n} \to \disc[S]$.
As with ordinary roots of unity, the semiadditive height gives an upper bound on the truncatedness of the spectrum $\disc[S]$.

\begin{prop}\label{Pi_Units_Height}
    Let $\cC \in \calg(\cat^{\sad{\infty}})$ and let $S\in\calg(\cC)$. If $S$ is of height $\le n$ at $p$, then the spectrum $\disc[S]$ is $n$-truncated.
\end{prop}
\begin{proof}
    For every $M\in \Sp_{(p)}^{\pfin{\pi}}$, the fold map $S[\Sigma^{n+1} M] \to S$ is an isomorphism by \cite[Proposition 3.2.3]{AmbiHeight}. Thus, applying the various adjunctions, we have
    \[
        \Map_{\Sp^{\cn}}(M, \Omega^{n+1}\disc[S]) \simeq
        \Map_{\Sp^{\cn}}(\Sigma^{n+1} M, \disc[S]) \simeq
    \]
    \[
        \Map_{\Sp^{\cn}}(\Sigma^{n+1} M, S^{\times}) \simeq
        \Map_{\calg_S(\cC)}(S[\Sigma^{n+1} M], S) \simeq
        \Map_{\calg_S(\cC)}(S, S) \simeq 
        \pt.
    \]
Since $\Sp_{(p)}^{\pfin{\pi}}$ generates $\Sp_{(p)}^{\ptors{\pi}}$ under colimits, it follows that $\Omega^{n+1}\disc[S]= 0$ and hence $\disc[S]$ is $n$-truncated.
\end{proof}

We therefore focus on the full subcategory $\mdef{\Sptors{n}_{(p)}} \sseq \Sp_{(p)}^{\ptors{\pi}}$ of $n$-truncated objects. Furthermore, passing from $\Spfin{n}_{(p)}$ to $\Sptors{n}_{(p)}$ is a purely formal operation. 

\begin{lem}\label{pTors_Ind}
    For every prime $p$ and integer $n\ge 0$, we have an equivalence of $\infty$-categories 
    \[
        \Sptors{n}_{(p)}\, \simeq\, \Ind(\Spfin{n}_{(p)}).
    \]
\end{lem}
\begin{proof}
    By definition, the $\pi$-finite objects in $\Sp_{(p)}^{\ptors{\pi}}$ generate it under colimits. Since $n$-truncation preserves colimits and the property of being $\pi$-finite, the same holds for $\Sptors{n}_{(p)}$. It hence remains to show that the $\pi$-finite objects are compact in $\Sptors{n}_{(p)}$. This follows from the fact that any $n$-truncated $\pi$-finite spectrum can be represented as the $n$-truncation of a finite spectrum.   
\end{proof}

With these preliminaries, we are ready to prove the following analogue of \Cref{Conn_Mu_p}:

\begin{thm}\label{orientation_discrepency}
    Let $\cC \in \calg(\Prsad)$ be $0$-connected and $(\FF_p,n)$-oriented. For every $d\le n$, the $\infty$-category $\cC$ is $(d+\frac{1}{2})$-connected at $p$ if and only if
    \[
\tau_{\ge n-d} (\disc) \:\simeq\: 
        \tau_{\ge n-d} (\Dual{\Sph_{(p)}}{n}) 
        \qin \Sp^\cn. 
    \]
\end{thm}

\begin{proof}
Throughout the proof we shall write $\Omega \colon \Sp^\cn \to \Sp^\cn$  for the loops functor from the $\infty$-category of connective spectra to \textit{itself}. With this convention, the above isomorphism is equivalent to the following one:
    \[
        \Omega^{n-d} (\disc) \simeq 
        \Omega^{n-d} (\Dual{\Sph_{(p)}}{n}) 
        \qin \Sp^\cn. 
    \]
    We begin with the `if' part. Assuming the above isomorphism we get
    \[
\Omega^{n-d}\mu_p(\one)=
        \Omega^{n-d}\hom_{\Sp^\cn}(C_p, \one^\times) \simeq 
        \Omega^{n-d}\hom_{\Sp^\cn}(C_p, \disc) \simeq
    \] 
    \[
        \hom_{\Sp^\cn}(C_p, \Omega^{n-d}\disc) \simeq
        \hom_{\Sp^\cn}(C_p, \Omega^{n-d} (\Dual{\Sph_{(p)}}{n})) \simeq
    \] 
    \[
        \Omega^{n-d} \hom_{\Sp^\cn}(C_p, (\Dual{\Sph_{(p)}}{n})) \simeq
        \Omega^{n-d} (\Dual{C_p}{n}) \simeq 
        \Omega^{n-d} (\Sigma^n C_p) \simeq
        \Sigma^d C_p.
    \]
    Hence, by \Cref{Conn_Mu_p}, $\cC$ is $d$-connected. It remains to show that $\cC$ is $(\tau_{\le d}\Sph_{(p)},n)$-orientable. Let $\omega$ be an $\FF_p$-orientation of height $n$ for $\cC$. By \Cref{local_ring_orientation}, it suffices to solve the lifting problem 
    \[\begin{tikzcd}
    	{\Dual{\FF_p}{n}} && {\Dual{\tau_{\le d}\Sph_{(p)}}{n}} \\
    	& {\disc.}
\arrow[from=1-1, to=1-3]
\arrow["\omega"', from=1-1, to=2-2]
\arrow["\cl{\omega}",dashed, from=1-3, to=2-2]
    \end{tikzcd}\]
    
    Since $\Dual{\FF_p}{n}$ is $(n-1)$-connected, it is in particular $(n-d-1)$-connected. Hence, $\omega$ factors through $\tau_{\ge n-d}(\disc)$, which we assumed to be isomorphic to 
    $\tau_{\ge n-d}(\Dual{\Sph_{(p)}}{n}) \simeq \Dual{\tau_{\le d}\Sph_{(p)}}{n}$. We shall be done by showing that every two non-zero maps 
    $\Dual{\FF_p}{n} \to \tau_{\ge n-d}(\Dual{\Sph_{(p)}}{n})$ 
    differ by an automorphism of the target. Indeed, we have
    \[
        \Map_{\Sp^\cn}(\Dual{\FF_p}{n}, \tau_{\ge n-d}(\Dual{\Sph_{(p)}}{n})) \simeq
        \Map_{\Sp^\cn}(\Dual{\FF_p}{n}, \Dual{\Sph_{(p)}}{n}) \simeq
        \Map_{\Sp^\cn}(\Sph_{(p)}, \FF_p) \simeq \FF_p.
    \]
    The non-zero maps correspond to $\FF_p^\times$ and multiplication by the Teichmuller lifts $\FF_p^\times \to \Sph_{p}^\times$ permutes them.
    
    We now show the `only if' part. Assume that $\cC$ is $(d+\frac{1}{2})$-connected. We have
    \[
        \Omega^{n-d}(\Dual{\Sph_{(p)}}{n}) \quad,\quad 
        \Omega^{n-d}(\disc)
        \qin \Sptors{d}_{(p)},
    \]
    by definition and \Cref{Pi_Units_Height} respectively.
    By the Yoneda lemma, it would thus suffice to show that both objects represent the same functor on $\Sptors{d}_{(p)}$. 
    In fact, since 
    $\Sptors{d}_{(p)} \simeq \Ind(\Spfin{d}_{(p)})$ 
    (\Cref{pTors_Ind}), it suffices to show that on $\Spfin{d}_{(p)}$, both objects represent the functor $\und{\Dual{(-)}{d}}$.
    Given $M\in\Spfin{d}_{(p)}$, we have on the one hand a chain of natural isomorphisms
    \[
        \Map_{\Sp^\cn}(M, \Omega^{n-d}(\Dual{\Sph_{(p)}}{n})) \simeq
        \Map_{\Sp^\cn}(\Sigma^{n-d}M, \Dual{\Sph_{(p)}}{n}) \simeq
        \und{\Dual{(\Sigma^{n-d}M)}{n}} \simeq 
        \und{\Dual{M}{d}}.
    \]
    On the other hand, we have a chain of natural isomorphisms
    \[
        \Map_{\Sp^\cn}(M, \Omega^{n-d}(\disc)) \simeq
        \Map_{\Sp^\cn}(\Sigma^{n-d}M, \disc) \simeq
    \]
    \[
        \Map_{\Sp^\cn}(\Sigma^{n-d}M, \one^\times) \simeq
        \Map_{\calg(\cC)}(\one[\Sigma^{n-d}M], \one).
    \]
    We now use the assumption that $\omega$ extends to a $\tau_{\le d}\Sph_{(p)}$-orientation. This gives a Fourier transform natural isomorphism
    \[
        \one[\Sigma^{n-d}M] \iso
        \one^{\und{\Dual{(\Sigma^{n-d}M)}{n}}} \simeq
        \one^{\und{\Dual{M}{d}}}.
    \]
    Using this, and the $d$-connectedness of $\cC$ at $p$, we can extend the above chain of natural isomorphisms by
    \[
        \Map_{\calg(\cC)}(\one[\Sigma^{n-d}M], \one) \simeq 
        \Map_{\calg(\cC)}(\one^{\und{\Dual{M}{d}}}, \one) \simeq
        \und{\Dual{M}{d}}.
    \]
\end{proof}

\begin{cor} \label{connected_torsion_units}
    $\cC \in \calg(\Prsad)$ is $(n+\frac{1}{2})$-connected if and only if it is $0$-connected, $(\FF_p,n)$-orientable
    and
    \[
\disc \:\simeq\: 
        \Dual{\Sph_{(p)}}{n}
        \qin \Sp^\cn. 
    \]
\end{cor}

\begin{proof}
    This follows immediately from the case $n=d$ in \Cref{orientation_discrepency}.
\end{proof}

\section{Chromatic Applications}\label{sec:chromatic}

In this final section, we apply the general theory of Fourier transforms and orientations to chromatic homotopy theory, specifically to the study of the monochromatic categories $\Sp_{K(n)}$, $\Sp_{T(n)}$, and $\LocMod_{E_n}$. In particular, we deduce theorems A-F stated in the introduction of the paper. 

    
    


\subsection{Chromatic preliminaries}\label{ssec:chromaticpreliminaries} 

We begin with a rapid review of some material from chromatic stable homotopy theory, geared towards our applications in the subsequent subsections. For a more comprehensive survey, we refer the interested to \cite{BB2020chromaticsurvey}. In the end of this section we also review briefly the theory of higher cyclotomic extensions in the monochromatic categories from \cite{carmeli2021chromatic}, and their relationship with Westerland's ring spectrum $R_n$ introduced and studied in \cite{Westerland}. 

\subsubsection{Telescopic localizations}\label{sssec:telescopiclocalization}

Let $\Sp$ be the symmetric monoidal $\infty$-category of spectra and let $p$ be a fixed prime. The thick subcategory theorem of Hopkins and Smith \cite{nilp2} classifies the thick subcategories of the $\infty$-category $\Sp_{(p)}^\omega$ of finite $p$-local spectra. These subcategories assemble into a strictly ascending \textit{chromatic filtration}
\[
    (0) = 
    \cC^\omega_{\infty} \subset 
    \ldots \subset 
    \cC^\omega_n \subset 
    \cC^\omega_{n-1} \subset 
    \ldots \subset 
    \cC^\omega_0 = 
    \Sp_{(p)}^\omega,
\]
which plays a fundamental role in chromatic stable homotopy theory. The objects of $\cC^\omega_n \smallsetminus \cC^\omega_{n+1}$ are said to be ($p$-local) \emph{finite spectra of type $n$}; by the thick subcategory theorem, any choice of finite type $n$ spectrum generates $\cC^\omega_n$ as a thick subcategory. Furthermore, the periodicity theorem says that any finite spectrum $F(n)$ of type $n$ admits a $v_n$-self map $\nu\colon \Sigma^{|\nu|}F(n) \to F(n)$ of some non-negative degree $|\nu|$ depending on $F(n)$. We will denote the corresponding \emph{height $n$ telescope} by $T(n)= F(n)[\nu^{-1}]$. For $n=0$, we use the convention that the self-map is taken to be multiplication by $p$. In particular one can take $F(0) = \Sph_{(p)}$ and $T(0) = \Sph_{(p)}[\inv{p}]\simeq \QQ$.

The chromatic filtration extends to the category of all $p$-local spectra $\Sp_{(p)}$, by considering for each height $n$ the subcategory $\cC_n \sseq \Sp_{(p)}$
generated by $\cC^\omega_n$ (equivalently, by $F(n)$)) under all colimits. 
The ``complement'' of $\cC_{n+1}$ is given by the Verdier quotient 
\[
    \cC_{n+1} \too 
    \Sp_{(p)} \too \Sp_n^f,
\]
and we write 
\[
    L_n^f \colon
    \Sp \to
    \Sp_{(p)} \to
    \Sp_n^f \into
    \Sp
\]
for the corresponding finite localization functor (leaving the prime $p$ implicit). 
The subcategory $\Sp_n^f \sseq \Sp_{(p)}$ consists of the $(T(0)\oplus \cdots \oplus T(n))$-local spectra in the sense of Bousfield. 
In particular, if $X\in \Sp$ is $T(m)$-acyclic for $0<m\le n$, then 
\[
    L_n^fX\simeq L_0^fX = X\otimes \QQ.  
\]
For example, as we shall repeatedly use, this is the case when $X$ is bounded above. 

We denote by $C_n^f$ the $L_n^f$-acyclification functor. Thus, for every $X\in \Sp$ we have a canonical fiber sequence of spectra 
\[
    C_n^fX \too 
    X_{(p)} \too 
    L_n^fX.
\]
Finally, we note that the functors $L_n^f$ and $(-)_{(p)}$, and hence $C_n^f$, are smashing. This means that they preserve all colimits (as endofunctors of $\Sp$). 

The $\infty$-categories $\Sp_{n}^f$ form a strictly ascending filtration
\[
    \Sp_{\QQ}= \Sp_{0}^f \subset \ldots \subset \Sp_{n}^f \subset \Sp_{n+1}^f \subset \ldots  \subset \Sp_{(p)},
\]
interpolating between the $\infty$-categories of rational and $p$-local spectra. The $n$-th filtration quotient can then be identified with the Bousfield localization of $\Sp_{(p)}$ at a height $n$ telescope $T(n)$:
\[
    \Sp_{n}^f/\Sp_{n-1}^f \iso  \Sp_{T(n)}.
\]
The localization $\Sp_{T(n)}$ inherits a symmetric monoidal structure, given by the $T(n)$-localized smash product. Moreover, by \cite[Theorem~A]{TeleAmbi}, $\Sp_{T(n)}$ is $\infty$-semiadditive of height $n$.

\subsubsection{$K(n)$-local homotopy}\label{sssec:k(n)localization}

There is a variant of the chromatic filtration that is constructed from the Morava $K$-theories. Continuing to work with an implicit fixed prime $p$, for every finite height $n\ge1$, let $K(n)$ denote the \emph{$n$-th Morava $K$-theory} spectrum associated to a formal group law $\Gamma$ of height $n$ over $\FF_p$. It has the structure of a complex oriented $\mathbb{E}_1$-ring spectrum, and its coefficient ring is given by
\[
    K(n)_* \cong \FF_p[v_n^{\pm 1}],
\] 
where $v_n$ is of degree $2(p^n-1)$. By convention, we set $K(0) = \QQ$ and $K(\infty) = \FF_p$. As a consequence of the nilpotence theorem \cite{nilp1, nilp2}, the Morava $K$-theories form the prime fields of $\Sp_{(p)}$.

Let $L_n\colon \Sp \to \Sp$ be the Bousfield localization functor with respect to $(K(0)\oplus\cdots\oplus K(n))$. Hopkins and Ravenel proved that $L_n$ is smashing as well. 
Setting $\Sp_{n} = L_n\Sp$, we obtain a filtration 
\[
    \Sp_{\QQ}= \Sp_{0} \subset \ldots \subset \Sp_{n} \subset \Sp_{n+1} \subset \ldots  \subset \Sp_{(p)},
\]
which is compatible with the chromatic filtration discussed above, in the sense that $\Sp_n \sseq\Sp_n^f$.
The $n$-th filtration quotients of this filtration are symmetric monoidally equivalent to the $K(n)$-local categories:
\[
    \Sp_{n}/\Sp_{n-1} \iso \Sp_{K(n)}.
\]
The question whether the localization functor $L_{K(n)}\colon \Sp_{T(n)} \to \Sp_{K(n)}$ is an equivalence is the content of Ravenel's telescope conjecture.

By construction, the $\infty$-category $\Sp_{K(n)}$ is compactly generated by the $K(n)$-localization of $F(n)$ for any finite spectrum $F(n)$ of type $n$. However, for $n\ge1$, the unit $\Sph_{K(n)} \in \Sp_{K(n)}$ is not compact. Finally, as the telescopic categories, each $\Sp_{K(n)}$ is $\infty$-semiadditive of height $n$ as well, as was proven previously by Hopkins and Lurie in \cite{AmbiKn}. 

\subsubsection{Lubin--Tate spectra and cyclotomic extensions}\label{sssec:lubintatespectra}

For $n\ge1$, let $E_n$ be the \emph{$n$-th Lubin--Tate spectrum} (or \emph{Morava $E$-theory spectrum}) at the prime $p$ associated to a formal group law $\Gamma$ of height $n$ over $\FF_p$. Namely,  it is the Landweber exact ring spectrum attached to the universal deformation of the base-change $\overline{\Gamma}$ of $\Gamma$ over $\overline{\FF}_p$ and has ring of coefficients
\[
    \pi_*E_n \simeq 
    \WW(\cl{\FF}_p)[[ u_1,\ldots, u_{n-1}]][u^{\pm 1}].
\]
Here, $\WW(\cl{\FF}_p)$ denotes the ring of Witt vector on $\overline{\FF}_p$, the power series variables $u_i$ have degree 0, and $u$ has degree $-2$. The spectrum $E_n$ is $K(n)$-local and has the same Bousfield class as $\bigoplus_{i=0}^nK(i)$; in particular, there is a natural equivalence of localization functors $L_{n} \simeq L_{E_n}$. 
By Goerss--Hopkins obstruction theory \cite{goerss2004moduli}, $E_n$ admits an essentially unique $\EE_{\infty}$-ring spectrum structure (see also \cite{Lurie_Ell2}). We write $\LocMod_{E_n}$ for the $\infty$-semiadditive height $n$ symmetric monoidal $\infty$-category of $K(n)$-local modules over $E_n$. 

While other conventions exist in the literature, for the purposes of this paper it is convenient to define the Lubin--Tate spectrum at height $0$ to be the even periodic commutative ring spectrum\footnote{see also \cite{Null} for further motivation for this convention.} 
\[
E_0 := \cl{\QQ}[u^{\pm 1}] \qin \calg(\Sp_\QQ),
\] where $u$ is in degree $-2$. 

For $n\ge1$, the $n$-th \emph{Morava stabilizer group}, defined as the profinite group
\[
    \GG_n = 
    \aut(\cl{\Gamma}/\cl{\FF}_p) \rtimes \Gal(\cl{\FF}_p/\FF_p),
\]
acts continuously on $E_n$ through $\EE_{\infty}$-ring maps, in a sense made precise by Devinatz and Hopkins in \cite{DH}. In fact, $E_n$ is algebraically closed and the canonical unit map $\Sph_{K(n)} \to E_n$ exhibits the target as a $K(n)$-local pro-Galois extension with Galois group $\GG_n$ (in fact, the algebraic closure); see for example \cite{RognesGal}. Note, however, that $E_0$ is \emph{not} a Galois extension of $\Sph_{K(0)} = \QQ$.  

This brings into focus intermediate extensions of $\Sph_{K(n)}$. One particularly relevant such extension appeared in work of Westerland \cite{Westerland}, where he constructed a certain $\ZZ_p^{\times}$-extension $R_n$ of $\Sph_{K(n)}$ for all odd primes and positive heights. 
In fact, as explained in \cite{carmeli2021chromatic}, this extension identifies with the \emph{(infinite) higher cyclotomic extension} $\orcyc[\Sph_{K(n)}]{p^\infty}{n}$, which we therefore denote by $R_n$ at all primes and heights (including $p=2$ and $n=0$). The cyclotomic extensions 
$R_n$ are given as the filtered colimits of the corresponding finite cyclotomic extensions $R_{n,r}:=\orcyc[\Sph_{K(n)}]{p^r}{n}$,   
\[
R_n   \simeq \colim R_{n,r} \qin \calg(\Sp_{K(n)}),
\]
where $R_{n,r}$ is a $(\ZZ/p^r)^\times$-Galois extension in $\Sp_{K(n)}$.
For $n\ge1$, under the $K(n)$-local Galois correspondence of \cite[Theorem 10.9]{AkhilGalois}, the Galois extensions $R_{n,r}$ correspond to group homomorphisms $\chi_{p,r}\colon \GG_n \to (\ZZ/p^r)^{\times}$. 
Consequently, $R_n$ corresponds to a continuous homomorphism
\[
    \chi_p\colon \GG_n \too \ZZ_p^{\times}
\] 
called the \textit{$p$-adic cyclotomic character 
}
of $\Sp_{K(n)}$, see \cite[\S5.2]{carmeli2021chromatic}. 
The kernel of $\chi_p$, denoted by $\GG_n^0 \triangleleft \GG_n$, is a closed subgroup, and we have an equivalence 
$
    R_n \simeq E_n^{h\GG_n^0} 
$ 
of faithful $\ZZ_p^{\times}$-Galois extensions of $\Sph_{K(n)}$; here, $(-)^{h\GG_n^0}$ stands for the \emph{continuous} fixed points.
For $n=0$, the ring $R_n$ is the classical cyclotomic extension $\QQ(\omega_{p^\infty})$, and it corresponds to the usual $p$-adic cyclotomic character of the absolute Galois group of $\QQ$. 


One advantage of the cyclotomic approach from \cite{carmeli2021chromatic} is that it allows for a telescopic lifts of $R_{n,r}$ and $R_n$. 
The corresponding $p^r$-th finite cyclotomic extensions $R_{n,r}^f:=\orcyc[\Sph_{T(n)}]{p^r}{n}$, constructed  in parallel to the finite cyclotomic extensions $R_{n,r}$, are again faithful $(\ZZ/p^r)^\times$-Galois for every $r$ (\cite[Proposition 5.2]{carmeli2021chromatic}). They assemble into an infinite cyclotomic extension 
\[
R_n^f := \orcyc[\Sph_{T(n)}]{p^\infty}{n} = \colim R_{n,r}^f \qin \calg(\Sp_{T(n)})
\]
which is a pro-finite $\ZZ_p^\times$-Galois extension in $\Sp_{T(n)}$. It is, however, not known whether $R_n^f$ is a \emph{faithful} Galois extension; we return to this point at the end of \cref{ssec:ortelescopic}.

\subsection{Orientations of the Lubin--Tate ring spectrum}\label{ssec:orlubintate}

We begin our discussion of the Fourier transform in chromatic homotopy theory with the case of $E_n$-modules. After reinterpreting \Cref{HL_orientation_Intro} of Hopkins and Lurie as providing $E_n$ with a $\ZZ_{(p)}$-orientation and descending it to $R_n$, we apply the results of \Cref{sec:examplesforspecificrings} to extend it to a spherical orientation (\Cref{Sphere_Or_Intro}), and the results of \cref{sec:categorification} to further categorify it  (\Cref{E_n_Categorification_Intro}). We then construct the pro-$\pi$-finite Galois $K(n)$-local spherical cyclotomic extension (\Cref{Sph_Gal_Intro}), and conclude the subsection  with the computation of the connective cover of the  $p$-localized discrepancy spectrum of $E_n$ (\Cref{Discrepancy_Intro}).

\subsubsection{$\ZZ_{(p)}$-Orientability}

As we shall now explain, the $(\ZZ_{(p)},n)$-orientability of $E_n$ is essentially equivalent to the following result of Hopkins and Lurie:
\begin{thm}[{{\cite[Corollary 5.3.26]{AmbiKn}}}] 
\label{HL_orientation}
    For every $n\ge 1$, there is a natural isomorphism
    \[
        E_n[M] \iso 
        E_n^{\und{\Sigma^n M^*}} 
        \qin \calg(\Sp_{K(n)}),
    \]
    for connective $\pi$-finite $p$-local $\ZZ$-modules $M$, where 
$M^*= \hom_\ZZ(M,\QQ/\ZZ)$. 
\end{thm}

Indeed, while this isomorphism is not constructed in \cite{AmbiKn} using Fourier theory, all natural transformations of the above form are essentially Fourier transforms. Thus, this result can be reformulated (and extended to height $n=0$) as follows: 
\begin{cor}\label{LT_Zp_Or}
    For every $n\ge 0$, the $\infty$-category $\LocMod_{E_n}$ is $(\ZZ_{(p)},n)$-orientable.
\end{cor}

\begin{proof}
    We start with the case $n=0$. The commutative ring spectrum $E_0$ is an algebra over $\cl{\QQ}$, and hence it suffices to show that $\cl{\QQ}$ is $(\ZZ_{(p)},0)$-orientable by \cref{push_orientations_functor}. Using \Cref{lem:bcdcolim}, the compatible system of $p$-power roots of unity $\exp(\frac{2\pi i}{p^k})\in \cl{\QQ}$ (viewed as a subfield of $\CC$)
     gives a $\ZZ_{(p)}$-pre-orientation of $\cl{\QQ}$ of height $0$. For a finite abelian $p$-group $M$, the resulting Fourier transform $\cl{\QQ}[M]\to \cl{\QQ}^{M^*}$ is the classical discrete Fourier transform, hence an isomorphism. 
    
    We turn to the case $n\ge 1$. 
    First, a $\pi$-finite $p$-local $\ZZ$-module is the same thing as a $\pi$-finite $\ZZ_{(p)}$-module. Furthermore, if $M$ is concentrated between degrees $0$ and $n$ then $\Sigma^n M^*\simeq \Dual{M}{n}$.  Hence, an isomorphism as in \Cref{HL_orientation} restricts to a natural isomorphism $E_n[-]\simeq E_n^{\und{\Dual{-}{n}}}$ of functors $\Modfin{\ZZ/p^r}{n} \to \calg(\Sp_{K(n)})$ for every $r\in\NN$. By \Cref{Four_is_POr} (and \Cref{Four_Pi_Finite}), such an isomorphism is the Fourier transform associated with an essentially unique $\ZZ/p^r$-orientation $\omega \colon \Sigma^n \ZZ/p^r \to E_n^\times$. Since these orientations are compatible with each other, they assemble into a $\ZZ_{(p)}$-orientation of $E_n$ of height $n$ by \Cref{universal_algebra_for_Zp}.
\end{proof}

\begin{rem}
The construction of the isomorphism in \cref{HL_orientation} depends on a choice of a \textit{normalization} $\nu$ of the $p$-divisible group $\Gamma$ associated with $E_n$, in the sense of \cite[Definition 5.3.1]{AmbiKn}. Hence, identifying this isomorphism with the Fourier transform, we associate to a normalization $\nu$ of $\Gamma$ a $(\ZZ_{(p)},n)$-orientation $\omega_\nu$. It is not hard to show that the association $\nu \mapsto \omega_\nu$ furnishes a \textit{bijection} between normalizations of $\Gamma$ and $(\ZZ_{(p)},n)$-orientations of $E_n$.
\end{rem}

The orientability of $\LocMod_{E_n}$ implies \textit{virtual} orientability of $\Sp_{K(n)}$.

\begin{cor}\label{Kn_Zp_Or}
    For every $n\ge 0$, the $\infty$-category $\Sp_{K(n)}$ is virtually $(\ZZ_{(p)},n)$-orientable (hence, in particular, virtually $(\FF_p,n)$-orientable).
\end{cor}

\begin{proof}
    Since $E_n$ is faithful in $\Sp_{K(n)}$, this follows from \Cref{LT_Zp_Or}. 
\end{proof}

While $\Sp_{K(n)}$ is virtually $(\ZZ_{(p)},n)$-orientable, it is not $(\ZZ_{(p)},n)$-\emph{orientable}. 
Namely, one can not replace $E_n$ with $\Sph_{K(n)}$ in the isomorphism of \Cref{HL_orientation}. However, the Fourier theoretic point of view does allow us to descend this isomorphism from $E_n$ to the intermediate extension $R_n$.

\begin{thm} \label{R_n_fourier}
    For every $n\ge 0$, there is a natural isomorphism
    \[
        R_n[M] \iso 
        R_n^{\und{\Sigma^n M^*}} 
        \qin \calg(\Sp_{K(n)})
    \]
    for connective $\pi$-finite $p$-local $\ZZ$-module spectra $M$. 
\end{thm}

\begin{proof}
Since $\Sp_{K(n)}$ is virtually $(\FF_p,n)$-orientable, by \Cref{Zp_Univ_Cyclo} we obtain that 
    \[
        R_n = 
        \orcyc[\Sph_{K(n)}]{p^\infty}{n} \simeq
        \orcyc[\Sph_{K(n)}]{\ZZ_{(p)}}{n} \qin \calg(\Sp_{K(n)}).
    \]
    In particular, the universal $(\ZZ_{(p)},n)$-orientation on $R_n$ provides the desired isomorphism.
\end{proof}

\begin{rem}
    The fact that $R_n$ carries the \emph{universal} $(\ZZ_{(p)},n)$-orientation among $K(n)$-local commutative ring spectra, shows that for $n\ge1$ the isomorphism in \Cref{HL_orientation} is obtained from the one in \Cref{R_n_fourier} by scalar extension along a map $R_n \to E_n$. This map identifies with the inclusion of the fixed point algebra
    $R_n \simeq E_n^{h\GG_n^0} \to E_n$, up to possibly pre-composing with an automorphism of $R_n$ (that is, an element of $\ZZ_{p}^\times$). Hence, \Cref{R_n_fourier} is essentially the claim that the isomorphism in \Cref{HL_orientation} is $\GG_n^0$-equivariant.   
\end{rem}

\subsubsection{$\Sph_{(p)}$-orientability}

Since $\Sp_{K(n)}$ is virtually $\ZZ_{(p)}$-orientable (by \cref{Kn_Zp_Or}), it is also virtually $\Sph_{(p)}$-orientable (by \cref{Virt_Zp_Bootstrap}). Namely, the $K(n)$-local spherical cyclotomic extension $\orcyc[\Sph_{K(n)}]{\Sph_{(p)}}{n}$ is \textit{faithful}. Our general results imply that it is a pro-$\pi$-finite Galois extension of $\Sph_{K(n)}$. 

\begin{thm}\label{Kn_Pro_Galois}
    For every $n \ge 0$, the commutative algebra $\orcyc[\Sph_{K(n)}]{\Sph_{(p)}}{n}$ is a pro-Galois extension of $\Sph_{K(n)}$ for the group $\tau_{\le n}\Sph_{(p)}^\times$, viewed as a pro-$\pi$-finite group.
\end{thm}
\begin{proof}
    By \cref{Kn_Zp_Or}, $\Sp_{K(n)}$ is virtually $(\FF_p,n)$-orientable. Thus, the result follows from \Cref{Spherical_Cyc_Galois}.
\end{proof}

While $\orcyc[\Sph_{K(n)}]{\Sph_{(p)}}{n}$ is the universal spherically oriented $K(n)$-local commutative algebra, we do not have an explicit description of it. In contrast, combining the theory of categorical connectedness from \Cref{sec:examplesforspecificrings} with the results on $\pic(E_n)$ from \cite{Null}, we can also construct a (non-universal) spherical orientation on $E_n$, which for $n\ge 1$ is an ordinary (pro-finite) Galois extension of $\Sph_{K(n)}$.

\begin{thm} \label{Sphere_or_E_n}
    For every $n\ge 0$, the $\infty$-category $\LocMod_{E_n}$ is $(n+\frac{1}{2})$-connected, hence in particular $(\Sph_{(p)},n)$-orientable. 
\end{thm}

\begin{proof}
    We check the assumptions of \Cref{Pic_Conn_Frac}. By \cref{LT_Zp_Or}, $\LocMod_{E_n}$ is $(\FF_p,n)$-orientable, and by \cite[Proposition 8.14]{Null}, we have
    \[
        \Map_{\Sp^\cn}(C_p,\pic(E_n)) \simeq B^{n+1}C_p.\qedhere
    \]
\end{proof}

Applying the general results on categorification of orientations from \Cref{sec:categorification}, we also get the following:

\begin{cor}\label{Ninga_Orientation}
    There is a natural equivalence of symmetric monoidal $\infty$-categories: 
    \[
        \Fun(\Omega^\infty M, \LocMod_{E_n})_\Day \iso 
        \Fun(\Omega^{\infty - (n+1) } (I_{\QQ_p/\ZZ_p} M), \LocMod_{E_n})_\Ptw,
    \]
    for $M$ a connective $(n+1)$-finite $p$-local spectrum, provided that the action map 
    \[
        \pi_{n+1}\Sph \otimes \pi_0 M \too \pi_{n+1} M
    \] 
    is zero. 
\end{cor}
\begin{proof}
    Since $\LocMod_{E_n}$ is $(\tau_{\le n}\Sph_{(p)},n)$-orientable (by \Cref{Sphere_or_E_n}), we get by \Cref{Orientability_Cat} that $\Mod_{\LocMod_{E_n}}$ is $(\tau_{\le n}\Sph_{(p)},n+1)$-orientable. By  \Cref{Truncated_Sph_p_Crit}, this translates to the above. 
\end{proof}   

We expect that the technical condition on $M$ can be removed:

\begin{conjecture}\label{con:cat_E}
    The $\infty$-category $\Mod_{\LocMod_{E_n}}$  is $(\Sph_{(p)},n+1)$-orientable.
\end{conjecture}

At least in height $0$, this conjecture can be verified by an explicit computation:
\begin{prop} 
    The $\infty$-category 
    $\Mod_{\Mod_{E_0}}$ 
    is $(\Sph_{(p)},1)$-orientable (for every prime $p$). 
\end{prop}

\begin{proof}
    It suffices to show that $\Mod_{\Mod_{E_0}}$ 
    is $(\tau_{\le 1}\Sph_{(p)},1)$-orientable (see \Cref{rem:n_trunc}). By \Cref{LT_Zp_Or}, $\Mod_{E_0}$ is $(\ZZ_{(p)},0)$-orientatable, and hence by \Cref{Orientability_Cat}, $\Mod_{\Mod_{E_0}}$ is $(\ZZ_{(p)},1)$-orientable. 
    For $p\neq 2$, we have $\tau_{\le 1}\Sph_{(p)}\simeq \ZZ_{(p)}$, so the result holds for odd primes. 
    It remains to treat the case $p=2$. In fact, we shall show that $\Mod_{\Mod_{E_0}}$ is $(1+\frac{1}{2})$-connected at $p=2$, which implies $(\tau_{\le 1}\Sph_{(p)},1)$-orientability, by definition.
    
    By \Cref{orientation_discrepency} applied to $\Mod_{\Mod_{E_0}}$, it would suffice to show that 
    $
        \pic(E_0)^{\ptors{\pi}}_{(2)} \:\simeq\: 
        I_2^{(1)}\Sph_{(2)}.
    $ 
    Since $\pi_*E_0$ is a 2-periodic even graded field, $\pi_0\pic(E_0) \simeq \ZZ/2$, with the non-zero element given by the isomorphism class of $\Sigma E_0$ (see, e.g., \cite[Theorem 37]{bakerrichter2005invertible}).
    Since 
    $\Omega \pic(E_0) \simeq E_0^\times$
    we conclude that 
    \[
        \pi_t(\pic(E_0)) =  
        \left\{\begin{array}{lr}
            \ZZ/2\ZZ &  t=0 \\
            \cl{\QQ}^\times &  t = 1\\
            \cl{\QQ} & t>1 \text{ odd} \\
            0 & \text{otherwise}
        \end{array}\right.
    \] 
    Let $X$ denote the fiber of the map 
    $\pic(E_0)\to \pic(E_0)[\inv{2}]$. Since 
    $\pic(E_0)[\inv{2}]^{\ptors{\pi}}_{(2)} \simeq 0,$ 
    we obtain that 
    $\pic(E_0)^{\ptors{\pi}}_{(2)}\simeq X^{\ptors{\pi}}_{(2)},$ 
    and we shall compute the latter. 
    
    From the long exact sequence of homotopy groups associated with the fiber sequence 
    \[
        X \to \pic(E_0) \to \pic(E_0)[\inv{2}],
    \]
    we see that the homotopy groups of $X$ are as follows:
    \[
        \pi_t(X) =  
        \left\{\begin{array}{lr}
            \ZZ/2\ZZ &  t=0 \\
            \QQ_2/\ZZ_2 & t = 1\\
            0 & t\ge 2.
        \end{array}\right.
    \]
    In particular, by \Cref{Cp_Tor_Crit}, $X \simeq X^{\ptors{\pi}}_{(2)}$, so it remains to show that $X\simeq I_2^{(1)}\Sph_{(2)}$.

    Since $X$ is an extension of $\ZZ/2\ZZ$ by $\Sigma \QQ_2/\ZZ_2$, it is classified by a map $\ZZ/2\ZZ \to \Sigma^2 \QQ_2 /\ZZ_2$ in $\Sp$. The collection of homotopy classes of such maps is given by
    \[
    \pi_0\Map_{\Sp}(\ZZ/2\ZZ, \Sigma^{2}\QQ_2/\ZZ_2) \simeq \pi_0\Map_{\Mod_\ZZ}(\ZZ/2\ZZ\otimes_\Sph \ZZ, \Sigma^{2}\QQ_2/\ZZ_2) 
    \simeq (\pi_2(\ZZ/2\ZZ\otimes_\Sph \ZZ))^* \simeq \ZZ/2\ZZ,
    \]
    where the last isomorphism follows from the classical computation of the integral (dual) Steenrod algebra. 
    Consequently, there are only two possible extensions of $\ZZ/2\ZZ$ by $\Sigma \QQ_2/\ZZ_2$, and we want to show that $X$ is the non-split one. 
    
    These two extensions can be distinguished using the multiplication-by-$\eta$ map from $\pi_0$ to $\pi_1$, which is $0$ for the split extension and the inclusion $\ZZ/2\ZZ\into \QQ_2/\ZZ_2$ for the non-split one. 
    Since the morphism $X\to \pic(E_0)$ induces an isomorphism on $\pi_0$ and an injection on $\pi_1$, it suffices to show that multiplication by $\eta$ is non-zero already on $\pi_0\pic(E_0)$. This follows from the fact that, for the class $[\Sigma E_0]\in \pi_0(P)$, we have (e.g., by \cite[Proposition 3.20]{carmeli2021chromatic}), 
    \[
        \eta\cdot[\Sigma E_0] = 
        \dim(\Sigma E_0) = 
        -1 \qin 
        \pi_1(\pic(E_0))\simeq \cl{\QQ}^\times,
    \]
    which is non-trivial.
\end{proof}

\subsubsection{The discrepancy spectrum} 

Another application of the categorical connectedness result of \Cref{Sphere_or_E_n}, is an explicit description of the spectrum $\disc[E_n]$. 

\begin{thm} \label{E_n_tor_units_bc}
    For every $n\ge 0$ there is an isomorphism
    \[
        \disc[E_n] \:\simeq\:
        \tau_{\ge 0}(\Sigma^n I_{\QQ_p/\ZZ_p})
        \qin \Sp_{(p)}^\cn.
    \]
\end{thm}
\begin{proof}
    The $\infty$-category $\LocMod_{E_n}$ is $(n+\frac{1}{2})$-connected by \Cref{Sphere_or_E_n}, so the claim follows by \Cref{connected_torsion_units}, keeping in mind that $\Dual{\Sph_{(p)}}{n} \simeq \tau_{\ge 0}(\Sigma^n I_{\QQ_p/\ZZ_p})$.
\end{proof}

The spectrum $\disc[E_n]$ turns out to be closely related to the \textit{discrepancy spectrum} of $E_n$. In \cite{stringorientation}, Ando, Hopkins, and Rezk defined the discrepency spectrum of an arbitrary $L_n$-local commutative ring spectrum $R$ as the fiber of the localization map $R^\times \to L_n R^\times$. 
For such a ring spectrum $R$, it essentially follows from \cite[Theorem 4.11]{stringorientation} that its discrepency spectrum agrees with its $p$-torsion $\pi$-finite units $\disc[R]$, after taking the connective cover and $p$-localizing. 
As explained in \cite{stringorientation} (see the discussion below \cite[Lemma 4.12]{stringorientation}), if $R$ is $L_n$-local then $L_nR^\times\simeq L_n^fR^\times$, so one can use $L_n^fR^\times$ instead of $L_nR^\times$ in the definition of the discrepancy spectrum. This variant has the advantage of providing a well-behaved notion of a discrepancy spectrum defined for all \emph{$L_n^f$-local} commutative ring spectra. 
In particular, for this definition, the $p$-localization of the discrepancy spectrum of $R\in \calg(L_n^f\Sp)$ is given by $C_n^fR^\times$.

Our goal in this subsection is to review the relation between the discrepancy spectrum and the $p$-torsion $\pi$-finite units of $L_n$-local commutative ring spectra from $\cite{stringorientation}$, and generalize it to the context of $L_n^f$-local commutative ring spectra. In fact, we will even work in the wider generality of almost $L_n^f$-local commutative ring spectra, in the following sense:
\begin{defn}\label{Almost_Lnf_Local}
    A $p$-local spectrum $X$ is \tdef{almost $L_n^f$-local}, if $\hom(Z,X)$ is bounded above for some (and hence all) finite spectra $Z$ of type $n+1$.
\end{defn}

Equivalently, $X$ is almost $L_n^f$-local if $X\otimes Z$ is bounded above for some (and hence all) finite spectra $Z$ of type $n+1$. Indeed, $\hom(Z,X) \simeq \DD Z \otimes X$ and $\DD Z$ is of type $n+1$ if and only if $Z$ is. Note also that the collection of almost $L_n^f$-local spectra itself forms a thick subcategory of $\Sp_{(p)}$.

\begin{rem}
    If we replace in \Cref{Almost_Lnf_Local} `bounded above' with `$\pi$-finite', we arrive at the stronger notion of \textit{fp-type $n$} in the sense of Mahowald and Rezk \cite{MR1999bcduality}. 
\end{rem}

\begin{example}
    \label{ex:lnf_almost_lnf}
    Every $L_n^f$-local (and in particular $L_n$-local) spectrum is almost $L_n^f$-local. Indeed, a $p$-local spectrum $X$ is $L_n^f$-local if and only if $\hom(Z,X)= 0$ for some finite spectrum $Z$ of type $n+1$.
\end{example}

\begin{example} 
    \label{ex:bounded_above_almost_lnf}
    Every $p$-local bounded above spectrum is almost $L_n^f$-local. Indeed, if $X$ is a bounded above spectrum and $Z$ is \emph{any} (in particular, type $n+1$) finite spectrum, then $Z\otimes X$ is also bounded above.    
\end{example}    

The example above implies that the almost $L_n^f$-locality of a spectrum can be checked after passing to an arbitrary connected cover of it. 
In fact, one can verify this property using only an arbitrary connected cover of its \emph{underlying space}. 

\begin{prop}\label{Almost_Lnf_Local_Eventual}
    Let $X,Y \in \Sp_{(p)}$ be such that 
    \[
        \Omega^{\infty + d} X \simeq \Omega^{\infty+d} Y \qin \Spc_*
    \] 
    for some (and hence all sufficiently large) $d \ge 0$. Then, $X$ is almost $L_n^f$-local if and only if $Y$ is.
\end{prop}

\begin{proof}
    Let $A\in \Spc_*$ be a finite pointed space whose reduced suspension spectrum $Z:= \cl{\Sph}[A]$ is of type $n+1$. Then, by definition, a spectrum $W$ is almost $L_n^f$-local if and only if $\hom(Z,W)$ is bounded above. This is the case if and only if there exists $e\in \NN$ such that 
    \[
            \Omega^{\infty + e}\hom(Z,W) \simeq \pt \quad \tag{$*$}
    \]
    By applying the functor $\Omega$ to the above isomorphism, if $(*)$ holds for some $e$ then it holds for any larger value of $e$ as well.
    Unwinding the definitions, we get
    \[
        \Omega^{\infty + e}\hom(Z,W) \simeq
        \Omega^{e}\Map_\Sp(Z,W) \simeq 
        \Omega^{e}\Map_{\Spc_*}(A,\Omega^\infty W) \simeq 
        \Map_{\Spc_*}(A,\Omega^{\infty+e} W).
    \]

    Now, if $\Omega^{\infty + d} X \simeq \Omega^{\infty+d} Y$ and $X$ is almost $L_n^f$-local, we may choose $e$ satisfying $(*)$ for $X$ such that $e \ge d$. But then 
    \[
    \Map_{\Spc_*}(A,\Omega^{\infty + e}Y)\simeq \Map_{\Spc_*}(A,\Omega^{\infty+e} X)\simeq \pt
    \] 
    and we deduce that $Y$ is almost $L_n^f$-local as well. By symmetry, if $Y$ is almost $L_n^f$-local then so is $X$ and therefore they are almost $L_n^f$-local together.  
\end{proof}
Given $R \in \calg(\Sp_{(p)})$, in addition to the  ``additive'' underlying $p$-local spectrum $R\in \Sp_{(p)}$, we can form the ``multiplicative'' $p$-localized  spectrum of units $R^\times_{(p)}\in \Sp_{(p)}$. Although these are very different $p$-local spectra in general, \Cref{Almost_Lnf_Local_Eventual} implies that they are almost $L_n^f$-local together. 
\begin{cor}\label{Almost_Lnf_Local_Units}
    A $p$-local commutative ring spectrum $R$ is almost $L_n^f$-local if and only if the $p$-localization of its spectrum of units $R^\times$ is almost $L_n^f$-local.   
\end{cor}

\begin{proof}
    This follows from \Cref{Almost_Lnf_Local_Eventual} and the fact that 
    \[
        \Omega^{\infty +1} R \simeq 
        \Omega^{\infty +1} R^\times \simeq
        \Omega^{\infty +1}(R^\times_{(p)}). \qedhere
    \]
\end{proof}

Note that, in particular, $R^\times_{(p)}$ is almost $L_n^f$-local for every $R\in \calg(\Sp_n)$. We thus obtain the following: 

\begin{example}\label{ex:E_n_units_almost_lnf}
    The spectrum $(E_n^\times)_{(p)}$ is almost $L_n^f$-local.
\end{example}

We proceed by analysing the behavior of the functor $C_n^f$ on arbitrary almost $L_n^f$-local spectra. 

\begin{prop}\label{Almost_Lnf_Local_Cnf}
    For an almost $L_n^f$-local spectrum $X$, the spectrum $C_n^fX$ is a filtered colimit of bounded above $p$-torsion spectra.
\end{prop}

\begin{proof}
    This follows from the standard fact (see \cite[Proposition 7.10]{hovey1999morava}) that $C_n^fX$ can be written as
    \(
    \colim \hom(Z_\alpha, X),      
    \)
    where the $Z_\alpha$-s are a cofiltered diagram of finite spectra of type $n+1$. Since $X$ is almost $L_n^f$-local and the $Z_\alpha$-s are finite  $p$-torsion spectra, each $\hom(Z_\alpha, X) \simeq \hom(Z_\alpha,C_n^fX)$ is bounded above and $p$-torsion.
\end{proof}


The functor $C_n^f$ is compatible with connective covers in the following sense:

\begin{lem}\label{lem:truncating}
    For $X\in \Sp$, the canonical map 
    \[ 
        \tau_{\ge 0} (C_n^f \tau_{\ge 0}X) \to 
        \tau_{\ge 0} (C_n^f X)
        \qin \Sp_{(p)}^\cn
    \] 
    is an isomorphism. 
\end{lem}
\begin{proof}
    Since $C_n^f$ is an exact functor, we have a cofiber sequence 
    \[ 
        C_n^f \tau_{\ge 0}X \too 
        C_n^f X \too
        C_n^f \tau_{\leq -1} X
        \qin \Sp_{(p)}.
    \]
    Applying the limit preserving functor 
    $\tau_{\ge 0}\colon \Sp_{(p)} \to \Sp_{(p)}^\cn$ 
    we obtain a fiber sequence
    \[ 
        \tau_{\ge0}C_n^f \tau_{\ge 0}X \too 
        \tau_{\ge0}C_n^f X \too
        \tau_{\ge0}C_n^f \tau_{\leq -1} X
        \qin \Sp_{(p)}^\cn.
    \]
    Thus to show that the first map is an isomorphism, it suffices to show that
    $\tau_{\ge 0}C_n^f \tau_{\leq -1} X \simeq 0$. Equivalently, for $Y\in \Sp$, we wish to show that if $\tau_{\ge 0}Y=0$ then $\tau_{\ge 0}C_n^f Y = 0$. First, such a $Y$ is bounded above, so $L_n^fY\simeq \QQ \otimes Y$ (as explained in \cref{ssec:chromaticpreliminaries}). 
    Consequently, 
    \[
        \tau_{\ge 0}L_n^f Y \simeq 
        \tau_{\ge 0}(\QQ \otimes Y)\simeq 
        \QQ \otimes \tau_{\ge 0}Y \simeq 0, 
    \]
    and similarly
    \[
        \tau_{\ge 0}Y_{(p)}\simeq 
        (\tau_{\ge 0}Y)_{(p)} \simeq 0. 
    \]
    Finally, $C_n^fY$ is the fiber of a map $Y_{(p)}\to L_n^fY$, and since the functor $\tau_{\ge 0}\colon\Sp\to \Sp^\cn$ preserves fibers, we get $\tau_{\ge 0} C_n^fY \simeq 0$.  \end{proof}
    
    Recall that $\disc[R]$ is defined as the $p$-local $\pi$-torsion part of $R^\times$. The relation between $\disc[R]$ and the discrepancy spectrum of $R$ is deduced from the following general fact:
    \begin{prop}\label{Almost_Lnf_Local_Connective_Cover}
    Let $X$ be an almost $L_n^f$-local spectrum. Then 
    \[
        \tau_{\ge 0}(C_n^fX) \simeq 
        (\tau_{\ge 0}X)^{\ptors{\pi}}_{(p)}
        \qin \Sp_{(p)}^\cn.
    \]
\end{prop}
\begin{proof}
    First, we can assume without loss of generality that $X$ is connective, by replacing it with $\tau_{\ge 0}X$. Indeed, by \Cref{ex:bounded_above_almost_lnf}, the spectrum $\tau_{\ge 0}X$ is also almost $L_n^f$-local, and by 
    \Cref{lem:truncating}, we have 
    $\tau_{\ge 0} (C_n^f X) \simeq \tau_{\ge 0} (C_n^f \tau_{\ge 0}X)$. 
    Now, by the definition of $(-)_{(p)}^{\ptors{\pi}}$, it would suffice to show that: 
    \begin{enumerate}
        \item $\tau_{\ge 0}C_n^fX \in \Sp_{(p)}^{\ptors{\pi}}$.
        \item For every $Z\in \Sp_{(p)}^{\ptors{\pi}}$, the map $\tau_{\ge 0}C_n^fX \to X$ induces an isomorphism 
        \[
            \Map(Z,\tau_{\ge 0}C_n^fX)\simeq
            \Map(Z,X). 
        \]
    \end{enumerate}
    
    For $(1)$, by \Cref{Almost_Lnf_Local_Cnf}, there is a filtered colimit presentation $C_n^f X \simeq \colim X_\alpha$ such that each $X_\alpha$ is bounded above and $p$-torsion. Since the formation of connective covers preserves filtered colimits, we obtain that
    \[
        \tau_{\ge 0} (C_n^f X) \simeq \colim \tau_{\ge 0}X_\alpha.
    \]
    Each $\tau_{\ge 0}X_\alpha$ is connective, bounded above and $p$-torsion. By \Cref{Cp_Tor_Crit}, $\tau_{\ge 0}X_\alpha$ belongs to $\Sp_{(p)}^{\ptors{\pi}}$, and hence so does $\tau_{\ge 0} (C_n^f X)$.
    
    For $(2)$, by definition, every $Z \in \Sp_{(p)}^{\ptors{\pi}}$ is a filtered colimit of bounded above $p$-torsion spectra. A bounded above spectrum $Y$ satisfies 
    $L_n^f Y \simeq Y\otimes \QQ$ and if $Y$ is also $p$-torsion, then it is $L_n^f$-acyclic. Since $L_n^f$-acyclic spectra are closed under colimits, we deduce that $Z$ itself is $L_n^f$-acyclic, hence $\Sp_{(p)}^{\ptors{\pi}} \subseteq C_n^f(\Sp_{(p)})$. Consequently,
    \[
        \Map(Z,X) \simeq 
        \Map(Z,C_n^f X) \simeq
        \Map(Z,\tau_{\ge 0} (C_n^f X)),
    \]
    where the composite of these isomorphisms is given by post-composing with the canonical map $\tau_{\ge 0}C_n^fX \to X$. 
\end{proof}

Putting everything together we get the main result of this subsection.

\begin{thm}[cf.~{\cite[Theorem 4.11]{stringorientation}}] \label{discrepency_fiber}
    For all $n \ge 0$ and for every almost $L_n^f$-local commutative ring spectrum $R$, we have 
    \[
        \disc[R]\simeq 
        \tau_{\ge 0}C_n^f(R^\times)
        \qin \Sp_{(p)}^\cn.
    \]
\end{thm}
\begin{proof}
    By \Cref{Almost_Lnf_Local_Units}, the $p$-localization of the spectrum $R^\times$ is almost $L_n^f$-local as well. Hence, by \Cref{Almost_Lnf_Local_Connective_Cover}, we have
    \[
        \tau_{\ge 0}C_n^f(R^\times)= \tau_{\ge 0}C_n^f((R^\times)_{(p)}) \simeq (R^\times)^{\ptors{\pi}}_{(p)} := \disc[R]. 
    \] 
\end{proof}

\begin{cor}
    The connective cover of the $p$-localized discrepancy spectrum of $E_n$ is isomorphic to $\tau_{\ge 0}(\Sigma^n I_{\QQ_p/\ZZ_p})$.
\end{cor}

\begin{proof}
    Recall that the connective cover of the $p$-localized discrepancy spectrum of $E_n$ is isomorphic to $\tau_{\ge 0}C_n^f(E_n^\times)$. By \Cref{discrepency_fiber}, we have $\tau_{\ge 0}C_n^f(E_n^\times) \simeq \disc[E_n]$. Hence, the result follows from \Cref{E_n_tor_units_bc}. 
\end{proof}

\begin{rem}
    Using the vanishing of the telescopic homology of sufficiently connected Eilenberg--MacLane spaces, established in \cite[Theorem E]{TeleAmbi}, the same argument as in the proof of
    \Cref{Pi_Units_Height}, shows that $\disc[R]$ is $n$-truncated for every $R \in \calg(L_n^f \Sp)$. Together with \Cref{discrepency_fiber}, this constitutes a telescopic generalization of \cite[Theorem 4.11]{stringorientation}.
\end{rem}

As another consequence of \Cref{discrepency_fiber}, we obtain the following property of the functor $\disc[-]$.
\begin{cor}\label{discrepency_filtered_colimits}
    The functor $\disc[-]\colon \calg(L_n^f\Sp)\to \Sp_{(p)}^{\ptors{\pi}}$ preserves filtered colimits. 
\end{cor}

\begin{proof}
First, observe that the fully faithful embedding $\Sp^{\ptors{\pi}}_{(p)} \into \Sp_{(p)}^\cn$ is conservative and preserves filtered colimits, so it suffices to prove the claim when regarding $\disc[-]$ as a functor into $p$-local connective spectra. 

By \Cref{discrepency_fiber}, we have $\disc[R]\simeq \tau_{\ge 0}C_n^f(R^\times)$ for $R\in \calg(L_n^f\Sp)$. Namely,  we may write the functor 
$\disc[-]$ as the composition  
\[
    \calg(L_n^f\Sp) \xhookrightarrow[(1)]{\quad} 
    \calg(\Sp) \xrightarrow[(2)]{(-)^\times} 
    \Sp \xrightarrow[(3)]{\:\: C_n^f\:\:} 
    \Sp_{(p)} \xrightarrow[(4)]{\tau_{\ge 0}}
    \Sp_{(p)}^{\cn}. 
\]
We will proceed by showing that each of the functors in this composite preserve filtered colimits. 
\begin{enumerate}

    \item The embedding $\calg(L_n^f\Sp) \into \calg(\Sp)$ is obtained from the lax symmetric monoidal, colimit preserving, fully faithful embedding $L_n^f\Sp \into \Sp$ by applying $\calg(-)$, and hence it preserves filtered colimits. 
    
    \item The argument for the functor  $(-)^\times \colon \calg(\Sp) \to \Sp$ is similar to \cite[Proposition 2.3.3]{mathew2016picard}, and we give it for completeness. The lax symmetric monoidal functor $\Omega^\infty \colon \Sp \to \Spc$ preserves filtered colimits, and therefore so does the induced functor $\calg(\Sp) \to \CMon(\Spc),$ taking a commutative algebra to its underlying space with the multiplicative commutative monoid structure. Furthermore, the functor $\CMon(\Spc) \to \Sp^\cn$, right adjoint to the inclusion of connective spectra as group-like commutative monoids, also preserves filtered colimits. Indeed, this can be checked after composing with the conservative filtered colimit preserving forgetful functor $\CMon(\Spc) \to \Spc$, and the invertible elements form a connected summand, so the claim can be easily verified on $\pi_0$. 
    
    \item The functor $C_n^f\colon \Sp\to \Sp_{(p)}$ is the acyclification functor associated with the \emph{smashing} localization $L_n^f$ and hence it preserves all small colimits. 
    \item The functor $\tau_{\ge 0} \colon  \Sp_{(p)}\to \Sp_{(p)}^\cn$ preserves filtered colimits since the formation of homotopy groups preserves filtered colimits. 
\end{enumerate}

We deduce that their composition $\disc[-]$ preserves filtered colimits. 
\end{proof}

\subsection{Virtual orientability of \texorpdfstring{$\Sp_{T(n)}$}{Sp T(n)}}\label{ssec:ortelescopic}

We now turn to the application of the theory of orientations and the Fourier transform to the $T(n)$-local setting. In particular, we construct a $T(n)$-local lift of the chromatic Fourier transform (\cref{Tn_orientation_Intro}) and prove the corresponding $T(n)$-local affineness, Eilenberg--Moore and Galois results (\cref{Tn_Applications_Intro} and \Cref{Higher_Kummer_Intro}).

\subsubsection{Virtual $\FF_p$-orientability \& applications}

In \cite[\S 5]{AmbiKn}, the $(\ZZ_{(p)},n)$-orientability of $E_n$ was used to prove affineness results for $\Sp_{K(n)}$, as well as Eilenberg--Moore type formulas for the cohomology of $\pi$-finite spaces with $K(n)$-local coefficients. From the Fourier-theoretic perspective, this is a formal consequence of $\Sp_{K(n)}$ being virtually $(\FF_p,n)$-orientable. Similarly, the next proposition allows us to lift all of these results to $\Sp_{T(n)}$.

\begin{prop}\label{Tn_Virt_Or}
For all $n \ge 0$, the $\infty$-category $\Sp_{T(n)}$ is virtually $(\FF_p,n)$-orientable. 
\end{prop}

\begin{proof}
    The $T(n)$-local commutative ring spectrum $E_n$ is $(\ZZ_{(p)},n)$-orientable by \Cref{LT_Zp_Or} and hence in particular $(\FF_p,n)$-orientable by \Cref{Or_Push}. Since the functor 
    \[
        E_n\otimes(-)\colon 
        \Sp_{T(n)} \too 
        \LocMod_{E_n}
    \] 
    is nil-conservative (see, e.g., \cite[Corollary 5.1.17]{TeleAmbi}\footnote{The height $0$ case, while not covered by the referred corollary, follows easily from the fact that $\overline{\QQ}$ is a retract of $E_0$ in $\Sp_\QQ$.}), we deduce from \Cref{nil_cons_virt_F_p} that $\Sp_{T(n)}$ is virtually $(\FF_p,n)$-orientable.
\end{proof}

\begin{rem}
    The virtual $(\FF_p,n)$-orientability of $\Sp_{T(n)}$ depends crucially on the fact that it is $\infty$-semiadditive of semiadditive height $n$ (\cite[Theorem 4.4.5]{AmbiHeight}). 
    However, those properties alone do not suffice to guarantee that a stable $\infty$-category is virtually $(\FF_p,n)$-orientable. Indeed, for $n=1$, the universal example of such an $\infty$-category is $\tsadi_1$ (constructed in \cite[Theorem 5.3.6]{AmbiHeight}). In \cite{Yuan2022TheSO}, Yuan constructs a commutative algebra $\Sph_p^{\tsadi_1}\in \calg(\tsadi_1)$, whose $p$-th cyclotomic extension is not Galois (\cite[Proposition 3.9]{Yuan2022TheSO}). By \Cref{Cyc_Galois}, this implies that $\tsadi_1$ is \emph{not} virtually $(\FF_p,1)$-orientable.  
\end{rem}

The virtual $(\FF_p,n)$-orientability of $\Sp_{T(n)}$ implies the following affineness result:

\begin{thm}
\label{Tn_Affineness}
    Let $n\ge 0$, and let $A$ be a $\pi$-finite space for which $\pi_1(A,a)$ is a $p$-group and $\pi_{n+1}(A,a)$ is of order prime to $p$, for every $a\in A$. Then, $A$ is $\Sp_{T(n)}$-affine.
\end{thm}

Recall that by \Cref{criterion_affineness}, this implies that for every  $R\in \alg(\Sp_{T(n)})^A$, the global sections functor induces an isomorphism
\[
    \Mod_R(\Sp_{T(n)}^A) \iso \Mod_{A_* R}(\Sp_{T(n)}).
\]

\begin{proof}
By \Cref{Tn_Virt_Or}, $\Sp_{T(n)}$ is virtually $(\FF_p,n)$-orientable, and as it is also stable and $p$-local, the $\Sp_{T(n)}$-affineness of $A$ follows from \Cref{p_Affineness} and \Cref{rem:non_p_spaces_affine}. 
\end{proof}

The affineness of the spaces in \Cref{Tn_Affineness} gives in turn corresponding Eilenberg--Moore type results:

\begin{cor}
\label{Tn_EM}
    Let $A$ be as in \Cref{Tn_Affineness} (e.g., an $n$-finite $p$-space) and let $R\in\alg(\Sp_{T(n)})$. For every $\pi$-finite space $B$ and an arbitrary space $B'$ mapping to $A$, we have
    \[
        R^B \otimes_{R^A} R^{B'} \simeq 
        R^{B\times_A B'} \qin \Sp_{T(n)}
    \]
\end{cor}
\begin{proof}
    The space $A$ and the map $B \to A$ are $\Sp_{T(n)}$-ambidextrous by the $\infty$-semiadditivity of $\Sp_{T(n)}$ (\cite[Theorem A]{TeleAmbi}), and $A$ is  $\Sp_{T(n)}$-affine, by \Cref{Tn_Affineness}. Thus, the claim follows from \Cref{Affiness_Eilenberg_Moore}.
\end{proof}

We also obtain the following result on the ubiquity of $T(n)$-local Galois extensions:

\begin{cor}\label{Tn_Galois}
    Let $A$ be as in \Cref{Tn_Affineness} (e.g., an $n$-finite $p$-space). Every $R \in \calg(\Sp_{T(n)})^A$ is an $A$-Galois extension of $A_*R$, in the sense of \Cref{def:Galois}. 
\end{cor}

In particular, for $n \ge 1$ and $G$ a finite $p$-group, every $R \in \calg(\Sp_{T(n)})^{BG}$ is a $G$-Galois extension of its fixed point algebra $R^{hG}$. 

\begin{proof}
    By \Cref{Tn_Affineness}, $A$ is $\Sp_{T(n)}$-affine. Hence, the claim follows from \Cref{Galois_Auto_Rel}.
\end{proof}

\begin{rem}
    The $K(n)$-local analogues of the above results follow easily, either  using the colimit preserving symmetric monoidal functor $L_{K(n)}\colon \Sp_{T(n)}\to \Sp_{K(n)}$ or by an identical argument starting from the virtual $(\FF_p,n)$-orientability of $\Sp_{K(n)}$. In particular, we recover the affineness result of \cite[Theorem 5.4.3]{AmbiKn} and the Eilenberg--Moore type result of \cite[Theorem 5.4.8]{AmbiKn}. 
\end{rem}

The virtual $\FF_p$-orientability of $\Sp_{T(n)}$, given by \Cref{Tn_Virt_Or}, bootstraps automatically to virtual $\OR$-orientability for every connective $n$-truncated $\pi$-finite $p$-local commutative ring spectrum $\OR$, by \Cref{Virt_Fp_Bootstrap}. Namely, the corresponding $\OR$-cyclotomic extensions, over which we have a Fourier transform isomorphism, are \textit{faithful}. Specializing to $\OR = \ZZ/p^r$, these are precisely the $T(n)$-local higher cyclotomic extensions $R_{n,r}^f$ constructed in \cite{carmeli2021chromatic}. We thus obtain Fourier isomorphisms over these extensions. 

\begin{thm} \label{Tn_Fourier}
    For every $n\ge 0$ and $r\ge 1$, there is a natural isomorphism of $T(n)$-local commutative $R^f_{n,r}$-algebras
    \[
        \Four_{\omega \colon 
        }R_{n,r}^f[M] \iso 
        (R_{n,r}^f)^{\und{\Sigma^n M^*}},
    \]
    where $M$ is a connective $\pi$-finite $\ZZ/p^r$-module and
    $M^*$ is its Pontryagin dual.
\end{thm}

\begin{proof}
    By \Cref{Cyclo_Univ_Or}, $R_{n,r}^f$ is the universal $(\ZZ/p^r,n)$-oriented commutative algebra in $\Sp_{T(n)}$.
\end{proof}


Finally, we also obtain a higher chromatic height analogue of Kummer theory. 

\begin{thm} 
    \label{Tn_Higher_Kummer}
    For every $n\ge 0$, every $R\in \calg(\Sp_{T(n)})$ admitting a primitive higher $p^r$-th root of unity, and every $M \in \Modfin{\ZZ/p^r}{n}$, there is a natural isomorphism of spaces
    \[
        \calg^{\und{M}-\gal}(R;\Sp_{T(n)}) \:\simeq \:
        \Map_{\Sp^\cn}(\Dual{M}{n},R^\times).
    \]
\end{thm}

\begin{proof}
    This is a special case of \Cref{Kummer}.
\end{proof}

Classically, Kummer theory is used to classify abelian Galois extensions. Applying \Cref{Tn_Higher_Kummer} to the case of finite abelian $p$-groups, we get a similar classification in arbitrary chromatic heights.  
\begin{cor}
    For every $n\ge 0$, every $R\in \calg(\Sp_{T(n)})$ admitting a primitive higher $p^r$-th root of unity, and every finite abelian $p^r$-torsion group $M$, there is a natural isomorphism
       \[
        \calg^{BM-\gal}(R;\Sp_{T(n)}) \:\simeq \:
        \Map_{\Sp^\cn}(\Sigma^nM^*,\pic(R)).
    \]
\end{cor}

\begin{proof}
The case $n=0$ follows from \cite[Theorem 3.18]{carmeli2021chromatic}. 
For $n\ge 1$ we have 
\[
    \Map(\Sigma^{n}M^*,\pic(R))\simeq
    \Map(\Sigma^{n-1} M^*,\Omega \pic(R)) \simeq 
\]
\[
    \Map(\Sigma^{n-1} M^*,R^\times)\simeq
    \Map(\Dual{\Sigma M}{n}, R^\times).
\]
By \Cref{Tn_Higher_Kummer}, the last space identifies with 

\[
\calg^{\und{\Sigma M}-\gal}(R;\Sp_{T(n)}) =
\calg^{BM-\gal}(R;\Sp_{T(n)}). \qedhere
\]
\end{proof}

\subsubsection{Virtual $\ZZ_{(p)}$-orientability \& speculations}

By \Cref{Zp_Univ_Cyclo}, the infinite cyclotomic extension $R_n^f$ is the universal $(\ZZ_{(p)},n)$-orientable $T(n)$-local commutative ring spectrum, so in particular it supports a Fourier transform for all $\pi$-finite $p$-local $\ZZ$-modules. Note that by universality, this lifts the $K(n)$-local Fourier transform over $R_n$ from \Cref{R_n_fourier}. However, in contrast with $R_n$, we do not know whether $R_n^f$ is \textit{faithful} (even though all the $R_{n,r}^f$-s are). This question can be re-formulated in a way that might shed some light on the relationship between $\Sp_{T(n)}$ and $\Sp_{K(n)}$. 

By \Cref{initial_virtually_R_oriented}, the Bousfield localization of $\Sp_{T(n)}$ with respect to $R_n^f$ is the universal virtually $\ZZ_{(p)}$-orientable symmetric monoidal localization of $\Sp_{T(n)}$. \Cref{Z_p_virt_smashing} tells us that this localization $\widehat{\Sp}_{T(n)} := (\Sp_{T(n)})_{R_n^f}$ is \textit{smashing} and that its unit is given by $(R_n^f)^{hG} \in \calg(\Sp_{T(n)})$, where
\[
    G:= T_p \times \ZZ \:\sseq\: 
    T_p \times \ZZ_p \:\simeq\: 
    \ZZ_p^\times.
\]
We also observe that since, essentially by construction, $\widehat{\Sp}_{T(n)}$ is virtually $\ZZ_{(p)}$-orientable, it is in fact virtually $\Sph_{(p)}$-orientable, by \Cref{Virt_Zp_Bootstrap}. 
Now,
Since $\Sp_{K(n)}$ is a virtually $\ZZ_{(p)}$-orientable localization of $\Sp_{T(n)}$, we obtain a chain of fully faithful embeddings
\[
    \Sp_{K(n)} \:\sseq\: \widehat{\Sp}_{T(n)} \:\sseq\: \Sp_{T(n)}.
\] 

The gap between $\Sp_{K(n)}$ and $\Sp_{T(n)}$ is the subject of 
Ravenel's celebrated \emph{telescope conjecture}, which would imply that all the above inclusions are in fact equalities. However, this conjecture is not only open, but also believed by many experts to be \emph{false} for heights greater than $1$. It is also not known whether there can be any Bousfield localization strictly in between $\Sp_{K(n)}$ and $\Sp_{T(n)}$. In this light, we propose the following:

\begin{ques} 
    What can be said about the location of the intermediate localization $\widehat{\Sp}_{T(n)}$? In particular, is $\widehat{\Sp}_{T(n)} = \Sp_{K(n)}$? Is $\widehat{\Sp}_{T(n)} = \Sp_{T(n)}$?
\end{ques}
  
On the one hand, $\widehat{\Sp}_{T(n)} = \Sp_{T(n)}$ if and only if $R_n^f$ is faithful in $\Sp_{T(n)}$. On the other, $\widehat{\Sp}_{T(n)} = \Sp_{K(n)}$ if and only if $R_n^f$ is itself $K(n)$-local, namely $R_n^f \simeq L_{K(n)}R_n^f \simeq R_n$.  Thus, the failure of the telescope conjecture is equivalent to at least one of these assertions being false, while the failure of both would produce a strictly intermediate localization.

\bibliographystyle{alpha}
\phantomsection\addcontentsline{toc}{section}{\refname}
\bibliography{four}

\end{document}